\documentclass[final,1p,times]{elsarticle}
\usepackage{amssymb}
 \usepackage{amsthm}
 \usepackage{caption}
\usepackage{amsmath,amssymb,amsopn,amsfonts,mathrsfs,amsbsy,amscd}
\usepackage{longtable}
\usepackage{multirow}
\usepackage[latin1]{inputenc}
\newcommand{\R}{\mathbb{R}}

\newcommand{\G}{{\mathfrak{g}}}

\newcommand{\h}{{\mathfrak{h}}}

\newcommand{\e}{\check{e}}
\newtheorem{Def}{Definition}
\newtheorem{theo}{Theorem}
\newtheorem{pr}{Proposition}
\newtheorem{Le}{Lemma}
\newtheorem{co}{Corollary}

\newtheorem{remark}{Remark}

\usepackage{adjustbox}
\usepackage{longtable}
\usepackage{geometry}
\begin{document}
\begin{frontmatter}
\title{Three-Dimensional Real Affine Lie Groups}
\author[]{T. A\"it Aissa and S. El Bourkadi and M. W. Mansouri  }
\address{Department of Mathematics, Faculty of Sciences, Ibn Tofail University\\
	Analysis, Geometry and Applications Laboratory $($LAGA$)$\\ Kenitra, Morocco\\e-mail:
	tarik.aitaissa @uit.ac.ma\\
	said.elbourkadi @uit.ac.ma\\	mansourimohammed.wadia@uit.ac.ma}
\begin{abstract}
We classify all left-invariant real affine connections in dimension three. Our approach reduces the three-dimensional problem to a two-dimensional one by decomposing each left-invariant affine connection into a two-dimensional part and an additional one-dimensional component. 
After characterizing all possible two-dimensional left-invariant affine connections, we return to the three-dimensional setting to obtain a simplified description of all three-dimensional left-invariant affine connections. We then explicitly solve the resulting simplified quadratic equations and perform a refined analysis up to isomorphism, leading to a complete classification. Furthermore, we determine several geometric and algebraic properties of these structures, including the Novikov, associative, radiant, and bi-symmetric conditions, as well as geodesic completeness.
\end{abstract}

\begin{keyword}
 Affine Lie groups, Flat Lie algebras, Flat torsion-free connection.\\
        \MSC 22E25, 17B30, 53B05. 
\end{keyword}
        \end{frontmatter}
\section{Introduction}
An \textit{affine structure} on a   manifold $M$ is given by a maximal atlas of charts with values in $\R^n$ whose transition maps are locally affine transformations of $\R^n$. It is equivalent to endow $M$ with a flat and torsion-free connection $\nabla$. It is also equivalent to give a local diffeomorphism $D$ from the universal cover $\widetilde{M}$ of $M$ into $\R^n$, called the \textit{developing map}, such that there exists a representation $h$ of the fundamental group $\pi_1(M)$ into the group $\mathrm{Aff}(\R^n)$ of affine transformations of $\R^n$, called the \textit{holonomy}, satisfying, for all $X$ in $\widetilde{M}$ and $\gamma$ in $\pi_1(M)$,
\[
D(\gamma \cdot X) = h(\gamma)D(X).
\]
If $g$ is in $\mathrm{Aff}(\mathbb{R}^n)$, the developing maps $D$ and $g \circ D$ are considered equivalent. The simplest examples are the tori $\mathbb{T}^n$, whose universal covering space is the affine space $\mathbb{R}^n$. If $D$ is a diffeomorphism, i.e., if $\tilde{M}$ is affinely diffeomorph to $\mathbb{R}^n$, then $M$ is called complete.

A manifold with an affine structure naturally induces a differentiable structure. Additionally, affine manifolds are characterized by flat geometric structures. More precisely, There is a one-to-one correspondence between the affine structures and the flat torsion-free affine connections $\nabla$ on the manifold $M$. Such a connection is a linear connection on the tangent 
bundle $TM$ whose torsion tensor and curvature tensor vanish identically. 
That is, for all vector fields $X,Y,Z \in \mathfrak{X}(M)$,
\begin{equation}\label{Torsionfree}
	T^\nabla(X,Y)
	=
	\nabla_XY-\nabla_YX-[X,Y]
	=
	0,
\end{equation}
and
\begin{equation}\label{Flatness}
	R^\nabla(X,Y)Z
	=
	\nabla_X\nabla_YZ
	-
	\nabla_Y\nabla_XZ
	-
	\nabla_{[X,Y]}Z
	=
	0.
\end{equation}
Riemannian-flat and Lorentz-flat manifolds are subclasses of affine manifolds. The existence of affine structures is a fundamental problem. An affine structure in dimension two exhibits a fundamental topological property.  A closed surface admits an affine structure if and only if its Euler characteristic vanishes \cite{Milnor}.

Lie group structures with left-invariant affine structures provide many examples of affine manifolds. An affine structure on a Lie group $G$ is left-invariant if the left-multiplication by $g$, $\mathrm{L}_g:G\to G$, is an automorphism of the affine structure. 

Suppose that $G$ admits a left-invariant flat, torsion-free connection $\nabla$ on $G$. It follows that for any two left-invariant vector fields $X, Y\in\G$, their covariant derivative $\nabla_XY\in\G$ is left-invariant as well.
 Therefore, the covariant derivative defines a bilinear product on $\G$, given by
 \[
 XY:=\nabla_XY,
 \]
 for all $X,Y\in\G$. Keeping it simple, we will write  $\nabla_XY$ as $XY$ or $X\cdot Y$. Due to local flatness and torsion-free properties of $\nabla$, we have by \eqref{Torsionfree} and \eqref{Flatness}:
 \begin{align}
 	[X,Y]&=X\cdot Y-Y\cdot X,\label{admi}\\
 	[X,Y]\cdot Z&=X\cdot(Y\cdot Z)-Y\cdot(X\cdot Z).\label{repre}
 \end{align}
We can rewrite~\eqref{repre} using~\eqref{admi} as
\[
(X,Y,Z)=(Y,X,Z),
\]
where $(X,Y,Z)$ denotes the associator of the three elements $X$, $Y$, and $Z$ in $\G$. Thus, $(\G,\cdot)$ is a left-symmetric algebra (or simply an LSA), see~\cite{Bu2,Seg}.  There are many names for left-symmetric algebras. Left-symmetric algebras are also known as flat Lie algebras, Vinberg algebras, Koszul algebras, or quasiassociative algebras. We refer the reader to the survey article \cite{Bu1} and its references for more information on left-symmetric algebras. This context raises the following important question, also posed by Milnor \cite{Milnor2} in his studies of fundamental groups of complete affine manifolds: Which Lie groups (Lie algebras) admit left-invariant affine structures?  
There is a special difficulty in this question for nilpotent Lie groups. Several evidences suggest that nilpotent Lie groups admit left-invariant affine structures \cite{Bu3}. In fact, Milnor postulated that this holds even for solvable Lie groups \cite{Milnor2}. Recently, however, counterexamples were discovered \cite{Ben}. Several nilmanifolds are not affine.   When the Lie group $G$ is semisimple (also perfect), it does not admit left-invariant affine structures \cite{Hel}. As a concrete example of affine manifolds, symplectic Lie groups can be regarded as a fundamental class of affine manifolds. Indeed, it is well known that every symplectic Lie group $(G,\omega)$ carries a torsion-free flat connection $\nabla^{\omega}$ on $G$, naturally associated with the symplectic structure (see~\cite[Theorem~6]{Chu}).
In addition, cosymplectic Lie groups (or Lie algebras) constitute another class of affine manifolds; see~\cite{EM}.

Another problem related to affine structures is their classification. We cite here several classification results in low dimensions for left-symmetric algebras (or flat Lie algebras): the two-dimensional case~\cite{And}; three-dimensional left-symmetric algebras over the field $\mathbb{C}$ of complex numbers~\cite{Bai1}; the abelian case~\cite{Goz}; four-dimensional reductive Lie groups over $\mathbb{R}$ and $\mathbb{C}$~\cite{AM3,Bu0}; Complete left-invariant affine structures on nilpotent Lie groups up to dimension~4~\cite{Kim}; and abelian left-symmetric algebras up to dimension~$5$ \cite{Deki}.
\\\\
The paper follows the following structure. In Section~\ref{Sec2}, we introduce important definitions and notions in the context of flat Lie algebras and  their basic properties. Section~\ref{Sec3} is devoted to flat Lie  groups (algebras) and the description of flat, torsion-free connections on this class of Lie algebras. We also provide several characterizations of isomorphisms between flat, torsion-free connections. In Section~\ref{Sec4}, we begin by classifying all torsion-free connections in dimension~2 and then proceed to the classification of three-dimensional real flat, torsion-free connections.  The appendix~\ref{App} contains detailed proofs of the classification of flat, torsion-free connections on abelian Lie algebra.
\\\\
\textbf{Notation and conventions.} Unless otherwise stated, we work over a fixed field $\mathbb{K}$ of characteristic zero for our results on affine Lie groups or flat Lie algebras. A geometric interpretation of simply connected Lie groups over the field $\mathbb{K}=\R$ of real numbers, is a natural one, which is therefore of principal interest in our investigations. Throughout this paper, the notations $(\G,\nabla)$ and $(\G,\cdot)$ will be used interchangeably to denote the flat Lie algebras, where $\nabla$ is flat, torsion-free connection on $\G$, and $``\cdot"$ is defined by $X\cdot Y=\nabla_XY$, for all $X,Y\in\G$.

\section{Flat Lie algebras}\label{Sec2}

First, we define flat Lie algebras, as a first definition, we begin with the notion of linear connection. Let $G$ be a connected and simply connected finite-dimensional  Lie group with Lie algebra $\G$.

A \emph{connection} on a Lie algebra $\G$ is a bilinear map
\[
\nabla:\G\times\G\to\G,
\qquad
(X,Y)\mapsto \nabla_X Y.
\]
For all $X,Y,Z\in\G$, the \emph{torsion tensor} $T=T^\nabla$ is defined by
\[
T(X,Y)=\nabla_X Y-\nabla_Y X-[X,Y],
\]
and the \emph{curvature tensor} $R=R^\nabla$ by
\[
R(X,Y)Z
=\nabla_X(\nabla_Y Z)
-\nabla_Y(\nabla_X Z)
-\nabla_{[X,Y]}Z.
\]
A connection $\nabla$ is said to be \emph{torsion-free} if its torsion vanishes, that is,
\[
T=0.
\]
The curvature tensor $R$ vanishes if and only if the map
\[
\tau^\nabla : \G \to \mathrm{End}(\G),
\qquad
X \mapsto \nabla_X,
\]
defines a representation of $\G$ on itself. In this case, the connection
$\nabla$ is said to be \emph{flat}.
\begin{Def}
	A flat Lie algebra is a Lie algebra $\G$ equipped with a flat, torsion-free connection $\nabla$. The pair $(\G,\nabla)$ is then called a flat Lie algebra.
\end{Def}

We recall the following fundamental definition.

\begin{Def}
	Let $\nabla^1$ and $\nabla^2$ be two flat, torsion-free connections on the Lie algebra $\G$. Then, $\nabla^1$ and $\nabla^2$ are isomorphic if and only if there exists a Lie algebra automorphism $\Psi:\G\to \G$ such that
	\begin{equation}
		\nabla^2_X=\Psi\circ\nabla^1_{\Psi^{-1}(X)}\circ\Psi^{-1},\quad\text{for all}~~X\in\G.
	\end{equation}
	
\end{Def}

Setting $
\nabla_X Y=X\cdot Y$, we define the left and right multiplication operators associated with the product $``\cdot"$ by
\[
\mathrm{L}_X(Y)=X\cdot Y=\nabla_X Y,
\qquad
\mathrm{R}_X(Y)=Y\cdot
 X=\nabla_Y X,
\]
for all $X,Y\in\G$.  
Moreover, the left-multiplication operator
\[
\mathrm{L}:\G\longrightarrow \mathrm{End}(\G),\qquad X\mapsto \mathrm{L}_X,
\]
defines a representation of the sub-adjacent Lie algebra of the flat Lie algebra.

There is a one-to-one correspondence between left-invariant affine structures
on $G$ and flat structures $``\cdot"$ (LSA-structures)  on $\G$ \cite{Seg}.

There are some very important subclasses of flat Lie algebras (Left-symmetric algebras):

\begin{Def}
Let $(\G,\cdot)$ be a flat Lie algebra. 
\begin{enumerate}
	\item	If, for every $X\in \G$, the operator $\mathrm{R}_X$ is nilpotent, then $\G$ is said to be transitive or complete. Transitivity corresponds to the completeness of an affine manifold. 
	\item  If, for every $X,Y\in \G$, $\mathrm{R}_X\circ \mathrm{R}_Y=\mathrm{R}_Y\circ \mathrm{R}_X$, then $\G$ is called a Novikov algebra.
	\item  
	If, for every $X,Y,Z\in \G$, the associator $(X,Y,Z)$ is right-symmetric, that is,
	\[
	(X,Y,Z)=(X,Z,Y),
	\]
	then $\G$ is said to be bi-symmetric. 
	\item If there exists a vector  $\zeta\in \G$   such that
	\[
	\nabla_X \zeta=X,
	\]
	for every vector  $X\in\G$, then $(\G,\cdot)$ is called a radiant flat Lie algebra.
\end{enumerate}
	
\end{Def}

\begin{remark}
	\begin{enumerate}
		\item 	The sub-adjacent Lie algebra of a complete left-symmetric algebra is solvable, $($see~\textsc{\cite{Kim,Medina1})}.
		\item  The notion of Novikov algebra was introduced in connection with Poisson brackets of hydrodynamic type and Hamiltonian operators arising in formal variational calculus, $($see~\textsc{\cite{Balin})}.
		\item  The notion of bi-symmetric coincides with that of an assosymmetric ring in the study of nearly associative algebras, $($see, \textsc{\cite{Kle}, \cite{Bai3})}.
		\item The notion of radiant manifolds was introduced in~\textsc{\cite{Gold1,Gold2}}. Its existence is related to a cohomological obstruction to the existence of a fixed point for the holonomy representation. The classification of radiant flat Lie algebras up to dimension 3 was established in \textsc{\cite{AM1}}.

		\end{enumerate}
\end{remark}

\section{Real Flat Solvable Lie Algebras}\label{Sec3}

Let $\G =   \langle \ell\rangle\ltimes \G_0$ be the semidirect sum of a Lie algebra $\G_0$ and a one-dimensional vector space $\langle \ell \rangle$. We define an alternating bilinear map $[\cdot, \cdot] : \G \times \mathfrak{g} \to \G$ by specifying that the only non-zero brackets are given by
\begin{align}
[x,y]_\G&=[x,y]_{\G_0},\quad\quad\text{for all}~~x,y\in\G_{0},\label{Bra1}\\
[\ell,x]_\G&=\mathrm{D}(x),\quad\quad\text{for all}~~x\in\G_{0},\label{Bra2}
\end{align}
where, $\mathrm{D}\in\mathrm{End}(\G_{0})$ is an endomorphism of $\G_0$. 

\begin{Le}\label{Derivation}
The alternating product $[\cdot,\cdot ]$ as declared in $(\ref{Bra1})$ - $(\ref{Bra2})$
above defines a Lie algebra $\G = (\G, [\cdot,\cdot ])$ if and only if $\mathrm{D}$ is a derivation on $\G.$
\end{Le}
\begin{proof}
Let $\G = \langle \ell\rangle\ltimes \G_0$ be the semidirect sum of a Lie algebra $\G_0$ and a one-dimensional vector space $\langle \ell \rangle$. Then
it is easily verified that the  condition that $\mathrm{D}$ is a derivation on $\G_0$, is equivalent to the Jacobi-identity for the  alternating product $[\cdot,\cdot ]$ as declared in $(\ref{Bra1})$ - $(\ref{Bra2})$.
\end{proof}

Conversely, let $\G$ be a solvable Lie algebra. Then $\G \neq [\G,\G]$, and hence there exists an element $\ell \in \G$ such that $\ell \notin [\G,\G]$. Therefore, $\G$ can be viewed as a semidirect sum of $\langle \ell \rangle$ and an ideal $\G_0 \subset \G$. The Lie bracket structure on $\G$ is then given by \eqref{Bra1} and \eqref{Bra2}, under the conditions stated in Lemma~\ref{Derivation}.

Let $\nabla$ be a torsion-free connection on $\G$. Then it is easy to see that $\nabla$ has the following form:
\begin{align}\label{Connetorsionfre}
\begin{split}
\nabla_xy&=\nabla^0_xy+\theta(x,y)\ell,\\
\nabla_x\ell&=\beta(x)+\gamma(x)\ell,\\
\nabla_\ell x&=\eta(x)+\gamma(x)\ell,\\
\nabla_\ell\ell&=\zeta+\lambda \ell,
\end{split}
\end{align}

for all $x, y \in \G_0$, where $\zeta\in\G_0$,  $\theta \in \mathcal{S}^2(\G_0)$ is a symmetric form, $\beta, \eta: \G_0 \to \G_0$ are endomorphisms of $\G_0$, $\gamma: \G_0 \to \R$ is a one-form, and $\nabla^0$ is a torsion-free connection on $\G_0$. Moreover, since $\nabla$ is torsion-free, we have $\mathrm{D} = \eta-\beta $.

From now on, and for the remainder of this paper, we consider a solvable Lie algebra $\G$ viewed as a semidirect sum of a Lie algebra $\G_0$ and a one-dimensional subalgebra $\langle \ell \rangle$, and its Lie brackets are described by $(\ref{Bra1})$ and $(\ref{Bra2})$.
\begin{Le}
Let $(\G, \nabla)$ be a solvable Lie algebra endowed with the torsion-free connection given in $(\ref{Connetorsionfre})$. Then $\nabla$ has vanishing curvature if and only if the following conditions hold$:$
\begin{enumerate}
\item[$(i)$] $\mathcal{K}^{\nabla^0}(x,y,z)=\theta(x,z)\beta(y)- \theta(y,z)\beta(x)$,
\item[$(ii)$] $\theta\big(x,\nabla^0_yz\big)-\theta\big(y,\nabla^0_xz\big)=\theta(x,z)\gamma(y)-\theta(y,z)\gamma(x)$,
\item[$(iii)$] $\nabla^0_x\zeta-\eta\big(\beta(x)\big)+\gamma(x)\zeta=-\lambda \beta(x)-\beta\big(\mathrm{D}(x)\big)$,
\item[$(iv)$] $\theta(x,\zeta)=\gamma\big(\beta(x)-\mathrm{D}(x)\big)$,
\item[$(v)$] $\nabla^0_{\mathrm{D}(x)}y-\eta\big(\nabla^0_xy\big)+\nabla^0_x\eta(y)=\theta(x,y)\zeta-\gamma(y)\beta(x)$,
\item[$(vi)$] $\theta\big(\mathrm{D}(x),y\big)-\gamma\big(\nabla^0_xy\big)=\lambda\theta(x,y)-\gamma(x)\gamma(y)-\theta\big(x,\eta(y)\big)$,
\end{enumerate}
for all $x, y, z \in \G_0$.
\end{Le}
\begin{proof}
Let $\mathfrak{g}$ be a solvable Lie algebra and $\nabla$ a torsion-free connection on $\mathfrak{g}$ as described in $(\ref{Connetorsionfre})$. We compute the curvature associated with $\nabla$, i.e.,
\begin{align}\label{Curvature}
\mathcal{K}^\nabla(x,y)z&=\nabla_x\nabla_yz-\nabla_y\nabla_xz-\nabla_{[x,y]_\G}z,\quad\quad x,y,z\in\G.
\end{align}
A closer inspection of Equation $(\ref{Curvature})$ reveals: The equation has two components, yielding two distinct equations. This situation occurs in each of the four cases we must consider. The first case occurs when both $x$, $y$ and $z$ are vectors in $\G_0$, which yields equations that are readily identified as $(i)$ and $(ii)$.  The next case occurs when $x$ and $y$ are vectors in $\G_0$ and $z$ is a vector in $\R\ell$. The first component leads to:
\begin{align}\label{C3}
\nabla^0_x\beta(y) -\nabla^0_y\beta(x)-\beta\big([x,y]_{\G_0}\big)&=\gamma(x)\beta(y)-\gamma(y)\beta(x),
\end{align}
while the second component yields:
\begin{align}\label{C4}
\theta\big(x, \beta(y)\big)- \theta(y, \beta(x)) =\gamma\big([x, y]_{\G_0}\big).
\end{align}
The third case occurs when $x \in \G_0$ and $y, z \in \langle \ell \rangle$. The first component of the resulting equation is:
\begin{align}\label{C5}
\theta(x, \zeta)  - \eta\big(\beta(x)\big) + \gamma(x)\zeta=- \lambda\beta(x)-\beta\big(\mathrm{D}(x)\big),
\end{align}
while the second component yields:
\begin{align}\label{C6}
\theta(x,\zeta)=\gamma\big(\beta(x)-\mathrm{D}(x)\big).
\end{align}

The final case occurs when $x,z \in \G_0$ and $y \in \langle \ell \rangle$. The first component of the resulting equation yields:
\begin{align}\label{C7}
 \nabla^0_{\mathrm{D}(x)}z-\eta\big(\nabla^0_xz\big)+\nabla^0_x\eta(z)&=\theta(x,z)\zeta-\gamma(z)\beta(x),
\end{align}
and the second component gives:
\begin{align}\label{C8}
\theta\big(\mathrm{D}(x),z\big)-\gamma\big(\nabla^0_xz\big)&=
\lambda\theta(x,z)-\gamma(x)\gamma(z)-\theta\big(x,\eta(z)\big).
\end{align}
Finally, it is easily verified that condition $(\ref{C8})$ implies condition $(\ref{C6})$. Moreover, a straightforward computation shows that $(\ref{C7})$ implies $(\ref{C3})$.
\end{proof}

\begin{Le}\label{Auto}
Let $\G = \mathbb{R}\ell \ltimes \G_0$ be a solvable Lie algebra, and let $\Psi \in \mathrm{Aut}(\G)$ be an automorphism of $\G$. Then $\Psi$ has the following form$:$
\begin{align}\label{Autg}
\begin{split}
\Psi(x)&=\Phi(x)+\sigma(x)\ell,\quad\quad\text{for all } x\in\G_0,\\
\Psi(\ell)&=u+\tau \ell,\quad\quad\quad\quad~ u\in \G_0,
~~\tau\in\R,
\end{split}
\end{align}
where, 
\begin{enumerate}
\item[$(c_1)$] $\Phi\in\mathrm{Aut}(\G_0)$,
\item[$(c_2)$] $\sigma\in Z^1(\G_0)$,
\item[$(c_3)$] $\mathrm{Im}(\mathrm{D})\subset\ker\sigma$,
\item[$(c_4)$] $\Phi\circ\mathrm{D}-\tau\mathrm{D}\circ\Phi=\mathrm{ad}^0_u\circ\Phi-\mathrm{D}(u)\sigma$.
\end{enumerate}

\end{Le}
\begin{proof}
Suppose that $\Psi:\G\longrightarrow\G$ is an automorphism of $\G$. Then, $\Psi$ can be expressed as:
\begin{align*}
\Psi(x)&=\Phi(x)+\sigma(x)\ell,&\text{for all } x\in\G_0,\\
\Psi(\ell)&=u+\tau \ell,&u\in \G_0,
~~\tau\in\R,
\end{align*}
where, $\Phi:\G_0\longrightarrow\G_0$ is an endomorphism of $\G_0$, and $\sigma\in\G_0^\ast$ is a $1$-form. The Lie brackets of $\G$ are given by
\begin{align}
[x,y]_\G&=[x,y]_{\G_0},&\text{for all } x,y\in\G_0,\\
[\ell,x]_\G&=\mathrm{D}(x),&\text{for all } x\in\G_0.
\end{align}
On the one hand, for all $x,y\in\G_0$, we have
\begin{align*}
\Psi\big([x,y]_\G\big)&=\Psi\big([x,y]_{\G_0}\big)\\
&=\Phi\big([x,y]_{\G_0}\big)+\sigma\big([x,y]_{\G_0}\big)\ell,
\end{align*}
and
\begin{align*}
[\Psi(x),\Psi(y)]_\G&=[\Phi(x)+\sigma(x)\ell,\Phi(y)+\sigma(y)\ell]\\
&=[\Phi(x),\Phi(y)]_{\G_0}.
\end{align*}
Thus,
\begin{align*}
    \mathrm{d}\sigma = 0, \text{ i.e., } \sigma \in Z^1(\G_0), \quad \text{and } \Phi \in \mathrm{Aut}(\G_0).
\end{align*}
On the other hand, for all $x\in\G_0$, we have
\begin{align*}
\Psi\big([\ell,x]_\G\big)&=\Psi\big(\mathrm{D}(x)\big)\\
&=\Phi\big(\mathrm{D}(x)\big)+\sigma\big(\mathrm{D}(x)\big)\ell,
\end{align*}
and
\begin{align*}
[\Psi(\ell),\Psi(y)]_\G&=[u+\tau\ell,\Phi(x)+\sigma(x)\ell]_\G\\
&=[u,\Phi(x)]_\G+\sigma(x)[u,\ell]_\G+\tau[\ell,\Phi(x)]_\G\\
&=\mathrm{ad}^0_u\circ\Phi(x)-\sigma(x)\mathrm{D}(u)+\tau\mathrm{D}\circ\Phi(x).
\end{align*}
This implies that, $\mathrm{Im}(\mathrm{D})\subset\ker\sigma$ and
\begin{align}
\Phi\circ\mathrm{D}-\tau\mathrm{D}\circ\Phi=\mathrm{ad}^0_u\circ\Phi-\mathrm{D}(u)\sigma.
\end{align}
\end{proof}

\begin{remark}\label{autodirectsum}
If the action of $\mathbb{R}\ell$ on $\mathfrak{g}_0$ is trivial, then any automorphism $\Psi$ of $\mathfrak{g}$ takes the form~$(\ref{Autg})$ with $\Phi\in\mathrm{Aut}(\G_0)$, $\sigma\in Z^1(\G_0)$ and $u$ belonging to the center $Z(\G_0)$.
\end{remark}
\begin{Le}\label{isoofconnection}
Let $\nabla^1$ and $\nabla^2$ be two flat torsion-free connections on a solvable Lie algebra $\G$. Then $\nabla^1$ and $\nabla^2$ are isomorphic if and only if there exist a Lie algebra automorphism
$\Phi: \G_0\longrightarrow\G_0$, $u\in\G_0$, $\sigma\in\G_0^\ast$, and $\tau\in\R$ such that
\begin{enumerate}
\item[$(i)$] $\nabla^{0,1}_{\Phi(x)}\Phi(y)-\Phi\big(\nabla^{0,2}_xy\big)=\theta_2(x,y)u-\sigma(x)\eta_1\big(\Phi(y)\big)-\sigma(y)\beta_1\big(\Phi(x)\big)-\sigma(x)\sigma(y)\zeta_1$,
\item[$(ii)$] $\big(\Phi^\ast\theta_1\big)(x,y)-\tau\theta_2(x,y)=\sigma\big(\nabla^{0,2}_xy\big)-\sigma(x)\gamma\big(\Phi(y)\big)-\sigma(y)\gamma\big(\Phi(x)\big)-\lambda_1\sigma(x)\sigma(y)$,
\item[$(iii)$] $\nabla^{0,1}_{\Phi(x)}u+\tau\beta_1\big(\Phi(x)\big)+\gamma(x)\eta_1(u)+\tau\sigma(x)\zeta_1=\Phi\big(\beta_2(x)\big)+\gamma(x)u$,
\item[$(iv)$] $\theta_1\big(\Phi(x),u\big)+\tau\gamma\big(\Phi(x)\big)+\sigma(x)\gamma(u)+\tau\sigma(x)\lambda_1=\sigma\big(\beta_2(x)\big)+\tau\gamma(x)$,
\item[$(v)$] $\nabla^{0,1}_uu+\tau\beta_1(u)+\tau\eta_1(u)+\tau^2\zeta_1=\Phi(\zeta_2)+\lambda_2u$,
\item[$(vi)$] $\theta_1(u,u)+2\tau\gamma(u)+\tau^2\lambda_1=\sigma(\zeta_2)+\lambda_2\tau$.
\end{enumerate}
\end{Le}
\begin{proof}
Let $\nabla$ be a flat torsion-free connection on $\mathfrak{g}$. Then $\nabla$ has the following form$:$
\begin{align}
\begin{split}
\nabla_xy&=\nabla^0_xy+\theta(x,y)\ell,\\
\nabla_x\ell&=\beta(x)+\gamma(x)\ell,\\
\nabla_\ell x&=\eta(x)+\gamma(x)\ell,\\
\nabla_\ell\ell&=\zeta+\lambda \ell,
\end{split}
\end{align}

for all $x, y \in \G_0$, where $\zeta\in\G_0$,  $\theta \in \mathcal{S}^2(\G_0)$ is a symmetric form, $\beta, \eta: \G_0 \to \G_0$ are endomorphisms such that $\mathrm{D} = \eta - \beta$, $\gamma: \G_0 \to \R$ is a one-form, and $\nabla^0$ is a torsion-free connection on $\G_0$. Let $\nabla^1$ and $\nabla^2$ be two flat torsion-free connections on $\G$, each associated with the data $(\nabla^{0,j}, \theta_j, \beta_j, \gamma, \eta_j,\lambda_j,\zeta_j)$ for $j = 1, 2$, where $\nabla^{0,j}$ is a torsion-free connection on $\G_0$. Furthermore, if $\nabla^1$ and $\nabla^2$ are isomorphic, then there exists a Lie algebra automorphism $\Psi : \mathfrak{g} \to \mathfrak{g}$ such that
\begin{align}\label{isoconne}
\nabla^2_x = \Psi \circ \nabla^1_{\Psi^{-1}(x)} \circ \Psi^{-1}, \quad \text{for all }~ x \in \G.
\end{align}

Here, $\Psi$ is characterized as in Lemma~$\ref{Auto}$. The structure of Equation $(\ref{isoconne})$ involves two components, leading to a pair of equations that must be satisfied. This situation occurs in each of the four cases under consideration. First, equation $(\ref{isoconne})$ is equivalent to the following
 \begin{align}\label{isoconne2}
\Psi\big(\nabla^2_xy\big) = \nabla^1_{\Psi(x)}\Psi(y), \quad \text{for all }~ x,y \in \G.
\end{align}  
The first case occurs when both $x$ and $y$ are vectors in $\G_0$, in which the resulting equations reduce to $1.$ and $2.$ The
next case is when $x$ is a vector in $\G_0$ and $y$ is a vector in $\langle\ell\rangle$. The first component of the resulting equation yields:
\begin{align}\label{eqflat1}
\nabla^{0,1}_{\Phi(x)}u+\tau\beta_1\big(\Phi(x)\big)+\gamma(x)\eta_1(u)+\tau\sigma(x)\zeta_1&=\Phi\big(\beta_2(x)\big)+\gamma(x)u,
\end{align}
while the second component gives:
\begin{align}\label{eqflat2}
\theta_1\big(\Phi(x),u\big)+\tau\gamma\big(\Phi(x)\big)+\sigma(x)\gamma(u)+\tau\sigma(x)\lambda_1&=\sigma\big(\beta_2(x)\big)+\tau\gamma(x).
\end{align}
The third case occurs when $x \in \langle \ell \rangle$ and $y \in \G_0$. The first component yields$:$
\begin{align}\label{eqflat3}
\nabla^{0,1}_u\Phi(y)+\tau\eta_1\big(\Phi(y)\big)+\gamma(y)\beta_1(u)+\tau\sigma(y)\zeta_1&=\Phi\big(\eta_2(y)\big)+\gamma(y)u,
\end{align}
while the second component yields:
\begin{align}\label{eqflat4}
\theta_1\big(u,\Phi(y)\big)+\tau\gamma\big(\Phi(y)\big)+\sigma(y)\gamma(u)+\tau\sigma(y)\lambda_1&=\sigma\big(\eta_2(y)\big)+\tau\gamma(y).
\end{align}
Now, equation $(\ref{eqflat4})$ is readily seen to be equivalent to equation $(\ref{eqflat2})$ since $\mathrm{Im}(\mathrm{D})\subset\ker\sigma$. Furthermore, a straightforward computation shows that $(\ref{eqflat3})$ and $(\ref{eqflat1})$ are equivalent. Making use of  condition $(c_4)$. Finally, consider the case where $x$ and $y$ are vectors in $\langle \ell \rangle$. The first component of the resulting equation yields$:$
\begin{align}
\nabla^{0,1}_uu+\tau\beta_1(u)+\tau\eta_1(u)+\tau^2\zeta_1&=\Phi(\zeta_2)+\lambda_2u,
\end{align}
while the second component gives:
\begin{align}
\theta_1(u,u)+2\tau\gamma(u)+\tau^2\lambda_1&=\sigma(\zeta_2)+\lambda_2\tau.
\end{align}
\end{proof}
\begin{remark}\label{sympleofauto}
In the classification, we can use the simplification that $u = 0$ and $\sigma = 0$. In this case, $\tau \neq 0$ and the conditions become:
\begin{enumerate}
\item $\nabla^{0,2}_x  =\Phi\circ \nabla^{0,1}_{\Phi^{-1}(x)} \circ\Phi^{-1}$,
\item $\theta_2 =\frac{1}{\tau} \Phi^\ast\theta_1$ and $\Phi^\ast\gamma = \gamma$,
\item $\beta_2 = \tau\Phi^{-1} \circ \beta_1 \circ \Phi$ and $\eta_2 = \tau\Phi^{-1} \circ \eta_1 \circ \Phi$,
\item $\zeta_2 = \tau^2\Phi^{-1}(\zeta_1)$ and $\lambda_2 = \tau\lambda_1$.
\end{enumerate}
\end{remark}

\section{Classification of flat Lie algberas in low-dimensions}\label{Sec4}

In what follows, we outline the main approach to classifying real flat Lie algebras in low dimensions. The classification of real flat Lie algebras is particularly difficult, as there is no general theory to simplify or solve the quadratic equations arising from the curvature condition of $\nabla$. This stands in contrast to the case of flat Lie algebras over the complex field $\mathbb{C}$; see \cite{Bai1},  where Lie's theorem can be applied to obtain upper-triangular representations associated to the flat torsion-free connection $\nabla$. Our approach aims to simplify these quadratic equations and solve them in a systematic manner. We begin with a torsion-free connection $\nabla^0$ on a Lie algebra $\G_0$, which we extend to a torsion-free connection $\nabla$ as described in $(\ref{Connetorsionfre})$. This approach provides information about the components of $\nabla$ relative to the basis of $\G =\R\ell  \ltimes \G_0$.

For example, if $\G$ is a three-dimensional solvable Lie algebra, then $\G$ is isomorphic to either $ \R\ell\ltimes \R^2$ or $\R\ell \ltimes\mathfrak{aff}(1,\R) $. This shows that $\G_0$ is either $\mathfrak{aff}(1,\R)$ or $\R^2$. The classification of torsion-free connections on $\mathfrak{aff}(1,\R)$ or $\R^2$ is straightforward due to their low dimensions. The remaining equations (conditions $(i),\ldots(vi))$ then become easier to solve.

The classification of torsion-free connections on $\G_0 = \mathfrak{aff}(1,\mathbb{R})$ or $\G_0 = \mathbb{R}^2$ is carried out modulo the following equivalence relation: two connections $\nabla^1$ and $\nabla^2$ are equivalent if there exists an automorphism $\Phi \in \mathrm{Aut}(\G_0)$ satisfying condition $(c_4)$ of Lemma~$\ref{Auto}$  such that
\begin{align}\label{equiconne}
\nabla^2_x = \Phi \circ \nabla^1_{\Phi^{-1}(x)} \circ \Phi^{-1}, \quad \forall x \in \mathfrak{g}_0.
\end{align}
for all $x\in\G_0$. Now, if $\mathbb{R}\ell$ acts trivially on $\G_0$ (i.e., we have a central extension), then the equivalence relation reduces to Condition~$(\ref{equiconne})$ without requiring any condition on the automorphism $\Phi \in \mathrm{Aut}(\G_0)$; see Remark~$\ref{autodirectsum}$.

Every two-dimensional Lie algebra $\G$ has a flat, torsion-free connection. The following classification is easily verified for dimension reasons based on the definition of flat, torsion-free connections. 

\begin{pr}\label{Pr 2flat}
Let $(\h,\nabla)$ be a two-dimensional  real flat Lie algebra. Then $\h$  is isomorphic to one of the following algebras:
\end{pr}
{\renewcommand*{\arraystretch}{1.5}
\captionof{table}{Flat torsion-free connection on the affine algebra of the real line $\mathfrak{aff}(1,\R)$.}
\setcounter{table}{0}
\begin{small} 
\setlength{\tabcolsep}{10pt} 
\begin{longtable}{@{}cllllllc@{}} 
			\hline
		Algebra&&&&\\
			\hline
$\mathfrak{a}_1$&$\nabla_{e_1}e_2=\lambda e_1$&$\nabla_{e_2}e_1=(\lambda-1)e_1$&$\nabla_{e_2}e_2=\lambda e_2$&$\lambda\in\R$\\
$\mathfrak{a}_2$&$\nabla_{e_2}e_1=-e_1$&$\nabla_{e_2}e_2=\mu e_2$&&$\mu\in\R^\ast$\\
$\mathfrak{a}_3$&$\nabla_{e_2}e_1=-e_1$&$\nabla_{e_2}e_2= e_1-e_2$&&\\
$\mathfrak{a}_4$&$\nabla_{e_1}e_2= e_1$&$\nabla_{e_2}e_2=e_1+ e_2$&&\\
$\mathfrak{a}_5$&$\nabla_{e_1}e_1=\varepsilon e_2$&$\nabla_{e_2}e_1=-e_1$&$\nabla_{e_2}e_2=-2e_2$&$\varepsilon=\pm1$\\\hline
\end{longtable}\label{Flataffine}
			\end{small}	
			}

{\renewcommand*{\arraystretch}{1.5}
\captionof{table}{Flat torsion-free connection on the affine algebra of the real line $2\G_1\cong\R^2$.}
\setcounter{table}{1}
\begin{small} 
\setlength{\tabcolsep}{10pt} 
\begin{longtable}{@{}cllllllc@{}} 
			\hline
		Algebra&&&&\\	
			\hline
$\mathfrak{b}_0$&$\nabla\equiv0$&&&\\
$\mathfrak{b}_1$&$\nabla_{e_1}e_1=e_1$&&&\\
$\mathfrak{b}_2$&$\nabla_{e_2}e_2=e_1$&&&\\
$\mathfrak{b}_3$&$\nabla_{e_1}e_1=e_1$&$\nabla_{e_2}e_2=e_2$&&\\
$\mathfrak{b}_4$&$\nabla_{e_1}e_1=e_1$&$\nabla_{e_1}e_2=e_2$&$\nabla_{e_2}e_1=e_2$&\\
$\mathfrak{b}_5$&$\nabla_{e_1}e_1=e_1$&$\nabla_{e_1}e_2=e_2$&$\nabla_{e_2}e_1=e_2$&$\nabla_{e_2}e_2=-e_1$
\\\hline
\end{longtable}\label{FlatR2}
			\end{small}	
			}
\subsection{Three-dimensional real flat  Lie algebras}
It is well known that there is no flat torsion-free connetion on a semisimple Lie algebra. Therefore, over the real
field $\R$, besides $3$-dimensional simple Lie algebras $\mathfrak{sl}_2(\R)$ and $\mathfrak{so}(3)$, up to isomorphisms, there
are the following (non-isomorphic) Lie algebras (Mubarakzyanov's classification, we only give the nonzero products):
\begin{enumerate}
\item[] $ 3\G_1$, abelian.
\item[] $3\G_{2,1}\oplus\G_1=\langle e_1,e_2,e_3~|~[e_1,e_2]=e_1\rangle$, decomposable solvable.
\item[] $\G_{3,1}=\langle e_1,e_2,e_3~|~[e_1,e_2]=e_3\rangle$, Heisenberg-Weyl algebra, nilpotent.
\item[] $\G_{3,2}=\langle e_1,e_2,e_3~|~[e_1,e_3]=e_1,~[e_2,e_3]=e_1+e_2\rangle$, solvable.
\item[] $\G_{3,3}=\langle e_1,e_2,e_3~|~[e_1,e_3]=e_1,~[e_2,e_3]=e_2\rangle$, solvable.
\item[] $\G_{3,4}=\langle e_1,e_2,e_3~|~[e_1,e_3]=e_1,~[e_2,e_3]=\alpha e_2, -1\leq\alpha<1,\alpha\neq 0\rangle$, solvable, Poincar\'e\\
\phantom{xxxx} algebra $\mathfrak{p}(1,1)$ when $\alpha=-1$.
\item[] $\G_{3,5}=\langle e_1,e_2,e_3~|~[e_1,e_3]=\beta e_1-e_2,~[e_2,e_3]=e_1+\beta e_2, \beta\geq 0\rangle$, solvable.
\end{enumerate}

\subsubsection{Flat Lie algebra $3\G_1$}
A straightforward and systematic computation yields the following classification of torsion-free connections on the abelian Lie algebra $2\G_1\cong\R^2$:
\begin{Le}\label{ClassitorsionfreeR}
Let $\nabla^0$ be a torsion-free connection on $\mathbb{R}^2$. Then $\nabla^0$ is equivalent to  one of the following  torsion-free connections$:$
\begin{small}
\[
\begin{alignedat}{4}
\nabla^{1}_{e_1}e_1 &= 0, &\quad\quad \nabla^{1}_{e_1}e_2 &= 0,&\quad\quad
\nabla^{1}_{e_2}e_1 &=  0, &\quad\quad \nabla^{1}_{e_2}e_2 &=\delta\,e_1.\\
\nabla^2_{e_1}e_1 &= e_2, &\quad \nabla^2_{e_1}e_2 &= 0,&\quad
\nabla^2_{e_2}e_1 &=  0, &\quad \nabla^2_{e_2}e_2 &=0.\\
\nabla^3_{e_1}e_1 &= e_2, &\quad\quad \nabla^3_{e_1}e_2 &= 0,&\quad\quad
\nabla^3_{e_2}e_1 &=  0, &\quad\quad \nabla^3_{e_2}e_2 &=e_1.\\
	\nabla^4_{e_1}e_1 &=\lambda_1\,  e_1+e_2, &\quad \nabla^1_{e_1}e_2 &= e_2,&\quad
\nabla^4_{e_2}e_1 &=  e_2, &\quad \nabla^4_{e_2}e_2 &= \mu_1\,e_1+\nu_1\,e_2.\\
\nabla^5_{e_1}e_1 &=\lambda_2\,  e_1, &\quad \nabla^5_{e_1}e_2 &= e_2,&\quad
\nabla^5_{e_2}e_1 &=  e_2, &\quad \nabla^5_{e_2}e_2 &= \mu_2\,e_1+e_2.\\
\nabla^6_{e_1}e_1 &=\lambda_3\,  e_1, &\quad \nabla^6_{e_1}e_2 &= e_2,&\quad
\nabla^6_{e_2}e_1 &=  e_2, &\quad \nabla^6_{e_2}e_2 &= \delta_{\varepsilon_1}\,e_1.\\
\nabla^7_{e_1}e_1 &= e_1+e_2, &\quad \nabla^7_{e_1}e_2 &= 0,&\quad
\nabla^7_{e_2}e_1 &=  0, &\quad \nabla^7_{e_2}e_2 &= \mu_4\,e_1+\nu_4\,e_2.
\end{alignedat}
\]
\end{small}
Where, $\lambda_j,\mu_j,\nu_j\in\R$, $\delta=0,1$ and $\delta_{\varepsilon_1}=0,\pm1$.
\end{Le}
\begin{proof}

Let $V = \langle e_1, e_2 \rangle$ be a two-dimensional vector space, viewed as an abelian Lie algebra $\G$, and let $\nabla$ be a torsion-free connection on $\G$. Then, with respect to the basis $\{e_1, e_2\}$, the  products are given by:
\begin{align}\label{GeneralfreeconnR}
	\nabla_{e_1}e_1 &=a_{11} e_1+a_{21}e_2, &\quad \nabla_{e_1}e_2 &= a_{12} e_1+a_{22}e_2,&\quad
	\nabla_{e_2}e_1 &=  a_{12} e_1+a_{22}e_2, &\quad \nabla_{e_2}e_2 &= b_{12}e_1+b_{22}e_2,
\end{align}
where, $a_{ij},b_{ij}\in\R$. Recall that two torsion-free connections $\nabla^1$ and $\nabla^2$ on a Lie algebra $\G$ are isomorphic if and only if there exists an automorphism $\Phi : \G \longrightarrow \G$ such that
\begin{align}\label{iso2torsion}
	\nabla^2_x&=\Phi\circ\nabla^1_{
		\Phi^{-1}(x)}\circ\Phi^{-1},\quad\quad\text{for all~~} x\in\G.
\end{align}
An automorphism of an abelian Lie algebra is simply an element $\Phi \in \mathrm{GL}_n(\mathbb{R})$.

Let $\nabla^{ij}$ denote the coefficients of
\[
\Phi\circ\nabla_{\Phi^{-1}(x)}\circ\Phi^{-1},
\quad\quad \text{for all } x\in\R^2.
\]
Fix $\Phi\in \mathrm{GL}_n(\R)$ given by
\begin{align}
	\Phi=
	\begin{pmatrix}
		x & y\\
		z & t
	\end{pmatrix},\quad xt-zy\neq0.
\end{align}
We may assume that $a_{12}=0$. Indeed, if $a_{12}\neq0$, then, by taking
\[
t=x,
\qquad
y=-z,
\]
We obtain that the numerator of $\nabla^{21}$ is 
\begin{align}
	a_{12}x^3
	+z(a_{11}-a_{22}-b_{12})x^2
	-z^2(a_{12}+a_{21}-b_{22})x
	+a_{22}z^3.
\end{align}
Hence, $\nabla^{12}$ necessarily admits a real solution, since it is a polynomial of odd degree.

From now on, assume that $a_{21}=0$. Applying the automorphism $\Psi$ to the connection given in \eqref{GeneralfreeconnR} yields the following equivalent connection:
\begin{align}\label{Auttcon}
	\nabla_{e_1}e_1 &=\tfrac{a_{11}}{x}  e_1+\tfrac{t\,a_{21}}{x^2} e_2, &\quad \nabla_{e_1}e_2 &= \tfrac{a_{22}}{x}e_2,&\quad
	\nabla_{e_2}e_1 &=  \tfrac{a_{22}}{x}e_2, &\quad \nabla_{e_2}e_2 &= \tfrac{x\,b_{12}}{x}e_1+\tfrac{b_{22}}{t}e_2,
\end{align}

Suppose that $a_{22}a_{21}\neq0$. Taking
\[
x=a_{22},
\qquad
t=\tfrac{a_{22}^2}{a_{21}},
\]
and setting
\[
\lambda=\tfrac{a_{11}}{a_{22}},
\qquad
\mu=\tfrac{a_{21}^2b_{12}}{a_{22}^3},
\qquad
\nu=\tfrac{a_{21}b_{22}}{a_{22}^2},
\]
we obtain 
\begin{align}\label{R2so1}
	\nabla^1_{e_1}e_1 &=\lambda\,  e_1+e_2, &\quad \nabla^1_{e_1}e_2 &= e_2,&\quad
	\nabla^1_{e_2}e_1 &=  e_2, &\quad \nabla^1_{e_2}e_2 &= \mu\,e_1+\nu\,e_2.
\end{align}
Note that this connection is flat if and only if $\mu=0$ and $\lambda=1-\nu$.  
In the same way, according to the values of the parameters defining the torsion-free connection \eqref{Auttcon}, we derive all possible cases and present them as follows:
\begin{small}
	\[
	\begin{alignedat}{4}\label{allabelian}
	\nabla^1_{e_1}e_1 &=\lambda_1\,  e_1+e_2, &\quad \nabla^1_{e_1}e_2 &= e_2,&\quad
	\nabla^1_{e_2}e_1 &=  e_2, &\quad \nabla^1_{e_2}e_2 &= \mu_1\,e_1+\nu_1\,e_2.\\
	\nabla^2_{e_1}e_1 &=\lambda_2\,  e_1, &\quad \nabla^2_{e_1}e_2 &= e_2,&\quad
	\nabla^2_{e_2}e_1 &=  e_2, &\quad \nabla^2_{e_2}e_2 &= \mu_2\,e_1+e_2.\\
	\nabla^3_{e_1}e_1 &=\lambda_3\,  e_1, &\quad \nabla^3_{e_1}e_2 &= e_2,&\quad
	\nabla^3_{e_2}e_1 &=  e_2, &\quad \nabla^3_{e_2}e_2 &= \delta_{\varepsilon_1}\,e_1.\\
	\nabla^4_{e_1}e_1 &= e_1+e_2, &\quad \nabla^4_{e_1}e_2 &= 0,&\quad
	\nabla^4_{e_2}e_1 &=  0, &\quad \nabla^4_{e_2}e_2 &= \mu_4\,e_1+\nu_4\,e_2.\\
	\nabla^5_{e_1}e_1 &= e_1, &\quad \nabla^5_{e_1}e_2 &= 0,&\quad
	\nabla^5_{e_2}e_1 &=  0, &\quad \nabla^5_{e_2}e_2 &= \mu_5\,e_1+e_2.\\
	\nabla^6_{e_1}e_1 &= e_1, &\quad \nabla^6_{e_1}e_2 &= 0,&\quad
	\nabla^6_{e_2}e_1 &=  0, &\quad \nabla^6_{e_2}e_2 &= \delta_{\varepsilon_2}\,e_1.\\
	\nabla^7_{e_1}e_1 &= e_2, &\quad \nabla^7_{e_1}e_2 &= 0,&\quad
	\nabla^7_{e_2}e_1 &=  0, &\quad \nabla^7_{e_2}e_2 &=e_1+\nu_7\,e_2.\\
	\nabla^8_{e_1}e_1 &= e_2, &\quad \nabla^8_{e_1}e_2 &= 0,&\quad
	\nabla^8_{e_2}e_1 &=  0, &\quad \nabla^8_{e_2}e_2 &=\delta_{\varepsilon_3}\,e_2.\\
	\nabla^9_{e_1}e_1 &= 0, &\quad \nabla^9_{e_1}e_2 &= 0,&\quad
	\nabla^9_{e_2}e_1 &=  0, &\quad \nabla^9_{e_2}e_2 &=e_1+e_2.\\
	\nabla^{10}_{e_1}e_1 &= 0, &\quad \nabla^{10}_{e_1}e_2 &= 0,&\quad
	\nabla^{10}_{e_2}e_1 &=  0, &\quad \nabla^{10}_{e_2}e_2 &=e_2.\\
	\nabla^{11}_{e_1}e_1 &= 0, &\quad \nabla^{11}_{e_1}e_2 &= 0,&\quad
	\nabla^{11}_{e_2}e_1 &=  0, &\quad \nabla^{11}_{e_2}e_2 &=\delta\,e_1.
	\end{alignedat}
	\]
\end{small}
Where, $\lambda_j,\mu_j,\nu_j\in\R$, $\delta=0,1$ and $\delta_\varepsilon=0,\pm1$.

A straightforward verification shows the following equivalences:
\begin{enumerate}
	\item $\nabla^{9}\cong \nabla^{10}$, \quad $\Phi:\quad e_1\mapsto e_1,\quad e_2\mapsto e_1+e_2$,
	\item  $\nabla^{6}\cong \nabla^{10}$, ($\delta_{\varepsilon_2}=0$) \quad $\Phi:\quad e_1\mapsto e_2,\quad e_2\mapsto e_1$,  (We may assume that $\delta_{\varepsilon_2}=\varepsilon=\pm1$)
	\item  $\nabla^{10}\cong \nabla^{4}$, ($\mu_4=\nu_4=0$) \quad $\Phi:\quad e_1\mapsto e_2,\quad e_2\mapsto e_1+e_2$,
	\item  $\nabla^{8}\cong \nabla^{6}$, ($\delta_{\varepsilon_3}\delta_{\varepsilon_2}\neq0$) \quad $\Phi:\quad e_1\mapsto z\,e_2,\quad e_2\mapsto \delta_{\varepsilon_3}\,e_1$, with $z^2=\frac{\delta_{\varepsilon_3}}{\delta_{\varepsilon_2}}$, (We may assume that $\delta_{\varepsilon_3}=0$)
		\item  $\nabla^{7}\cong \nabla^{4}$, ($\nu_7\neq0, ~~\nu_4=0,~~\mu_4=\frac{1}{\nu_7^3}$) \quad $\Phi:\quad e_1\mapsto \nu_7^2\,e_2,\quad e_2\mapsto \nu_7\,e_1$, (We may assume that $\nu_7=0$)
			\item  $\nabla^{6}\cong \nabla^{2}$, ($\delta_{\varepsilon_2}=-1$) \quad $\Phi:\quad e_1\mapsto e_2,\quad e_2\mapsto e_1-e_2$,  (We may assume that $\delta_{\varepsilon_2}=\varepsilon=1$)
				\item  $\nabla^{6}\cong \nabla^{4}$, ($\delta_{\varepsilon_2}=1,~~\nu_4=\mu_4=1$) \quad $\Phi:\quad e_1\mapsto \frac{1}{2}( e_1+e_2),\quad e_2\mapsto  \frac{1}{2}( e_1-e_2)$,  (This case can be omitted)
					\item  $\nabla^{5}\cong \nabla^{4}$, ($\mu_4=0,~~\mu_5=\nu_4\neq0$) \quad $\Phi:\quad e_1\mapsto  e_1+\frac{1}{\mu_5}e_2,\quad e_2\mapsto  e_2$, (We may assume that $\mu_5=0$)
					\item  $\nabla^{5}\cong \nabla^{3}$, ($\delta_{\varepsilon_1}=1,~~\lambda_3=1,~~\mu_5=0$) \quad $\Phi:\quad e_1\mapsto \frac{1}{2}( e_1-e_2),\quad e_2\mapsto  \frac{1}{2}( e_1+e_2)$,  (This case can be omitted)

\end{enumerate}
For the remaining cases, i.e., $\nabla^j$, $j=1,2,3,4$, the corresponding connections are isomorphic under certain conditions. In order to obtain an equivalence classification, it is not necessary to determine their representatives up to isomorphism.

\end{proof}

\begin{Le}\label{no flat to flat}
	Let $\nabla$ be a non flat  torsion-free connection on the abelian Lie algebra $\mathbb{R}^2$. Then there exist a flat  torsion-free connection $\nabla^{0}$ on $\R^2$, a vector $v \in \mathbb{R}^2$, and  a symmetric $(0,2)$-tensor $\mathcal{S} :\R\times\R\to\R$ such that
	\[
	\nabla_x y - \nabla^{0}_x y = \mathcal{S}(x,y)\,v,
	\]
	for all $x,y \in \mathbb{R}^2$.
\end{Le}
\begin{proof}
Let $V = \langle e_1, e_2 \rangle$ be a two-dimensional vector space, viewed as an abelian Lie algebra $\G$, and	let $\nabla$ be no flat, torsion-free connection on $\G$. Then, $\nabla$ is equivalent to one of the connections $\nabla^j$, $j=3,\ldots,7$ given in Lemma~\ref{ClassitorsionfreeR}. 
	Consider the first case, namely $\nabla=\nabla^3$, and consider the following data:
\begin{align}
	\mathcal{S}_3=e^1\otimes e^1,\quad v=e_2,\quad\text{and}\quad \nabla^{0,3}_{e_2}e_2=e_1.
\end{align}
Moreover, $\nabla^{0,3}$ is a flat, torsion-free connection on $\G$. 
	It is straightforward to show that 	\[
	\nabla^3_x y - \nabla^{0,3}_x y = \mathcal{S}_3(x,y)\,v,\quad\text{for all}\quad x,y\in\G.
	\]

Set 
\begin{align}
	\mathcal{S}_4=e^2\otimes e^2,\quad v=\mu_1\,e_1+(\nu_1+\lambda_1-1)\,e_2,
\end{align}
and
\begin{align}
	\nabla^{0,4}_{e_1}e_1&=\lambda_1\,e_1+e_2,&\nabla^{0,4}_{e_1}e_2&=e_2,&\nabla^{0,4}_{e_2}e_1&=e_2,&\nabla^{0,4}_{e_2}e_2&=(1-\lambda_1)\,e_2.
\end{align}
Note that the connection $\nabla^4$ is flat if and only if $\mu_1=0$ and $\nu_1=1-\lambda_1$.
Then, $\nabla^{0,4}$ is a flat, torsion-free connection on $\G$. 
We obtain,	\[
\nabla^4_x y - \nabla^{0,4}_x y = \mathcal{S}_4(x,y)\,v,\quad\text{for all}\quad x,y\in\G.
\]

Let now
\begin{align}
	\mathcal{S}_5=(\lambda_2-1)\,e^1\otimes e^1,\quad v=e_1,\quad\text{and}\quad \nabla^{0,5}_{e_1}e_1=e_1,~~\nabla^{0,5}_{e_1}e_2=e_2,~~\nabla^{0,5}_{e_2}e_1=e_2,~~\nabla^{0,5}_{e_2}e_2=\mu_2\,e_1+e_2.
\end{align}
Note that the connection $\nabla^5$ is flat if and only if $\lambda_2=1$.
Then, $\nabla^{0,5}$ is a flat, torsion-free connection on $\G$. 
We obtain,	\[
\nabla^5_x y - \nabla^{0,5}_x y = \mathcal{S}_5(x,y)\,v,\quad\text{for all}\quad x,y\in\G.
\]

	Consider
		\begin{align}
		\mathcal{S}_6=(\lambda_3-1)\,e^1\otimes e^1,\quad v=e_1,\quad\text{and}\quad \nabla^{0,6}_{e_1}e_1=e_1,~~\nabla^{0,6}_{e_1}e_2=e_2,~~\nabla^{0,6}_{e_2}e_1=e_2,~~\nabla^{0,6}_{e_2}e_2=\delta_{\varepsilon_1}\,e_1.
	\end{align}
	Note that the connection $\nabla^6$ is flat if and only if $\lambda_3=1$.
	Then, $\nabla^{0,6}$ is a flat, torsion-free connection on $\G$. 
	We obtain,	\[
	\nabla^6_x y - \nabla^{0,6}_x y = \mathcal{S}_6(x,y)\,v,\quad\text{for all}\quad x,y\in\G.
	\]

	Let
	\begin{align}
		\mathcal{S}_7=e^1\otimes e^1+e^2\otimes e^2,\quad v=e_1+e_2,\quad\text{and}\quad \nabla^{0,7}_{e_2}e_2=(\mu_4-1)\,e_1+(\nu_4-1)\,e_2.
	\end{align}
	Then, $\nabla^{0,7}$ is a flat, torsion-free connection on $\G$. 
	We obtain,	\[
	\nabla^7_x y - \nabla^{0,7}_x y = \mathcal{S}_7(x,y)\,v,\quad\text{for all}\quad x,y\in\G.
	\]
\end{proof}

\begin{pr}\label{FlatinR3}
Let $(\G, \nabla)$ be a three-dimensional real flat  abelian Lie algebra. Then $(\G, \nabla)$ is isomorphic to exactly one of the flat Lie algebras listed in Table~$\ref{FlatR3}$.

{\renewcommand*{\arraystretch}{1.8}
\captionof{table}{Flat torsion-free connection on the Lie algebra $3\G_1$.}
\setcounter{table}{2}
\begin{footnotesize} 
\setlength{\tabcolsep}{5pt} 
\begin{longtable}{@{}cllllllc@{}} 
			\hline
		Algebra&\multicolumn{2}{@{}l@{}}{~~Flat torsion-free connection}&&&&Remarks \\
			\hline
			$\h_{0,0}$&$\nabla\equiv0$&&&&&\\
$\h_{0,1}$&$\nabla_{e_2}e_2=e_1$&$\nabla_{e_2}e_3=e_1$&$\nabla_{e_3}e_2=e_1$&$\nabla_{e_3}e_3=e_2-e_3$&&\\
$\h_{0,2}$&$\nabla_{e_3}e_3=e_2+e_3$&&&&&\\
$\h_{0,3}$&$\nabla_{e_2}e_3=e_2$&$\nabla_{e_3}e_2=e_2$&$\nabla_{e_3}e_3=e_3$&&&\\
$\h_{0,4}$&$\nabla_{e_2}e_3=e_1$&$\nabla_{e_3}e_2=e_1$&$\nabla_{e_3}e_3=e_2$&&&\\
$\h_{0,5}$&$\nabla_{e_2}e_2=\varepsilon e_3$&$\nabla_{e_2}e_3=e_2$&$\nabla_{e_3}e_2=e_2$&$\nabla_{e_3}e_3=e_3$&&$\varepsilon=\pm1$\\
$\h_{0,6}$&$\nabla_{e_3}e_3=e_2$&&&&&\\
$\h_{0,7}$&$\nabla_{e_1}e_1=e_3$&$\nabla_{e_2}e_2=\varepsilon e_3$&&&&$\varepsilon=\pm1$\\
$\h_{0,8}$&$\nabla_{e_1}e_3=e_1$&$\nabla_{e_2}e_3=e_2$&$\nabla_{e_3}e_j=e_j$&&&$j=1,2,3$\\\hline
$\h_{1,1}$&$\nabla_{e_1}e_1=e_1$&$\nabla_{e_3}e_3=e_3$&&&&\\
$\h_{1,2}$&$\nabla_{e_1}e_1 =e_1$&$\nabla_{e_2}e_3=e_2$&$\nabla_{e_3}e_2=e_2$&$\nabla_{e_3}e_3=e_3$&&\\
$\h_{1,3}$& $\nabla_{e_1}e_1 = e_1$ & $\nabla_{e_2}e_2= e_2$& $\nabla_{e_2}e_3 = e_3$ &
$\nabla_{e_3}e_2 = e_3$ &
$\nabla_{e_3}e_3 = e_2$&\\\hline
$\h_{2,1}$&$\nabla_{e_1}e_3=e_1$&$\nabla_{e_2}e_2=e_1$&$\nabla_{e_2}e_3=e_1+e_2$&$ \nabla_{e_3}e_1=e_1$&$\nabla_{e_3}e_2=e_1+e_2$&\\
&$\nabla_{e_3}e_3=e_2+e_3$&&&&&\\\hline
$\h_{3,1}$&$\nabla_{e_1} e_1 = e_1$&$\nabla_{e_2} e_2 = e_2,$&$\nabla_{e_2} e_3 = e_3$&$\nabla_{e_3} e_2 = e_3$&$\nabla_{e_3} e_3 = - e_2$&\\\hline
$\h_{4,1}$&$	\nabla_{e_1} e_1 = e_1$&$\nabla_{e_1} e_2=e_2$&$\nabla_{e_1} e_3=e_3$&$\nabla_{e_2} e_1=e_2$&$\nabla_{e_2} e_2= e_3$&\\
&$\nabla_{e_2} e_3=\lambda e_1+\varepsilon_1e_2$&$\nabla_{e_3} e_1=e_3$&$\nabla_{e_3}e_2=\lambda e_1+\varepsilon_1e_2$&$\nabla_{e_3}e_3=\lambda e_2+\varepsilon_1e_3$&&$\lambda\in\R^\ast,~\varepsilon_1=\pm1$

\\\hline		

			\end{longtable}
			\label{FlatR3}
			\end{footnotesize}	
			}

\end{pr}
\begin{proof}
	Let $\nabla$ be a linear connection on $\R^3$,  viewed as $\R^2\oplus\R e_3$, with basis $\{e_1,e_2,e_3\}$. Then, $\nabla$ can be expressed as 
	\begin{align}\label{Connegeneral}
		\begin{split}
			\nabla_xy&=\nabla^0_xy+\theta(x,y)e_3,\\
			\nabla_xe_3&=\beta(x)+\gamma(x)e_3,\\
			\nabla_{e_3}x&=\beta(x)+\gamma(x)e_3,\\
			\nabla_{e_3}e_3&=\zeta+\lambda e_3,
		\end{split}
	\end{align}
	
	for all $x, y \in \R^2$, where $\zeta\in\R^2$,  $\theta \in \mathcal{S}^2(\R^2)$ is a symmetric form, $\beta, \eta: \R^2 \to \R^2$ are endomorphisms of $\R^2$, $\gamma: \R \to \R$ is a one-form, and $\nabla^0$ is a torsion-free connection on $\R^2$. The curvature tensor $\mathcal{R}^\nabla$ of $\nabla$ is given by
	\begin{align}\label{flatness-equations}
		\mathcal{R}^\nabla(x,y)z=\nabla_x\nabla_yz-\nabla_y\nabla_xz,\quad\quad\text{for all~} x,y,z\in\R^3.
	\end{align}
	The condition for $\nabla$ to be flat is $\mathcal{R}^\nabla = 0$. We will refer to the corresponding system of equations as the \textit{flatness-equations}. In the basis $\{e_1, e_2, e_3\}$, the operators $\nabla_{e_1}$, $\nabla_{e_2}$, and $\nabla_{e_3}$ are given respectively by the matrices:
	\begin{align}
		\nabla_{e_1}&=\left( \begin {array}{ccc} a_{11}&a_{12}&a_{{13}}\\ \noalign{\medskip}a_{21}&a_{22}&a_{{23}}\\ \noalign{\medskip}a_{{31}}&a_{{32}}&a_{{33}}\end {array}
		\right),&\nabla_{e_2}&=\left( \begin {array}{ccc} a_{12}&b_{12}&b_{{13}}\\ \noalign{\medskip}a_{22}&b_{22}&b_{{23}}\\ \noalign{\medskip}a_{{32}}&b_{{32}}&b_{{33}}\end {array}
		\right),&\nabla_{e_3}&=\left( \begin {array}{ccc} a_{{13}}&b_{{13}}&c_{{13}}
		\\ \noalign{\medskip}a_{{23}}&b_{{23}}&c_{{23}}
		\\ \noalign{\medskip}a_{{33}}&b_{{33}}&c_{{33}}\end {array}
		\right),
	\end{align}
	where, $a_{ij}, b_{ij}, c_{ij}\in\R$.

	According to Lemma~\ref{no flat to flat}, we may assume that $\nabla^0$ is flat, torsion-free. In this case, solving the flatness-equations is straightforward by considering the six cases of flat, torsion-free connections listed in Table~\ref{FlatR2}. We refer the reader to Appendix~\ref{App}  for the detailed proof of this classification.
\end{proof}
\begin{co}
	With the notations as above, among the flat Lie algebras on $3\G_{1}$, we have
	\begin{enumerate}
		\item[i)] Associative algebras$:$\hspace{0.275cm} $\h_{0,0}$, $\h_{0,2}$, $\h_{0,3}$, $\h_{0,4}$, $\h_{0,5}$, $\h_{0,6}$, $\h_{0,7}$, $\h_{0,8}$, $\h_{1,1}$, $\h_{1,2}$, $\h_{3,1}$, $\h_{4,1}$.
		\item[ii)] Novikov algebras$:$\hspace{0.73cm} $\h_{0,0}$, $\h_{0,2}$, $\h_{0,3}$, $\h_{0,4}$, $\h_{0,5}$, $\h_{0,6}$, $\h_{0,7}$, $\h_{0,8}$, $\h_{1,1}$, $\h_{1,2}$, $\h_{3,1}$, $\h_{4,1}$.
		\item[iii)] Bi-symmetric algebras$:$ $\h_{0,0}$, $\h_{0,2}$, $\h_{0,3}$, $\h_{0,4}$, $\h_{0,5}$, $\h_{0,6}$, $\h_{0,7}$, $\h_{0,8}$, $\h_{1,1}$, $\h_{1,2}$, $\h_{3,1}$, $\h_{4,1}$.
		\item[iv)] Complete algebras$:$\hspace{0.62cm}$\h_{0,0}$, $\h_{0,4}$, $\h_{0,6}$, $\h_{0,7}$.
	\end{enumerate}
\end{co}

\subsubsection{Flat Lie algebra $2\G_{2,1}\oplus\G_1$}

\begin{Le}\label{Torsionfreeaffine}
Let $\nabla^0$ be a torsion-free connection on $\mathfrak{aff}(1,\R)$. Then $\nabla^0$ is equivalent to either one of the flat, torsion-free connections on $\mathfrak{aff}(1,\R)$ given in Table~$\ref{Flataffine}$, or to one of the following no flat torsion-free connections$:$
\begin{small}
\[
\begin{alignedat}{4}
	\nabla^{0,1}_{e_1}e_1&=\mu e_1+\varepsilon e_2,\quad\quad &\nabla^{0,1}_{e_1}e_2&=\lambda e_1,\quad\quad &\nabla^{0,1}_{e_2}e_1&=(\lambda-1) e_1,\quad\quad  &\nabla^{0,1}_{e_2}e_2&=\nu e_1+\eta e_2,\\
	 \nabla^{0,2}_{e_1}e_1&=\mu_1 e_1,\quad\quad &\nabla^{0,2}_{e_1}e_2&= e_2, \quad\quad&\nabla^{0,2}_{e_2}e_1&=-e_1+e_2, \quad\quad &\nabla^{0,2}_{e_2}e_2&=\nu_1 e_1+\eta_1 e_2,\\
	 \nabla^{0,3}_{e_1}e_1&=e_1,\quad\quad &\nabla^{0,3}_{e_1}e_2&=0,\quad\quad &\nabla^{0,3}_{e_2}e_1&=- e_1,\quad\quad  &\nabla^{0,3}_{e_2}e_2&=\nu_2 e_1+\eta_2 e_2,\\
	 \nabla^{0,4}_{e_1}e_1&=e_1,\quad\quad &\nabla^{0,4}_{e_1}e_2&=\lambda_3 e_1+e_2,\quad\quad &\nabla^{0,4}_{e_2}e_1&=(\lambda_3-1) e_1, \quad\quad &\nabla^{0,4}_{e_2}e_2&=\nu_3 e_1,\\
	 \nabla^{0,5}_{e_1}e_1&=0,\quad\quad &\nabla^{0,5}_{e_1}e_2&=\lambda_4 e_1,\quad\quad &\nabla^{0,5}_{e_2}e_1&=(\lambda_4-1) e_1, \quad\quad &\nabla^{0,5}_{e_2}e_2&=\delta e_1+\eta_4 e_2,
\end{alignedat}
\]
\end{small}
where the parameters associated with the above connections satisfy the assumptions that $\delta=0$ whenever $\eta_4-2\lambda_4+1\neq0$, and
\[
\lambda_4(\lambda_4-\eta_4)\neq 0,
\qquad
(\mu,\lambda,\nu,\eta)\neq(0,0,0,-2).
\]
\end{Le}
\begin{proof}
Let $\{e_1,e_2\}$ be a basis of $\mathfrak{aff}(1,\mathbb{R})$, and let $\nabla$ be a non-flat, torsion-free connection on $\mathfrak{aff}(1,\mathbb{R})$. Then, with respect to the basis $\{e_1,e_2\}$, $\nabla$ is given by
\begin{equation}\label{conaff1}
	\nabla_{e_1}=\begin{pmatrix}
		a_{11}&a_{12}\\
		a_{21}&a_{22}
	\end{pmatrix},
	\qquad
	\nabla_{e_2}=\begin{pmatrix}
		-1+a_{12}&b_{12}\\
		a_{22}&b_{22}
	\end{pmatrix}.
\end{equation}
Recall that two torsion-free connections $\nabla^1$ and $\nabla^2$ on a Lie algebra $\mathfrak{g}$ are isomorphic if and only if there exists an automorphism $\Psi : \mathfrak{g} \to \mathfrak{g}$ such that
\begin{align}\label{eqiso2torsion}
	\nabla^2_x = \Psi \circ \nabla^1_{\Psi^{-1}(x)} \circ \Psi^{-1}, \quad \text{for all } x \in \mathfrak{g}.
\end{align}
An automorphism of $\mathfrak{aff}(1,\mathbb{R})$ is given by
\begin{align}
	\Psi=\begin{pmatrix}
		x_{11}&x_{12}\\
		0&1
	\end{pmatrix},\qquad x_{11}\neq0.
\end{align}
Assume that $a_{21} \neq 0$, and consider the following automorphism
\begin{align*}
	\Psi(e_1) &= \tfrac{\sqrt{\varepsilon\, a_{21}}}{\varepsilon}\, e_1, &
	\Psi(e_2) &= \tfrac{a_{22}\sqrt{\varepsilon\, a_{21}}}{\varepsilon\, a_{21}}\, e_1 + e_2, \qquad \varepsilon = \pm 1.
\end{align*}
Applying it to the connection given in \eqref{conaff1} yields the following equivalent connection:
\begin{equation}\label{aff1}
	\nabla^1_{e_1}=\begin{pmatrix}
		\mu&\lambda\\
		\varepsilon&0
	\end{pmatrix},
	\qquad
	\nabla^1_{e_2}=\begin{pmatrix}
		\lambda-1&\nu\\
		0&\eta
	\end{pmatrix},\qquad \lambda,\mu,\nu,\eta\in\R, ~~\varepsilon=\pm1.
\end{equation}
This connection is flat if and only if $\mu = \lambda = \nu = 0$ and $\eta = -2$. We can also verify that $\nabla^{\mu,\eta,\lambda,\nu} \cong \nabla^{\mu,\eta,-\lambda,-\nu}$. Thus, $\nabla^{\mu,\eta,\lambda,\nu}$ is defined as in \eqref{aff1} under the assumptions $\lambda, \nu \in \mathbb{R}^+$.

Assume now that $a_{21} = 0$ and $(a_{11} - a_{22})a_{22} \neq 0$, and consider the following automorphism
\begin{align*}
	\Psi(e_1) &= a_{22}\, e_1, &
	\Psi(e_2) &= \tfrac{a_{12}a_{22}}{a_{11}-a_{22}}\, e_1 + e_2,
\end{align*}
 Applying it to the connection \eqref{conaff1} yields the following equivalent connection:
\begin{equation}\label{aff2}
	\nabla^2_{e_1}=\begin{pmatrix}
		\mu_1&0\\
		0&1
	\end{pmatrix},
	\qquad
	\nabla^2_{e_2}=\begin{pmatrix}
		-1&\nu_1\\
		1&\eta_1
	\end{pmatrix},\qquad \mu_1,\nu_1,\eta_1\in\R.
\end{equation}

If $a_{21} = 0$, $a_{11} \neq 0$, and $a_{22} = 0$, consider the following automorphism
\begin{align*}
	\Psi(e_1) &= a_{11}\, e_1, &
	\Psi(e_2) &=a_{12}e_1 + e_2.
\end{align*}
Applying it to the connection given in \eqref{conaff1} yields the following equivalent connection:
\begin{equation}\label{aff3}
	\nabla^3_{e_1}=\begin{pmatrix}
		1&0\\
		0&0
	\end{pmatrix},
	\qquad
	\nabla^3_{e_2}=\begin{pmatrix}
		-1&\nu_2\\
		0&\eta_2
	\end{pmatrix},\qquad \nu_2,\eta_2\in\R.
\end{equation}

If $a_{21} = 0$, $a_{11}=a_{22}$, and $a_{22} \neq 0$, consider the following automorphism
\begin{align*}
	\Psi(e_1) &= a_{22}\, e_1, &
	\Psi(e_2) &=\tfrac{b_{22}}{2} e_1 + e_2.
\end{align*}
Applying it to the connection given in \eqref{conaff1} yields the following equivalent connection:
\begin{equation}\label{aff4}
	\nabla^4_{e_1}=\begin{pmatrix}
		1&\lambda_3\\
		0&1
	\end{pmatrix},
	\qquad
	\nabla^4_{e_2}=\begin{pmatrix}
		-1+\lambda_3&\nu_3\\
		1&0
	\end{pmatrix},\qquad \lambda_3,\nu_3\in\R.
\end{equation}

If $a_{21} = 0$ and $a_{11} = a_{22} = 0$, consider the following automorphism
\begin{align*}
	\Psi(e_1) &= e_1, &
	\Psi(e_2) &= x\, e_1 + e_2,
\end{align*}
where $x \in \mathbb{R}^\ast$. Applying it, for a suitable choice of $x$, to the connection given in \eqref{conaff1} yields the following equivalent connection:
\begin{equation}\label{aff4}
	\nabla^4_{e_1}=\begin{pmatrix}
		0&\lambda_4\\
		0&0
	\end{pmatrix},
	\qquad
	\nabla^4_{e_2}=\begin{pmatrix}
		-1+\lambda_4&\delta\\
		0&\eta_4
	\end{pmatrix},\qquad \lambda_4,\eta_4\in\R,~~\delta=0,1.
\end{equation}
This connection is flat if and only if $\lambda_4(\eta_4-\lambda_4)=0$. A straightforward computation shows that $\delta=0$ whenever $\eta_4-2\lambda_4+1\neq0$.

\end{proof}
\begin{pr}
Let $(\G, \nabla)$ be a three-dimensional real flat  Lie algebra with $\G = 2\G_{2,1}\oplus\G_1$. Then $(\G, \nabla)$ is isomorphic to exactly one of the flat Lie algebras listed in Table~$\ref{g2}$.
{\renewcommand*{\arraystretch}{1.8}
\captionof{table}{Flat torsion-free connection on the Lie algebra $3\G_{2,1}\oplus\G_1$.}
\setcounter{table}{3}
\begin{footnotesize} 
\setlength{\tabcolsep}{1pt} 
\begin{longtable}{@{}cllllllc@{}}
			\hline
		Flat algebra&\multicolumn{2}{@{}l@{}}{~~Flat torsion-free connection}&&&&Remarks & \\
			\hline
$\h_{1,1}$&$\nabla_{e_1}e_2=\lambda e_1$&$\nabla_{e_2}e_1=(\lambda-1)e_1$&$\nabla_{e_2}e_2=\lambda e_2$&$\nabla_{e_3}e_3=\delta e_3$&&$\lambda\in\R$, $\delta=0,1$&\\		
$\h_{1,2}$&$\nabla_{e_1}e_2=-e_1+e_3$&$\nabla_{e_2}e_1=-2e_1+e_3$&$\nabla_{e_2}e_2=-e_2$&$\nabla_{e_2}e_3=-e_3$&$\nabla_{e_3}e_2=-e_3$&&\\
$\h_{1,3}$&$\nabla_{e_1}e_2=e_1$&$\nabla_{e_2}e_2=e_2$&$\nabla_{e_3}e_3= e_1+ \delta e_3$&&&$\delta=0,1$&\\
$\h_{1,4}$&$\nabla_{e_1}e_1=\varepsilon_1 e_3$&$\nabla_{e_1}e_2=2 e_1$&$\nabla_{e_2}e_1=e_1$&$\nabla_{e_2}e_2=2e_2$&&$\varepsilon_1=\pm1$&\\
$\h_{1,5}$&$\nabla_{e_1}e_2=\lambda e_1$&$\nabla_{e_2}e_1=(\lambda-1)e_1$&$\nabla_{e_2}e_2=\lambda e_2+\delta e_3$&$\nabla_{e_2}e_3=\lambda e_3$&$\nabla_{e_3}e_2=\lambda e_3$&$\lambda\in\R$, $\delta=0,1$&\\
$\h_{1,6}$&$\nabla_{e_1}e_2=e_1$&$\nabla_{e_2}e_2=e_2+\delta e_3$&$\nabla_{e_2}e_3=e_1+e_3$&$\nabla_{e_3}e_2=e_1+e_3$&&$\delta=0,1$&\\
$\h_{1,7}$&$\nabla_{e_1}e_2=-e_1$&$\nabla_{e_1}e_3=e_1$&$\nabla_{e_2}e_1=-2e_1$&$\nabla_{e_2}e_2=-e_2$&$\nabla_{e_2}e_3=e_2$&&\\
&$\nabla_{e_3}e_1=e_1$&$\nabla_{e_3}e_2=e_2$&$\nabla_{e_3}e_3=e_3$&&&&\\
$\h_{1,8}$&$\nabla_{e_2}e_1=-e_1$&$\nabla_{e_2}e_2=\lambda_1 e_3$&$\nabla_{e_2}e_3=e_2$&$\nabla_{e_3}e_2=e_2$&$\nabla_{e_3}e_3=e_3$&$\lambda_1\in\R$&\\
$\h_{1,9}$&$\nabla_{e_1}e_2=\lambda e_1$&$\nabla_{e_1}e_3=e_1$&$\nabla_{e_2}e_1=(\lambda-1)e_1$&$\nabla_{e_2}e_2=\lambda e_2$&$\nabla_{e_2}e_3=e_2$&$\lambda\in\R,\lambda_1\in\R^\ast$&\\
&$\nabla_{e_3}e_1=e_1$&$\nabla_{e_3}e_2=e_2$&$\nabla_{e_3}e_3=\lambda_1 e_2+(1-\lambda\lambda_1)e_3$&&&&\\
$\h_{1,10}$&$\nabla_{e_2}e_1=-e_1$&$\nabla_{e_3}e_3=\varepsilon e_2+\lambda_1 e_3$&&&&$\varepsilon=\pm1$, $\lambda_1\in\R$&\\
$\h_{1,11}$&$\nabla_{e_1}e_2=e_1$&$\nabla_{e_1}e_3=e_1$&$\nabla_{e_2}e_2=e_2$&$\nabla_{e_2}e_3=e_2$&$\nabla_{e_3}e_1=e_1$&&\\
&$\nabla_{e_3}e_2=e_2$&$\nabla_{e_3}e_3=e_1+e_3$&&&&&\\
$\h_{1,12}$&$\nabla_{e_1}e_3=e_1$&$\nabla_{e_2}e_1=-e_1$&$\nabla_{e_2}e_3=e_2$&$\nabla_{e_3}e_1=e_1$&$\nabla_{e_3}e_2=e_2$&&\\
&$\nabla_{e_3}e_3=\lambda_1 e_2+e_3$&&&&&$\lambda_1\in\R$&\\
$\h_{1,13}$&$\nabla_{e_1}e_3=e_1$&$\nabla_{e_2}e_1=-e_1$&$\nabla_{e_3}e_1=e_1$&$\nabla_{e_3}e_3=\lambda_2 e_2+e_3$&&$\lambda_2\in\R$&\\
\\\hline
$\h_{2,1}$&$\nabla_{e_2}e_1=-e_1$&$\nabla_{e_2}e_2=\mu e_2$&$\nabla_{e_3}e_3=\delta e_3$&&&$\mu\in\R^\ast$, $\delta=0,1$&\\
$\h_{2,2}$&$\nabla_{e_2}e_1=-e_1$&$\nabla_{e_2}e_2=\mu e_2$&$\nabla_{e_2}e_3=\mu e_3$&$\nabla_{e_3}e_2=\mu e_3$&&$\mu\in\R^\ast$&\\
$\h_{2,3}$&$\nabla_{e_2}e_1=-e_1$&$\nabla_{e_2}e_2=- e_2$&$\nabla_{e_2}e_3=- e_3$&$\nabla_{e_3}e_2=- e_3$&&&\\
$\h_{2,4}$&$\nabla_{e_2}e_1=-e_1$&$\nabla_{e_2}e_2=- e_2$&$\nabla_{e_2}e_3=-e_3$&$\nabla_{e_3}e_2=- e_3$&$\nabla_{e_3}e_3=e_1+\delta e_3$&$\delta=0,1$&\\
$\h_{2,5}$&$\nabla_{e_1}e_1=\varepsilon_1 e_3$&$\nabla_{e_2}e_1=-e_1$&$\nabla_{e_2}e_2=- 2e_2$&$\nabla_{e_2}e_3=-2e_3$&$\nabla_{e_3}e_2=- 2e_3$&$\varepsilon_1=\pm1$&\\
$\h_{2,6}$&$\nabla_{e_1}e_2= e_3$&$\nabla_{e_2}e_1=-e_1+e_3$&$\nabla_{e_2}e_2=e_2$&&&&\\
$\h_{2,7}$&$\nabla_{e_1}e_3= e_1$&$\nabla_{e_2}e_1=-e_1$&$\nabla_{e_2}e_2=\mu e_2$&$\nabla_{e_3}e_1=e_1$&$\nabla_{e_3}e_3=e_3$&$\mu\in\R^\ast$&\\
$\h_{2,8}$&$\nabla_{e_1}e_3=\delta_1 e_1+\delta_2e_2$&$\nabla_{e_2}e_1=-e_1$&$\nabla_{e_2}e_2=-e_2$&$\nabla_{e_3}e_1=\delta_1 e_1+\delta_2e_2$&$\nabla_{e_3}e_3=\delta_1 e_3$&$\delta_1^2+\delta_2^2\neq0$, $\delta_1,\delta_2=0,1$&\\
$\h_{2,9}$&$\nabla_{e_2}e_1=-e_1$&$\nabla_{e_2}e_2=\mu e_2$&$\nabla_{e_2}e_3=\mu e_3$&$\nabla_{e_3}e_2=\mu e_3$&$\nabla_{e_3}e_3=\varepsilon e_2+\mu_1 e_3$&$\varepsilon=\pm1$, $\mu\in\R^\ast,\mu_1\in\R$&\\
$\h_{2,10}$&$\nabla_{e_1}e_3=e_1$&$\nabla_{e_2}e_1=-e_1$&$\nabla_{e_2}e_2=\mu e_2$&$\nabla_{e_2}e_3=\mu_1e_2$&$\nabla_{e_3}e_1=e_1$&$\mu\in\R^\ast,\mu_1\in\R$&\\
&$\nabla_{e_3}e_2=\mu_1e_2$&$\nabla_{e_3}e_3=\frac{\mu_1(\mu_1-1)}{\mu}e_2+e_3$&&&&&\\
$\h_{2,11}$&$\nabla_{e_2}e_1=-e_1$&$\nabla_{e_2}e_2=\mu e_2$&$\nabla_{e_2}e_3= e_2$&$\nabla_{e_3}e_2= e_2$&$\nabla_{e_3}e_3=\frac{1}{\mu}e_2$&$\mu\in\R^\ast$&\\
$\h_{2,12}$&$\nabla_{e_2}e_1=-e_1$&$\nabla_{e_2}e_2=-e_2$&$\nabla_{e_2}e_3=e_2$&$\nabla_{e_3}e_2=e_2$&$\nabla_{e_3}e_3=e_3$&&\\
$\h_{2,13}$&$\nabla_{e_2}e_1=-e_1$&$\nabla_{e_2}e_2=-e_2$&$\nabla_{e_2}e_3=e_1$&$\nabla_{e_3}e_2=e_1$&&&\\
$\h_{2,14}$&$\nabla_{e_1}e_3=e_1$&$\nabla_{e_2}e_1=-e_1$&$\nabla_{e_2}e_2=-e_2$&$\nabla_{e_2}e_3=e_1+e_2$&$\nabla_{e_3}e_1=e_1$&&\\
&$\nabla_{e_3}e_2=e_1+e_2$&$\nabla_{e_3}e_3=-e_1+e_3$&&&&&\\
$\h_{2,15}$&$\nabla_{e_2}e_1=-e_1$&$\nabla_{e_2}e_2=-e_2+\mu_2e_3$&$\nabla_{e_2}e_3=e_2$&$\nabla_{e_3}e_2=e_2$&$\nabla_{e_3}e_3=e_3$&$\mu_2\in\R^\ast$&\\
$\h_{2,16}$&$\nabla_{e_2}e_1=-e_1$&$\nabla_{e_2}e_2=-e_2+\delta e_3$&$\nabla_{e_2}e_3=e_1$&$\nabla_{e_3}e_2=e_1$&&$\delta=0,1$&
\\\hline
$\h_{3,1}$&$\nabla_{e_2}e_1=-e_1$&$\nabla_{e_2}e_2=e_1- e_2$&$\nabla_{e_3}e_3=\delta e_3$&&&$\delta=0,1$&\\
$\h_{3,2}$&$\nabla_{e_2}e_1=-e_1$&$\nabla_{e_2}e_2=e_1- e_2$&$\nabla_{e_2}e_3=-e_3$&$\nabla_{e_3}e_2=- e_3$&&&\\
$\h_{3,3}$&$\nabla_{e_2}e_1=-e_1$&$\nabla_{e_2}e_2=e_1- e_2$&$\nabla_{e_2}e_3=-e_3$&$\nabla_{e_3}e_2=- e_3$&$\nabla_{e_3}e_3=\lambda_1e_1+e_3$&$\lambda_1\in\R^\ast$&\\
$\h_{3,4}$&$\nabla_{e_2}e_1=-e_1$&$\nabla_{e_2}e_2=e_1- e_2$&$\nabla_{e_2}e_3=-e_3$&$\nabla_{e_3}e_2=- e_3$&$\nabla_{e_3}e_3=\varepsilon_1 e_1$&$\varepsilon_1=\pm1$&\\
$\h_{3,5}$&$\nabla_{e_1}e_3=e_1$&$\nabla_{e_2}e_1=-e_1$&$\nabla_{e_2}e_2=e_1-e_2$&$\nabla_{e_2}e_3=e_2$&$\nabla_{e_3}e_1=e_1$&&\\
&$\nabla_{e_3}e_2=e_2$&$\nabla_{e_3}e_3=e_3$&&&&&\\
$\h_{3,6}$&$\nabla_{e_1}e_3=e_1$&$\nabla_{e_2}e_1=-e_1$&$\nabla_{e_2}e_2=e_1-e_2$&$\nabla_{e_2}e_3=\lambda_2 e_1+e_2$&$\nabla_{e_3}e_1=e_1$&$\lambda_2\in\R^\ast$&\\
&$\nabla_{e_3}e_2=\lambda_2e_1+e_2$&$\nabla_{e_3}e_3=-\lambda_2e_1+e_3$&&&&&
\\\hline
$\h_{4,1}$&$\nabla_{e_1}e_2=e_1$&$\nabla_{e_2}e_2=e_1+e_2$&$\nabla_{e_3}e_3=\delta e_3$&&&$\delta=0,1$&\\
$\h_{4,2}$&$\nabla_{e_1}e_2=e_1$&$\nabla_{e_2}e_2=e_1+e_2$&$\nabla_{e_2}e_3=e_3$&$\nabla_{e_3}e_2=e_3$&&&\\
$\h_{4,3}$&$\nabla_{e_1}e_2=e_1$&$\nabla_{e_2}e_2=e_1+e_2$&$\nabla_{e_3}e_3=\lambda_1 e_1+ e_3$&&&$\lambda_1\in\R^\ast$&\\
$\h_{4,4}$&$\nabla_{e_1}e_2=e_1$&$\nabla_{e_2}e_2=e_1+e_2$&$\nabla_{e_3}e_3=\varepsilon_1 e_1$&&&$\varepsilon=\pm1$&\\
$\h_{4,5}$&$\nabla_{e_1}e_2=e_1$&$\nabla_{e_1}e_3=e_1$&$\nabla_{e_2}e_2=e_1+e_2$&$\nabla_{e_2}e_3=e_2$&$\nabla_{e_3}e_1=e_1$&$\lambda_1\in\R^\ast$&\\
&$\nabla_{e_3}e_2=e_2$&$\nabla_{e_3}e_3=\lambda_1 e_1+e_3$&&&&&\\
$\h_{4,6}$&$\nabla_{e_1}e_2=e_1$&$\nabla_{e_2}e_2=e_1+e_2$&$\nabla_{e_3}e_3=\varepsilon e_1+ \lambda_2e_3$&&& $\varepsilon=\pm1$, $\lambda_2\in\R, \lambda_2\neq\pm1$&
\\\hline
$\h_{5,1}$&$\nabla_{e_1}e_1=\varepsilon e_2$&$\nabla_{e_2}e_1=-e_1$&$\nabla_{e_2}e_2=-2e_2$&$\nabla_{e_3}e_3=\delta e_3$&&$\delta=0,1$&\\
$\h_{5,2}$&$\nabla_{e_1}e_1=\varepsilon e_2$&$\nabla_{e_1}e_3=e_1$&$\nabla_{e_2}e_1=-e_1$&$\nabla_{e_2}e_2=-2e_2$&$\nabla_{e_2}e_3=e_2$&$\varepsilon=\pm1$&\\
&$\nabla_{e_3}e_1=e_1$&$\nabla_{e_3}e_2=e_2$&$\nabla_{e_3}e_3=e_3$&&&&
\\\hline		
			\end{longtable}
			\label{g2} 
			\end{footnotesize}	
			}

\end{pr}
\begin{proof}
	Let $\nabla$ be a linear connection on $2\G_{2,1}\oplus\G_1=\mathfrak{aff}(1,\R)\oplus\R$, with basis $\{e_1,e_2,e_3\}$. Then, $\nabla$ can be expressed as 
	\begin{align}\label{exten}
		\begin{split}
			\nabla_xy&=\nabla^0_xy+\theta(x,y)e_3,\\
			\nabla_xe_3&=\beta(x)+\gamma(x)e_3,\\
			\nabla_{e_3}x&=\beta(x)+\gamma(x)e_3,\\
			\nabla_{e_3}e_3&=\zeta+\lambda e_3,
		\end{split}
	\end{align}
	
	for all $x, y \in\mathfrak{aff}(1,\R)$, where $\zeta\in\mathfrak{aff}(1,\R)$,  $\theta \in \mathcal{S}^2(\mathfrak{aff}(1,\R)$ is a symmetric form, $\beta: \mathfrak{aff}(1,\R) \to \mathfrak{aff}(1,\R)$ is an endomorphism of $\mathfrak{aff}(1,\R)$, $\gamma: \mathfrak{aff}(1,\R) \to \R$ is a one-form, and $\nabla^0$ is a torsion-free connection on $\mathfrak{aff}(1,\R)$. The curvature tensor $\mathcal{R}^\nabla$ of $\nabla$ is given by
	\begin{align}\label{flatness-equations1}
		\mathcal{R}^\nabla(x,y)z=\nabla_x\nabla_yz-\nabla_y\nabla_xz-\nabla_{[x,y]}z,\quad\quad\text{for all~} x,y,z\in\mathfrak{aff}(1,\R)\oplus\R.
	\end{align}
	The condition for $\nabla$ to be flat is $\mathcal{R}^\nabla = 0$. We will refer to the corresponding system of equations as the \textit{flatness-equations}. In the basis $\{e_1, e_2, e_3\}$, the operators $\nabla_{e_1}$, $\nabla_{e_2}$, and $\nabla_{e_3}$ are given respectively by the matrices:
	\begin{align}\label{AFF}
		\nabla_{e_1}&=\left( \begin {array}{ccc} a_{11}&a_{12}&a_{{13}}\\ \noalign{\medskip}a_{21}&a_{22}&a_{{23}}\\ \noalign{\medskip}a_{{31}}&a_{{32}}&a_{{33}}\end {array}
		\right),&\nabla_{e_2}&=\left( \begin {array}{ccc} a_{12}-1&b_{12}&b_{{13}}\\ \noalign{\medskip}a_{22}&b_{22}&b_{{23}}\\ \noalign{\medskip}a_{{32}}&b_{{32}}&b_{{33}}\end {array}
		\right),&\nabla_{e_3}&=\left( \begin {array}{ccc} a_{{13}}&b_{{13}}&c_{{13}}
		\\ \noalign{\medskip}a_{{23}}&b_{{23}}&c_{{23}}
		\\ \noalign{\medskip}a_{{33}}&b_{{33}}&c_{{33}}\end {array}
		\right),
	\end{align}
	where, $a_{ij}, b_{ij}, c_{ij}\in\R$.

	As a result of Lemma~\ref{isoofconnection} and Remark~\ref{sympleofauto}, we can choose $\nabla^0$ up to isomorphism.  Assume that $\nabla^0$ is a non-flat torsion-free connection. In this case, $\nabla^0$ is isomorphic to one of the torsion-free connections listed in  Lemma~\ref{Torsionfreeaffine}.

	The case where $\nabla^0$ is a flat torsion-free connection is treated separately outside the scope of this paper. All solutions to the flatness equations are given in Appendix~\ref{App}, since this case is comparatively simple. Furthermore, as will be seen in the subsequent analysis, the non-flat torsion-free cases reduce, up to isomorphism, to the study of the flat torsion-free case.

	$\bullet$ \textbf{Case 1 : } In the case where $\nabla^0=\nabla^{0,1}$, Equation~\eqref{AFF} reduces to
		\begin{align}\label{AFF1}
		\nabla_{e_1}&=\left( \begin {array}{ccc} \mu&\lambda&a_{{13}}\\ \noalign{\medskip}\varepsilon&0&a_{{23}}\\ \noalign{\medskip}a_{{31}}&a_{{32}}&a_{{33}}\end {array}
		\right),&\nabla_{e_2}&=\left( \begin {array}{ccc} \lambda-1&\nu&b_{{13}}\\ \noalign{\medskip}0&\eta&b_{{23}}\\ \noalign{\medskip}a_{{32}}&b_{{32}}&b_{{33}}\end {array}
		\right),&\nabla_{e_3}&=\left( \begin {array}{ccc} a_{{13}}&b_{{13}}&c_{{13}}
		\\ \noalign{\medskip}a_{{23}}&b_{{23}}&c_{{23}}
		\\ \noalign{\medskip}a_{{33}}&b_{{33}}&c_{{33}}\end {array}
		\right),
	\end{align}
	where, $a_{ij}, b_{ij}, c_{ij}\in\R$, $\varepsilon=\pm1$ and $(\mu,\lambda,\nu,\eta)\neq(0,0,0,-2)$.
	
Based on a straightforward calculation, the flatness-equations associated with connection~\eqref{AFF1} have a unique solution, given by the following flat torsion-free connection:
	\begin{equation}
		\begin{aligned}\label{AffSOL1}
			\nabla_{e_1} e_1 &=\varepsilon e_2+a_{31}e_3, &\nabla_{e_1}e_2&=-\varepsilon\,c_{33}a_{31}e_1,&\nabla_{e_2}e_1&=(-\varepsilon\,c_{33}a_{31}-1)e_1, \\\nabla_{e_2}e_2&=(-2\,\varepsilon\,c_{33}a_{31}-2)e_2+(-a_{31}^2c_{33}-2\,\varepsilon\,a_{31})e_3,&
			 \nabla_{e_j}e_3&=c_{33}e_j,&j&=1,2,3.
		\end{aligned}
	\end{equation}
	Let us consider the following automorphism
	\begin{align*}
		\Psi(e_1) &= e_1, &
		\Psi(e_2) &=e_2-\varepsilon\,a_{31}\,x e_3, &
		\Psi(e_3) &= x\, e_3,
	\end{align*}
	and applying it to the connection given in \eqref{AffSOL1}, for a suitable choice of $x \in \mathbb{R}^\ast$, yields the following isomorphic connection:
		\begin{equation}
		\begin{aligned}\label{AffSOL1,1}
			\nabla_{e_1} e_1 &=\varepsilon e_2, &\nabla_{e_2}e_1&=-e_1, &\nabla_{e_2}e_2&=-2\,e_2,&
			\nabla_{e_j}e_3&=\delta_1\,e_j,\quad\delta_1=0,1
		\end{aligned}
	\end{equation}
Observe first that $\nabla^0$ associated with this connection coincides with the flat torsion-free connection corresponding to the flat Lie algebra $\mathfrak{a}_5$ listed in Table~\ref{Flataffine}. If $\delta_1=1$, the connection given in \eqref{AffSOL1,1} coincides with the one associated with the flat Lie algebra $\h_{5,2}$ listed in Table~\ref{g2}. Otherwise, if $\delta_1=0$, this connection coincides with the one associated with the flat Lie algebra $\h_{5,1}$ listed in Table~\ref{g2}, in which case $\delta=0$.

	$\bullet$ \textbf{Case 2 : } If $\nabla^0=\nabla^{0,2}$. A straightforward calculation shows that the flatness-equations have a unique solution, given by the following flat torsion-free connection:
		\begin{equation}
		\begin{aligned}\label{AffSOL2}
			\nabla_{e_1} e_2 &=e_2+a_{32}e_3, &\nabla_{e_1}e_3&=-\tfrac{1}{a_{32}}e_2-e_3, &\nabla_{e_2}e_1&=-e_1+e_2+a_{32}e_3,\\
			\nabla_{e_2}e_2&=(-2\,a_{32}b_{23}-1)e_2-a_{32}(a_{32}b_{23}+1)e_3, &\nabla_{e_2}e_3&=b_{23}e_2, &\nabla_{e_3}e_1&=-\tfrac{1}{a_{32}}e_2-e_3,\\
			\nabla_{e_3}e_2&=b_{23}e_2, &\nabla_{e_3}e_3&=b_{23}e_3.
		\end{aligned}
	\end{equation}
	
	Consider the following automorphism
	\begin{align*}
		\Psi(e_1) &= e_1, &
		\Psi(e_2) &=e_2-\varepsilon\,a_{31}\,x e_3, &
		\Psi(e_3) &= x\, e_3,
	\end{align*}
	and applying it to the connection given in \eqref{AffSOL2}, for a suitable choice of $x \in \R^\ast$, yields the following isomorphic connection:
	\begin{equation}
		\begin{aligned}\label{AffSOL2,1}
			\nabla_{e_1} e_3 &=e_2, &\nabla_{e_2}e_1&=-e_1, &\nabla_{e_2}e_2&=-e_2,&
			\nabla_{e_2}e_3&=\delta_1\,e_2, &\nabla_{e_3}e_1&=e_2, &\nabla_{e_3}e_2&=\delta_1\,e_2,\\ \nabla_{e_3}e_3&=\delta_1\,e_3, &\delta_1&=0,1.
		\end{aligned}
	\end{equation}
	Observe that $\nabla^0$ associated with this connection coincides with the flat torsion-free connection corresponding to the flat Lie algebra $\mathfrak{a}_2$ listed in Table~\ref{Flataffine}, in which $\mu=-1$.  If $\delta_1=0$, the connection given in \eqref{AffSOL2,1} coincides with the one associated with the flat Lie algebra $\h_{2,8}$ listed in Table~\ref{g2}, in which $\delta_1=0$ and $\delta_2=1$. Otherwise, the following automorphism
	\begin{align*}
		\Psi(e_1) &=  e_1, &
		\Psi(e_2) &= e_1+ e_2, &
		\Psi(e_3) &= e_3,
	\end{align*}
	establishes an isomorphism between the connection given in \eqref{AffSOL2,1} and the one associated with the flat Lie algebra $\h_{2,8}$.

		$\bullet$ \textbf{Cases 3}  and \textbf{4 : } If $\nabla^0=\nabla^{0,3}$ or $\nabla^0=\nabla^{0,4}$. Simple and straightforward calculations show that there is no solution to the flatness-equations.

$\bullet$ \textbf{Case 5 : } If $\nabla^0=\nabla^{0,5}$.	Suppose that $\eta_4-2\,\lambda_4+1\neq0$, then $\delta=0$. A straightforward computation also leads to a unique solution, namely the following flat  torsion-free connection :
	\begin{equation}
		\begin{aligned}\label{AffSOL5}
			\nabla_{e_1} e_2 &=\lambda_4 e_1, \quad\nabla_{e_1}e_3=\nabla_{e_3}e_1=\tfrac{\lambda_4\,b_{33}+b_{32}c_{33}-b_{33}^2}{b_{32}}e_1, &\nabla_{e_2}e_1&=(\lambda_4-1)e_1, \\\nabla_{e_2}e_2&=\tfrac{\lambda_4^2-\lambda_4\,b_{33}-b_{32}c_{33}+b_{33}^2}{\lambda_4}e_2+b_{32}e_3,
			&\nabla_{e_2}e_3& =\nabla_{e_3}e_2=\tfrac{(\lambda_4\,b_{33}+b_{32}c_{33}-b_{33}^2)(\lambda_4-b_{33})}{\lambda_4\,b_{32}}e_2+b_{33}e_3,\\
			\nabla_{e_3}e_3&=\tfrac{b_{33}(\lambda_4\,b_{33}+b_{32}c_{33}-b_{33}^2)(\lambda_4-b_{33})}{\lambda_4\,b_{32}^2}e_2+c_{33}e_3
		\end{aligned}
	\end{equation}
	Assume that $\lambda_4\,b_{33}+b_{32}c_{33}-b_{33}^2\neq0$, and consider the following automorphism
	\begin{align*}
		\Psi(e_1) &= e_1, &
		\Psi(e_2) &=e_2+\lambda_4\, e_3, &
		\Psi(e_3) &= \tfrac{\lambda_4\,b_{33}+b_{32}c_{33}-b_{33}^2}{b_{32}} e_3.
	\end{align*}
	The following family of isomorphic connections (three connections) is obtained by applying the automorphism $\Psi$ to the connections given in \eqref{AffSOL5}, together with a simple change of variables:
	\begin{equation}
	\begin{aligned}\label{AffSOL5,1}
		\nabla_{e_1} e_3 &=e_1, &\nabla_{e_2}e_1&=-e_1, &\nabla_{e_2}e_2&=\mu_0\,e_2, &\nabla_{e_2}e_3&=\lambda_0 \,e_2, &\nabla_{e_3}e_1&=e_1, &\nabla_{e_3}e_2&=\lambda_0e_2, \\\nabla_{e_3}e_3&=\tfrac{\lambda_0(\lambda_0-1)}{\mu_0}e_2+e_3.
	\end{aligned}
\end{equation}
This one actually coincides with the connection associated with the flat Lie algebra $\h_{2,10}$ listed in Table~\ref{g2}, where $\mu_0=\mu$ and $\lambda_0=\mu_1$.
	
The second isomorphic connection is given by	
	\begin{equation}
		\begin{aligned}\label{AffSOL5,2}
			\nabla_{e_1} e_3 &=e_1, &\nabla_{e_2}e_1&=-e_1,  &\nabla_{e_2}e_3&=e_2, &\nabla_{e_3}e_1&=e_1, &\nabla_{e_3}e_2&=e_2, &\nabla_{e_3}e_3&=\mu_2\, e_2+e_3,\quad\mu_2\in\R.
		\end{aligned}
	\end{equation}
	This one actually coincides with the connection associated with the flat Lie algebra $\h_{1,12}$ listed in Table~\ref{g2}, where $\mu_2=\lambda_1$.
	
	The third isomorphic connection is given by	
	\begin{equation}
		\begin{aligned}\label{AffSOL5,3}
			\nabla_{e_1} e_3 &=e_1, &\nabla_{e_2}e_1&=-e_1,   &\nabla_{e_3}e_1&=e_1,  &\nabla_{e_3}e_3&=\mu_3\, e_2+e_3,\quad\mu_3\in\R.
		\end{aligned}
	\end{equation}
		This one actually coincides with the connection associated with the flat Lie algebra $\h_{1,13}$ listed in Table~\ref{g2}, where $\mu_2=\lambda_2$.

Suppose now that $\eta_4 - 2\lambda_4 + 1 = 0$ and $\delta = 0$. In this case, $\lambda_4\neq0,1$. It is straightforward to solve the flatness equations and show that they admit a unique solution, given by the following flat torsion-free connection:
\begin{equation}
	\begin{aligned}\label{AffSOL6}
		\nabla_{e_1} e_2 &=\lambda_4\,e_1, ~~\nabla_{e_1}e_3=-\tfrac{(\lambda_4-1)\lambda_4}{b_{32}}e_1, ~~\nabla_{e_2}e_1=(\lambda_4-1)e_1, ~~\nabla_{e_2}e_2=(2\,\lambda_4-1)e_2+b_{32}e_3,\\
		\nabla_{e_2}e_3&=b_{13}e_1-\tfrac{(\lambda_4-1)(\lambda_4-b_{33})}{b_{32}}e_2+b_{33}e_3,  ~~\nabla_{e_3}e_1=-\tfrac{(\lambda_4-1)\lambda_4}{b_{32}}e_1,~~\nabla_{e_3}e_2=b_{13}e_1-\tfrac{(\lambda_4-1)(\lambda_4-b_{33})}{b_{32}}e_2+b_{33}e_3,\\
		\nabla_{e_3}e_3&=-\tfrac{b_{13}(\lambda_4-b_{33})}{b_{32}}e_1-\tfrac{(\lambda_4-1)(\lambda_4-b_{33})b_{33}}{b_{32}^2}e_2+\tfrac{\lambda_4-\lambda_4^2+b_{33}^2-\lambda_4b_{33}}{b_{32}}e_3.
	\end{aligned}
\end{equation}
	Suppose $b_{33}\neq0$, and consider the following automorphism
	\begin{align*}
		\Psi(e_1) &=  e_1, &
		\Psi(e_2) &=-\tfrac{b_{13}b_{32}}{b_{33}(\lambda_4-1)} e_1+ e_2+\lambda_4\,e_3, &
		\Psi(e_3) &=-\tfrac{(\lambda_4-1)\lambda_4}{b_{32}} e_3.
	\end{align*}
	Applying it to the connection given in  \eqref{AffSOL6} and changing variables results in the following equivalent connection:
\begin{equation}
	\begin{aligned}\label{AffSOL6,1}
		\nabla_{e_1} e_3 &=e_1, &\nabla_{e_2}e_1&=-e_1, &\nabla_{e_2}e_2&=a\,e_2, &\nabla_{e_2}e_3&=b\,e_2, &\nabla_{e_3}e_1&=e_1, &\nabla_{e_3}e_2&=b\,e_2, \\\nabla_{e_3}e_3&=c\,e_2+e_3, &ac&=b-b^2.
	\end{aligned}
\end{equation}
This family of connections coincides with the three flat torsion-free connections given in \eqref{AffSOL5,1}, \eqref{AffSOL5,2}, and \eqref{AffSOL5,3}.

Suppose that $b_{33}=0$ and $(\lambda_4-2)b_{13}\neq0$. Consider the following automorphism
		\begin{align*}
		\Psi(e_1) &= \tfrac{(\lambda_4-1)^2}{(\lambda_4-2)b_{13}b_{32}} e_1, &
		\Psi(e_2) &= e_2+\lambda_4\,e_3, &
		\Psi(e_3) &=-\tfrac{(\lambda_4-1)\lambda_4}{b_{32}} e_3.
	\end{align*}
	Applying it to the connection given in  \eqref{AffSOL6} and changing variables results in the following equivalent connection:
\begin{equation}
	\begin{aligned}\label{AffSOL6,2}
		\nabla_{e_1} e_3 &=e_1, &\nabla_{e_2}e_1&=-e_1, &\nabla_{e_2}e_2&=e_1-e_2, &\nabla_{e_2}e_3&=\tfrac{1}{(\lambda_4-2)\,\lambda_4}e_1+e_2, \\\nabla_{e_3}e_1&=e_1, &\nabla_{e_3}e_2&=\tfrac{1}{(\lambda_4-2)\,\lambda_4}e_1+e_2, &\nabla_{e_3}e_3&=-\tfrac{1}{(\lambda_4-2)\,\lambda_4}e_1+e_3.
	\end{aligned}
\end{equation}
	This connection coincides with the one associated with the flat Lie algebra $\h_{3,6}$ listed in Table~\ref{g2}, in which case $\lambda_2=\tfrac{1}{(\lambda_4-2)\,\lambda_4}$.

Suppose that $b_{33}=0$ and $(\lambda_4-2)b_{13}=0$. We begin with the case $\lambda_4=2$. Consider the following automorphism 
\begin{align*}
	\Psi(e_1) &=x\,  e_1, &
	\Psi(e_2) &= e_2+2\,e_3, &
	\Psi(e_3) &=-\tfrac{2}{b_{32}} e_3.
\end{align*}
Applying it to the connection given in \eqref{AffSOL6}, for a suitable parameter $x\in\mathbb{R}^\ast$, yields the following equivalent connection:
\begin{equation}
	\begin{aligned}\label{AffSOL6,3}
		\nabla_{e_1} e_3 &=e_1, &\nabla_{e_2}e_1&=-e_1, &\nabla_{e_2}e_2&=-e_2, &\nabla_{e_2}e_3&=\delta_0\,e_1+e_2, \\\nabla_{e_3}e_1&=e_1, &\nabla_{e_3}e_2&=\delta_0\,e_1+e_2, &\nabla_{e_3}e_3&=-\delta_0\,e_1+e_3,\quad\delta_0=0,1.
	\end{aligned}
\end{equation}
If $\delta_0=1$, this connection coincides with the one associated with the flat Lie algebra $\h_{2,14}$. Otherwise, if $\delta_0=0$, the connection \eqref{AffSOL6,3} coincides with the one associated with the flat Lie algebra $\h_{2,10}$, in which case $\mu_1=1$ and $\mu=-1$.

Suppose that $b_{33}=0$ and $(\lambda_4-2)b_{13}=0$. Now let's examine the second case, namely  $b_{13}=0$. Consider the following automorphism:
\begin{align*}
	\Psi(e_1) &= e_1, &
	\Psi(e_2) &= e_2+\lambda_4\,e_3, &
	\Psi(e_3) &=-\tfrac{(\lambda_4-1)\lambda_4}{b_{32}} e_3.
\end{align*}
Applying it to the connection given in  \eqref{AffSOL6}, yields the following equivalent connection:
\begin{equation}
	\begin{aligned}\label{AffSOL6,4}
		\nabla_{e_1} e_3 &=e_1, &\nabla_{e_2}e_1&=-e_1, &\nabla_{e_2}e_2&=-e_2, &\nabla_{e_2}e_3&=e_2, \\\nabla_{e_3}e_1&=e_1, &\nabla_{e_3}e_2&=e_2, &\nabla_{e_3}e_3&=e_3.
	\end{aligned}
\end{equation}
In fact, this connection coincides with the previous one in \eqref{AffSOL6,3}, in which case $\delta_0=0$, so no further analysis is needed.

Suppose now that $\eta_4 - 2\lambda_4 + 1 = 0$ and $\delta = 1$. In this case, $\lambda_4\neq0,1$. Solving flatness-equations and showing they admit a unique solution is straightforward, given by the following flat torsion-free connection:
\begin{equation}
	\begin{aligned}\label{AffSOL7}
		\nabla_{e_1} e_2 &=\lambda_4\,e_1, &\nabla_{e_1}e_3&=-\tfrac{(\lambda_4-1)\,\lambda_4}{b_{32}}e_1, &\nabla_{e_2}e_1&=(\lambda_4-1)e_1, \\\nabla_{e_2}e_2&=e_1+(2\,\lambda_4-1)e_2+b_{32}e_3, &\nabla_{e_2}e_3&=b_{13}e_1-\tfrac{(\lambda_4-1)\,\lambda_4}{b_{32}}e_2, &\nabla_{e_3}e_1&=-\tfrac{(\lambda_4-1)\,\lambda_4}{b_{32}}e_1,\\ \nabla_{e_3}e_2&=b_{13}e_1-\tfrac{(\lambda_4-1)\,\lambda_4}{b_{32}}e_2, &\nabla_{e_3}e_3&=-\tfrac{b_{13}\,\lambda_4}{b_{32}}e_1-\tfrac{(\lambda_4-1)\,\lambda_4}{b_{32}}e_3.
	\end{aligned}
\end{equation}
Suppose $b_{13}\neq0$. Consider the following automorphism
\begin{align*}
	\Psi(e_1) &=\tfrac{(\lambda_4-1)^2\,\lambda_4}{b_{13}b_{32}} e_1, &
	\Psi(e_2) &= e_2+\lambda_4\,e_3, &
	\Psi(e_3) &=-\tfrac{(\lambda_4-1)\lambda_4}{b_{32}} e_3.
\end{align*}
Applying it to the connection given in  \eqref{AffSOL7}, yields the following equivalent connection:
\begin{equation}
	\begin{aligned}\label{AffSOL7,1}
		\nabla_{e_1} e_3 &=e_1, &\nabla_{e_2}e_1&=-e_1, &\nabla_{e_2}e_2&=\lambda_0\,e_1-e_2, &\nabla_{e_2}e_3&=e_1+e_2, &\nabla_{e_3}e_1&=e_1,\\
		\nabla_{e_3}e_2&=e_1+e_2, &\nabla_{e_3}e_3&=-e_1+e_3.
			\end{aligned}
\end{equation}
If $\lambda_0=0$, this connection coincides with the one associated with the flat Lie algebra $\h_{2,14}$ listed in Table~\ref{g2}. Otherwise, if $\lambda_0\neq0$, one can easily verify that this connection is in fact isomorphic to the one associated with the flat Lie algebra $\h_{3,6}$ listed in Table~\ref{g2}, in which case $\lambda_2=\frac{1}{\lambda_0}$, via the following automorphism:
\begin{align*}
	\Psi(e_1) &=\tfrac{1}{\lambda_0} e_1, &
	\Psi(e_2) &= e_2, &
	\Psi(e_3) &=e_3.
\end{align*}

Suppose $b_{13}=0$. Consider the following automorphism
\begin{align*}
	\Psi(e_1) &=e_1, &
	\Psi(e_2) &= e_2+\lambda_4\,e_3, &
	\Psi(e_3) &=-\tfrac{(\lambda_4-1)\lambda_4}{b_{32}} e_3.
\end{align*}
Applying it to the connection given in  \eqref{AffSOL7}, yields the following equivalent connection:
\begin{equation}
	\begin{aligned}\label{AffSOL7,2}
		\nabla_{e_1} e_3 &=e_1, &\nabla_{e_2}e_1&=-e_1, &\nabla_{e_2}e_2&=e_1-e_2, &\nabla_{e_2}e_3&=e_2, &\nabla_{e_3}e_1&=e_1,\\
		\nabla_{e_3}e_2&=e_2, &\nabla_{e_3}e_3&=e_3.
	\end{aligned}
\end{equation}
In fact, this connection corresponds with that associated with the flat Lie algebra  $\h_{3,5}$ listed in Table~~\ref{g2}.

 With this procedure, one can start with a no flat torsion-free connection in dimension $2$ and extend it, as in \eqref{exten}, via an isomorphism, to an upper block flat torsion-free connection in dimension $3$.

\end{proof}

\begin{co}
	With the notations as above, among the flat Lie algebras on $2\G_{2,1}\oplus\G_1$, we have
	\begin{enumerate}
		\item[i)] Associative algebras$:$\hspace{0.275cm} $\h_{1,1}^{\lambda=1}$, $\h_{1,5}^{\lambda=1}$, $\h_{2,1}^{\mu=-1}$, $\h_{2,2}^{\mu=-1}$, $\h_{2,3}$, $\h_{2,10}^{\mu=-1, \mu_1=1}$.
		\item[ii)] Novikov algebras$:$\hspace{0.73cm} $\h_{1,1}$, $\h_{1,5}$, $\h_{4,1}$, $\h_{4,2}$.
		\item[iii)] Bi-symmetric algebras$:$ $\h_{1,1}^{\lambda=1}$, $\h_{1,5}^{\lambda=1}$, $\h_{2,1}^{\mu=-1}$, $\h_{2,2}^{\mu=-1}$, $\h_{2,3}$, $\h_{3,1}$, $\h_{3,2}$, $\h_{3,5}$, $\h_{4,1}$, $\h_{4,2}$.
		\item[iv)] Complete algebras$:$\hspace{0.62cm}$\h_{1,1}^{\lambda=0,\delta=0}$, $\h_{1,5}^{\lambda=0}$, $\h_{1,10}^{\lambda_1=0}$.
	\end{enumerate}
\end{co}

\subsubsection{Flat Lie algebra $\G_{3,j}$}

\begin{Le}\label{Lemg3j}
Let $\nabla^0$ be a torsion-free connection on $\R^2$ extended to a flat torsion-free connection on a three-dimensional solvable Lie algebra $\G_{3,j}$, $j = 1, \ldots, 5$. Then $\nabla^0$ is equivalent to  one of the following  torsion-free connections:\\
$\triangleright$ Lie algebra $\G_{3,1} :$
\begin{small}
		\[
	\begin{alignedat}{4}
		\nabla^1_{e_1}e_1&=\lambda_1\,e_1+\mu_1\,e_2,\quad\quad &\nabla^1_{e_1}e_2&=\nu_1\,e_2,\quad\quad &\nabla^1_{e_2}e_1&=\nu_1\,e_2,\quad\quad &\nabla^1_{e_2}e_2&=e_1+e_2.\\
		\nabla^2_{e_1}e_1&=\lambda_2\,e_1+\mu_2\,e_2,\quad\quad &\nabla^2_{e_1}e_2&=\varepsilon_1\,e_2,\quad\quad &\nabla^2_{e_2}e_1&=\varepsilon_1\,e_2, \quad\quad&\nabla^2_{e_2}e_2&=e_1.\\
		\nabla^3_{e_1}e_1&=\lambda_3\,e_1+e_2,\quad\quad &\nabla^3_{e_1}e_2&=0,\quad\quad &\nabla^3_{e_2}e_1&=0,\quad\quad &\nabla^3_{e_2}e_2&=e_1.\\
		\nabla^4_{e_1}e_1&=\delta_\varepsilon\,e_1,\quad\quad &\nabla^4_{e_1}e_2&=0,\quad\quad &\nabla^4_{e_2}e_1&=0,\quad\quad &\nabla^4_{e_2}e_2&=e_1.\\
			\nabla^5_{e_1}e_1&=\lambda_5\,e_2,\quad\quad &\nabla^5_{e_1}e_2&=e_1+e_2,\quad\quad &\nabla^5_{e_2}e_1&=e_1+e_2,\quad\quad &\nabla^5_{e_2}e_2&=\mu_5\,e_2.\\
			\nabla^6_{e_1}e_1&=\delta_\varepsilon\,e_2,\quad\quad &\nabla^6_{e_1}e_2&=e_1,\quad\quad &\nabla^6_{e_2}e_1&=e_1,\quad\quad &\nabla^6_{e_2}e_2&=\mu_6\,e_2.\\
			\nabla^7_{e_1}e_1&=e_1+\lambda_7\,e_2, \quad\quad&\nabla^7_{e_1}e_2&=0,\quad\quad &\nabla^7_{e_2}e_1&=0,\quad\quad &\nabla^7_{e_2}e_2&=e_2.\\
			\nabla^8_{e_1}e_1&=\delta_{\varepsilon_1}\,e_2,\quad\quad &\nabla^8_{e_1}e_2&=0,\quad\quad &\nabla^8_{e_2}e_1&=0,\quad\quad &\nabla^8_{e_2}e_2&=e_2.\\
			\nabla^9_{e_1}e_1&=\lambda_9\,e_1, \quad\quad &\nabla^9_{e_1}e_2&=e_2,\quad\quad  &\nabla^9_{e_2}e_1&=e_2,\quad\quad  &\nabla^9_{e_2}e_2&=0.\\
				\nabla^{10}_{e_1}e_1&=2\,e_1+e_2,\quad\quad  &\nabla^{10}_{e_1}e_2&=e_2,\quad\quad  &\nabla^{10}_{e_2}e_1&=e_2, &\nabla^{10}_{e_2}e_2&=0.\\
				\nabla^{11}_{e_1}e_1&=e_1,\quad\quad  &\nabla^{11}_{e_1}e_2&=0,\quad\quad  &\nabla^{11}_{e_2}e_1&=0,\quad\quad  &\nabla^{11}_{e_2}e_2&=0.\\
				\nabla^{12}_{e_1}e_1&=\delta\,e_2,\quad\quad  &\nabla^{12}_{e_1}e_2&=0,\quad\quad  &\nabla^{12}_{e_2}e_1&=0,\quad\quad  &\nabla^{12}_{e_2}e_2&=0.
	\end{alignedat}
	\]
	
\end{small}
$\triangleright$ Lie algebra $\G_{3,2} :$
\begin{small}
	\[
	\begin{alignedat}{4}
	\nabla^1_{e_1}e_1&=e_2,\quad\quad &\nabla^1_{e_1}e_2&=\lambda_1\,e_1+\mu_1\,e_2,\quad\quad &\nabla^1_{e_2}e_1&=\lambda_1\,e_1+\mu_1\,e_2,\quad\quad &\nabla^1_{e_2}e_2&=\nu_1\,e_1+\eta_1\,e_2.\\
		\nabla^2_{e_1}e_1&=\lambda_2\,e_1,\quad\quad &\nabla^2_{e_1}e_2&=e_2,\quad\quad &\nabla^2_{e_2}e_1&=e_2,\quad\quad &\nabla^2_{e_2}e_2&=\nu_2\,e_1+\eta_2\,e_2.\\
			\nabla^3_{e_1}e_1&=e_1,\quad\quad &\nabla^3_{e_1}e_2&=0,\quad\quad &\nabla^3_{e_2}e_1&=0,\quad\quad &\nabla^3_{e_2}e_2&=\nu_3\,e_1+\eta_3\,e_2.\\
	\nabla^4_{e_1}e_1&=e_1,\quad\quad &\nabla^4_{e_1}e_2&=\lambda_4\,e_1+e_2,\quad\quad &\nabla^4_{e_2}e_1&=\lambda_4\,e_1+e_2,\quad\quad &\nabla^4_{e_2}e_2&=\nu_4\,e_1.\\
	\nabla^5_{e_1}e_1&=0,\quad\quad &\nabla^5_{e_1}e_2&=e_1,\quad\quad &\nabla^5_{e_2}e_1&=e_1,\quad\quad &\nabla^5_{e_2}e_2&=\eta_5\,e_2.\\
	\nabla^6_{e_1}e_1&=0,\quad\quad &\nabla^6_{e_1}e_2&=e_1,\quad\quad &\nabla^6_{e_2}e_1&=e_1,\quad\quad &\nabla^6_{e_2}e_2&=\eta_6\,e_2+2\,e_2.\\
	\nabla^7_{e_1}e_1&=0,\quad\quad &\nabla^7_{e_1}e_2&=0,\quad\quad &\nabla^7_{e_2}e_1&=0,\quad\quad &\nabla^7_{e_2}e_2&=e_2.\\
	\nabla^8_{e_1}e_1&=0,\quad\quad &\nabla^8_{e_1}e_2&=0,\quad\quad &\nabla^8_{e_2}e_1&=0,\quad\quad &\nabla^8_{e_2}e_2&=\delta\,e_1.		
	\end{alignedat}
	\]
\end{small}
$\triangleright$ Lie algebra $\G_{3,3} :$
\begin{small}
	\[
	\begin{alignedat}{4}
		\nabla^{1}_{e_1}e_1 &= 0, &\quad\quad \nabla^{1}_{e_1}e_2 &= 0,&\quad\quad
		\nabla^{1}_{e_2}e_1 &=  0, &\quad\quad \nabla^{1}_{e_2}e_2 &=\delta\,e_1.\\
		\nabla^2_{e_1}e_1 &= e_2, &\quad \nabla^2_{e_1}e_2 &= 0,&\quad
		\nabla^2_{e_2}e_1 &=  0, &\quad \nabla^2_{e_2}e_2 &=0.\\
		\nabla^3_{e_1}e_1 &= e_2, &\quad\quad \nabla^3_{e_1}e_2 &= 0,&\quad\quad
		\nabla^3_{e_2}e_1 &=  0, &\quad\quad \nabla^3_{e_2}e_2 &=e_1.\\
		\nabla^4_{e_1}e_1 &=\lambda_1\,  e_1+e_2, &\quad \nabla^1_{e_1}e_2 &= e_2,&\quad
		\nabla^4_{e_2}e_1 &=  e_2, &\quad \nabla^4_{e_2}e_2 &= \mu_1\,e_1+\nu_1\,e_2.\\
		\nabla^5_{e_1}e_1 &=\lambda_2\,  e_1, &\quad \nabla^5_{e_1}e_2 &= e_2,&\quad
		\nabla^5_{e_2}e_1 &=  e_2, &\quad \nabla^5_{e_2}e_2 &= \mu_2\,e_1+e_2.\\
		\nabla^6_{e_1}e_1 &=\lambda_3\,  e_1, &\quad \nabla^6_{e_1}e_2 &= e_2,&\quad
		\nabla^6_{e_2}e_1 &=  e_2, &\quad \nabla^6_{e_2}e_2 &= \delta_{\varepsilon_1}\,e_1.\\
		\nabla^7_{e_1}e_1 &= e_1+e_2, &\quad \nabla^7_{e_1}e_2 &= 0,&\quad
		\nabla^7_{e_2}e_1 &=  0, &\quad \nabla^7_{e_2}e_2 &= \mu_4\,e_1+\nu_4\,e_2.
	\end{alignedat}
	\]
\end{small}
\\\\
$\triangleright$ Lie algebra $\G_{3,4} :$
\begin{small}
\[
\begin{alignedat}{4}
	\nabla^1_{e_1}e_1&=\lambda_1\,e_1+\mu_1\,e_2,\quad\quad &\nabla^1_{e_1}e_2&=e_1+e_2,\quad\quad &\nabla^1_{e_2}e_1&=e_1+e_2,\quad\quad &\nabla^1_{e_2}e_2&=\nu_1\,e_1+\eta_1\,e_2.\\
	\nabla^2_{e_1}e_1&=e_1+\mu_2\,e_2,\quad\quad &\nabla^2_{e_1}e_2&=e_1,\quad\quad &\nabla^2_{e_2}e_1&=e_1,\quad\quad &\nabla^2_{e_2}e_2&=\nu_2\,e_1+\eta_2\,e_2.\\
	\nabla^3_{e_1}e_1&=\varepsilon_1\,e_2,\quad\quad &\nabla^3_{e_1}e_2&=e_1,\quad\quad &\nabla^3_{e_2}e_1&=e_1,\quad\quad &\nabla^3_{e_2}e_2&=\nu_3\,e_1+\eta_3\,e_2.\\
	\nabla^4_{e_1}e_1&=0,\quad\quad &\nabla^4_{e_1}e_2&=e_1,\quad\quad &\nabla^4_{e_2}e_1&=e_1,\quad\quad &\nabla^4_{e_2}e_2&=e_1+\eta_4\,e_2.
	\\
	\nabla^5_{e_1}e_1&=0,\quad\quad &\nabla^5_{e_1}e_2&=e_1,\quad\quad &\nabla^5_{e_2}e_1&=e_1,\quad\quad &\nabla^5_{e_2}e_2&=\eta_5\,e_2.\\	\nabla^6_{e_1}e_1&=\lambda_6\,e_1+e_2,\quad\quad &\nabla^6_{e_1}e_2&=e_2,\quad\quad &\nabla^6_{e_2}e_1&=e_2,\quad\quad &\nabla^6_{e_2}e_2&=\nu_6\,e_1+\eta_6\,e_2.
	\\	\nabla^7_{e_1}e_1&=\lambda_7\,e_1,\quad\quad &\nabla^7_{e_1}e_2&=e_2,\quad\quad &\nabla^7_{e_2}e_1&=e_2,\quad\quad &\nabla^7_{e_2}e_2&=\varepsilon_2\,e_1+\eta_7\,e_2.
	\\	\nabla^8_{e_1}e_1&=\lambda_8\,e_1,\quad\quad &\nabla^8_{e_1}e_2&=e_2,\quad\quad &\nabla^8_{e_2}e_1&=e_2,\quad\quad &\nabla^8_{e_2}e_2&=\delta_1\,e_2.
	\\	\nabla^9_{e_1}e_1&=e_1+e_2,\quad\quad &\nabla^9_{e_1}e_2&=0,\quad\quad &\nabla^9_{e_2}e_1&=0,\quad\quad &\nabla^9_{e_2}e_2&=\nu_9\,e_1+\eta_9\,e_2.	\\	\nabla^{10}_{e_1}e_1&=e_1,\quad\quad &\nabla^{10}_{e_1}e_2&=0,\quad\quad &\nabla^{10}_{e_2}e_1&=0,\quad\quad &\nabla^{10}_{e_2}e_2&=\varepsilon_2\,e_1+\eta_{10}\,e_2.\\	\nabla^{11}_{e_1}e_1&=e_2,\quad\quad &\nabla^{11}_{e_1}e_2&=0,\quad\quad &\nabla^{11}_{e_2}e_1&=0,\quad\quad &\nabla^{11}_{e_2}e_2&=e_1+\eta_{11}\,e_2.\\
	\nabla^{12}_{e_1}e_1&=e_2,\quad\quad &\nabla^{12}_{e_1}e_2&=0,\quad\quad &\nabla^{12}_{e_2}e_1&=0,\quad\quad &\nabla^{12}_{e_2}e_2&=\delta_{\varepsilon}\,e_2.\\
		\nabla^{13}_{e_1}e_1&=0,\quad\quad &\nabla^{13}_{e_1}e_2&=0,\quad\quad &\nabla^{13}_{e_2}e_1&=0,\quad\quad &\nabla^{13}_{e_2}e_2&=e_1+e_2.\\
		\nabla^{14}_{e_1}e_1&=0,\quad\quad &\nabla^{14}_{e_1}e_2&=0,\quad\quad &\nabla^{14}_{e_2}e_1&=0,\quad\quad &\nabla^{14}_{e_2}e_2&=e_1.\\	\nabla^{15}_{e_1}e_1&=0,\quad\quad &\nabla^{15}_{e_1}e_2&=0,\quad\quad &\nabla^{15}_{e_2}e_1&=0,\quad\quad &\nabla^{15}_{e_2}e_2&=\delta_2\,e_2.
\end{alignedat}
\]
\end{small}
$\triangleright$ Lie algebra $\G_{3,5} :$
\begin{small}
\[
\begin{alignedat}{4}
\nabla^{1}_{e_1}e_1 &=\lambda_1\, e_1+\mu_1\,e_2, &\quad\quad \nabla^{1}_{e_1}e_2 &= e_2,&\quad\quad
\nabla^{1}_{e_2}e_1 &= e_2, &\quad\quad \nabla^{1}_{e_2}e_2 &=\nu_1\, e_1+\eta_1 \, e_2.\\
\nabla^{2}_{e_1}e_1 &=\lambda_2\, e_1+e_2, &\quad\quad \nabla^{2}_{e_1}e_2 &=0,&\quad\quad
\nabla^{2}_{e_2}e_1 &=0, &\quad\quad \nabla^{2}_{e_2}e_2 &=\nu_2\, e_1+\eta_2 \, e_2.\\
\nabla^{3}_{e_1}e_1 &=e_1, &\quad\quad \nabla^{3}_{e_1}e_2 &=0,&\quad\quad
\nabla^{3}_{e_2}e_1 &=0, &\quad\quad \nabla^{3}_{e_2}e_2 &=\nu_3\, e_1+\eta_3 \, e_2.\\
\nabla^{4}_{e_1}e_1 &=0, &\quad\quad \nabla^{4}_{e_1}e_2 &=0,&\quad\quad
\nabla^{4}_{e_2}e_1 &=0, &\quad\quad \nabla^{4}_{e_2}e_2 &= e_1+\eta_4 \, e_2.\\
\nabla^{5}_{e_1}e_1 &=0, &\quad\quad \nabla^{5}_{e_1}e_2 &=0,&\quad\quad
\nabla^{5}_{e_2}e_1 &=0, &\quad\quad \nabla^{5}_{e_2}e_2 &=\delta \, e_2.
\end{alignedat}
\]
\end{small}
Where, $\lambda_j,\mu_j,\nu_j,\eta_j\in\R$, $\varepsilon_j=\pm1$,  $\delta_j=0,1$, and $\delta_{\varepsilon}=0,\pm1$.\\\\

\end{Le}
\begin{proof}
	Let $\{e_1,e_2\}$ be a basis of $\R^2$, and let $\nabla$ be a non-flat, torsion-free connection on $\R^2$. Then, with respect to the basis $\{e_1,e_2\}$, $\nabla$ is given by
	\begin{equation}\label{conabelian1}
		\nabla_{e_1}=\begin{pmatrix}
			a_{11}&a_{12}\\
			a_{21}&a_{22}
		\end{pmatrix},
		\qquad
		\nabla_{e_2}=\begin{pmatrix}
			a_{12}&b_{12}\\
			a_{22}&b_{22}
		\end{pmatrix}.
	\end{equation}
	Every solvable Lie algebra $\G_{3,j}$ can be viewed as a semidirect sum $\mathbb{R}\ltimes\mathbb{R}^2$.  Lemma~\ref{Auto} describes precisely how an automorphism of these algebras must satisfy certain conditions in order to be restricted to $\mathbb{R}^2$. Let $\Phi$ be an automorphism of $\mathbb{R}^2$. According to Lemma~\ref{Auto}, $\Phi$ has the following form:
	\begin{align*}
		\Phi_{3,1}&=\begin{pmatrix}
			x&0\\
			z&t
		\end{pmatrix},&\Phi_{3,2}&=\begin{pmatrix}
			x&y\\
			0&x
		\end{pmatrix},&\Phi_{3,3}&=\begin{pmatrix}
			x&y\\
			z&t
		\end{pmatrix},\\
		\quad\Phi_{3,4}^{\alpha\neq-1}&=\begin{pmatrix}
			x&0\\
			0&t
		\end{pmatrix}, &\Phi_{3,4}^{\alpha=-1}&=\begin{pmatrix}
		x&0\\
		0&t
		\end{pmatrix},\begin{pmatrix}
			0&y\\
			z&0
		\end{pmatrix},&\Phi_{3,5}&=\begin{pmatrix}
			x&-z\\
			z&x
		\end{pmatrix}.
	\end{align*}
	Let $\nabla^{ij}_{e_\ell}$ denote the coefficients of
	\[
	\Phi_{3,k}\circ\nabla_{\Phi_{3,k}^{-1}(e_\ell)}\circ\Phi_{3,k}^{-1},
	\quad\quad \text{for all } x\in\R^2.
	\]
	$\bullet$ \textbf{Lie algebra $\G_{3,1} :$ }  
	
	$\triangleright$ Suppose that $b_{12}\neq0$. If $a_{12}+b_{22}\neq0$, consider the following automorphism:
	\begin{align*}
		\Phi(e_1) &=\tfrac{(a_{12}+b_{22})^2}{b_{12}} e_1+\tfrac{a_{12}(a_{12}+b_{22})}{b_{12}} e_2, &
		\Phi(e_2) &=(a_{12}+b_{22})\,e_2.
	\end{align*}
	Applying the automorphism $\Phi$  to the connection given in \eqref{conabelian1} yields the following equivalent connection:	
		\begin{equation}\label{g3,1, c1}
			\begin{aligned}
				\nabla_{e_1}e_1&=\lambda_1\,e_1+\mu_1\,e_2, &\nabla_{e_1}e_2&=\nu_1\,e_2, &\nabla_{e_2}e_1&=\nu_1\,e_2, &\nabla_{e_2}e_2&=e_1+e_2.
			\end{aligned}
		\end{equation}
		If $a_{12}+b_{22}=0$, and $a_{12}^2+a_{22}b_{12}\neq0$. Set $b_{22}=-a_{12}$, and consider the following automorphism
		\begin{align*}
			\Phi(e_1) &=\tfrac{t^2}{b_{12}} e_1+\tfrac{t\,a_{12}}{b_{12}}e_2, &
			\Phi(e_2) &=t\,e_2.
		\end{align*}
		For a suitable choice of the parameter $t\in\R^\ast$, applying the automorphism $\Phi$  to the connection given in \eqref{conabelian1} yields the following equivalent connection:
		\begin{equation}\label{g3,1, c2}
			\begin{aligned}
				\nabla_{e_1}e_1&=\lambda_2\,e_1+\mu_2\,e_2, &\nabla_{e_1}e_2&=\varepsilon_1\,e_2, &\nabla_{e_2}e_1&=\varepsilon_1\,e_2, &\nabla_{e_2}e_2&=e_1.
			\end{aligned}
		\end{equation}

		If $a_{12}+b_{22}=0$,  $a_{12}^2+a_{22}b_{12}=0$ and $a_{1 1}a_{1 2} + a_{2 1}b_{1 2}\neq0$. Set $b_{22}=-a_{12}$ and $a_{22}=\frac{a_{12}^2}{b_{12}}$, and consider the following automorphism
		\begin{align*}
			\Phi(e_1) &=\tfrac{t^2}{b_{12}} e_1+\tfrac{t\,a_{12}}{b_{12}}e_2, &
			\Phi(e_2) &=t\,e_2.
		\end{align*}
		For a suitable choice of the parameter $t\in\R^\ast$, applying the automorphism $\Phi$  to the connection given in \eqref{conabelian1} yields the following equivalent connection:
		\begin{equation}\label{g3,1, c3}
			\begin{aligned}
				\nabla_{e_1}e_1&=\lambda_3\,e_1+e_2, &\nabla_{e_1}e_2&=0, &\nabla_{e_2}e_1&=0, &\nabla_{e_2}e_2&=e_1.
			\end{aligned}
		\end{equation}

			If $a_{12}+b_{22}=0$,  $a_{12}^2+a_{22}b_{12}=0$ and $a_{1 1}a_{1 2} + a_{2 1}b_{1 2}=0$. Set $b_{22}=-a_{12}$,  $a_{22}=\frac{a_{12}^2}{b_{12}}$ and $a_{21}=-\frac{a_{11}a_{12}}{b_{12}}$, and consider the following automorphism
		\begin{align*}
			\Phi(e_1) &=\tfrac{t^2}{b_{12}} e_1+\tfrac{t\,a_{12}}{b_{12}}e_2, &
			\Phi(e_2) &=t\,e_2.
		\end{align*}
		For a suitable choice of the parameter $t\in\R^\ast$, applying the automorphism $\Phi$  to the connection given in \eqref{conabelian1} yields the following equivalent connection:
		\begin{equation}\label{g3,1, c4}
			\begin{aligned}
				\nabla_{e_1}e_1&=\delta_\varepsilon\,e_1, &\nabla_{e_1}e_2&=0, &\nabla_{e_2}e_1&=0, &\nabla_{e_2}e_2&=e_1.
			\end{aligned}
		\end{equation}

$\triangleright$ Suppose that $b_{12}=0$. If $a_{12}\neq0$ and $(a_{1 1} + 2\,a_{2 2})a_{1 2} - a_{1 1}b_{2 2}\neq0$, consider the following automorphism:
\begin{align*}
	\Phi(e_1) &=x\,e_1+\tfrac{a_{11}}{2}e_2, &
	\Phi(e_2) &=a_{12}e_2.
\end{align*}
For a suitable choice of the parameter $x\in\R^\ast$, applying the automorphism $\Phi$  to the connection given in \eqref{conabelian1} yields the following equivalent connection:	
\begin{equation}\label{g3,1, c5}
	\begin{aligned}
		\nabla_{e_1}e_1&=\lambda_5\,e_2, &\nabla_{e_1}e_2&=e_1+e_2, &\nabla_{e_2}e_1&=e_1+e_2, &\nabla_{e_2}e_2&=\mu_5\,e_2.
	\end{aligned}
\end{equation}

If $a_{12}\neq0$, $(a_{1 1} + 2\,a_{2 2})a_{1 2} - a_{1 1}b_{2 2}=0$ and $(2\,a_{1 2} - b_{2 2})a_{1 1}^2 + 4\,a_{2 1}a_{1 2}^2\neq0$. Set $a_{22}=\frac{a_{11}(b_{22}-a_{12})}{2\,a_{12}}$, and consider the following automorphism:
\begin{align*}
	\Phi(e_1) &=x\,e_1+\tfrac{a_{11}}{2}e_2, &
	\Phi(e_2) &=a_{12}e_2.
\end{align*}
For a suitable choice of the parameter $x\in\R^\ast$, applying the automorphism $\Phi$  to the connection given in \eqref{conabelian1} yields the following equivalent connection:	
\begin{equation}\label{g3,1, c6}
	\begin{aligned}
		\nabla_{e_1}e_1&=\delta_\varepsilon\,e_2, &\nabla_{e_1}e_2&=e_1, &\nabla_{e_2}e_1&=e_1, &\nabla_{e_2}e_2&=\mu_6\,e_2.
	\end{aligned}
\end{equation}		

$\triangleright$ Suppose that $b_{12}=a_{12}=0$. If $b_{22}\neq0$ and $a_{11}\neq0$, consider the following automorphism:
\begin{align*}
	\Phi(e_1) &=a_{11}\,e_1+a_{22}\,e_2, &
	\Phi(e_2) &=b_{22}\,e_2.
\end{align*}
Applying the automorphism $\Phi$  to the connection given in \eqref{conabelian1} yields the following equivalent connection:	
\begin{equation}\label{g3,1, c7}
	\begin{aligned}
		\nabla_{e_1}e_1&=e_1+\lambda_7\,e_2, &\nabla_{e_1}e_2&=0, &\nabla_{e_2}e_1&=0, &\nabla_{e_2}e_2&=e_2.
	\end{aligned}
\end{equation}

 If $b_{22}\neq0$ and $a_{11}=0$, consider the following automorphism:
\begin{align*}
	\Phi(e_1) &=x\,\,e_1+a_{22}\,e_2, &
	\Phi(e_2) &=b_{22}\,e_2.
\end{align*}
For a suitable choice of the parameter $x\in\R^\ast$, applying the automorphism $\Phi$  to the connection given in \eqref{conabelian1} yields the following equivalent connection:	
\begin{equation}\label{g3,1, c8}
	\begin{aligned}
		\nabla_{e_1}e_1&=\delta_{\varepsilon_1}\,e_2, &\nabla_{e_1}e_2&=0, &\nabla_{e_2}e_1&=0, &\nabla_{e_2}e_2&=e_2.
	\end{aligned}
\end{equation}

$\triangleright$ Suppose that $b_{12}=a_{12}=b_{22}=0$. If $a_{22}\neq0$ and $a_{11}-2\,a_{22}\neq0$, consider the following automorphism:
\begin{align*}
	\Phi(e_1) &=a_{22}\,e_1+\tfrac{a_{21}}{2\,a_{22}-a_{11}} e_2, &
	\Phi(e_2) &=e_2.
\end{align*}
Applying the automorphism $\Phi$  to the connection given in \eqref{conabelian1} yields the following equivalent connection:	
\begin{equation}\label{g3,1, c9}
	\begin{aligned}
		\nabla_{e_1}e_1&=\lambda_9\,e_1, &\nabla_{e_1}e_2&=e_2, &\nabla_{e_2}e_1&=e_2, &\nabla_{e_2}e_2&=0.
	\end{aligned}
\end{equation}	

If $a_{22}\neq0$ and $a_{11}-2\,a_{22}=0$. Set $a_{11}=2\,a_{22}$, and consider the following automorphism:
\begin{align*}
	\Phi(e_1) &=a_{22}\,e_1, &
	\Phi(e_2) &=t\,e_2.
\end{align*}
For a suitable choice of the parameter $t\in\R^\ast$, applying the automorphism $\Phi$  to the connection given in \eqref{conabelian1} yields the following equivalent connection:	
\begin{equation}\label{g3,1, c10}
	\begin{aligned}
		\nabla_{e_1}e_1&=2\,e_1+\delta\,e_2, &\nabla_{e_1}e_2&=e_2, &\nabla_{e_2}e_1&=e_2, &\nabla_{e_2}e_2&=0.
	\end{aligned}
\end{equation}	
Note that if $\delta=0$, then this torsion-free connection coincides with the one given in~\eqref{g3,1, c9}, in which case $\lambda_9=2$. Therefore, for the equivalence classification, we may assume that $\delta=1$.

$\triangleright$ Suppose that $b_{12}=a_{12}=b_{22}=a_{22}=0$. If $a_{11}\neq0$, consider the following automorphism:
\begin{align*}
	\Phi(e_1) &=a_{11}\,e_1-\tfrac{a_{21}}{a_{11}} e_2, &
	\Phi(e_2) &=e_2.
\end{align*}
Applying the automorphism $\Phi$  to the connection given in \eqref{conabelian1} yields the following equivalent connection:	
\begin{equation}\label{g3,1, c11}
	\begin{aligned}
		\nabla_{e_1}e_1&=e_1, &\nabla_{e_1}e_2&=0, &\nabla_{e_2}e_1&=0, &\nabla_{e_2}e_2&=0.
	\end{aligned}
\end{equation}	
If $a_{11}=0$, consider the following automorphism:
\begin{align*}
	\Phi(e_1) &=e_1, &
	\Phi(e_2) &=t\,e_2.
\end{align*}
 For a suitable choice of the parameter $t\in\R^\ast$, applying the automorphism $\Phi$  to the connection given in \eqref{conabelian1} yields the following equivalent connection:	
\begin{equation}\label{g3,1, c12}
	\begin{aligned}
		\nabla_{e_1}e_1&=\delta\,e_2, &\nabla_{e_1}e_2&=0, &\nabla_{e_2}e_1&=0, &\nabla_{e_2}e_2&=0.
	\end{aligned}
\end{equation}	
	$\bullet$ \textbf{Lie algebra $\G_{3,2} :$ } \\ $\triangleright$ Suppose that $a_{21}\neq0$, and consider the following automorphism:
	\begin{align*}
		\Phi(e_1) &=a_{21}\, e_1, &
		\Phi(e_2) &=-a_{11}\,e_1+a_{21}\,e_2.
	\end{align*}
Applying the automorphism $\Phi$  to the connection given in \eqref{conabelian1} yields the following equivalent connection:	
	\begin{equation}\label{g3,2, c1}
		\begin{aligned}
		\nabla_{e_1}e_1&=e_2, &\nabla_{e_1}e_2&=\lambda_1\,e_1+\mu_1\,e_2, &\nabla_{e_2}e_1&=\lambda_1\,e_1+\mu_1\,e_2, &\nabla_{e_2}e_2&=\nu_1\,e_1+\eta_1\,e_2.
		\end{aligned}
	\end{equation}

 $\triangleright$ Suppose that $a_{21}=0$. If $a_{11}-a_{22}\neq0$ and $a_{22}\neq0$, consider the following automorphism:
\begin{align*}
	\Phi(e_1) &=a_{22}\, e_1, &
	\Phi(e_2) &=\tfrac{a_{12}a_{22}}{a_{11}-a_{22}}e_1+a_{22}\,e_2.
\end{align*}
Applying the automorphism $\Phi$  to the connection given in \eqref{conabelian1} yields the following equivalent connection:	
\begin{equation}\label{g3,2, c2}
	\begin{aligned}
		\nabla_{e_1}e_1&=\lambda_2\,e_1, &\nabla_{e_1}e_2&=e_2, &\nabla_{e_2}e_1&=e_2, &\nabla_{e_2}e_2&=\nu_2\,e_1+\eta_2\,e_2.
	\end{aligned}
\end{equation}

If $a_{11}-a_{22}\neq0$ and $a_{22}=0$, consider the following automorphism:
\begin{align*}
	\Phi(e_1) &=a_{11}\, e_1, &
	\Phi(e_2) &=a_{12}\,e_1+a_{11}\,e_2.
\end{align*}
Applying the automorphism $\Phi$  to the connection given in \eqref{conabelian1} yields the following equivalent connection:	
\begin{equation}\label{g3,2, c3}
	\begin{aligned}
		\nabla_{e_1}e_1&=e_1, &\nabla_{e_1}e_2&=0, &\nabla_{e_2}e_1&=0, &\nabla_{e_2}e_2&=\nu_3\,e_1+\eta_3\,e_2.
	\end{aligned}
\end{equation}

If $a_{22}=a_{11}$ and $a_{11}\neq0$, consider the following automorphism:
\begin{align*}
	\Phi(e_1) &=a_{11}\, e_1, &
	\Phi(e_2) &=\tfrac{b_{22}}{a_{11}}e_1+a_{11}\,e_2.
\end{align*}
Applying the automorphism $\Phi$  to the connection given in \eqref{conabelian1} yields the following equivalent connection:	
\begin{equation}\label{g3,2, c4}
	\begin{aligned}
		\nabla_{e_1}e_1&=e_1, &\nabla_{e_1}e_2&=\lambda_4\,e_1+e_2, &\nabla_{e_2}e_1&=\lambda_4\,e_1+e_2, &\nabla_{e_2}e_2&=\nu_4\,e_1.
	\end{aligned}
\end{equation}

 $\triangleright$ Suppose that $a_{21}=a_{11}=a_{22}=0$. If $a_{12}\neq0$ and $2\,a_{12}-b_{22}\neq0$, consider the following automorphism:
\begin{align*}
	\Phi(e_1) &=a_{12}\, e_1, &
	\Phi(e_2) &=\tfrac{b_{12}a_{12}}{2\,a_{12}-b_{22}}e_1+a_{12}\,e_2.
\end{align*}
Applying the automorphism $\Phi$  to the connection given in \eqref{conabelian1} yields the following equivalent connection:	
\begin{equation}\label{g3,2, c5}
	\begin{aligned}
		\nabla_{e_1}e_1&=0, &\nabla_{e_1}e_2&=e_1, &\nabla_{e_2}e_1&=e_1, &\nabla_{e_2}e_2&=\eta_5\,e_2.
	\end{aligned}
\end{equation}

	If $a_{12}\neq0$ and $b_{22}=2\,a_{12}$, consider the following automorphism:
	\begin{align*}
		\Phi(e_1) &=a_{12}\, e_1, &
		\Phi(e_2) &=a_{12}\,e_2.
	\end{align*}
	Applying the automorphism $\Phi$  to the connection given in \eqref{conabelian1} yields the following equivalent connection:	
	\begin{equation}\label{g3,2, c6}
		\begin{aligned}
			\nabla_{e_1}e_1&=0, &\nabla_{e_1}e_2&=e_1, &\nabla_{e_2}e_1&=e_1, &\nabla_{e_2}e_2&=\eta_6\,e_2+2\,e_2.
		\end{aligned}
	\end{equation}

$\triangleright$ Suppose that $a_{21}=a_{11}=a_{22}=a_{12}=0$. If $b_{22}\neq0$, consider the following automorphism:
\begin{align*}
	\Phi(e_1) &=b_{22}\, e_1, &
	\Phi(e_2) &=-b_{12}\,e_1+b_{22}\,e_2.
\end{align*}
Applying the automorphism $\Phi$  to the connection given in \eqref{conabelian1} yields the following equivalent connection:	
\begin{equation}\label{g3,2, c7}
	\begin{aligned}
		\nabla_{e_1}e_1&=0, &\nabla_{e_1}e_2&=0, &\nabla_{e_2}e_1&=0, &\nabla_{e_2}e_2&=e_2.
	\end{aligned}
\end{equation}

	$\triangleright$ Suppose that $a_{21}=a_{11}=a_{22}=a_{12}=b_{22}=0$. If $b_{22}\neq0$, consider the following automorphism:
	\begin{align*}
		\Phi(e_1) &=x\, e_1, &
		\Phi(e_2) &=x\,e_2.
	\end{align*}
	For a suitable choice of the parameter $x\in\R^\ast$, applying the automorphism $\Phi$  to the connection given in \eqref{conabelian1} yields the following equivalent connection:	
	\begin{equation}\label{g3,2, c8}
		\begin{aligned}
			\nabla_{e_1}e_1&=0, &\nabla_{e_1}e_2&=0, &\nabla_{e_2}e_1&=0, &\nabla_{e_2}e_2&=\delta\,e_1,\quad\delta=0,1.
		\end{aligned}
	\end{equation}

	$\bullet$ \textbf{Lie algebra $\G_{3,3} :$ } Since every restricted automorphism $\Phi=\Psi|_{\R^2}$ of $\G_{3,3}=\R\ltimes\R^2$ is an element of $\mathrm{GL}_2(\R)$, this situation coincides with that of the abelian Lie algebra $\R^2$. Consequently, the non-flat torsion-free connections extended to $\G_{3,3}$ are precisely those arising in the case of $\R^2$.
	\\\\
	$\bullet$ \textbf{Lie algebra $\G_{3,4} :$ } Assume that $\alpha\neq-1$.  Applying the automorphism $\Phi_{3,4}^{\alpha\neq-1}$  to the connection given in \eqref{conabelian1} yields the following equivalent connection:
	\begin{align}
		\nabla_{e_1}e_1&=\tfrac{a_{11}}{x}e_1+\tfrac{t\,a_{21}}{x^2}e_2, &\nabla_{e_1}e_2&=\tfrac{a_{12}}{t}e_1+\tfrac{a_{22}}{x}e_2, &\nabla_{e_2}e_1&=\tfrac{a_{12}}{t}e_1+\tfrac{a_{22}}{x}e_2, &\nabla_{e_2}e_2&=\tfrac{x\,b_{12}}{t^2}e_1+\tfrac{b_{22}}{t}e_2.
	\end{align}
	Since the parameters $x,t\in\mathbb{R}^\ast$ are arbitrary, the parameters of the connection can be normalized by an appropriate choice of $x$ and $t$. For example, when $a_{12}\neq0$ and $a_{22}\neq0$, setting $t=a_{12}$ and $x=a_{22}$ yields precisely the first connection associated with the Lie algebra $\G_{3,4}$. As for the other cases, they are handled similarly.

	Let us now assume that $\alpha=-1$. Applying the automorphism $\Phi_{3,4}^{\alpha=-1}$  to the connection given in \eqref{conabelian1} yields the following equivalent connection:
	\begin{align}
		\nabla_{e_1}e_1&=\tfrac{b_{22}}{y}e_1+\tfrac{z\,b_{12}}{y^2}e_2, &\nabla_{e_1}e_2&=\tfrac{a_{22}}{z}e_1+\tfrac{a_{12}}{y}e_2, &\nabla_{e_2}e_1&=\tfrac{a_{22}}{z}e_1+\tfrac{a_{12}}{y}e_2, &\nabla_{e_2}e_2&=\tfrac{y\,a_{21}}{z^2}e_1+\tfrac{a_{11}}{z}e_2.
	\end{align}
	Observe first that the parameters $x,t$ in the automorphism  $\Phi^{\alpha\neq-1}_{3,4}$ are simply replaced by $y$ and $z$. Moreover, there is a permutation of the parameters of the connection given in~\eqref{conabelian1} under the action of $\Phi^{\alpha=-1}_{3,4}$. Therefore, without loss of generality and up to isomorphism, all flat, torsion-free connections under the action of $\Phi^{\alpha=-1}_{3,4}$ are precisely those obtained  under the action of $\Phi^{\alpha\neq-1}_{3,4}$. 
	Depending on the parameters of the connection, there may be isomorphisms between some flat, torsion-free connections via the second automorphism. The classification of these cases does not require us to take into account these cases.
	\\\\
	$\bullet$ \textbf{Lie algebra $\G_{3,5} :$ } We may assume that $a_{12}=0$. Indeed, if $a_{12}\neq0$,  then, we obtain that the numerator of $\nabla^{12}_{e_1}$ is 
	\begin{align}
		a_{12}x^3
		+z(a_{11}-a_{22}-b_{12})x^2
		-z^2(a_{12}+a_{21}-b_{22})x
		+a_{22}z^3.
	\end{align}
	Hence, $\nabla^{12}$ necessarily admits a real solution, since it is a polynomial of odd degree.

	From now on, we assume that $a_{12}=0$. Applying the automorphism $\Phi_{3,5}$ with $z=0$ to the connection given in \eqref{conabelian1} yields the following equivalent connection:
\begin{align}
	\nabla_{e_1}e_1&=\tfrac{a_{11}}{x}e_1+\tfrac{a_{21}}{x}e_2, &\nabla_{e_1}e_2&=\tfrac{a_{22}}{x}e_2, &\nabla_{e_2}e_1&=\tfrac{a_{22}}{x}e_2, &\nabla_{e_2}e_2&=\tfrac{b_{12}}{x}e_1+\tfrac{b_{22}}{x}e_2.
\end{align}
	
	We then normalize the parameters of this connection by choosing a suitable parameter $x\in\R^{\ast}$. For each admissible choice of the parameters, an appropriate change of variables allows us to reduce the connection to one of the normal forms listed in this lemma. For example, when $a_{22}\neq 0$, we set $x=a_{22}$; after a suitable change of variables, this yields the first connection appearing in the statement. The remaining three connections are obtained in an analogous manner. 
	
	The case $a_{22}=a_{21}=a_{11}=0$ corresponds to a flat torsion-free connection and is therefore excluded from the present classification.

	Finally, 
	a straightforward computation shows that the connections listed in Lemma~\ref{Lemg3j} on the flat Lie algebra $\G_{3,5}$ are pairwise non-isomorphic, except for certain cases that are not taken into consideration, since the corresponding isomorphisms hold only locally and not globally.

	Based on the assumptions in Lemma~\ref{Lemg3j}, it can be observed that the connections under consideration are not flat. As part of the proof, only conditions are derived that guarantee that the resulting connections are not isomorphic pairwise. Our aim is merely to establish an equivalence class and not a classification up to isomorphism, so this requirement is not essential for our purposes.
	
\end{proof}

\begin{pr}
Let $(\G, \nabla)$ be a three-dimensional real flat  Lie algebra with $\G = \G_{3,1}$. Then $(\G, \nabla)$ is isomorphic to exactly one of the flat Lie algebras listed in Table~$\ref{g3,1}$.
{\renewcommand*{\arraystretch}{1.8}
\captionof{table}{Flat torsion-free connection on the Lie algebra $\G_{3,1}$.}
\setcounter{table}{4}
\begin{footnotesize} 
\setlength{\tabcolsep}{5pt} 
\begin{longtable}{@{}cllllllc@{}} 
			\hline
		Flat algebra&\multicolumn{2}{@{}l@{}}{~~Flat torsion-free connection} &&&& Remarks \\
			\hline
$\h_{0,1}$&$\nabla_{e_1}e_1=e_1$&$\nabla_{e_1}e_2=e_3$&&&&&\\
$\h_{0,2}$&$\nabla_{e_1}e_1=e_2$&$\nabla_{e_1}e_2=e_3$&&&&&\\
$\h_{0,3}$&$\nabla_{e_1}e_2=e_3$&&&&&&\\
$\h_{0,4}$&$\nabla_{e_1}e_1=e_1$&$\nabla_{e_1}e_2=e_2+e_3$&$\nabla_{e_1}e_3=e_3$&$\nabla_{e_2}e_1=e_2$&$\nabla_{e_3}e_1=e_3$&&\\
$\h_{0,5}$&$\nabla_{e_1}e_1=e_2$&$\nabla_{e_1}e_2=(1+\lambda)e_3$&$\nabla_{e_2}e_1=\lambda e_3$&&&$\lambda\in\R^\ast$&\\
$\h_{0,6}$&$\nabla_{e_1}e_2=\frac{1}{2}e_3$&$\nabla_{e_2}e_1=-\frac{1}{2}e_3$&&&&&
\\\hline
$\h_{1,1}$&$\nabla_{e_1}e_1=e_3$&$\nabla_{e_2}e_1=-e_3$&$\nabla_{e_2}e_2=e_2$&&&&\\\hline
$\h_{4,1}$&$\nabla_{e_1}e_1=e_3$&$\nabla_{e_1}e_2=e_1$&$\nabla_{e_2}e_1=e_1-e_3$&$\nabla_{e_2}e_2=e_2$&$\nabla_{e_2}e_3=e_3$&&\\
&$\nabla_{e_3}e_2=e_3$&&&&&&\\\hline
$\h_{6,1}$&$\nabla_{e_1}e_1=\lambda e_3$&$\nabla_{e_1}e_2=e_3$&$\nabla_{e_2}e_2= e_3$&&&$\lambda\in\R^\ast$&
			\\\hline		
			\end{longtable}
			\label{g3,1}
			\end{footnotesize}	
			}
	
\end{pr}
\begin{proof}
Consider the Heisenberg algebra $\G_{3,1}$. Then, $\G_{3,1}$ can be viewed as the semidirect sum of $\R e_1$ and $\R^2=\langle e_2,e_3\rangle$. In the basis $\lbrace e_1, e_2, e_3 \rbrace$, the operators $\nabla_{e_1}$, $\nabla_{e_2}$ and $\nabla_{e_3}$ are given respectively by: 
	\begin{equation}\label{HeisCon}
		\nabla_{e_1}=\left( \begin {array}{ccc} a_{11}&a_{12}&a_{13}\\ \noalign{\medskip}
		a_{21}&a_{22}&a_{23}\\ \noalign{\medskip}a_{31}&a_{32}&a_{33}\end {array} \right),\quad
		\nabla_{e_2}=\left( \begin {array}{ccc} a_{12}&b_{12}&b_{13}\\ \noalign{\medskip}
		a_{22}&b_{22}&b_{23}\\ \noalign{\medskip}a_{32}-1&b_{32}&b_{33}\end {array} \right),\quad
		\nabla_{e_3}=\left( \begin {array}{ccc} a_{13} &b_{13}&c_{13}\\ \noalign{\medskip}
		a_{23}&b_{23}&c_{23}\\ \noalign{\medskip}a_{33}&b_{33}&c_{33}\end {array} \right),\quad
	\end{equation}
	where $a_{ij}$, $b_{ij}$, $c_{ij}\in \mathbb{R}$.

	Assume that $\nabla^0$ is a  torsion-free connection. Then $\nabla^0$ is equivalent to one of the connections listed in Lemma~\ref{Lemg3j} under the Lie algebra $\G_{3,1}$. We first consider the case $\nabla^0=\nabla^1$. Then, the connection given in \eqref{HeisCon}, becomes
	\begin{equation}\label{HeisCon1}
		\nabla_{e_1}=\left( \begin {array}{ccc} a_{11}&a_{12}&a_{13}\\ \noalign{\medskip}
		a_{21}&a_{22}&a_{23}\\ \noalign{\medskip}a_{31}&a_{32}&a_{33}\end {array} \right),\quad
		\nabla_{e_2}=\left( \begin {array}{ccc} a_{12}&b_{12}&b_{13}\\ \noalign{\medskip}
		a_{22}&\lambda_1&0\\ \noalign{\medskip}a_{32}-1&\mu_1&\nu_1\end {array} \right),\quad
		\nabla_{e_3}=\left( \begin {array}{ccc} a_{13} &b_{13}&c_{13}\\ \noalign{\medskip}
		a_{23}&0&1\\ \noalign{\medskip}a_{33}&\nu_1&1\end {array} \right),\quad
	\end{equation}
	A straightforward computation shows that the flatness-equations associated with the connection given in~\eqref{HeisCon1} admits no solution.

Similarly, if $\nabla^0=\nabla^i$, $j=2,\ldots,8$ a straightforward computation and analysis show that the flatness-equation admits no solution.

If $\nabla^0=\nabla^{9}$, then a straightforward computation shows that the flatness-equation admits three solutions: The first solution is given by the following flat, torsion-free connection:
\begin{equation}\label{g3,1,nabla9,sol1}
\begin{aligned}
\nabla_{e_1}e_1&=a_{31}e_3, &\nabla_{e_1}e_2&=e_1+a_{32}e_3, &\nabla_{e_2}e_1&=e_1+(a_{32}-1)e_3, \\\nabla_{e_2}e_2&=\tfrac{a_{32}}{a_{31}}e_1+e_2, &\nabla_{e_2}e_3&=e_3,
&\nabla_{e_3}e_2&=e_3.
\end{aligned}
	\end{equation}
	Consider the following automorphism
	\begin{align*}
		\Psi(e_1) &=a_{31}\,e_1,&
		\Psi(e_2) &=a_{32}\,e_1+e_2+(a_{32}-a_{32}^2)e_3,&
		\Psi(e_3) &=a_{31}\,e_3.
	\end{align*}
	Applying $\Psi$ to the connection given in~\eqref{g3,1,nabla9,sol1} yields the following equivalent connection:
	\begin{equation}\label{g3,1,nabla9,sol1,1}
		\begin{aligned}
			\nabla_{e_1}e_1&=e_3, &\nabla_{e_1}e_2&=e_1, &\nabla_{e_2}e_1&=e_1-e_3, &\nabla_{e_2}e_2&=e_2, &\nabla_{e_2}e_3&=e_3,
			&\nabla_{e_3}e_2&=e_3.
		\end{aligned}
	\end{equation}
		This connection is precisely the one appearing in Table~\ref{g3,1}, associated with the flat Lie algebra $\h_{4,1}$.

	The second solution is given by the following flat, torsion-free connection:
		\begin{equation}\label{g3,1,nabla9,sol2}
		\begin{aligned}
			\nabla_{e_1}e_1&=2\,a_{33}e_1-a_{33}^2e_2+a_{31}e_3, &\nabla_{e_1}e_2&=e_1, &\nabla_{e_1}e_3&=a_{33}e_3, &\nabla_{e_2}e_1&=e_1-e_3, &\nabla_{e_2}e_2&=e_2,\\
			\nabla_{e_2}e_3&=e_3, &\nabla_{e_3}e_1&=a_{33}e_3, &\nabla_{e_3}e_2&=e_3.
		\end{aligned}
	\end{equation}
	 Consider then the following automorphism
	\begin{align*}
		\Psi(e_1) &=a_{33}\,e_1+x\,e_2,&
		\Psi(e_2) &= e_1,&
		\Psi(e_3) &=-x\,e_3.
	\end{align*}
	For a suitable choice of the parameter $x\in\R^\ast$, applying $\Psi$ to the connection given in~\eqref{g3,1,nabla9,sol2} yields the following equivalent connection:
	\begin{equation}\label{g3,1,nabla9,sol2,1}
		\begin{aligned}
			\nabla_{e_1}e_1&=e_1, &\nabla_{e_1}e_2&=e_2+e_3, &\nabla_{e_1}e_3&=e_3, &\nabla_{e_2}e_1&=e_2, &\nabla_{e_2}e_2&=\delta\,e_3, &\nabla_{e_3}e_1&=e_3,\quad\delta=0,1.
		\end{aligned}
	\end{equation}
	Note that if $\delta=0$, then this connection coincides with the one associated with the flat Lie algebra $\h_{0,4}$. On the other hand, if $\delta=1$, then the following automorphism
	\begin{align*}
		\Psi(e_1) &=e_2,&
		\Psi(e_2) &=- e_1,&
		\Psi(e_3) &=e_3.
	\end{align*}
	establishes an isomorphism between the connection given in~\eqref{g3,1,nabla9,sol2,1} and the one given  in~\eqref{g3,1,nabla9,sol1,1}.

	The third solution corresponds to the following flat, torsion-free connection:
		\begin{equation}\label{g3,1,nabla9,sol3}
		\begin{aligned}
			\nabla_{e_1}e_1&=a_{3 3}(a_{3 3}b_{1 2} + 2)\,e_1-(a_{3 3}^3b_{1 2} - a_{3 3}^2)\,e_2+\tfrac{(2\,a_{3 2}a_{3 3} - a_{3 3})b_{1 2} + a_{3 2}}{b_{12}}\,e_3,\\\nabla_{e_1}e_2&=(a_{33}b_{12}+1)e_1-a_{33}^2b_{12}e_2+a_{32}\,e_3,\quad\nabla_{e_1}e_3=a_{33}\,e_3,\\
			\nabla_{e_1}e_2&=(a_{33}b_{12}+1)\,e_1-a_{33}^2b_{12}\,e_2+(a_{32}-1)\,e_3,\quad\nabla_{e_2}e_2=b_{12}\,e_1+(1-a_{33}b_{12})\,e_2,\\
			\nabla_{e_2}e_3&=e_3,\quad\nabla_{e_3}e_1=a_{33}\,e_3, \quad\nabla_{e_3}e_2=e_2.
		\end{aligned}
	\end{equation}
Suppose that $a_{33}b_{12}+1\neq0$, and consider the following automorphism	
	\begin{align*}
		\Psi(e_1) &=x\,e_1+a_{33}\,e_2,&
		\Psi(e_2) &=\tfrac{x\,b_{12}}{a_{33}b_{12}+1} e_1+e_2+\tfrac{x\,b_{12}\,(1-a_{32})}{(a_{33}b_{12}+1)^2}e_3,&
		\Psi(e_3) &=\tfrac{x}{a_{33}b_{12}+1} e_3.
	\end{align*}
	For a suitable choice of the parameter $x\in\R^\ast$, applying $\Psi$ to the connection given in~\eqref{g3,1,nabla9,sol3} yields the following equivalent connection:
		\begin{equation}\label{g3,1,nabla9,sol3,1}
		\begin{aligned}
			\nabla_{e_1}e_1&=\delta\,e_3, &\nabla_{e_1}e_2&=e_1, &\nabla_{e_2}e_1&=e_1-e_3, &\nabla_{e_2}e_2&=e_2, &\nabla_{e_2}e_3&=e_3,
			&\nabla_{e_3}e_2&=e_3.
		\end{aligned}
	\end{equation}
In fact, this connection is isomorphic to the one given in \eqref{g3,1,nabla9,sol2,1} via the following automorphism
	\begin{align*}
		\Psi(e_1) &=e_2,&
		\Psi(e_2) &=- e_1,&
		\Psi(e_3) &=e_3.
	\end{align*}
	There is therefore no need to perform any further analysis.
	
	Suppose now that $a_{33}b_{12}+1=0$. Set $a_{33}=-\frac{1}{b_{12}}$, and consider the following automorphsim
	\begin{align*}
		\Psi(e_1) &=-\tfrac{1}{b_{12}}e_2+\tfrac{x\,(a_{32}-1)}{b_{12}}e_3,&
		\Psi(e_2) &=x\,e_1+e_2,&
		\Psi(e_3) &=\tfrac{x}{b_{12}}e_3.
	\end{align*}
	For a suitable choice of the parameter $x\in\R^\ast$, applying $\Psi$ to the connection given in~\eqref{g3,1,nabla9,sol3} yields exactly the one given in \eqref{g3,1,nabla9,sol3,1}. 
	Therefore, no further analysis is required.

	If $\nabla^0=\nabla^{10}$.
	Then, a straightforward computation shows that the flatness-equation admits a unique solution, which is given by the following flat, torsion-free connection:
	\begin{equation}\label{g3,1,nabla10,sol1}
		\begin{aligned}
			\nabla_{e_1}e_1&=a_{33}\,e_1+a_{33}(a_{32}-1)\,e_3, &\nabla_{e_1}e_2&=a_{33}\,e_2+a_{32}\,e_3, &\nabla_{e_1}e_3&=a_{33}\,e_3, \\\nabla_{e_2}e_1&=a_{33}\,e_2+(a_{32}-1)\,e_3, &\nabla_{e_2}e_2&=-\tfrac{1}{a_{33}}e_1+2\,e_2+\delta\,e_3, &\nabla_{e_2}e_3&=e_3,\\
			\nabla_{e_3}e_1&=a_{33}\,e_3, &\nabla_{e_3}e_2&=e_3.
		\end{aligned}
	\end{equation}
	Consider the following automorphsim
	\begin{align*}
		\Psi(e_1) &=a_{33}\,e_1+a_{33}\,x\,(a_{32}-1)\,e_3,&
		\Psi(e_2) &=e_1+x\,e_2,&
		\Psi(e_3) &=a_{33}\,x\,e_3.
	\end{align*}
	For a suitable choice of the parameter $x\in\R^\ast$, applying $\Psi$ to the connection given in~\eqref{g3,1,nabla10,sol1} yields exactly the one given in \eqref{g3,1,nabla9,sol2,1}. 
	Therefore, no further analysis is required.

	If $\nabla^0=\nabla^{11}$, then a straightforward computation shows that the flatness-equation admits three solutions: The first solution is given by the following flat, torsion-free connection:
		\begin{equation}\label{g3,1,nabla11,sol1}
		\begin{aligned}
			\nabla_{e_1}e_1&=a_{31}\,e_3, &\nabla_{e_1}e_2&=-a_{31}b_{12}\,e_3, &\nabla_{e_2}e_1&=-(a_{31}b_{12}+1)\,e_3, &\nabla_{e_2}e_2&=b_{12}\,e_1+e_2.
		\end{aligned}
	\end{equation}
	Consider the following automorphsim
	\begin{align*}
		\Psi(e_1) &=x\,e_1,&
		\Psi(e_2) &=-x\,b_{12}\,e_1+e_2+x\,b_{12}\,(a_{31}b_{12}+1)e_3,&
		\Psi(e_3) &=x\,e_3.
	\end{align*}
	For a suitable choice of the parameter $x\in\R^\ast$, applying $\Psi$ to the connection given in~\eqref{g3,1,nabla11,sol1} yields the following equivalent connection:
		\begin{equation}\label{g3,1,nabla11,sol1,1}
		\begin{aligned}
			\nabla_{e_1}e_1&=\delta\,e_3, &\nabla_{e_2}e_1&=-e_3,  &\nabla_{e_2}e_2&=e_2,\quad\delta=0,1.
		\end{aligned}
	\end{equation}
Note that if $\delta=1$, then this conection coincides with the one associated to the flat Lie algebra $\h_{1,1}$. On the other hand, if $\delta=0$, then the following automorphism 
	\begin{align*}
	\Psi(e_1) &=e_2,&
	\Psi(e_2) &=e_1,&
	\Psi(e_3) &=-e_3.
\end{align*}
maps the  connection given in~\eqref{g3,1,nabla11,sol1,1} to the one associated with the flat Lie algebra $\h_{0,1}$.

The second solution is given by the following flat, torsion-free connection:	
		\begin{equation}\label{g3,1,nabla11,sol2}
		\begin{aligned}
			\nabla_{e_1}e_1&=a_{22}^2\,e_2+a_{31}e_3, &\nabla_{e_1}e_2&=a_{22}\,e_2, &\nabla_{e_2}e_1&=a_{22}\,e_2-e_3,&\nabla_{e_2}e_2&=e_2.
		\end{aligned}
	\end{equation}
	Consider the following automorphsim
	\begin{align*}
		\Psi(e_1) &=x\,e_1+a_{22}\,e_2,&
		\Psi(e_2) &=e_2,&
		\Psi(e_3) &=x\,e_3.
	\end{align*}
	For a suitable choice of the parameter $x\in\R^\ast$, applying $\Psi$ to the connection given in~\eqref{g3,1,nabla11,sol2}  yields exactly the one given in \eqref{g3,1,nabla11,sol1,1}. 
	Therefore, no further analysis is required.

		The third solution corresponds to the following flat, torsion-free connection:
	\begin{equation}\label{g3,1,nabla11,sol3}
		\begin{aligned}
			\nabla_{e_1}e_1&=\tfrac{a_{12}^2}{b_{12}}e_1+\tfrac{a_{12}^2}{b_{12}^2}e_2+\tfrac{(a_{32}-1)a_{12}-a_{32}}{b_{12}}e_3, \quad\nabla_{e_1}e_2=a_{12}\,e_1+\tfrac{a_{12}}{b_{12}}e_2+a_{32}\,e_3,\\
			\nabla_{e_2}e_1&=a_{12}\,e_1+\tfrac{a_{12}}{b_{12}}e_2+(a_{32}-1)\,e_3,~~~\quad\nabla_{e_2}e_2=b_{12}\,e_1+e_2.
		\end{aligned}
	\end{equation}
	
	Suppose that $a_{12}+1\neq0$, and	consider the following automorphsim
	\begin{align*}
		\Psi(e_1) &=\tfrac{a_{12}\,(a_{12}+1)}{b_{12}}e_1+x\,e_2,&
		\Psi(e_2) &=(a_{12}+1)\,e_1-b_{12}\,x\,e_2+x\,b_{12}\,(a_{32}-1)(a_{12}+1)\,e_3,&
		\Psi(e_3) &=-x\,(a_{12}+1)^2e_3.
	\end{align*}
	For a suitable choice of the parameter $x\in\R^\ast$, applying $\Psi$ to the connection given in~\eqref{g3,1,nabla11,sol3} yields the following equivalent connection:
		\begin{equation}\label{g3,1,nabla11,sol3,1}
		\begin{aligned}
			\nabla_{e_1}e_1&=e_1, &\nabla_{e_1}e_2&=e_3, &\nabla_{e_2}e_2&=\delta\,e_3,\quad\delta=0,1.
		\end{aligned}
	\end{equation}
Observe first that if $\delta=0$, the this connection coincides with the one associated to the flat Lie algebra $\h_{0,1}$. On the other hand, if $\delta=1$, then the following automorphism 
\begin{align*}
	\Psi(e_1) &=e_2,&
	\Psi(e_2) &=-e_1,&
	\Psi(e_3) &=e_3.
\end{align*}
maps the  connection given in~\eqref{g3,1,nabla11,sol3,1} to the one associated with the flat Lie algebra $\h_{1,1}$.

	Suppose now that $a_{12}=-1$, and	consider the following automorphsim
	\begin{align*}
		\Psi(e_1) &=e_1,&
		\Psi(e_2) &=-b_{12}\,e_1+b_{12}^2e_2+b_{13}^3(2\,a_{32}-1)e_3,&
		\Psi(e_3) &=b_{12}^2e_3.
	\end{align*}
Applying $\Psi$ to the connection given in~\eqref{g3,1,nabla11,sol3} yields the following equivalent connection:
	\begin{equation}\label{g3,1,nabla11,sol3,2}
		\begin{aligned}
			\nabla_{e_1}e_1&=e_2, &\nabla_{e_1}e_2&=(\lambda+1)\,e_3, &\nabla_{e_2}e_2&=\lambda\,e_3,\quad\lambda\in\R.
		\end{aligned}
	\end{equation}
	In this case, $\lambda=-a_{32}$.  This connection is precisely the one appearing in Table~\ref{g3,1}, associated with the flat Lie algebra $\h_{0,5}$.  The additional conditions $\lambda\neq0$ in the flat Lie algebra $\h_{0,5}$ is imposed to ensure that the corresponding connection is not isomorphic to the  algebras $\h_{0,1}$.

	If $\nabla^0=\nabla^{12}$ and $\delta=1$, then a straightforward computation shows that the flatness-equation admits three solutions: The first solution is given by the following flat, torsion-free connection:	
		\begin{equation}\label{g3,1,nabla12,delta=1,sol1}
		\begin{aligned}
			\nabla_{e_1}e_1&=a_{22}\,e_1+a_{22}\,(a_{32}-1)\,e_2+a_{31}\,e_3,&\nabla_{e_1}e_2&=a_{22}\,e_2+a_{32}\,e_3,&\nabla_{e_1}e_3&=a_{22}\,e_3,\\
			\nabla_{e_2}e_1&=a_{22}\,e_2+(a_{32}-1)\,e_3, &\nabla_{e_2}e_2&=e_3, &\nabla_{e_3}e_1&=a_{22}\,e_3.
		\end{aligned}
	\end{equation}
Suppose that $a_{22}\neq0$, and consider the following automorphism	
\begin{align*}
	\Psi(e_1) &=(a_{22}-a_{22}a_{32})\,e_1+a_{22}\,e_2+a_{22}\,(-a_{32}^2+a_{31}+a_{32})\,e_3,&
	\Psi(e_2) &=-a_{22}\,e_1,&
	\Psi(e_3) &=a_{22}^2\,e_2.
\end{align*}
Applying $\Psi$ to the connection given in~\eqref{g3,1,nabla12,delta=1,sol1}  yields exactly the one given in \eqref{g3,1,nabla9,sol1,1}. 
Therefore, no further analysis is required.

	Suppose now that $a_{22}=0$, and consider the following automorphism	
	\begin{align*}
		\Psi(e_1) &=(a_{22}-a_{22}a_{32})\,e_1+a_{22}\,e_2+a_{22}\,(-a_{32}^2+a_{31}+a_{32})\,e_3,&
		\Psi(e_2) &=-a_{22}\,e_1,&
		\Psi(e_3) &=a_{22}^2\,e_2.
	\end{align*}
	Applying $\Psi$ to the connection given in~\eqref{g3,1,nabla12,delta=1,sol1}  yields the following equivalent connection:
	\begin{equation}\label{g3,1,nabla12,delta=1,sol1,2}
		\begin{aligned}
			\nabla_{e_1}e_1&=\lambda\,e_3, &\nabla_{e_1}e_2&=e_3, &\nabla_{e_2}e_2&=e_3.
		\end{aligned}
	\end{equation}
	
	 	In this case, $\lambda=-a_{32}^2+a_{31}+a_{32}$.  This connection is precisely the one appearing in Table~\ref{g3,1}, associated with the flat Lie algebra $\h_{6,1}$.  The additional conditions $\lambda\neq0$ in the flat Lie algebra $\h_{6,1}$ is imposed to ensure that the corresponding connection is not isomorphic to the  algebras $\h_{0,3}$.

	The second solution is given by the following flat, torsion-free connection:
		\begin{equation}\label{g3,1,nabla12,delta=1,sol2}
		\begin{aligned}
			\nabla_{e_1}e_1&=a_{11}\,e_1+(a_{11}-a_{11}a_{32})\,e_2+a_{31}\,e_3, &\nabla_{e_1}e_2&=a_{32}\,e_3, &\nabla_{e_2}e_1&=(a_{32}-1)\,e_3, &\nabla_{e_2}e_2&=e_3.
		\end{aligned}
	\end{equation}
	Suppose that $a_{11}\neq0$, and consider the following automorphism
	\begin{align*}
		\Psi(e_1) &=(a_{11}-a_{11}a_{32})\,e_1+a_{11}\,e_2+a_{11}\,(-a_{32}^2+a_{31}+a_{32})\,e_3,&
		\Psi(e_2) &=-a_{11}\,e_1,&
		\Psi(e_3) &=a_{11}^2\,e_3.
	\end{align*}
	Applying $\Psi$ to the connection given in~\eqref{g3,1,nabla12,delta=1,sol2}  yields exactly the one given in \eqref{g3,1,nabla11,sol1,1}, in which case $\delta=1$. 
	Therefore, no further analysis is required.

	Suppose now that $a_{11}=0$, and consider the following automorphism	
	\begin{align*}
		\Psi(e_1) &=e_1+(a_{32}-1)\,e_2,&
		\Psi(e_2) &=e_2,&
		\Psi(e_3) &=e_3.
	\end{align*}
	Applying $\Psi$ to the connection given in~\eqref{g3,1,nabla12,delta=1,sol2}  yields exactly the one given in \eqref{g3,1,nabla12,delta=1,sol1,2}. 
	Therefore, no further analysis is required.

		The third solution corresponds to the following flat, torsion-free connection:
		\begin{equation}\label{g3,1,nabla12,delta=1,sol3}
		\begin{aligned}
			\nabla_{e_1}e_1&=\tfrac{a_{12}^2}{b_{12}}e_1+\tfrac{a_{12}\,(a_{32}-1)}{b_{12}}e_3, &\nabla_{e_1}e_2&=a_{12}\,e_1+a_{32}\,e_2, &\nabla_{e_2}e_1&=a_{12}\,e_1+(a_{32}-1)\,e_3, \\\nabla_{e_2}e_2&=b_{12}\,e_1+e_3.
		\end{aligned}
	\end{equation}

	Suppose that $a_{12}\neq0$, and consider the following automorphism	
	\begin{align*}
		\Psi(e_1) &=\tfrac{a_{12}^2}{b_{12}}e_1-\tfrac{x\,a_{12}\,(a_{32}-1)}{b_{12}}e_3,&
		\Psi(e_2) &=a_{12}\,e_1+x\,e_2,&
		\Psi(e_3) &=\tfrac{x\,a_{12}^2}{b_{12}}e_3.
	\end{align*}
	For a suitable choice of the parameter $x\in\R^\ast$, applying $\Psi$ to the connection given in~\eqref{g3,1,nabla12,delta=1,sol3} yields the following equivalent connection:
	\begin{equation}\label{g3,1,nabla12,delta=1,sol3,1}
		\begin{aligned}
			\nabla_{e_1}e_1&=e_1, &\nabla_{e_1}e_2&=e_3, &\nabla_{e_2}e_2&=\delta\,e_3,\quad\delta=0,1.
		\end{aligned}
	\end{equation}
	Note that if $\delta=0$, then this connection coincides with the one associated to the flat Lie algebra $\h_{0,1}$. On the other hand, if $\delta=1$, then the following automorphism 
	\begin{align*}
		\Psi(e_1) &=e_2,&
		\Psi(e_2) &=-e_1,&
		\Psi(e_3) &=e_3.
	\end{align*}
	maps the  connection given in~\eqref{g3,1,nabla12,delta=1,sol3,1} to the one associated with the flat Lie algebra $\h_{1,1}$.

		Suppose that $a_{12}=0$, and consider the following automorphism	
	\begin{align*}
		\Psi(e_1) &=\tfrac{1}{b_{12}}e_2+\tfrac{1}{b_{12}^2}e_3,&
		\Psi(e_2) &=e_1,&
		\Psi(e_3) &=-\tfrac{1}{b_{12}}e_3.
	\end{align*}
Applying $\Psi$ to the connection given in~\eqref{g3,1,nabla12,delta=1,sol3}  yields exactly the one given in \eqref{g3,1,nabla11,sol3,2}, in this case $\lambda=-a_{32}$. 
Therefore, no further analysis is required.

If $\nabla^0=\nabla^{12}$ and $\delta=0$, then a straightforward computation shows that the flatness-equation admits five solutions: The first solution is given by the following flat, torsion-free connection:	
\begin{equation}\label{g3,1,nabla12,delta=0,sol1}
	\begin{aligned}
		\nabla_{e_1}e_1&=\tfrac{a_{12}^2}{b_{12}}e_1+a_{31}\,e_3, &\nabla_{e_1}e_2&=a_{12}\,e_1+\tfrac{a_{12}+a_{31}b_{12}}{a_{12}}e_3,  &\nabla_{e_2}e_1&=a_{12}\,e_1+\tfrac{a_{31}b_{12}}{a_{12}}e_3, &\nabla_{e_2}e_2&=b_{12}\,e_1.
	\end{aligned}
\end{equation}

Consider the following automorphism	
\begin{align*}
	\Psi(e_1) &=\tfrac{a_{12}^2}{b_{12}}e_2+x\,a_{31}\,e_3,&
	\Psi(e_2) &=x\,e_1+a_{12}\,e_2,&
	\Psi(e_3) &=-\tfrac{x\,a_{12}^2}{b_{12}}e_3.
\end{align*}
For a suitable choice of the parameter $x\in\R^\ast$, applying $\Psi$ to the connection given in~\eqref{g3,1,nabla12,delta=0,sol1} yields the following equivalent connection:
\begin{equation}\label{g3,1,nabla12,delta=0,sol1,1}
	\begin{aligned}
		\nabla_{e_1}e_1&=\delta\,e_3,  &\nabla_{e_2}e_1&=-e_3, &\nabla_{e_2}e_2&=e_2,\quad\delta=0,1.
	\end{aligned}
\end{equation}

Note that if $\delta=1$, then this connection coincides with the one associated to the flat Lie algebra $\h_{1,1}$. On the other hand, if $\delta=0$, then the following automorphism 
\begin{align*}
	\Psi(e_1) &=e_2,&
	\Psi(e_2) &=e_1,&
	\Psi(e_3) &=-e_3.
\end{align*}
maps the  connection given in~\eqref{g3,1,nabla12,delta=0,sol1,1} to the one associated with the flat Lie algebra $\h_{0,1}$.

The second solution is given by the following flat, torsion-free connection:
\begin{equation}\label{g3,1,nabla12,delta=0,sol2}
	\begin{aligned}
		\nabla_{e_1}e_1&=a_{22}\,e_1+a_{21}\,e_2+a_{31}\,e_3, &\nabla_{e_1}e_2&=a_{22}\,e_2+e_3, &\nabla_{e_1}e_3&=a_{22}\,e_3, &\nabla_{e_2}e_1&=a_{22}\,e_2, \\\nabla_{e_3}e_1&=a_{22}\,e_3.
	\end{aligned}
\end{equation}
Suppose that $a_{22}\neq0$, and consider the following automorphism
\begin{align*}
	\Psi(e_1) &=e_2,&
	\Psi(e_2) &=e_1,&
	\Psi(e_3) &=-e_3.
\end{align*}
 Applying $\Psi$ to the connection given in~\eqref{g3,1,nabla12,delta=0,sol2} yields exactly the one given in Table~\ref{g3,1}, under the flat Lie algebra $\h_{0,4}$.

 Suppose that $a_{22}=0$, and consider the following automorphism
 \begin{align*}
 	\Psi(e_1) &=e_1+x\,a_{31}\,e_2,&
 	\Psi(e_2) &=x\,e_2,&
 	\Psi(e_3) &=x\,e_3.
 \end{align*}
For a suitable choice of the parameter $x\in\R^\ast$, applying  $\Psi$ to the connection given in~\eqref{g3,1,nabla12,delta=0,sol2} yields the following equivalent connection:
 \begin{equation}\label{g3,1,nabla12,delta=0,sol2,1}
 	\begin{aligned}
 		\nabla_{e_1}e_1&=\delta\,e_2, &\nabla_{e_1}e_2&=e_3, \quad\delta=0,1.
 	\end{aligned}
 \end{equation}
 Observe that if $\delta=0$, then this flat, torsion-free connection coincides with the one given in Table~\ref{g3,1} under the flat Lie algebra $\h_{0,3}$. On the other hand, if $\delta=1$, then it is clear that the connection given in~\eqref{g3,1,nabla12,delta=0,sol2,1} is identical to the one associated with the flat Lie algebra $\h_{0,2}$.

 	The third solution corresponds to the following flat, torsion-free connection:
 \begin{equation}\label{g3,1,nabla12,delta=0,sol3}
 	\begin{aligned}
 		\nabla_{e_1}e_2&=a_{32}\,e_3, &\nabla_{e_2}e_1&=(a_{32}-1)\,e_3, &\nabla_{e_2}e_2&=b_{12}\,e_1.
 	\end{aligned}
 \end{equation}
 
 Consider the following automorphism
 \begin{align*}
 	\Psi(e_1) &=x\,e_2,&
 	\Psi(e_2) &=e_1,&
 	\Psi(e_3) &=-x\,e_3.
 \end{align*}
 For a suitable choice of the parameter $x\in\R^\ast$, applying  $\Psi$ to the connection given in~\eqref{g3,1,nabla12,delta=0,sol3} yields the following equivalent connection:
 \begin{equation}\label{g3,1,nabla12,delta=0,sol3,1}
 	\begin{aligned}
 		\nabla_{e_1}e_1&=\delta\,e_2, &\nabla_{e_1}e_2&=(\lambda_1+1)\,e_3 &\nabla_{e_2}e_1&=\lambda_1\,e_3,\quad\delta=0,1.
 	\end{aligned}
 \end{equation}
 In this case, $\lambda_1=-a_{32}$. Note that if $\delta=1$, then this connection coincides with the one associated to the flat Lie algebra $\h_{0,5}$. On the other hand, if $\delta=0$  and $\lambda_1=0$, then this connection coincides with the one associated to the flat Lie algebra $\h_{0,3}$. If $\delta=0$ and $\lambda_1=-1$, then the following automorphism 
 \begin{align*}
 	\Psi(e_1) &=e_2,&
 	\Psi(e_2) &=e_1,&
 	\Psi(e_3) &=-e_3.
 \end{align*}
 maps the  connection given in~\eqref{g3,1,nabla12,delta=0,sol3,1} to the one associated with the flat Lie algebra $\h_{0,3}$.

 In addition, if $\delta=0$ and $\lambda_1\neq-\frac{1}{2}$, then the following automorphism 
 \begin{align*}
 	\Psi(e_1) &=-\tfrac{1+\lambda}{\lambda_1}e_1+e_2,&
 	\Psi(e_2) &=e_1+\tfrac{\lambda_1}{\lambda}e_2,&
 	\Psi(e_3) &=-\tfrac{1+\lambda}{\lambda} e_3.
 \end{align*}
 maps the  connection given in~\eqref{g3,1,nabla12,delta=0,sol3,1} to the one associated with the flat Lie algebra $\h_{0,5}$.
  Moreover, if $\delta=0$ and $\lambda_1=-\frac{1}{2}$, then the connection given in \eqref{g3,1,nabla12,delta=0,sol3,1}   is precisely the one appearing in Table~\ref{g3,1}, associated with the flat Lie algebra $\h_{0,6}$

 The fourth solution corresponds to the following flat, torsion-free connection:
  \begin{equation}\label{g3,1,nabla12,delta=0,sol4}
 	\begin{aligned}
 		\nabla_{e_1}e_1&=a_{21}\,e_2+a_{31}\,e_3, &\nabla_{e_1}e_2&=a_{32}\,e_3,&\nabla_{e_2}e_1&=(a_{32}-1)\,e_3.
 	\end{aligned}
 \end{equation}
 Suppose that $a_{21}\neq0$, and consider the following automorphism
  \begin{align*}
 	\Psi(e_1) &=e_1,&
 	\Psi(e_2) &=\tfrac{1}{a_{21}}e_2-\tfrac{a_{31}}{a_{21}^2}e_3,&
 	\Psi(e_3) &=\tfrac{1}{a_{21}}e_3.
 \end{align*}
 Applying  $\Psi$ to the connection given in~\eqref{g3,1,nabla12,delta=0,sol4} yields the following equivalent connection:
  \begin{equation}\label{g3,1,nabla12,delta=0,sol4,1}
 	\begin{aligned}
 		\nabla_{e_1}e_1&=e_2, &\nabla_{e_1}e_2&=\lambda_2\,e_3,&\nabla_{e_2}e_1&=(\lambda_2-1)\,e_3.
 	\end{aligned}
 \end{equation}
In this case, $\lambda=a_{32}$. Note that this connection coincides with the one given in \eqref{g3,1,nabla12,delta=0,sol3,1}, in which case $\delta=1$ and $\lambda_1=\lambda_2-1$.
 
Suppose now that $a_{21}=0$, and consider the following automorphism 
 \begin{align*}
 	\Psi(e_1) &=e_1,&
 	\Psi(e_2) &=x\,e_2,&
 	\Psi(e_3) &=x\,e_3.
 \end{align*}
 For a suitable choice of the parameter $x\in\R^\ast$, applying  $\Psi$ to the connection given in~\eqref{g3,1,nabla12,delta=0,sol4} yields the following equivalent connection:
  \begin{equation}\label{g3,1,nabla12,delta=0,sol4,2}
 	\begin{aligned}
 		\nabla_{e_1}e_1&=\delta\,e_3, &\nabla_{e_1}e_2&=\lambda_3\,e_3,&\nabla_{e_2}e_1&=(\lambda_3-1)\,e_3.
 	\end{aligned}
 \end{equation}
 In this case, $\lambda_3=a_{32}$. Observe first that if $\delta=0$, then this connection coincides with the one given in \eqref{g3,1,nabla12,delta=0,sol3,1}, in which case, $\delta=0$ and $\lambda_1=\lambda_3-1$. On the other hand, if $\delta=1$, then the connection given in \eqref{g3,1,nabla12,delta=0,sol4,2}, is in fact, equivalent (up to isomorphism) to the connection associated with the flat Lie algebra $\h_{6,1}$ through the following automorphism:
  \begin{align*}
 	\Psi(e_1) &=-e_2,&
 	\Psi(e_2) &=e_1-\lambda_3\,e_2,&
 	\Psi(e_3) &=e_3.
 \end{align*}

 The fifth solution corresponds to the following flat, torsion-free connection:
 \begin{equation}\label{g3,1,nabla12,delta=0,sol5}
 	\begin{aligned}
 		\nabla_{e_1}e_1&=a_{11}\,e_1+a_{21}\,e_2+a_{31}\,e_3, &\nabla_{e_1}e_2&=e_3.
 	\end{aligned}
 \end{equation}
 Suppose that $a_{11}\neq0$, and consider the following automorphism
 \begin{align*}
 	\Psi(e_1) &=a_{11}\,e_1-\tfrac{a_{22}}{a_{11}}e_2-\tfrac{a_{11}a_{31}+a_{21}}{a_{11}}e_3,&
 	\Psi(e_2) &=e_2,&
 	\Psi(e_3) &=a_{11}\,e_3.
 \end{align*}
Applying  $\Psi$ to the connection given in~\eqref{g3,1,nabla12,delta=0,sol5} yields the following equivalent connection:
\begin{equation}\label{g3,1,nabla12,delta=0,sol5,1}
	\begin{aligned}
		\nabla_{e_1}e_1&=e_1, &\nabla_{e_1}e_2&=e_3,
	\end{aligned}
\end{equation} 
 which is exactly the one associated to the flat Lie algebra $\h_{0,1}$.
 
 Suppose now that $a_{11}=0$, and consider the following automorphism
 \begin{align*}
 	\Psi(e_1) &=e_1+a_{31}\,x\,e_2,&
 	\Psi(e_2) &=x\,e_2,&
 	\Psi(e_3) &=x\,e_3.
 \end{align*}
 For a suitable choice of the parameter $x\in\R^\ast$, applying  $\Psi$ to the connection given in~\eqref{g3,1,nabla12,delta=0,sol5} yields the following equivalent connection:
 \begin{equation}\label{g3,1,nabla12,delta=0,sol5,2}
 	\begin{aligned}
 		\nabla_{e_1}e_1&=\delta\,e_2, &\nabla_{e_1}e_2&=e_3,\quad\delta=0,1.
 	\end{aligned}
 \end{equation} 
 Observe that if $\delta=1$, then this connection coincides with the one associated to the flat Lie algebra $\h_{0,2}$. Otherwise, if $\delta=0$, then the connection given in \eqref{g3,1,nabla12,delta=0,sol5,2}, is identical to the one associated to the flat Lie algebra $\h_{0,3}$.

 Lastly, it is not difficult to verify that all flat, torsion-free connections presented in Table~\ref{g3,1} are not pairwise isomorphic. We have now completed the classification of flat, torsion-free connections on the flat Lie algebra $\G_{3,1}$.
\end{proof}

\begin{co}
	With the notations as above, among the flat Lie algebras on $\G_{3,1}$, we have
	\begin{enumerate}
		\item[i)] Associative algebras$:$\hspace{0.275cm} $\h_{0,3}$, $\h_{0,6}$, $\h_{6,1}$.
		\item[ii)] Novikov algebras$:$\hspace{0.73cm} $\h_{0,2}$, $\h_{0,3}$, $\h_{0,4}$, $\h_{0,5}$, $\h_{0,6}$, $\h_{4,1}$, $\h_{6,1}$.
		\item[iii)] Bi-symmetric algebras$:$ $\h_{0,2}$, $\h_{0,3}$, $\h_{0,5}$, $\h_{0,6}$, $\h_{6,1}$.
		\item[iv)] Complete algebras$:$\hspace{0.62cm}$\h_{0,2}$, $\h_{0,3}$, $\h_{0,5}$, $\h_{0,6}$, $\h_{6,1}$.
	\end{enumerate}
\end{co}

\begin{pr}
Let $(\G, \nabla)$ be a three-dimensional real flat  Lie algebra with $\G = \G_{3,2}$. Then $(\G, \nabla)$ is isomorphic to exactly one of the flat Lie algebras listed in Table~$\ref{g3,2}$.
{\renewcommand*{\arraystretch}{1.8}
\captionof{table}{Flat torsion-free connection on the Lie algebra $\G_{3,2}$.}
\setcounter{table}{5}
\begin{footnotesize} 
\setlength{\tabcolsep}{5pt} 
\begin{longtable}{@{}cllllllc@{}} 
			\hline
		Flat algebra&\multicolumn{2}{@{}l@{}}{~~Flat torsion-free connection}&&&&Remarks  \\
			\hline
$\h_{0,1}$&$\nabla_{e_2}e_2=\varepsilon e_3$&$\nabla_{e_3}e_1=-e_1$&$\nabla_{e_3}e_2=-e_1-e_2$&$\nabla_{e_3}e_3=-2e_3$&&$\varepsilon=\pm1$&\\
$\h_{0,2}$&$\nabla_{e_3}e_1=-e_1$&$\nabla_{e_3}e_2=-e_1-e_2$&$\nabla_{e_3}e_3=\lambda e_3$&&&$\lambda\in\R$&\\
$\h_{0,3}$&$\nabla_{e_3}e_1=-e_1$&$\nabla_{e_3}e_2=-e_1-e_2$&$\nabla_{e_3}e_3=e_2-e_3$&&&&\\
$\h_{0,4}$&$\nabla_{e_2}e_3=\lambda e_1$&$\nabla_{e_3}e_1=-e_1$&$\nabla_{e_3}e_2=(\lambda-1)e_1-e_2$&&&$\lambda\in\R^\ast$&\\
$\h_{0,5}$&$\nabla_{e_1}e_3=\lambda e_1$&$\nabla_{e_2}e_3=\lambda e_2$&$\nabla_{e_3}e_1=(\lambda-1)e_1$&$\nabla_{e_3}e_2=-e_1+(\lambda-1)e_2$&$\nabla_{e_3}e_3=\lambda e_3$&$\lambda\in\R^\ast$&\\
$\h_{0,6}$&$\nabla_{e_1}e_3= e_1$&$\nabla_{e_2}e_3= e_2$&$\nabla_{e_3}e_2=-e_2$&$\nabla_{e_3}e_3=e_2+e_3$&&&\\\hline
$\h_{2,1}$&$\nabla_{e_2}e_2=e_1$&$\nabla_{e_2}e_3=e_2$&$\nabla_{e_3}e_1=-e_1$&$\nabla_{e_3}e_2=-e_1$&$\nabla_{e_3}e_3=-e_2+e_3$&&\\
$\h_{2,2}$&$\nabla_{e_1}e_3=-e_1$&$\nabla_{e_2}e_2=e_1$&$\nabla_{e_3}e_1=-2e_1$&$\nabla_{e_3}e_2=-e_1-e_2$&$\nabla_{e_3}e_3=-e_2-e_3$&&
			\\\hline		
			\end{longtable}
			\label{g3,2}
			\end{footnotesize}	
			}
	
\end{pr}
\begin{proof}
	 In the basis $\lbrace e_1, e_2, e_3 \rbrace$, the operators $\nabla_{e_1}$, $\nabla_{e_2}$ and $\nabla_{e_3}$ are given respectively by: 
	\begin{equation}
		\nabla_{e_1}=\left( \begin {array}{ccc} a_{11}&a_{12}&a_{13}\\ \noalign{\medskip}
		a_{21}&a_{22}&a_{23}\\ \noalign{\medskip}a_{31}&a_{32}&a_{33}\end {array} \right),\quad
		\nabla_{e_2}=\left( \begin {array}{ccc} a_{12}&b_{12}&b_{13}\\ \noalign{\medskip}
		a_{22}&b_{22}&b_{23}\\ \noalign{\medskip}a_{32}&b_{32}&b_{33}\end {array} \right),\quad
		\nabla_{e_3}=\left( \begin {array}{ccc} a_{13}-1 &b_{13}-1&c_{13}\\ \noalign{\medskip}
		a_{23}&b_{23}-1&c_{23}\\ \noalign{\medskip}a_{33}&b_{33}&c_{33}\end {array} \right),\quad
	\end{equation}
	where $a_{ij}$, $b_{ij}$, $c_{ij}\in \mathbb{R}$.

		Assume that $\nabla^0$ is a  torsion-free connection. Then $\nabla^0$ is equivalent to one of the connections listed in Lemma~\ref{Lemg3j} under the Lie algebra $\G_{3,3}$.
		We begin by considering the cases  $\nabla^0=\nabla^j$, $j=1,\ldots,6$. A straightforward computation shows that the corresponding flatness-equations admit no solutions.
		
		If $\nabla^0=\nabla^{7}$, then a straightforward computation shows that the flatness-equation admits a unique solution, and the corresponding flat, torsion-free connection is given as follows:
	\begin{equation}\label{g3,2,nabla7,sol1}
		\begin{aligned}
	 \nabla_{e_2}e_2&=e_2+b_{32}e_3, \quad\nabla_{e_2}e_3=-\tfrac{1}{b_{32}}e_2-e_3, &\nabla_{e_3}e_1&=-e_1, &\nabla_{e_3}e_2&=-e_1-\tfrac{b_{32}+1}{b_{32}}e_2-e_3,\\
	 \nabla_{e_3}e_3&=\tfrac{1}{b_{32}}e_1+\tfrac{1-b_{32}}{b_{32}^2}e_2+\tfrac{1-2\,b_{32}}{b_{32}}e_3.
		\end{aligned}
	\end{equation}
		Consider the following automorphism
		\begin{align*}
			\Psi(e_1) &=\tfrac{\sqrt{\varepsilon\,b_{32}}}{\varepsilon}e_1,&
			\Psi(e_2) &=\tfrac{\sqrt{\varepsilon\,b_{32}}}{\varepsilon}e_2,&
			\Psi(e_3) &=-\tfrac{\sqrt{\varepsilon\,b_{32}}}{\varepsilon\,b_{32}}e_3.
		\end{align*}
		Applying $\Psi$ to the connection given in~\eqref{g3,2,nabla7,sol1} yields the following equivalent connection:
		\begin{equation}\label{g3,2,nabla7,sol1,1}
		\begin{aligned}
			\nabla_{e_2}e_2&=\varepsilon\, e_3, &\nabla_{e_3}e_1&=-e_1, &\nabla_{e_3}e_2&=-e_1-e_2,&
			\nabla_{e_3}e_3&=-2\,e_3.
		\end{aligned}
	\end{equation}
	This connection is precisely the one appearing in Table~\ref{g3,2}, associated with the flat Lie algebra $\h_{0,1}$.
	
	Suppose that $\nabla^0=\nabla^8$, and $\delta=1$. By a straightforward computation, the flatness-equations admit three real solutions.
	The flat, torsion-free connection associated with first solution is given by the following connection: 
	\begin{equation}\label{g3,2,nabla8,sol1}
		\begin{aligned}
			\nabla_{e_1}e_3&=-e_1, \quad\nabla_{e_2}e_2=e_1, &\nabla_{e_2}e_3&=b_{13}e_1, &\nabla_{e_3}e_1&=-2\,e_1, &\nabla_{e_3}e_2&=(b_{13}-1)e_1-e_2,\\
			\nabla_{e_3}e_3&=c_{13}e_1-e_2-e_3.
		\end{aligned}
	\end{equation}
	Consider the following automorphism
	\begin{align*}
		\Psi(e_1) &=e_1,&
		\Psi(e_2) &=e_2,&
		\Psi(e_3) &=\tfrac{b_{13}^2-b_{13}-c_{13}}{2}e_1+b_{13}e_2+e_3.
	\end{align*}
	Applying $\Psi$ to the connection given in~\eqref{g3,2,nabla8,sol1} yields the following equivalent connection:
	\begin{equation}\label{g3,2,nabla8,sol1,1}
		\begin{aligned}
			\nabla_{e_1}e_3&=-e_1, \quad\nabla_{e_2}e_2=e_1,  &\nabla_{e_3}e_1&=-2\,e_1, &\nabla_{e_3}e_2&=-e_1-e_2,&
			\nabla_{e_3}e_3&=-e_2-e_3.
		\end{aligned}
	\end{equation}
		This connection is precisely the one appearing in Table~\ref{g3,2}, associated with the flat Lie algebra $\h_{2,2}$.

The second solution is given by the following flat, torsion-free connection:		
		\begin{equation}\label{g3,2,nabla8,sol2}
			\begin{aligned}
\nabla_{e_2}e_2&=e_1, &\nabla_{e_2}e_3&=b_{13}e_1+e_2, &\nabla_{e_3}e_1&=-e_1, &\nabla_{e_3}e_2&=(b_{13}-1)e_1,&
				\nabla_{e_3}e_3&=c_{13}e_1-e_2+e_3.
			\end{aligned}
		\end{equation}
Consider the following automorphism
\begin{align*}
	\Psi(e_1) &=e_1,&
	\Psi(e_2) &=e_2,&
	\Psi(e_3) &=\tfrac{b_{13}^2-b_{13}-c_{13}}{2}e_1+b_{13}e_2+e_3.
\end{align*}
Applying $\Psi$ to the connection given in~\eqref{g3,2,nabla8,sol2} yields the following equivalent connection:		
		\begin{equation}\label{g3,2,nabla8,sol2,1}
			\begin{aligned}
				\nabla_{e_2}e_2&=e_1, &\nabla_{e_2}e_3&=e_2, &\nabla_{e_3}e_1&=-e_1, &\nabla_{e_3}e_2&=-e_1,&
				\nabla_{e_3}e_3&=-e_2+e_3.
			\end{aligned}
		\end{equation}
			This connection is precisely the one appearing in Table~\ref{g3,2}, associated with the flat Lie algebra $\h_{2,1}$.

	The third solution corresponds to the following flat, torsion-free connection:
		\begin{equation}\label{g3,2,nabla8,sol3}
		\begin{aligned}
			\nabla_{e_2}e_2&=e_1+b_{32}e_3, &\nabla_{e_3}e_1&=-e_1, &\nabla_{e_3}e_2&=-e_1-e_2, &\nabla_{e_3}e_3&=-\tfrac{1}{b_{31}}e_1-2\,e_3.
		\end{aligned}
	\end{equation}
	Consider the following automorphism
	\begin{align*}
		\Psi(e_1) &=\tfrac{\sqrt{\varepsilon\,b_{32}}}{\varepsilon}e_1,&
		\Psi(e_2) &=\tfrac{\sqrt{\varepsilon\,b_{32}}}{\varepsilon}e_2,&
		\Psi(e_3) &=-\tfrac{\sqrt{\varepsilon\,b_{32}}}{\varepsilon\,b_{32}}e_1+e_3.
	\end{align*}
		Applying $\Psi$ to the connection given in~\eqref{g3,2,nabla8,sol3} yields the following equivalent connection:
	\begin{equation}\label{g3,2,nabla8,sol3,1}
		\begin{aligned}
			\nabla_{e_2}e_2&=\varepsilon\, e_3, &\nabla_{e_3}e_1&=-e_1, &\nabla_{e_3}e_2&=-e_1-e_2,&
			\nabla_{e_3}e_3&=-2\,e_3.
		\end{aligned}
	\end{equation}
	This connection is precisely the one appearing in Table~\ref{g3,2}, associated with the flat Lie algebra $\h_{0,1}$.

	Suppose now that $\nabla^0=\nabla^8$, and $\delta=0$. By a straightforward computation, the flatness-equations admit four real solutions.
	The flat, torsion-free connection associated with first solution is given by the following connection:
		\begin{equation}\label{g3,2,nabla8,delta=0,sol1}
		\begin{aligned}
	\nabla_{e_1}e_3&=b_{23}e_1, \quad\nabla_{e_2}e_3=b_{23}e_2,
	&\nabla_{e_3}e_1&=(b_{23}-1)e_1, &\nabla_{e_3}e_2&=-e_1+(b_{23}-1)e_2,\\
	\nabla_{e_3}e_3&=c_{13}e_1+c_{23}e_2+b_{23}e_3.
		\end{aligned}
	\end{equation}
	Suppose that $b_{23}\neq1$, and consider the following automorphism
	\begin{align*}
		\Psi(e_1) &=e_1,&
		\Psi(e_2) &=e_2,&
		\Psi(e_3) &=\tfrac{(b_{23}-1)c_{13}+c_{23}}{(b_{23}-1)^2}e_1+\tfrac{c_{23}}{b_{23}-1}e_2+e_3.
	\end{align*}
	Applying $\Psi$ to the connection given in~\eqref{g3,2,nabla8,delta=0,sol1} yields the following equivalent connection:
	\begin{equation}\label{g3,2,nabla8,delta=0,sol1,1}
		\begin{aligned}
			\nabla_{e_1}e_3&=\lambda\,e_1, &\nabla_{e_2}e_3&=\lambda\,e_2,
			&\nabla_{e_3}e_1&=(\lambda-1)e_1, &\nabla_{e_3}e_2&=-e_1+(\lambda-1)e_2,&
			\nabla_{e_3}e_3&=\lambda\,e_3.
		\end{aligned}
	\end{equation}
	In this case, $\lambda=b_{23}\neq1$. 
	
	Suppose that $b_{23}=1$. Consider then the following automorphism
	\begin{align*}
		\Psi(e_1) &=x\,e_1,&
		\Psi(e_2) &= x\,e_2,&
		\Psi(e_3) &=-c_{13}\,x\,e_2+e_3.
	\end{align*}
	For a suitable choice of the parameter $x\in\R^\ast$, applying $\Psi$ to the connection given in~\eqref{g3,2,nabla8,delta=0,sol1} yields the following equivalent connection:
		\begin{equation}\label{g3,2,nabla8,delta=0,sol1,2}
		\begin{aligned}
			\nabla_{e_1}e_3&=e_1, &\nabla_{e_2}e_3&=e_2,
		 &\nabla_{e_3}e_2&=-e_1,&
			\nabla_{e_3}e_3&=\delta\,e_2+e_3,\quad\delta=0,1.
		\end{aligned}
	\end{equation}
	Observe first that if $\delta=0$, then this connection coincides with the one considered in~\eqref{g3,2,nabla8,delta=0,sol1,1}, in which case $\lambda=1$. Therefore, we set $\lambda\in\mathbb{R}$ in~ \eqref{g3,2,nabla8,delta=0,sol1,1} and assume that $\delta=1$. 
	In view of the previous discussion, the connections given in \eqref{g3,2,nabla8,delta=0,sol1,1} and \eqref{g3,2,nabla8,delta=0,sol1,2} correspond to the flat Lie algebras $\h_{0,5}$ and $\h_{0,6}$, respectively.   The additional conditions $\lambda\neq0$ in the flat Lie algebra $\h_{0,5}$ is imposed to ensure that the corresponding connection is not isomorphic to the  algebras $\h_{0,2}$.

	The second solution is given by the following flat, torsion-free connection:
	\begin{equation}\label{g3,2,nabla8,delta=0,sol2}
		\begin{aligned}
			\nabla_{e_3}e_1&=-e_1, &\nabla_{e_3}e_2&=-e_1-e_2, &\nabla_{e_3}e_3&=c_{13}e_1+c_{23}e_2+c_{33}e_3.
		\end{aligned}
	\end{equation}
	Suppose that $c_{33}\neq-1$, and consider the following automorphism
	\begin{align*}
		\Psi(e_1) &=e_1,&
		\Psi(e_2) &=e_2,&
		\Psi(e_3) &=\tfrac{(-c_{33}-1)c_{13}+c_{23}}{(c_{33}+1)^2}e_1-\tfrac{c_{23}}{c_{33}+1}e_2+e_3.
	\end{align*}
	Applying $\Psi$ to the connection given in~\eqref{g3,2,nabla8,delta=0,sol2} yields the following equivalent connection:
	\begin{equation}\label{g3,2,nabla8,delta=0,sol2,1}
		\begin{aligned}
			\nabla_{e_3}e_1&=-e_1, &\nabla_{e_3}e_2&=-e_1-e_2, &\nabla_{e_3}e_3&=\lambda\,e_3.
		\end{aligned}
	\end{equation}
	In this case, $\lambda=c_{33}\neq-1$.

	Suppose that $c_{33}=-1$. Consider then the following automorphism
	\begin{align*}
		\Psi(e_1) &=x\,e_1,&
	\Psi(e_2) &= x\,e_2,&
	\Psi(e_3) &=-c_{13}\,x\,e_2+e_3.
	\end{align*}
	For a suitable choice of the parameter $x\in\R^\ast$, applying $\Psi$ to the connection given in~\eqref{g3,2,nabla8,delta=0,sol2} yields the following equivalent connection:
	\begin{equation}\label{g3,2,nabla8,delta=0,sol2,2}
		\begin{aligned}
			\nabla_{e_3}e_1&=-e_1, &\nabla_{e_3}e_2&=-e_1-e_2,
			&\nabla_{e_3}e_3&=\delta\,e_2-e_3,\quad\delta=0,1.
		\end{aligned}
	\end{equation}
	Observe first that if $\delta=0$, then this connection coincides with the one considered in~\eqref{g3,2,nabla8,delta=0,sol2,1}, in which case $\lambda=-1$. Therefore, we set $\lambda\in\mathbb{R}$ in~ \eqref{g3,2,nabla8,delta=0,sol2,1} and assume that $\delta=1$. 
	In view of the previous discussion, the connections given in \eqref{g3,2,nabla8,delta=0,sol2,1} and \eqref{g3,2,nabla8,delta=0,sol2,2} correspond to the flat Lie algebras $\h_{0,2}$ and $\h_{0,3}$, respectively.

	 The third solution corresponds to the following flat, torsion-free connection:
		\begin{equation}\label{g3,2,nabla8,delta=0,sol3}
		\begin{aligned}
			\nabla_{e_2}e_3&=b_{13}e_1,
			&\nabla_{e_3}e_1&=-e_1, &\nabla_{e_3}e_2&=(b_{13}-1)e_1-e_2,
			&\nabla_{e_3}e_3&=c_{13}e_1+c_{23}e_2.
		\end{aligned}
	\end{equation}
	 Consider then the following automorphism
	\begin{align*}
		\Psi(e_1) &=e_1,&
		\Psi(e_2) &= e_2,&
		\Psi(e_3) &=(c_{23}-c_{13}-2\,c_{23}b_{13})e_1-c_{23}e_2+e_3.
	\end{align*}
Applying $\Psi$ to the connection given in~\eqref{g3,2,nabla8,delta=0,sol3} yields the following equivalent connection:
		\begin{equation}\label{g3,2,nabla8,delta=0,sol3,1}
		\begin{aligned}
			\nabla_{e_2}e_3&=\lambda\,e_1,
			&\nabla_{e_3}e_1&=-e_1, &\nabla_{e_3}e_2&=(\lambda-1)e_1-e_2.
		\end{aligned}
	\end{equation}
	This connection is precisely the one appearing in Table~\ref{g3,2}, associated with the flat Lie algebra $\h_{0,4}$.

	The fourth solution corresponds to the following flat, torsion-free connection:
	\begin{equation}\label{g3,2,nabla8,delta=0,sol4}
		\begin{aligned}
			\nabla_{e_2}e_2&=b_{32}e_3, &\nabla_{e_3}e_1&=-e_1, &\nabla_{e_3}e_2&=-e_1-e_2, &\nabla_{e_3}e_3&=-2\,e_3.
		\end{aligned}
	\end{equation}
	Consider then the following automorphism
	\begin{align*}
		\Psi(e_1) &=x\,e_1,&
		\Psi(e_2) &= x\,e_2,&
		\Psi(e_3) &=e_3.
	\end{align*}
	For a suitable choice of the parameter $x\in\R^\ast$, applying $\Psi$ to the connection given in~\eqref{g3,2,nabla8,delta=0,sol4} yields the following equivalent connection:
		\begin{equation}\label{g3,2,nabla8,delta=0,sol4,1}
		\begin{aligned}
			\nabla_{e_2}e_2&=\delta_\varepsilon e_3, &\nabla_{e_3}e_1&=-e_1, &\nabla_{e_3}e_2&=-e_1-e_2, &\nabla_{e_3}e_3&=-2\,e_3,\quad\delta_\varepsilon=0,\pm1.
		\end{aligned}
	\end{equation}
	Note that if $\delta_\varepsilon=0$, then this connection coincides with the one associated to the flat Lie algebra $\h_{0,2}$, in which case, $\lambda=-2$. On the other hand, if $\delta_\varepsilon=\varepsilon=\pm1$, we obtai the one associated to the flat Lie algebra $\h_{0,1}$ listed in Table~\ref{g3,2}. torsion-free connections on the flat Lie algebra $\G_{3,2}$.   It is straightforward to verify that all flat, torsion-free connections listed in Table~\ref{g3,2}5 are pairwise non-isomorphic.

\end{proof}

\begin{co}
	With the notations as above, among the flat Lie algebras on $\G_{3,2}$, we have
	\begin{enumerate}
		\item[i)] Associative algebras$:$
		\item[ii)] Novikov algebras$:$\hspace{0.73cm} 
		$\h_{0,2}^{\lambda=0}$, $\h_{0,5}$, $\h_{0,6}$.
		\item[iii)] Bi-symmetric algebras$:$
		\item[iv)] Complete algebras$:$\hspace{0.62cm}$\h_{0,2}^{\lambda=0}$, $\h_{0,4}$.
	\end{enumerate}
\end{co}

\begin{pr}
Let $(\G, \nabla)$ be a three-dimensional real flat  Lie algebra with $\G = \G_{3,3}$. Then $(\G, \nabla)$ is isomorphic to exactly one of the flat Lie algebras listed in Table~$\ref{g3,3}$.
{\renewcommand*{\arraystretch}{1.8}
\captionof{table}{Flat torsion-free connection on the Lie algebra $\G_{3,3}$.}
\setcounter{table}{6}
\begin{footnotesize} 
\setlength{\tabcolsep}{5pt} 
\begin{longtable}{@{}cllllllc@{}} 
			\hline
		Flat algebra& \multicolumn{2}{@{}l@{}}{~~Flat torsion-free connection}&&&&Remarks\\
			\hline
$\h_{0,1}$&$\nabla_{e_1}e_3=\lambda e_1$&$\nabla_{e_2}e_3=\lambda e_2$&$\nabla_{e_3}e_1=(\lambda-1) e_1$&$\nabla_{e_3}e_2=(\lambda-1) e_2$&$\nabla_{e_3}e_3=\lambda e_3$&$\lambda\in\R$&\\
$\h_{0,2}$&$\nabla_{e_1}e_3= e_1$&$\nabla_{e_2}e_3= e_2$&$\nabla_{e_3}e_3=e_2+ e_3$&&&&\\
$\h_{0,3}$&$\nabla_{e_1}e_3=\lambda e_1$&$\nabla_{e_3}e_1=(\lambda-1) e_1$&$\nabla_{e_3}e_2=-e_2$&$\nabla_{e_3}e_3=\lambda e_3$&&$\lambda\in\R^{\ast}$&\\
$\h_{0,4}$&$\nabla_{e_1}e_3=- e_1$&$\nabla_{e_3}e_1=-2 e_1$&$\nabla_{e_3}e_2=-e_2$&$\nabla_{e_3}e_3=e_2- e_3$&&&\\
$\h_{0,5}$&$\nabla_{e_1}e_3= e_1$&$\nabla_{e_3}e_2=-e_2$&$\nabla_{e_3}e_3=e_1+ e_3$&&&&\\
$\h_{0,6}$&$\nabla_{e_2}e_3= e_1$&$\nabla_{e_3}e_1=-e_1$&$\nabla_{e_3}e_2=e_1- e_2$&&&&\\
$\h_{0,7}$&$\nabla_{e_3}e_1=- e_1$&$\nabla_{e_3}e_2=-e_2$&$\nabla_{e_3}e_3=\lambda e_3$&&&$\lambda\in\R^\ast$&\\
$\h_{0,8}$&$\nabla_{e_3}e_1=- e_1$&$\nabla_{e_3}e_2=-e_2$&$\nabla_{e_3}e_3=e_1- e_3$&&&&\\
$\h_{0,9}$&$\nabla_{e_2}e_2=\varepsilon e_3$&$\nabla_{e_3}e_1=- e_1$&$\nabla_{e_3}e_2=-e_2$&$\nabla_{e_3}e_3=-2 e_3$&&$\varepsilon=\pm1$&\\
$\h_{0,10}$&$\nabla_{e_1}e_2= e_3$&$\nabla_{e_2}e_1= e_3$&$\nabla_{e_3}e_1=- e_1$&$\nabla_{e_3}e_2=-e_2$&$\nabla_{e_3}e_3=-2 e_3$&&\\
$\h_{0,11}$&$\nabla_{e_1}e_1=\varepsilon_1 e_3$&$\nabla_{e_2}e_2=\varepsilon_2 e_3$&$\nabla_{e_3}e_1=-e_1$&$\nabla_{e_3}e_2=-e_2$&$\nabla_{e_3}e_3=-2e_3$&$\varepsilon_1,\varepsilon_2=\pm1$&
\\\hline
$\h_{2,1}$&$\nabla_{e_1}e_3= -e_1$&$\nabla_{e_2}e_2= e_1$&$\nabla_{e_3}e_1=-2 e_1$&$\nabla_{e_3}e_2= -e_2$&$\nabla_{e_3}e_3= -e_3$&&\\
$\h_{2,2}$&$\nabla_{e_2}e_2= e_1$&$\nabla_{e_2}e_3= e_2$&$\nabla_{e_3}e_1= -e_1$&$\nabla_{e_3}e_3= e_3$&&&
\\\hline		
			\end{longtable}
			\label{g3,3}
			\end{footnotesize}	
			}

\end{pr}
\begin{proof}
	 In the basis $\lbrace e_1, e_2, e_3 \rbrace$, the operators $\nabla_{e_1}$, $\nabla_{e_2}$ and $\nabla_{e_3}$ are given respectively by: 
	\begin{equation}
		\nabla_{e_1}=\left( \begin {array}{ccc} a_{11}&a_{12}&a_{13}\\ \noalign{\medskip}
		a_{21}&a_{22}&a_{23}\\ \noalign{\medskip}a_{31}&a_{32}&a_{33}\end {array} \right),\quad
		\nabla_{e_2}=\left( \begin {array}{ccc} a_{12}&b_{12}&b_{13}\\ \noalign{\medskip}
		a_{22}&b_{22}&b_{23}\\ \noalign{\medskip}a_{32}&b_{32}&b_{33}\end {array} \right),\quad
		\nabla_{e_3}=\left( \begin {array}{ccc} a_{13}-1 &b_{13}&c_{13}\\ \noalign{\medskip}
		a_{23}&b_{23}-1&c_{23}\\ \noalign{\medskip}a_{33}&b_{33}&c_{33}\end {array} \right),\quad
	\end{equation}
	where $a_{ij}$, $b_{ij}$, $c_{ij}\in \mathbb{R}$.

Assume that $\nabla^0$ is a torsion-free connection. Then, by Lemma~\ref{no flat to flat}, we may assume that $\nabla^0$ is flat and torsion-free. Consequently, $\nabla^0$ is isomorphic to one of the flat, torsion-free connections presented in Table~\ref{FlatR2}.
 We first consider the case $\nabla^0\equiv0$. Using a straightforward computation, the flatness-equations can be solved with seven real solutions. The first one is given by the following flat, torsion-free connection:
 \begin{equation}\label{g3,3,nabla1,sol1}
 	\begin{aligned}
 		\nabla_{e_1} e_1 &=a_{31}e_3, &\nabla_{e_1}e_2&=a_{32}e_3, &\nabla_{e_2}e_1&=a_{32}e_3, &\nabla_{e_2}e_2&=b_{32}e_3,\\ \nabla_{e_3}e_1&=-e_1, &\nabla_{e_3}e_2&=-e_2, &\nabla_{e_3}e_3&=-2\,e_3.
 	\end{aligned}
 \end{equation}
 
 Suppose that $b_{32}\neq0$, and consider the following automorphism
 \begin{align*}
 	\Psi(e_1) &=x\, e_1+\tfrac{a_{32}\sqrt{\varepsilon_2\,b_{32}}}{\varepsilon_2\,b_{32}}e_2,&
 	\Psi(e_2) &=\tfrac{\sqrt{\varepsilon_2\,b_{32}}}{\varepsilon_2}e_2,&
 	\Psi(e_3) &=e_3,\quad\varepsilon_2=\pm1.
 \end{align*}
 For a suitable choice of the parameter $x\in\R^\ast$, applying $\Psi$ to the connection given in~\eqref{g3,3,nabla1,sol1} yields the following equivalent connection:
 \begin{equation}\label{g3,3,nabla1,sol1,1}
 	\begin{aligned}
 		\nabla_{e_1} e_1 &=\delta_\varepsilon\, e_3,  &\nabla_{e_2}e_2&=\epsilon_2\, e_3,&\nabla_{e_3}e_1&=-e_1, &\nabla_{e_3}e_2&=-e_2, &\nabla_{e_3}e_3&=-2\,e_3.
 	\end{aligned}
 \end{equation}
 If $\delta_\varepsilon=0$, we obtain the flat, torsion-free connection associated with the flat Lie algebra $\h_{0,9}$, in which case $\varepsilon_2=\varepsilon$. Otherwise, if $\delta_\varepsilon=\pm1$, we obtain the flat, torsion-free connection associated with the Lie algebra $\h_{0,11}$, in this case $\delta_\varepsilon=\varepsilon_1$.

 Suppose that $b_{32}=0$ and $a_{32}\neq0$, and consider the following automorphism
 \begin{align*}
 	\Psi(e_1) &=\tfrac{a_{32}}{2}e_1+e_2,&
 	\Psi(e_2) &=a_{32}e_2,&
 	\Psi(e_3) &=e_3.
 \end{align*}
 Applying $\Psi$ to the connection given in~\eqref{g3,3,nabla1,sol1} yields the following equivalent connection:
  \begin{equation}\label{g3,3,nabla1,sol1,2}
 	\begin{aligned}
 		\nabla_{e_1} e_2 &= e_3,  &\nabla_{e_2}e_1&=e_3,&\nabla_{e_3}e_1&=-e_1, &\nabla_{e_3}e_2&=-e_2, &\nabla_{e_3}e_3&=-2\,e_3.
 	\end{aligned}
 \end{equation}
 Then  this connection is precisely the one associated with the flat Lie algebra $\h_{0,10}$ listed in Table~\ref{g3,3}.

  Suppose now that $b_{32}=0$ and $a_{32}=0$, and consider the following automorphism
  \begin{align*}
 	\Psi(e_1) &=x\,e_2,&
 	\Psi(e_2) &=e_1,&
 	\Psi(e_3) &=e_3.
 \end{align*}
  For a suitable choice of the parameter $x\in\R^\ast$, applying $\Psi$ to the connection given in~\eqref{g3,3,nabla1,sol1} yields the following equivalent connection:
  \begin{equation}\label{g3,3,nabla1,sol1,3}
 	\begin{aligned}
 		\nabla_{e_2} e_2 &=\delta_\varepsilon\, e_3,  &\nabla_{e_3}e_1&=-e_1, &\nabla_{e_3}e_2&=-e_2, &\nabla_{e_3}e_3&=-2\,e_3,\quad\delta_\varepsilon=0,\pm1.
 	\end{aligned}
 \end{equation}
 	If $\delta_\varepsilon=0$, then this connection coincides with the one associated with the flat Lie algebra $\h_{0,7}$, in which case, $\lambda=-2$. On the other hand, if $\delta_\varepsilon=\pm1$, it coincides with the one associated with the flat Lie algebra $\h_{0,9}$, in this case, $\delta_\varepsilon=\varepsilon$.
 	
 	The second solution is given by the following flat, torsion-free connection:
  \begin{equation}\label{g3,3,nabla1,sol2}
 	\begin{aligned}
 		\nabla_{e_1} e_3 &=a_{13}e_1+a_{23}e_2, &\nabla_{e_2}e_3&=\tfrac{b_{23}a_{13}}{a_{23}}e_1+b_{23}e_2, &\nabla_{e_3}e_1&=(a_{13}-1)e_1+a_{23}e_2,\\
 		\nabla_{e_3}e_2&=\tfrac{b_{23}a_{13}}{a_{23}}e_1+(b_{23}-1)e_2, &\nabla_{e_3}e_3&=c_{13}e_1+c_{23}e_2+(a_{13}+b_{23})e_3.
 	\end{aligned}
 \end{equation}
 Suppose that $a_{13}\neq0$ and $a_{1 3} + b_{2 3}\neq\pm1$. Consider the following automorphism
  \begin{align*}
 	\Psi(e_1) &=e_1-\tfrac{a_{23}}{a_{13}}e_2,&
 	\Psi(e_2) &=\tfrac{b_{23}}{a_{23}} e_1+e_2,&
 	\Psi(e_3) &=\tfrac{a_{23}c_{13}+b_{23}c_{23}}{a_{23}(a_{13}+b_{23}-1)}e_1+\tfrac{a_{23}c_{13}-a_{13}c_{23}}{a_{23}(a_{13}+b_{23}+1)}e_2+e_3.
 \end{align*}
Applying $\Psi$ to the connection given in~\eqref{g3,3,nabla1,sol2} yields the following equivalent connection:
 \begin{equation}\label{g3,3,nabla1,sol2,1}
 	\begin{aligned}
 		\nabla_{e_1} e_3 &=\lambda\,e_1, &\nabla_{e_3}e_1&=(\lambda-1)e_1, &\nabla_{e_3}e_2&=-e_2, &\nabla_{e_3}e_3&=\lambda\,e_3.
 	\end{aligned}
 \end{equation}
In this case $\lambda=a_{13}+b_{23}\neq\pm1$.

 Suppose that $a_{13}\neq0$ and $a_{1 3} + b_{2 3}=-1$. Consider the following automorphism
 \begin{align*}
 	\Psi(e_1) &=e_1-\tfrac{x\,a_{23}}{a_{13}}e_2,&
 	\Psi(e_2) &=\tfrac{b_{23}}{a_{23}} e_1+x\,e_2,&
 	\Psi(e_3) &=\tfrac{a_{13}c_{23}-c_{13}a_{23}+c_{23}}{2\,a_{23}}e_1+e_3.
 \end{align*}
 For a suitable choice of the parameter $x\in\R^\ast$, applying $\Psi$ to the connection given in~\eqref{g3,3,nabla1,sol2} yields the following equivalent connection:
 \begin{equation}\label{g3,3,nabla1,sol2,2}
 	\begin{aligned}
 		\nabla_{e_1} e_3 &=-e_1, &\nabla_{e_3}e_1&=-2\,e_1, &\nabla_{e_3}e_2&=-e_2, &\nabla_{e_3}e_3&=\delta_1\,e_2-e_3,\quad\delta_1=0,1.
 	\end{aligned}
 \end{equation}

  Suppose now that $a_{13}\neq0$ and $a_{1 3} + b_{2 3}=1$. Consider the following automorphism
 \begin{align*}
 	\Psi(e_1) &=x\,e_1-\tfrac{x\,a_{23}}{a_{13}}e_2,&
 	\Psi(e_2) &=\tfrac{x\,b_{23}}{a_{23}} e_1+e_2,&
 	\Psi(e_3) &=\tfrac{c_{13}a_{23}-a_{13}c_{23}}{2\,a_{23}}e_2+e_3.
 \end{align*}
 For a suitable choice of the parameter $x\in\R^\ast$, applying $\Psi$ to the connection given in~\eqref{g3,3,nabla1,sol2} yields the following equivalent connection:
 \begin{equation}\label{g3,3,nabla1,sol2,3}
 	\begin{aligned}
 		\nabla_{e_1} e_3 &=e_1,  &\nabla_{e_3}e_2&=-e_2, &\nabla_{e_3}e_3&=\delta_2\,e_1+e_3,\quad\delta_2=0,1.
 	\end{aligned}
 \end{equation}

 Observe  that the flat, torsion-free connections given in \eqref{g3,3,nabla1,sol2,1} and  \eqref{g3,3,nabla1,sol2,2} are isomorphic if and only if $\delta_1=0$ and $\lambda=-1$. For this reason, we set $\lambda\in\R \setminus{\{-1\}}$ and assume that $\delta_1=1$. Similarly, one can easily show that \eqref{g3,3,nabla1,sol2,1} and  \eqref{g3,3,nabla1,sol2,3} are isomorphic if and only if $\delta_2=0$ and $\lambda=1$.  For this reason, we set $\lambda\in\R$ and assume that $\delta_2=1$.

 The connections given in \eqref{g3,3,nabla1,sol2,1}, \eqref{g3,3,nabla1,sol2,2}, and \eqref{g3,3,nabla1,sol2,3} correspond to the flat, torsion-free connections associated with  the flat Lie algebras $\h_{0,3}$, $\h_{0,4}$ and $\h_{0,5}$, respectively. The additional conditions $\lambda\neq0$ in the flat Lie algebra $\h_{0,3}$ is imposed to ensure that the corresponding connection is not isomorphic to the  algebras $\h_{0,1}$, as will be shown later.

 Suppose that $a_{13}=0$. If $b_{23}\neq\pm1$, consider the following automorphism
 \begin{align*}
 	\Psi(e_1) &=e_1,&
 	\Psi(e_2) &=\tfrac{b_{23}}{a_{23}} e_1+\tfrac{1}{a_{23}} e_2,&
 	\Psi(e_3) &=\tfrac{c_{13}a_{23}+b_{23}c_{23}}{a_{23}(b_{23}-1)}e_1+\tfrac{2\,c_{13}a_{23}+b_{23}c_{23}+c_{23}}{a_{23}(b_{23}^2-1)}e_2+e_3.
 \end{align*}
  Applying $\Psi$ to the connection given in~\eqref{g3,3,nabla1,sol2} yields the following equivalent connection:
  \begin{equation}\label{g3,3,nabla1,sol2,4}
  	\begin{aligned}
  		\nabla_{e_1} e_3 &=\lambda_2\,e_1+e_2,  &\nabla_{e_3}e_1&=(\lambda_2-1)e_1+e_2, &\nabla_{e_3}e_2&=-e_2, &\nabla_{e_3}e_3&=\lambda_2\,e_3.
  	\end{aligned}
  \end{equation}
  In this case $\lambda_2=b_{23}\neq\pm1$.
 This flat, torsion-free connection is, in fact, equivalent (up to isomorphism) to the connection associated with the flat Lie algebra $\h_{0,3}$ through the following automorphism:
\begin{align*}
	\Psi(e_1) &=e_1+e_2,&
	\Psi(e_2) &=-\lambda_2\, e_2,&
	\Psi(e_3) &=e_3,
\end{align*} 
 with $\lambda_1=\lambda_2\neq0$.

 If $\lambda_2=0$, consider the following automorphism:
 \begin{align*}
 	\Psi(e_1) &= e_2, &
 	\Psi(e_2) &= e_1, &
 	\Psi(e_3) &= e_3.
 \end{align*}
 Applying this automorphism to the connection given in~\eqref{g3,3,nabla1,sol2,4} yields an equivalent connection
 \begin{equation}\label{g3,3,nabla1,sol2,5}
 	\begin{aligned}
 		\nabla_{e_2} e_3 &=e_1,  &\nabla_{e_3}e_1&=-e_1, &\nabla_{e_3}e_2&=e_1-e_2,
 	\end{aligned}
 \end{equation}
  which coincides with the one associated with the flat Lie algebra $\h_{0,6}$ listed in Table~\ref{g3,3}.

  Suppose that $a_{13}=0$ and  $b_{23}=-1$, Consider the following automorphism
  \begin{align*}
  	\Psi(e_1) &=x\,a_{23}\,e_1,&
  	\Psi(e_2) &=-x\, e_1+x\,e_2,&
  	\Psi(e_3) &=\tfrac{x\,c_{23}}{2}e_1+e_3.
  \end{align*}
  For a suitable choice of the parameter $x\in\R^\ast$, applying $\Psi$ to the connection given in~\eqref{g3,3,nabla1,sol2} yields the following equivalent connection:
  \begin{equation}\label{g3,3,nabla1,sol2,6}
 	\begin{aligned}
 		\nabla_{e_1} e_3 &=-e_1+e_2,  &\nabla_{e_3}e_1&=-2\,e_1+e_2, &\nabla_{e_3}e_2&=-e_2, &\nabla_{e_3}e_3&=\delta_3\,e_1-e_3,\quad\delta_3=0,1.
 	\end{aligned}
 \end{equation}

 Note that if  $\delta_3=0$, then this connection coincides with the one given in~\eqref{g3,3,nabla1,sol2,4} with $\lambda_2=-1$. Therefore, it is isomorphic to the connection associated with the flat Lie algebra $\h_{0,3}$, which corresponds to the case $\lambda_1=-1$.
 
 Otherwise, if $\delta_3=1$, the following automorphism
 \begin{align*}
 	\Psi(e_1) &= -2\,e_1+e_2, &
 	\Psi(e_2) &= e_2, &
 	\Psi(e_3) &=e_1+e_3,
 \end{align*}
 establishes an isomorphism between the connection given in~\eqref{g3,3,nabla1,sol2,6} and the one given in~\eqref{g3,3,nabla1,sol2,2} with $\delta_1=1$.

  Suppose that $a_{13}=0$ and  $b_{23}=1$, consider the following automorphism
 \begin{align*}
 	\Psi(e_1) &=x\,a_{23}\,e_1,&
 	\Psi(e_2) &=x\, e_1+x\,e_2,&
 	\Psi(e_3) &=\tfrac{x\,c_{23}}{2}e_1+e_3.
 \end{align*}
 For a suitable choice of the parameter $x\in\R^\ast$, applying $\Psi$ to the connection given in~\eqref{g3,3,nabla1,sol2} yields the following equivalent connection:
 \begin{equation}\label{g3,3,nabla1,sol2,7}
 	\begin{aligned}
 		\nabla_{e_1} e_3 &=e_1+e_2,  &\nabla_{e_3}e_2&=-e_2, &\nabla_{e_3}e_3&=\delta_4\,e_1+e_3,\quad\delta_4=0,1.
 	\end{aligned}
 \end{equation}

 Note that if  $\delta_4=0$, then this connection coincides with the one given in~\eqref{g3,3,nabla1,sol2,4} with $\lambda_2=1$. Therefore, it is isomorphic to the connection associated with the flat Lie algebra $\h_{0,3}$, which corresponds to the case $\lambda_1=1$.
 
 Otherwise, if $\delta_4=1$, the following automorphism
 \begin{align*}
 	\Psi(e_1) &= e_1-2\,e_2, &
 	\Psi(e_2) &= 2\,e_2, &
 	\Psi(e_3) &=e_2+e_3,
 \end{align*}
 establishes an isomorphism between the connection given in~\eqref{g3,3,nabla1,sol2,6} and the one given in~\eqref{g3,3,nabla1,sol2,3} with $\delta_2=1$.
 
 The third solution corresponds to the following flat, torsion-free connection:
 \begin{equation}\label{g3,3,nabla1,sol3}
 	\begin{aligned}
 		\nabla_{e_1} e_3 &=a_{13}e_1, \quad\nabla_{e_2}e_3=b_{13}e_1, &\nabla_{e_3}e_1&=(a_{13}-1)e_1, &\nabla_{e_3}e_2&=b_{13}e_1-e_2, \\\nabla_{e_3}e_3&=c_{13}e_1+c_{23}e_2+a_{13}e_3.
 	\end{aligned}
 \end{equation}
 Suppose that $a_{13}\neq0$. If $a_{13}\neq\pm1$, consider the following automorphism
 \begin{align*}
 	\Psi(e_1) &= e_1+\tfrac{1}{a_{13}}e_2, &
 	\Psi(e_2) &= \tfrac{b_{13}}{a_{13}}e_1+e_2, &
 	\Psi(e_3) &=\tfrac{a_{13}c_{13}+b_{13}c_{23}}{a_{13}\,(a_{13}-1)}e_1+\tfrac{-a_{13}^2c_{23}+(c_{13}+c_{23})a_{13}+2\,b_{13}c_{23}+c_{13}}{a_{13}\,(a_{13}^2-1)}e_2+e_3.
 \end{align*}
 
 Applying $\Psi$ to the connection given in~\eqref{g3,3,nabla1,sol3} yields the following equivalent connection:
 \begin{equation}\label{g3,3,nabla1,sol3,1}
 	\begin{aligned}
 		\nabla_{e_1} e_3 &=\lambda_3\,e_1+e_2,  &\nabla_{e_3}e_1&=(\lambda_3-1)e_1+e_2, &\nabla_{e_3}e_2&=-e_2, &\nabla_{e_3}e_3&=\lambda_3\,e_3.
 	\end{aligned}
 \end{equation}
 In this case $\lambda_3=a_{13}\neq0,\pm1$. 
 Observe that this connection coincides with the one given in~\eqref{g3,3,nabla1,sol2,4}. Therefore, $\lambda_3=\lambda_2$, and the result follows without further analysis.

  Suppose that $a_{13}\neq0$. If $a_{13}=-1$, consider the following automorphism
  \begin{align*}
  	\Psi(e_1) &=-y\,e_1+y\,e_2,&
  	\Psi(e_2) &=y\,b_{13}\, e_1+x\,e_2,&
  	\Psi(e_3) &=\tfrac{x\,c_{23}+y\,c_{13}}{2}e_1+e_3.
  \end{align*}
  For a suitable choice of the parameters $x\in\R$ or $y\in\R^\ast$, applying $\Psi$ to the connection given in~\eqref{g3,3,nabla1,sol3} yields the following equivalent connection:
  \begin{equation}\label{g3,3,nabla1,sol3,2}
 	\begin{aligned}
 		\nabla_{e_1} e_3 &=-e_1+e_2,  &\nabla_{e_3}e_1&=-2\,e_1+e_2, &\nabla_{e_3}e_2&=-e_2, &\nabla_{e_3}e_3&=\delta_5\,e_1-e_3,\quad\delta_5=0,1.
 	\end{aligned}
 \end{equation}
This flat, torsion-free connection coincides with the one considered in \eqref{g3,3,nabla1,sol2,6}, in which case, $\delta_5=\delta_3$.

 Suppose that $a_{13}\neq0$. If $a_{13}=1$, consider the following automorphism
 \begin{align*}
 	\Psi(e_1) &=y\,e_1+y\,e_2,&
 	\Psi(e_2) &=y\,b_{13}\, e_1+x\,e_2,&
 	\Psi(e_3) &=\tfrac{x\,c_{23}+y\,c_{13}}{2}e_1+e_3.
 \end{align*}
 For a suitable choice of the parameters $x\in\R^\ast$ or $y\in\R$, applying $\Psi$ to the connection given in~\eqref{g3,3,nabla1,sol3} yields the following equivalent connection:
 \begin{equation}\label{g3,3,nabla1,sol3,3}
 	\begin{aligned}
 		\nabla_{e_1} e_3 &=e_1+e_2,   &\nabla_{e_3}e_2&=-e_2, &\nabla_{e_3}e_3&=\delta_6\,e_1+e_3,\quad\delta_6=0,1.
 	\end{aligned}
 \end{equation}
 This flat, torsion-free connection coincides with the one considered in \eqref{g3,3,nabla1,sol2,7}, in which case, $\delta_6=\delta_4$.
 
Suppose that $a_{13}=0$. Consider the following automorphism 
\begin{align*}
	\Psi(e_1) &=e_2,&
	\Psi(e_2) &=x\, e_1,&
	\Psi(e_3) &=-x\,c_{23}e_1+(-2\,b_{13}c_{23}-c_{13})e_2+e_3.
\end{align*}
For a suitable choice of the parameters $x\in\R^\ast$, applying $\Psi$ to the connection given in~\eqref{g3,3,nabla1,sol3} yields the following equivalent connection:
 \begin{equation}\label{g3,3,nabla1,sol3,4}
 	\begin{aligned}
 		\nabla_{e_1} e_3 &=\delta\,e_2,
 		&\nabla_{e_3}e_1&=-e_1+\delta\,e_2,   &\nabla_{e_3}e_2&=-e_2,\quad\delta=0,1.
 	\end{aligned}
 \end{equation}
 Observe that when $\delta=1$, this flat, torsion-free connection coincides with the one considered in~\eqref{g3,3,nabla1,sol3,1}, yielding the case $\lambda_3=0$.
 
 On the other hand, when $\delta=0$, the connection given in~\eqref{g3,3,nabla1,sol3,4} is equivalent to the one associated with the flat Lie algebra listed in Table~\ref{g3,3}, which corresponds to the case $\lambda=0$.

 The fourth solution corresponds to the following flat, torsion-free connection:
 \begin{equation}\label{g3,3,nabla1,sol4}
 	\begin{aligned}
 		\nabla_{e_1} e_3 &=b_{23}e_1, \quad\nabla_{e_2}e_3=b_{23}e_2, &\nabla_{e_3}e_1&=(b_{23}-1)e_1, &\nabla_{e_3}e_2&=(b_{23}-1)e_2, \\\nabla_{e_3}e_3&=c_{13}e_1+c_{23}e_2+b_{23}e_3.
 	\end{aligned}
 \end{equation}
 
 If $b_{23}\neq1$. Consider the following automorphism
 \begin{align*}
 	\Psi(e_1) &=e_1,&
 	\Psi(e_2) &=e_2,&
 	\Psi(e_3) &=\tfrac{c_{13}}{b_{23}-1} e_1+\tfrac{c_{23}}{b_{23}-1}e_2+e_3.
 \end{align*}
 Applying $\Psi$ to the connection given in~\eqref{g3,3,nabla1,sol4} yields the following equivalent connection:
 \begin{equation}\label{g3,3,nabla1,sol4,1}
 	\begin{aligned}
 		\nabla_{e_1} e_3 &=\lambda\,e_1, \quad\nabla_{e_2}e_3=\lambda\,e_2, &\nabla_{e_3}e_1&=(\lambda-1)e_1, &\nabla_{e_3}e_2&=(\lambda-1)e_2, &\nabla_{e_3}e_3&=\lambda\,e_3.
 	\end{aligned}
 \end{equation}
 In this case, $\lambda=b_{23}\neq1$.

  If $b_{23}=1$. Consider the following automorphism
 \begin{align*}
 	\Psi(e_1) &=x\,e_1,&
 	\Psi(e_2) &=y\,e_2,&
 	\Psi(e_3) &=e_3.
 \end{align*}
 For a suitable choice of the parameters $x\in\R^\ast$, applying $\Psi$ to the connection given in~\eqref{g3,3,nabla1,sol4} yields the following equivalent connection:
 \begin{equation}\label{g3,3,nabla1,sol4,2}
 	\begin{aligned}
 		\nabla_{e_1} e_3 &=e_1, &\nabla_{e_2}e_3&=e_2,  &\nabla_{e_3}e_3&=\delta_1\,e_1+\delta_2\,e_2+e_3.
 	\end{aligned}
 \end{equation}
 Observe first that if $\delta_1=\delta_2=0$, then this connection coincides with the one considered in~\eqref{g3,3,nabla1,sol4,1}, in which case $\lambda=1$. Therefore, we set $\lambda\in\mathbb{R}$ in~ \eqref{g3,3,nabla1,sol4,1} and assume that $\delta_1^2+\delta_2^2\neq0$. On the other hand, we may assume that $\delta_1=0$ and $\delta_2=1$. Indeed, the connection given in~\eqref{g3,3,nabla1,sol4,2}  with $\delta_1=1$ and $\delta_2=1$ or $\delta_1=1$ and $\delta_2=0$ is isomorphic to the same connection with $\delta_1=0$ and $\delta_2=1$ via the following automorphisms:
 \begin{align*}
 	\Psi(e_1) &=-e_1+e_2,&
 	\Psi(e_2) &=e_1,&
 	\Psi(e_3) &=e_3.
 \end{align*}
 or
 \begin{align*}
 	\Psi(e_1) &=e_1-e_2,&
 	\Psi(e_2) &=e_2,&
 	\Psi(e_3) &=e_3.
 \end{align*}
 
 In view of the previous discussion, the connections given in \eqref{g3,3,nabla1,sol4,1} and \eqref{g3,3,nabla1,sol4,2} correspond to the flat Lie algebras $\h_{0,1}$ and $\h_{0,2}$, respectively.

 	The fifth solution corresponds to the following flat, torsion-free connection:
 \begin{equation}\label{g3,3,nabla1,sol5}
 	\begin{aligned}
 		\nabla_{e_1} e_3 &=a_{23}e_2, \quad\nabla_{e_2}e_3=c_{33}e_2, &\nabla_{e_3}e_1&=-e_1+a_{23}e_2, &\nabla_{e_3}e_2&=(c_{33}-1)e_2, \\\nabla_{e_3}e_3&=c_{13}e_1+c_{23}e_2+c_{33}e_3.
 	\end{aligned}
 \end{equation}

 Suppose that $c_{33}\neq0$. If $c_{33}\neq\pm1$, and consider the following automorphism
 \begin{align*}
 	\Psi(e_1) &=\tfrac{a_{23}}{c_{33}} e_1+e_2,&
 	\Psi(e_2) &=e_1,&
 	\Psi(e_3) &=\tfrac{c_{13}a_{23}+c_{23}c_{33}}{c_{33}(c_{33}-1)}e_1-\tfrac{c_{13}}{c_{33}+1}e_2+e_3.
 \end{align*}
 
  Applying $\Psi$ to the connection given in~\eqref{g3,3,nabla1,sol5} yields the following equivalent connection:
 \begin{equation}\label{g3,3,nabla1,sol5,1}
 	\begin{aligned}
 		\nabla_{e_1} e_3 &=\lambda\,e_1,  &\nabla_{e_3}e_1&=(\lambda-1)e_1, &\nabla_{e_3}e_2&=-e_2, &\nabla_{e_3}e_3&=\lambda\,e_3.
 	\end{aligned}
 \end{equation}
 In this case, $\lambda=c_{33}\neq0,\pm1$. 
 Moreover, it is evident that this  connection is identical to the one given in~\eqref{g3,3,nabla1,sol2,1}; therefore, no further analysis is required.

 Suppose that $c_{33}\neq0$. If $c_{33}=-1$, and consider the following automorphism
 \begin{align*}
 	\Psi(e_1) &=-a_{23}e_1+x\,e_2,&
 	\Psi(e_2) &=e_1,&
 	\Psi(e_3) &=\tfrac{c_{13}a_{23}-c_{23}}{2}e_1+e_3.
 \end{align*}
 For a suitable choice of the parameter $x\in\R^\ast$, applying $\Psi$ to the connection given in~\eqref{g3,3,nabla1,sol5} yields the following equivalent connection:
 \begin{equation}\label{g3,3,nabla1,sol5,2}
 	\begin{aligned}
 		\nabla_{e_1} e_3 &=-e_1, &\nabla_{e_3}e_1&=-2\,e_1, &\nabla_{e_3}e_2&=-e_2, &\nabla_{e_3}e_3&=\delta\,e_2-e_3,\quad\delta=0,1.
 	\end{aligned}
 \end{equation}
 Similarly, this connection is identical to the one given in~\eqref{g3,3,nabla1,sol2,2}; therefore, no further analysis is required.

 Suppose now that $c_{33}\neq0$. If $c_{33}=1$, and consider the following automorphism
 \begin{align*}
 	\Psi(e_1) &=a_{23}\,x\,e_1+x\,e_2,&
 	\Psi(e_2) &=x\,e_1,&
 	\Psi(e_3) &=-\tfrac{c_{13}}{2}e_2+e_3.
 \end{align*}
 For a suitable choice of the parameter $x\in\R^\ast$, applying $\Psi$ to the connection given in~\eqref{g3,3,nabla1,sol5} yields the following equivalent connection:
 \begin{equation}\label{g3,3,nabla1,sol5,2}
 	\begin{aligned}
 		\nabla_{e_1} e_3 &=e_1, &\nabla_{e_3}e_2&=-e_2, &\nabla_{e_3}e_3&=\delta\,e_1+e_3,\quad\delta=0,1.
 	\end{aligned}
 \end{equation}
 Similarly, this connection is identical to the one given in~\eqref{g3,3,nabla1,sol2,3}; therefore, no further analysis is required.

 If $c_{33}=0$. Consider the following automorphism
 \begin{align*}
 	\Psi(e_1) &=e_2,&
 	\Psi(e_2) &=x\, e_1,&
 	\Psi(e_3) &=x(-2\,c_{13}a_{23}-c_{23})e_1-c_{13}e_2+e_3.
 \end{align*}
 For a suitable choice of the parameter $x\in\R^\ast$, applying $\Psi$ to the connection given in~\eqref{g3,3,nabla1,sol5} yields the following equivalent connection:
 \begin{equation}\label{g3,3,nabla1,sol5,3}
 	\begin{aligned}
 		\nabla_{e_2} e_3 &=\delta\,e_1, &\nabla_{e_3}e_1&=-e_1, &\nabla_{e_3}e_2&=\delta\,e_1-e_2,\quad\delta=0,1.
 	\end{aligned}
 \end{equation}

 If $\delta=1$, then this connection corresponds to the one associated with the affine manifold $\h_{0,6}$, listed in Table~\ref{g3,3}. Otherwise, if $\delta=0$, then this connection coincides with the one associated with the flat Lie algebra $\h_{0,7}$, in which case $\lambda=0$.

 The sixth solution corresponds to the flat, torsion-free connection:
 \begin{equation}\label{g3,3,nabla1,sol6}
 	\begin{aligned}
 		\nabla_{e_2} e_3 &=b_{13}e_1+c_{33}e_2, &\nabla_{e_3}e_1&=-e_1, &\nabla_{e_3}e_2&=b_{13}e_1+(c_{33}-1)e_2, &\nabla_{e_3}e_3&=c_{13}e_1+c_{23}e_2+c_{33}e_3.
 	\end{aligned}
 \end{equation}
 Suppose that $c_{33}\neq0$. If $c_{33}\neq\pm1$, consider then the following automorphism
 \begin{align*}
 	\Psi(e_1) &=e_2,&
 	\Psi(e_2) &= e_1-\tfrac{b_{13}}{c_{33}}e_2,&
 	\Psi(e_3) &=\tfrac{c_{23}}{c_{33}-1}e_1+\tfrac{b_{13}c_{23}-c_{13}c_{33}}{c_{33}(c_{33}+1)}e_2+e_3.
 \end{align*}
 Applying $\Psi$ to the connection given in~\eqref{g3,3,nabla1,sol6} yields the following equivalent connection:
 \begin{equation}\label{g3,3,nabla1,sol6,1}
 	\begin{aligned}
 		\nabla_{e_1} e_3 &=\lambda\,e_1, &\nabla_{e_3}e_1&=(\lambda-1)e_1, &\nabla_{e_3}e_2&=-e_2, &\nabla_{e_3}e_3&=\lambda\,e_3.
 	\end{aligned}
 \end{equation}
 In this case $\lambda=c_{33}\neq0,\pm1$. Moreover, it is evident that this  connection is identical to the one given in~\eqref{g3,3,nabla1,sol2,1}; therefore, no further analysis is required.

  Suppose that $c_{33}\neq0$. If $c_{33}=-1$, consider then the following automorphism
 \begin{align*}
 	\Psi(e_1) &=x\,e_2,&
 	\Psi(e_2) &= e_1+x\,b_{13}e_2,&
 	\Psi(e_3) &=\tfrac{c_{23}}{2}e_1+e_3.
 \end{align*}
 For a suitable choice of the parameter $x\in\R^\ast$, applying $\Psi$ to the connection given in~\eqref{g3,3,nabla1,sol6} yields the following equivalent connection:
 \begin{equation}\label{g3,3,nabla1,sol6,2}
 \begin{aligned}
 	\nabla_{e_1} e_3 &=-e_1, &\nabla_{e_3}e_1&=-2\,e_1, &\nabla_{e_3}e_2&=-e_2, &\nabla_{e_3}e_3&=\delta_1\,e_2-e_3,\quad\delta_1=0,1.
 \end{aligned}
 \end{equation}
 Moreover,  this  connection is identical to the one given in~\eqref{g3,3,nabla1,sol2,2}; therefore, no further analysis is required.

  Suppose that $c_{33}\neq0$. If $c_{33}=1$, consider then the following automorphism
 \begin{align*}
 	\Psi(e_1) &=e_2,&
 	\Psi(e_2) &= x\,e_1-b_{13}e_2,&
 	\Psi(e_3) &=\tfrac{b_{13}c_{23}-c_{13}}{2}e_2+e_3.
 \end{align*}
 For a suitable choice of the parameter $x\in\R^\ast$, applying $\Psi$ to the connection given in~\eqref{g3,3,nabla1,sol6} yields the following equivalent connection:
 \begin{equation}\label{g3,3,nabla1,sol6,3}
 	\begin{aligned}
 		\nabla_{e_1} e_3 &=e_1,  &\nabla_{e_3}e_2&=-e_2, &\nabla_{e_3}e_3&=\delta_2\,e_1+e_3,\quad\delta_2=0,1.
 	\end{aligned}
 \end{equation}
 Moreover,  this  connection is identical to the one given in~\eqref{g3,3,nabla1,sol2,3}; therefore, no further analysis is required.

 Suppose now that $c_{33}=0$.  Consider then the following automorphism
 \begin{align*}
 	\Psi(e_1) &=e_1,&
 	\Psi(e_2) &= x\,e_2,&
 	\Psi(e_3) &=\-(2\,b_{13}c_{23}+c_{13})e_1-x\,c_{23}e_2+e_3.
 \end{align*}
 For a suitable choice of the parameter $x\in\R^\ast$, applying $\Psi$ to the connection given in~\eqref{g3,3,nabla1,sol6} yields the following equivalent connection:
 \begin{equation}\label{g3,3,nabla1,sol6,4}
 	\begin{aligned}
 		\nabla_{e_2} e_3 &=\delta\,e_1, &\nabla_{e_3}e_1&=-e_1,
 		 &\nabla_{e_3}e_2&=\delta\,e_1-e_2, \quad\delta=0,1.
 	\end{aligned}
 \end{equation}
Observe first that, if $\delta=1$, then this connection coincides with the one associated to the flat Lie algebra $\h_{0,6}$. Otherwise, if $\delta=0$, then the connection defined in \eqref{g3,3,nabla1,sol6,4} coincides with the one associated to the flat Lie algebra $\h_{0,7}$, in which case $\lambda=0$.

 The seventh solution is given by the following flat, torsion-free connection:
 \begin{equation}\label{g3,3,nabla1,sol7}
 	\begin{aligned}
 		\nabla_{e_3} e_1 &=-e_1, &\nabla_{e_3}e_2&=-e_2,
 		&\nabla_{e_3}e_3&=c_{13}e_1+c_{23}e_2+c_{33}e_3
 	\end{aligned}
 \end{equation}
 Suppose that $c_{33}\neq-1$, and consider the following automorphism
 \begin{align*}
 	\Psi(e_1) &=e_1,&
 	\Psi(e_2) &= e_2,&
 	\Psi(e_3) &=-\tfrac{c_{13}}{c_{33}+1}e_1-\tfrac{c_{23}}{c_{33}+1}e_2+e_3.
 \end{align*}
 Applying $\Psi$ to the connection given in~\eqref{g3,3,nabla1,sol7} yields the following equivalent connection:
 \begin{equation}\label{g3,3,nabla1,sol7,1}
 	\begin{aligned}
 		\nabla_{e_3} e_1 &=-e_1, &\nabla_{e_3}e_2&=-e_2,
 		&\nabla_{e_3}e_3&=\lambda\,e_3.
 	\end{aligned}
 \end{equation}
 In this case, $\lambda=c_{33}\neq-1$.
 
 Suppose now that $c_{33}=-1$.  Consider then the following automorphism
 \begin{align*}
 	\Psi(e_1) &=x\,e_1,&
 	\Psi(e_2) &= y\,e_2,&
 	\Psi(e_3) &=e_3.
 \end{align*}
 For a suitable choice of the parameters $x,y\in\R^\ast$, applying $\Psi$ to the connection given in~\eqref{g3,3,nabla1,sol7} yields the following equivalent connection:
 \begin{equation}\label{g3,3,nabla1,sol7,2}
 	\begin{aligned}
 		\nabla_{e_3} e_1 &=-e_1, &\nabla_{e_3}e_2&=-e_2,
 		&\nabla_{e_3}e_3&=\delta_1\,e_1+\delta_2\,e_2-e_3,\quad\delta_1,\delta_2=0,1.
 	\end{aligned}
 \end{equation}

 Observe first that if $\delta_1=\delta_2=0$, then this connection coincides with the one considered in~\eqref{g3,3,nabla1,sol7,1}, in which case $\lambda=-1$. Therefore, we set $\lambda\in\mathbb{R}$ in~ \eqref{g3,3,nabla1,sol7,1} and assume that $\delta_1^2+\delta_2^2\neq0$. On the other hand, we may assume that $\delta_1=0$ and $\delta_2=1$. Indeed, the connection given in~\eqref{g3,3,nabla1,sol7,2}  with $\delta_1=1$ and $\delta_2=1$ or $\delta_1=1$ and $\delta_2=0$ is isomorphic to the same connection with $\delta_1=0$ and $\delta_2=1$ via the following automorphisms:
 \begin{align*}
 	\Psi(e_1) &=e_1,&
 	\Psi(e_2) &=e_1+e_2,&
 	\Psi(e_3) &=e_3.
 \end{align*}
 or
 \begin{align*}
 	\Psi(e_1) &=e_2,&
 	\Psi(e_2) &=e_1,&
 	\Psi(e_3) &=e_3.
 \end{align*}
 
 In view of the previous discussion, the connections given in \eqref{g3,3,nabla1,sol7,1} and \eqref{g3,3,nabla1,sol7,2} correspond to the flat Lie algebras $\h_{0,7}$ and $\h_{0,8}$, respectively.   The additional conditions $\lambda\neq0$ in the flat Lie algebra $\h_{0,7}$ is imposed to ensure that the corresponding connection is not isomorphic to the  algebras $\h_{0,1}$.

If $\nabla^0\equiv\nabla^1$. Using a straightforward computation, the flatness-equations can be solved with a unique real solution. 
The flat, torsion-free connection associated with this solution is given by the following connection:
\begin{equation}\label{g3,3,nabla2,sol1}
	\begin{aligned}
		\nabla_{e_1} e_1 &=e_1+a_{31}e_3, \quad\nabla_{e_1}e_3=-\tfrac{1}{a_{31}}e_1-e_3, &\nabla_{e_3}e_1&=-\tfrac{a_{31}+1}{a_{31}}e_1-e_3, &\nabla_{e_3}e_2&=-e_2, \\\nabla_{e_3}e_3&=\tfrac{1-a_{31}}{a_{31}^2}e_1+\tfrac{1-2\,a_{31}}{a_{31}}e_3.
	\end{aligned}
\end{equation}
Consider the following automorphism:
\begin{align*}
	\Psi(e_1) &=\tfrac{\sqrt{\varepsilon_0\,a_{31}}}{\varepsilon_0} e_2,&
	\Psi(e_2) &=e_1,&
	\Psi(e_3) &=-\tfrac{\sqrt{\varepsilon_0\,a_{31}}}{\varepsilon_0\,a_{31}}e_2+e_3.
\end{align*}
Applying it to the connection given in~\eqref{g3,3,nabla2,sol1} yields the following equivalent connection:
\begin{equation}\label{g3,3,nabla2,sol1,1}
	\begin{aligned}
		\nabla_{e_2} e_2 &=\varepsilon_0\, e_3,  &\nabla_{e_3}e_1&=-e_1, &\nabla_{e_3}e_2&=-e_2, &\nabla_{e_3}e_3&=-2\,e_3,\quad\varepsilon_0=\pm1.
	\end{aligned}
\end{equation}
 This  connection is identical to the one associated to the flat Lie algebra $\h_{0,9}$.

 If $\nabla^0\equiv\nabla^2$. 
 By a straightforward computation, the flatness equations admit three real solutions.
The flat, torsion-free connection associated with first solution is given by the following connection:
 \begin{equation}\label{g3,3,nabla2,sol2,1}
 	\begin{aligned}
\nabla_{e_1}e_3&=-e_1, &\nabla_{e_2}e_2&=e_1, &\nabla_{e_2}e_3&=b_{13}e_1, \\\nabla_{e_3}e_1&=-2\,e_1, &\nabla_{e_3}e_2&=b_{13}e_1-e_2, &\nabla_{e_3}e_3&=c_{13}e_1-e_3.
 	\end{aligned}
 \end{equation}
 Consider the following automorphism:
 \begin{align*}
 	\Psi(e_1) &= e_1,&
 	\Psi(e_2) &=-b_{13}e_1+e_2,&
 	\Psi(e_3) &=-\tfrac{c_{13}}{2}e_1+e_3.
 \end{align*}
 Applying it to the connection given in~\eqref{g3,3,nabla2,sol2,1} yields the following equivalent connection:
 \begin{equation}\label{g3,3,nabla2,sol2,1,1}
 	\begin{aligned}
 		\nabla_{e_1}e_3&=-e_1, &\nabla_{e_2}e_2&=e_1,  &\nabla_{e_3}e_1&=-2\,e_1, &\nabla_{e_3}e_2&=-e_2, &\nabla_{e_3}e_3&=-e_3.
 	\end{aligned}
 \end{equation}
 This flat, torsion-free connection corresponds to the flat Lie algebra $\h_{2,1}$ listed in Table~\ref{g3,3}.

The second solution is given by the following flat, torsion-free connection:
\begin{equation}\label{g3,3,nabla2,sol2,2}
	\begin{aligned}
 \nabla_{e_2}e_2&=e_1, &\nabla_{e_2}e_3&=b_{13}e_1+e_2, &\nabla_{e_3}e_1&=-e_1, &\nabla_{e_3}e_2&=b_{13}e_1, &\nabla_{e_3}e_3&=c_{13}e_1+e_3.
	\end{aligned}
\end{equation}
Consider the following automorphism:
\begin{align*}
	\Psi(e_1) &= e_1,&
	\Psi(e_2) &=e_2,&
	\Psi(e_3) &=\tfrac{b_{13}^2-c_{13}}{2}e_1+b_{13}e_2+e_3.
\end{align*} 
 Applying it to the connection given in~\eqref{g3,3,nabla2,sol2,2} yields the following equivalent connection:
 \begin{equation}\label{g3,3,nabla2,sol2,2,1}
 	\begin{aligned}
 		\nabla_{e_2}e_2&=e_1,
 		 &\nabla_{e_2}e_3&=e_2, &\nabla_{e_3}e_1&=-e_1, &\nabla_{e_3}e_3&=e_3.
 	\end{aligned}
 \end{equation}
 This flat, torsion-free connection corresponds to the flat Lie algebra $\h_{2,2}$ listed in Table~\ref{g3,3}.

 The third solution corresponds to the following flat, torsion-free connection:
 \begin{equation}\label{g3,3,nabla2,sol2,3}
 	\begin{aligned}
 		\nabla_{e_2}e_2&=e_1+b_{32}e_3, &\nabla_{e_3}e_1&=-e_1, &\nabla_{e_3}e_2&=-e_2, &\nabla_{e_3}e_3&=-\tfrac{1}{b_{32}}e_1-2\,e_3.
 	\end{aligned}
 \end{equation}
 Consider the following automorphism:
 \begin{align*}
 	\Psi(e_1) &= e_1,&
 	\Psi(e_2) &=\tfrac{\sqrt{\varepsilon_1\,b_{32}}}{\varepsilon_1}e_2,&
 	\Psi(e_3) &=-\tfrac{1}{b_{32}}e_1+e_3.
 \end{align*}
 Applying it to the connection given in~\eqref{g3,3,nabla2,sol2,3} yields the following equivalent connection:
 \begin{equation}\label{g3,3,nabla2,sol2,3,1}
 	\begin{aligned}
 		\nabla_{e_2} e_2 &=\varepsilon_1\, e_3,  &\nabla_{e_3}e_1&=-e_1, &\nabla_{e_3}e_2&=-e_2, &\nabla_{e_3}e_3&=-2\,e_3,\quad\varepsilon_1=\pm1.
 	\end{aligned}
 \end{equation}
 This  connection is identical to the one associated to the flat Lie algebra $\h_{0,9}$. 
\end{proof}

If $\nabla^0\equiv \nabla^3$ or $\nabla^0\equiv \nabla^4$ or $\nabla^0\equiv\nabla^5$. 
A straightforward computation shows that the flatness-equation admits no solution. This completes the classification of flat, torsion-free connections on the flat Lie algebra $\G_{3,3}$. 
All flat, torsion-free connections listed in Table~\ref{g3,3} are pairwise non-isomorphic, and this can be verified by a straightforward computation.

\begin{co}
	With the notations as above, among the flat Lie algebras on $\G_{3,3}$, we have
	\begin{enumerate}
		\item[i)] Associative algebras$:$\hspace{0.275cm} $\h_{0,1}^{\lambda=1}$, $\h_{0,7}^{\lambda=-1}$,
		\item[ii)] Novikov algebras$:$\hspace{0.73cm} 
		$\h_{0,1}$, $\h_{0,2}$,
		\item[iii)] Bi-symmetric algebras$:$ $\h_{0,1}^{\lambda=1}$, $\h_{0,2}$, $\h_{0,7}^{\lambda=-1}$, $\h_{0,8}$.
		\item[iv)] Complete algebras$:$\hspace{0.62cm}$\h_{0,1}^{\lambda=0}$, $\h_{0,6}$.
	\end{enumerate}
\end{co}

\begin{pr}
Let $(\G, \nabla)$ be a three-dimensional real flat  Lie algebra with $\G = \G_{3,4}$. Then $(\G, \nabla)$ is isomorphic to exactly one of the flat Lie algebras listed in Table~$\ref{g3,4}$ and $\ref{g3,4,alpha}$.
{\renewcommand*{\arraystretch}{1.8}
\captionof{table}{Flat torsion-free connection on the Lie algebra $\G_{3,4}^{\alpha=-1}$.}
\setcounter{table}{7}
\begin{footnotesize} 
\setlength{\tabcolsep}{3pt} 
\begin{longtable}{@{}cllllllc@{}} 
			\hline
		Flat algebra&\multicolumn{2}{@{}l@{}}{~~Flat torsion-free connection}&&&&Remarks \\
			\hline
$\h_{0,1}$&$\nabla_{e_2}e_3=\delta e_1$&$\nabla_{e_3}e_1=-e_1$&$\nabla_{e_3}e_2=\delta 
e_1+e_2$&$\nabla_{e_3}e_3=-2e_3$&&$\delta=0,1$&\\
$\h_{0,2}$&$\nabla_{e_1}e_3=2e_1$&$\nabla_{e_2}e_3=\delta e_1+2e_2$&$\nabla_{e_3}e_1=e_1$&$\nabla_{e_3}e_2=\delta e_1+3e_2$&$\nabla_{e_3}e_3=2e_3$&$\delta=0,1$&\\
$\h_{0,3}$&$\nabla_{e_1}e_1=\varepsilon e_3$&$\nabla_{e_3}e_1=-e_1$&$\nabla_{e_3}e_2=e_2$&$\nabla_{e_3}e_3=-2e_3$&&$\varepsilon\pm1$&\\
$\h_{0,4}$&$\nabla_{e_3}e_1=-e_1$&$\nabla_{e_3}e_2=e_2$&$\nabla_{e_3}e_3=\lambda e_3$&&&$\lambda\in\R,~\lambda\neq\pm2$&\\
$\h_{0,5}$&$\nabla_{e_3}e_1=-e_1$&$\nabla_{e_3}e_2=e_2$&$\nabla_{e_3}e_3=e_1-e_3$&&&&\\
$\h_{0,6}$&$\nabla_{e_3}e_1=-e_1$&$\nabla_{e_3}e_2=e_2$&$\nabla_{e_3}e_3=e_2+e_3$&&&&\\
$\h_{0,7}$&$\nabla_{e_2}e_3=\lambda e_2$&$\nabla_{e_3}e_1=-e_1$&$\nabla_{e_3}e_2=(1+\lambda)e_2$&$\nabla_{e_3}e_3=\lambda e_3$&&$\lambda\in\R^\ast$&\\
$\h_{0,8}$&$\nabla_{e_2}e_3=-e_2$&$\nabla_{e_3}e_1=-e_1$&$\nabla_{e_3}e_3=\delta_1 e_1+\delta_2 e_2-e_3$&&&$\delta_1,\delta_2=0,1,~\delta_1^2+\delta_2^2\neq0$&\\
$\h_{0,9}$&$\nabla_{e_1}e_3=\lambda e_1$&$\nabla_{e_2}e_3=\lambda e_2$&$\nabla_{e_3}e_1=(\lambda-1) e_1$&$\nabla_{e_3}e_2=(1+\lambda) e_2$&$\nabla_{e_3}e_3=\lambda e_3$&$\lambda\in\R^\ast,~\lambda\neq\pm2$&
\\
$\h_{0,10}$&$\nabla_{e_1}e_3= e_1$&$\nabla_{e_2}e_3= e_2$&$\nabla_{e_3}e_2=2e_2$&$\nabla_{e_3}e_3=e_1+e_3$&&&\\
$\h_{0,11}$&$\nabla_{e_1}e_2=e_3$&$\nabla_{e_2}e_1=e_3$&$\nabla_{e_3}e_1=-e_1$&$\nabla_{e_3}e_2=e_2$&&&
\\\hline
$\h_{2,1}$&$\nabla_{e_1}e_3=3e_1$&$\nabla_{e_2}e_2=e_1$&$\nabla_{e_3}e_1=2e_1$&$\nabla_{e_3}e_2=e_2$&$\nabla_{e_3}e_3=3e_3$&&\\
$\h_{2,2}$&$\nabla_{e_2}e_2=e_1$&$\nabla_{e_2}e_3=-3e_2$&$\nabla_{e_3}e_1=-e_1$&$\nabla_{e_3}e_2=-2e_2$&$\nabla_{e_3}e_3=-3e_3$&&
\\\hline	
			\end{longtable}
			\label{g3,4}
			\end{footnotesize}	
			}
			\newgeometry{left=0.5cm,right=2cm,top=2cm,bottom=2cm}
{\renewcommand*{\arraystretch}{1.8}
\captionof{table}{Flat torsion-free connection on the Lie algebra $\G_{3,4}^{\alpha\neq-1}$.}
\setcounter{table}{8}
\begin{footnotesize} 
\setlength{\tabcolsep}{3pt} 
\begin{longtable}{@{}cllllllc@{}} 
			\hline
		Flat algebra&\multicolumn{2}{@{}l@{}}{~~Flat torsion-free connection}&&&&Remarks  \\
			\hline
$\h_{0,1}$&$\nabla_{e_2}e_3=e_1$&$\nabla_{e_3}e_1=-e_1$&$\nabla_{e_3}e_2=e_1-\alpha e_2$&$\nabla_{e_3}e_3=(\alpha-1)e_3$&&&\\
$\h_{0,2}$&$\nabla_{e_2}e_3=e_1$&$\nabla_{e_3}e_1=-e_1$&$\nabla_{e_3}e_2=e_1-\frac{1}{2}e_2$&$\nabla_{e_3}e_3=e_2-\frac{1}{2}e_3$&&$(\alpha=\frac{1}{2})$&\\
$\h_{0,3}$&$\nabla_{e_3}e_1=-e_1$&$\nabla_{e_3}e_2=-\alpha e_2$&$\nabla_{e_3}e_3=\delta e_2+(\alpha-1)e_3$&&&$\delta=0,1$&\\
$\h_{0,4}$&$\nabla_{e_1}e_3=(1-\alpha)e_1$&$\nabla_{e_2}e_3=e_1+(1-\alpha)e_2$&$\nabla_{e_3}e_1=-\alpha e_1$&$\nabla_{e_3}e_2=e_1+(1-2\alpha)e_2$&$\nabla_{e_3}e_3=(1-\alpha)e_3$&&\\
$\h_{0,5}$&$\nabla_{e_1}e_3=\frac{1}{2}e_1$&$\nabla_{e_2}e_3=e_1+\frac{1}{2}e_2$&$\nabla_{e_3}e_1=-\frac{1}{2}e_1$&$\nabla_{e_3}e_2=e_1$&$\nabla_{e_3}e_3=e_2+\frac{1}{2}e_3$&$(\alpha=\frac{1}{2})$&\\
$\h_{0,6}$&$\nabla_{e_1}e_3=(1-\alpha)e_1$&$\nabla_{e_2}e_3=(1-\alpha)e_2$&$\nabla_{e_3}e_1=-\alpha e_1$&$\nabla_{e_3}e_2=(1-2\alpha)e_2$&$\nabla_{e_3}e_3=\delta e_2+(1-\alpha)e_3$&$\delta=0,1$&\\
$\h_{0,7}$&$\nabla_{e_1}e_1=\varepsilon e_3$&$\nabla_{e_3}e_1=-e_1$&$\nabla_{e_3}e_2=-\alpha e_2$&$\nabla_{e_3}e_3=-2e_3$&&$\varepsilon=\pm1$&\\
$\h_{0,8}$&$\nabla_{e_1}e_3=\delta e_2$&$\nabla_{e_3}e_1=-e_1+\delta e_2$&$\nabla_{e_3}e_2=-\alpha e_2$&$\nabla_{e_3}e_3=(1-\alpha)e_3$&&$\delta=0,1$&\\
$\h_{0,9}$&$\nabla_{e_1}e_3=(\alpha-1)e_1+\delta e_2$&$\nabla_{e_2}e_3=(\alpha-1)e_2$&$\nabla_{e_3}e_1=(\alpha-2)e_1+\delta e_2$&$\nabla_{e_3}e_2=-e_2$&$\nabla_{e_3}e_3=(\alpha-1)e_3$&$\delta=0,1$&\\
$\h_{0,10}$&$\nabla_{e_3}e_1=-e_1$&$\nabla_{e_3}e_2=-\alpha e_2$&$\nabla_{e_3}e_3=\lambda e_3$&&&$\lambda\in\R,~\lambda\neq\pm(\alpha-1)$&\\
$\h_{0,11}$&$\nabla_{e_3}e_1=-e_1$&$\nabla_{e_3}e_2=-\alpha e_2$&$\nabla_{e_3}e_3=e_2-\alpha e_3$&&&$(\alpha\neq\frac{1}{2})$&\\
$\h_{0,12}$&$\nabla_{e_3}e_1=-e_1$&$\nabla_{e_3}e_2=-\alpha e_2$&$\nabla_{e_3}e_3=e_1-e_3$&&&&\\
$\h_{0,13}$&$\nabla_{e_2}e_3=\lambda e_2$&$\nabla_{e_3}e_1=-e_1$&$\nabla_{e_3}e_2=(\lambda-\alpha)e_2$&$\nabla_{e_3}e_3=\lambda e_3$&&$\lambda\in\R^\ast$&\\
$\h_{0,14}$&$\nabla_{e_2}e_3=\alpha e_2$&$\nabla_{e_3}e_1=-e_1$&$\nabla_{e_3}e_3=e_2+\alpha e_3$&&&&\\
$\h_{0,15}$&$\nabla_{e_2}e_3=-e_2$&$\nabla_{e_3}e_1=-e_1$&$\nabla_{e_3}e_2=-(\alpha+1)e_2$&$\nabla_{e_3}e_3=e_1-e_3$&&&\\
$\h_{0,16}$&$\nabla_{e_1}e_3=\lambda e_1$&$\nabla_{e_3}e_1=(\lambda-1)e_1$&$\nabla_{e_3}e_2=-\alpha e_2$&$\nabla_{e_3}e_3=\lambda e_3$&&$\lambda\in\R^\ast$&\\
$\h_{0,17}$&$\nabla_{e_1}e_3=-\alpha e_1$&$\nabla_{e_3}e_1=-(\alpha+1)e_1$&$\nabla_{e_3}e_2=-\alpha e_2$&$\nabla_{e_3}e_3=e_2-\alpha e_3$&&&\\
$\h_{0,18}$&$\nabla_{e_1}e_3= e_1$&$\nabla_{e_3}e_2=-\alpha e_2$&$\nabla_{e_3}e_3=e_1+ e_3$&&&&\\
$\h_{0,19}$&$\nabla_{e_1}e_3=\lambda e_1$&$\nabla_{e_2}e_3=\lambda e_2$&$\nabla_{e_3}e_1=(\lambda-1)e_1$&$\nabla_{e_3}e_2=(\lambda-\alpha)e_2$&$\nabla_{e_3}e_3=\lambda e_3$&$\lambda\in\R^\ast,~\lambda\neq\pm(\alpha-1)$&\\
$\h_{0,20}$&$\nabla_{e_1}e_3=\alpha e_1$&$\nabla_{e_2}e_3=\alpha e_2$&$\nabla_{e_3}e_1=(\alpha-1)e_1$&$\nabla_{e_3}e_3=e_2+\alpha e_3$&&$(\alpha\neq\frac{1}{2})$&\\
$\h_{0,21}$&$\nabla_{e_1}e_3=e_1$&$\nabla_{e_2}e_3=e_2$&$\nabla_{e_3}e_2=(1-\alpha)e_2$&$\nabla_{e_3}e_3=e_1+e_3$&&\\
$\h_{0,22}$&$\nabla_{e_1}e_2=e_3$&$\nabla_{e_2}e_1=e_3$&$\nabla_{e_3}e_1=-e_1$&$\nabla_{e_3}e_2=-\alpha e_2$&$\nabla_{e_3}e_3=-(\alpha+1)e_3$&&\\
$\h_{0,23}$&$\nabla_{e_2}e_2=\varepsilon e_3$&$\nabla_{e_3}e_1=-e_1$&$\nabla_{e_3}e_2=-\alpha e_2$&$\nabla_{e_3}e_3=-2\alpha e_3$&&$\varepsilon=\pm1$&\\
$\h_{0,24}$&$\nabla_{e_2}e_2=\varepsilon e_3$&$\nabla_{e_2}e_3=e_1$&$\nabla_{e_3}e_1=-e_1$&$\nabla_{e_3}e_2=e_1-\frac{1}{3}e_2$&$\nabla_{e_3}e_3=-\frac{2}{3}e_3$&$\varepsilon=\pm1,~(\alpha=\frac{1}{3})$&
\\\hline
$\h_{2,1}$&$\nabla_{e_1}e_3=(1-2\alpha)e_1$&$\nabla_{e_2}e_2=e_1$&$\nabla_{e_3}e_1=-2\alpha e_1$&$\nabla_{e_3}e_2=-\alpha e_2$&$\nabla_{e_3}e_3=(1-2\alpha)e_3$&&\\
$\h_{2,2}$&$\nabla_{e_2}e_2=e_1$&$\nabla_{e_3}e_1=-e_1$&$\nabla_{e_3}e_2=-\frac{1}{2}e_2$&$\nabla_{e_3}e_3=\lambda e_3$&&$\lambda\in\R^\ast,~(\alpha=\frac{1}{2})$&\\
$\h_{2,3}$&$\nabla_{e_2}e_2=e_1$&$\nabla_{e_3}e_1=-e_1$&$\nabla_{e_3}e_2=-\frac{1}{2}e_2$&$\nabla_{e_3}e_3=\varepsilon e_1- e_3$&&$\varepsilon=\pm1,~(\alpha=\frac{1}{2})$&\\
$\h_{2,4}$&$\nabla_{e_2}e_2=e_1$&$\nabla_{e_2}e_3=(2\alpha-1)e_2$&$\nabla_{e_3}e_1=-e_1$&$\nabla_{e_3}e_2=(\alpha-1)e_2$&$\nabla_{e_3}e_3=(2\alpha-1)e_3$&$(\alpha\neq\frac{1}{2})$&\\
$\h_{2,5}$&$\nabla_{e_1}e_3=\lambda e_1$&$\nabla_{e_2}e_2=e_1$&$\nabla_{e_2}e_3=\lambda e_2$&$\nabla_{e_3}e_1=(\lambda-1)e_1$&$\nabla_{e_3}e_2=(\lambda-\frac{1}{2})e_2$&$\lambda\in\R^\ast,~~(\alpha=\frac{1}{2})$&\\
&$\nabla_{e_3}e_3=\lambda e_3$&&&&&&\\
$\h_{2,6}$&$\nabla_{e_1}e_3=e_1$&$\nabla_{e_2}e_2=e_1$&$\nabla_{e_2}e_3=e_2$&$\nabla_{e_3}e_2=\frac{1}{2}e_2$&$\nabla_{e_3}e_3=\varepsilon e_1+e_3$&$\varepsilon=\pm1,~(\alpha=\frac{1}{2})$&\\
$\h_{2,7}$&$\nabla_{e_1}e_2=e_3$&$\nabla_{e_2}e_1=e_3$&$\nabla_{e_2}e_2=e_1$&$\nabla_{e_3}e_1=-e_1$&$\nabla_{e_3}e_2=-\frac{1}{2}e_2$&$(\alpha=\frac{1}{2})$&\\
&$\nabla_{e_3}e_3=-\frac{3}{2}e_3$&&&&&&
\\\hline
$\h_{2,8}$&$\nabla_{e_1}e_1=e_2$&$\nabla_{e_2}e_3=(\alpha-2)e_2$&$\nabla_{e_3}e_1=-e_1$&$\nabla_{e_3}e_2=-2e_2$&$\nabla_{e_3}e_3=(\alpha-2)e_3$&&\\
$\h_{2,9}$&$\nabla_{e_1}e_1=e_2$&$\nabla_{e_1}e_3=(2-\alpha)e_1$&$\nabla_{e_3}e_1=(1-\alpha)e_1$&$\nabla_{e_3}e_2=-\alpha e_2$&$\nabla_{e_3}e_3=(2-\alpha)e_3$&&
\\\hline		
			\end{longtable}
			\label{g3,4,alpha}
			\end{footnotesize}	
			}
					
			\clearpage
			\restoregeometry
	
\end{pr}
\begin{proof}
 In the basis $\lbrace e_1, e_2, e_3 \rbrace$, the operators $\nabla_{e_1}$, $\nabla_{e_2}$ and $\nabla_{e_3}$ are given respectively by: 
	\begin{equation}
		\nabla_{e_1}=\left( \begin {array}{ccc} a_{11}&a_{12}&a_{13}\\ \noalign{\medskip}
		a_{21}&a_{22}&a_{23}\\ \noalign{\medskip}a_{31}&a_{32}&a_{33}\end {array} \right),\quad
		\nabla_{e_2}=\left( \begin {array}{ccc} a_{12}&b_{12}&b_{13}\\ \noalign{\medskip}
		a_{22}&b_{22}&b_{23}\\ \noalign{\medskip}a_{32}&b_{32}&b_{33}\end {array} \right),\quad
		\nabla_{e_3}=\left( \begin {array}{ccc} a_{13}-1 &b_{13}&c_{13}\\ \noalign{\medskip}
		a_{23}&b_{23}-\alpha&c_{23}\\ \noalign{\medskip}a_{33}&b_{33}&c_{33}\end {array} \right),\quad
	\end{equation}
	where $a_{ij}$, $b_{ij}$, $c_{ij}\in \mathbb{R}$, $\alpha\neq0$ and $-1\leq\alpha<1$.

	Assume that $\nabla^0$ is a  torsion-free connection. Then $\nabla^0$ is equivalent to one of the connections listed in Lemma~\ref{Lemg3j} under the Lie algebra $\G_{3,4}$. Assume that, $0<|\alpha|<1$. We first consider the case $\nabla^0=\nabla^1$. Using a straightforward computation, the flatness-equations can be solved with two real solutions. The first one is given by the following flat, torsion-free connection:
	\begin{equation}\label{g3,4,alpha,nabla1,sol1}
		\begin{aligned}
			\nabla_{e_1} e_2 &=e_1+e_2+a_{3 2}e_3, \hspace{1cm}\nabla_{e_1}e_3=-\tfrac{1}{a_{32}}e_1-\tfrac{1}{a_{32}}e_2-e_3,\hspace{0.75cm} \nabla_{e_2}e_1=e_1+e_2+a_{3 2}e_3,\\
			\nabla_{e_2}e_3&=-\tfrac{1}{a_{32}}e_1-\tfrac{1}{a_{32}}e_2-e_3,\hspace{0.2cm}\nabla_{e_3}e_1=-\tfrac{a_{32}+1}{a_{32}}e_1-\tfrac{1}{a_{32}}e_2-e_3, \hspace{0.2cm}\nabla_{e_3}e_2=-\tfrac{1}{a_{32}}e_1-\tfrac{1+\alpha\,a_{32}}{a_{32}}e_2-e_3,\\
			\nabla_{e_3}e_3&=\tfrac{2-\alpha\,a_{32}}{a_{32}^2}e_1+\tfrac{2-a_{32}}{a_{32}^2}e_2+\tfrac{2-(\alpha+1)a_{32}}{a_{32}}e_3.
		\end{aligned}
	\end{equation}
	Applying the following automorphism
	\begin{align*}
		\Psi(e_1) &=a_{32}\, e_1,&
		\Psi(e_2) &=e_2,&
		\Psi(e_3) &=-e_1-\tfrac{1}{a_{32}}e_2+ e_3.
	\end{align*}
	 yields the flat, torsion-free connection associated with the flat Lie algebra $\h_{0,22}$.
	 
	 The second solution is given by the following flat, torsion-free connection:
\begin{equation}\label{g3,4,alpha,nabla1,sol2}
	\begin{aligned}
		\nabla_{e_1} e_2 &=e_1+e_2+a_{3 2}e_3, \hspace{0.2cm}\nabla_{e_1}e_3=-\tfrac{1}{a_{32}}e_1-\tfrac{1}{a_{32}}e_2-e_3,\hspace{0.2cm} \nabla_{e_2}e_1=e_1+e_2+a_{3 2}e_3,\\
		\nabla_{e_2}e_2&=\nu_1e_1, \hspace{0.2cm}\nabla_{e_2}e_3=-\tfrac{1+\nu_1}{a_{32}}e_1-\tfrac{1}{a_{32}}e_2-e_3,\hspace{0.2cm}\nabla_{e_3}e_1=-\tfrac{a_{32}+1}{a_{32}}e_1-\tfrac{1}{a_{32}}e_2-e_3, \\\nabla_{e_3}e_2&=-\tfrac{1+\nu_1}{a_{32}}e_1-\tfrac{2+a_{32}}{2\,a_{32}}e_2-e_3,\hspace{0.2cm}
		\nabla_{e_3}e_3=\tfrac{2\,\nu_1+4-a_{32}}{2\,a_{32}^2}e_1+\tfrac{2-a_{32}}{a_{32}^2}e_2+\tfrac{4-3\,a_{32}}{a_{32}}e_3.
	\end{aligned}
\end{equation}
In this case, $\alpha=\frac{1}{2}$.
	Consider the following automorphism
	\begin{align*}
		\Psi(e_1) &=\tfrac{a_{32}}{x} e_1,&
		\Psi(e_2) &=x\,e_2,&
		\Psi(e_3) &=-\tfrac{1}{x} e_1-\tfrac{x}{a_{32}}e_2+ e_3.
	\end{align*}
 For a suitable choice of the parameter $x$, applying $\Psi$ to the connection given in~\eqref{g3,4,alpha,nabla1,sol2} yields the following equivalent connection:
\begin{equation}\label{g3,4,alpha,nabla1,sol2,1}
	\begin{aligned}
		\nabla_{e_1} e_2 &=e_3, &\nabla_{e_2}e_1&=e_3, &\nabla_{e_2}e_2&=\delta\,e_1, &\nabla_{e_3}e_1&=-e_1, &\nabla_{e_3}e_2&=-\tfrac{1}{2}e_2, &\nabla_{e_3}e_3&=-\tfrac{3}{2}e_3,\quad\delta=0,1.
	\end{aligned}
\end{equation}
If $\delta=1$, then this connection is precisely the one associated with the flat Lie algebra $\h_{2,7}$. If, on the other hand, $\delta=0$, then the connection~\eqref{g3,4,alpha,nabla1,sol2,1} coincides with the one associated with the flat Lie algebra $\h_{0,22}$, and necessarily $\alpha=\frac{1}{2}$.

If $\nabla^0=\nabla^2$ or $\nabla^0=\nabla^3$, then a straightforward computation shows that the flatness equations admit no solutions for any value of the parameter $\alpha$ satisfying the above assumptions and for any choice of the parameters defining the torsion-free connections on $\nabla^2$ and $\nabla^3$.

If $\nabla^0=\nabla^4$, then a straightforward computation shows that the flatness equations admit a solution if and only if $\alpha=\frac{1}{2}$. In this case, the solution is unique, and the flat torsion-free connection follows:
\begin{equation}\label{g3,4,alpha,nabla1,sol4}
	\begin{aligned}
		\nabla_{e_1} e_2 &=e_1-\tfrac{1}{2\,c_{13}}e_3, &\nabla_{e_2}e_1&=e_1-\tfrac{1}{2\,c_{13}}e_3, &\nabla_{e_2}e_3&=2\,c_{13}e_1-e_3, \\
		\nabla_{e_3}e_1&=-e_1, &\nabla_{e_3}e_2&=2\,c_{13}e_1-\tfrac{1}{2}e_2-e_3, &\nabla_{e_3}e_3&=c_{13}e_1-\tfrac{3}{2}e_2.
	\end{aligned}
\end{equation}

Consider the following automorphism
\begin{align*}
	\Psi(e_1) &=\tfrac{1}{(4\,c_{13}^2)^{\frac{1}{3}}} e_1,&
	\Psi(e_2) &=-\tfrac{1}{(2\,c_{13})^{\frac{1}{3}}} e_2,&
	\Psi(e_3) &=(2\,c_{13})^{\frac{1}{3}} e_1+ e_3.
\end{align*}
Applying $\Psi$ to the connection given in~\eqref{g3,4,alpha,nabla1,sol4} yields the following equivalent connection:
	\begin{equation}\label{g3,4,alpha,nabla1,sol4,1}
		\begin{aligned}
			\nabla_{e_1} e_2 &=e_3, &\nabla_{e_2}e_1&=e_3, &\nabla_{e_2}e_2&=e_1, &\nabla_{e_3}e_1&=-e_1, &\nabla_{e_3}e_2&=-\tfrac{1}{2}e_2, &\nabla_{e_3}e_3&=-\tfrac{3}{2}e_3.
		\end{aligned}
	\end{equation}
	As a matter of fact, this connection coincides with that associated with the flat Lie algebra $\h_{2,7}$.

	If $\nabla^0=\nabla^5$, then a straightforward computation shows that the flatness equations admit a unique solution given by the following flat, torsion-free connection:
	
	\begin{equation}\label{g3,4,alpha,nabla1,sol5}
		\begin{aligned}
			\nabla_{e_1} e_2 &=e_1+a_{32}e_3, &\nabla_{e_2}e_1&=e_1+a_{32}e_3, &\nabla_{e_2}e_3&=-\tfrac{1}{a_{32}}e_1-e_3,\\
			\nabla_{e_3}e_1&=-e_1, &\nabla_{e_3}e_2&=-\tfrac{1}{a_{32}}e_1-\alpha\,e_2-e_3, &\nabla_{e_3}e_3&=-\tfrac{\alpha}{a_{32}}e_1-(\alpha+1)e_3.
		\end{aligned}
	\end{equation}

	Consider the following automorphism
	\begin{align*}
		\Psi(e_1) &= e_1,&
		\Psi(e_2) &=a_{32}\, e_2,&
		\Psi(e_3) &=-\tfrac{1}{a_{32}} e_1+ e_3.
	\end{align*}
	Applying $\Psi$ to the connection given in~~\eqref{g3,4,alpha,nabla1,sol5} yields precisely the flat, torsion-free connection associated with the flat Lie algebra $\h_{0,22}$.

If $\nabla^0=\nabla^6$, then a straightforward computation shows that the flatness equations admit no solutions for any value of the parameter $\alpha$ satisfying the above assumptions.

	If $\nabla^0=\nabla^7$, then a straightforward computation shows that the flatness equations admit a solution if and only if $\alpha=\frac{1}{2}$. The corresponding 
	  flat torsion-free connection is given as follows:
	\begin{equation}\label{g3,4,alpha,nabla1,sol7}
	\begin{aligned}
		\nabla_{e_1} e_2 &=e_2-\tfrac{1}{a_{23}}e_3, &\nabla_{e_1}e_3&=a_{23}e_2-e_3, &\nabla_{e_2}e_1&=e_2-\tfrac{1}{a_{23}}e_3, &\nabla_{e_2}e_2&=\varepsilon_2e_1,\\ \nabla_{e_2}e_3&=\varepsilon_2\,a_{23}e_1,\
		&\nabla_{e_3}e_1&=-e_1+a_{23}e_2-e_3, &\nabla_{e_3}e_2&=\varepsilon_2\,a_{23}e_1-\tfrac{1}{2}e_2,\\ \nabla_{e_3}e_3&=\varepsilon_2\,a_{23}^2e_1+a_{23}e_2-\tfrac{3}{2}e_3.
			\end{aligned}
\end{equation}
	
		Consider the following automorphism
	\begin{align*}
		\Psi(e_1) &=-\tfrac{1}{x\,a_{23}} e_1,&
		\Psi(e_2) &=x\,e_2,&
		\Psi(e_3) &=x\,a_{23}e_2+ e_3.
	\end{align*}
	For a suitable choice of the parameter $x$, applying $\Psi$ to the connection given in~\eqref{g3,4,alpha,nabla1,sol7} yields the following equivalent connection:
	\begin{equation}\label{g3,4,alpha,nabla1,sol7,1}
		\begin{aligned}
			\nabla_{e_1} e_2 &=e_3, &\nabla_{e_2}e_1&=e_3, &\nabla_{e_2}e_2&=e_1, &\nabla_{e_3}e_1&=-e_1, &\nabla_{e_3}e_2&=-\tfrac{1}{2}e_2, &\nabla_{e_3}e_3&=-\tfrac{3}{2}e_3.
		\end{aligned}
	\end{equation}
	It is actually the connection associated with the flat Lie algebra $\h_{2,7}$.

		If $\nabla^0=\nabla^8$, then a straightforward computation shows that the flatness equations admit a solution if and only if $\delta_1=0$. The corresponding 
	flat torsion-free connection is given as follows:
	\begin{equation}\label{g3,4,alpha,nabla1,sol8}
	\begin{aligned}
		\nabla_{e_1} e_2 &=e_2-\tfrac{1}{a_{23}}e_3, &\nabla_{e_1}e_3&=a_{23}e_2-e_3, &\nabla_{e_2}e_1&=e_2-\tfrac{1}{a_{23}}e_3,\\ \nabla_{e_3}e_1&=-e_1+a_{23}e_2-e_3, 
		&\nabla_{e_3}e_2&=\alpha\,e_2, &\nabla_{e_3}e_3&=a_{23}e_2-(\alpha+1)e_3.
	\end{aligned}
	\end{equation}
		Consider the following automorphism
	\begin{align*}
		\Psi(e_1) &=-\tfrac{1}{a_{23}} e_1,&
		\Psi(e_2) &= e_2,&
		\Psi(e_3) &=a_{23} e_2+ e_3.
	\end{align*}
	Applying $\Psi$ to the connection given in~~\eqref{g3,4,alpha,nabla1,sol8} yields precisely the flat, torsion-free connection associated with the flat Lie algebra $\h_{0,22}$.

		If $\nabla^0=\nabla^9$, then a straightforward computation shows that the flatness equations admit a unique solution. The corresponding 
	flat torsion-free connection is given as follows:
	\begin{equation}\label{g3,4,alpha,nabla1,sol9}
		\begin{aligned}
			\nabla_{e_1} e_1 &=e_1+e_2+a_{31}e_3, &\nabla_{e_1}e_3&=-\tfrac{1}{a_{31}}e_1-\tfrac{1}{a_{31}}e_2-e_3, &\nabla_{e_3}e_1&=-\tfrac{1+a_{31}}{a_{31}}e_1-\tfrac{1}{a_{31}}e_2-e_3,\\
			\nabla_{e_3}e_2&=-\alpha\,e_2, &\nabla_{e_3}e_3&=\tfrac{1-a_{31}}{a_{31}^2}e_1+\tfrac{1+(\alpha-2)a_{31}}{a_{31}^2}e_2+\tfrac{1-2\,a_{31}}{a_{31}}e_3.
		\end{aligned}
	\end{equation}

		Consider the following automorphism
	\begin{align*}
		\Psi(e_1) &=\tfrac{\sqrt{\varepsilon\,a_{31}}}{\varepsilon} e_1,&
		\Psi(e_2) &= e_2,&
		\Psi(e_3) &=-\tfrac{\sqrt{\varepsilon\,a_{31}}}{\varepsilon\,a_{31}} e_1-\tfrac{1}{a_{31}}e_2+e_3,\quad\varepsilon=\pm1.
	\end{align*}
	Applying $\Psi$ to the connection given in~~\eqref{g3,4,alpha,nabla1,sol9} yields precisely the flat, torsion-free connection associated with the flat Lie algebra $\h_{0,7}$.

	If $\nabla^0=\nabla^{10}$ or $\nabla^0=\nabla^{11}$, then a straightforward computation shows that the flatness equations admit no solutions for any value of the parameter $\alpha$ satisfying the above assumptions and for any choice of the parameters defining the torsion-free connections on $\nabla^{10}$ and $\nabla^{11}$.

	If $\nabla^0=\nabla^{12}$ and $\delta_{\varepsilon}=0$, then a straightforward computation shows that the flatness equations admit three solutions. The first solution is given by the following  
	flat torsion-free connection:
	\begin{equation}\label{g3,4,alpha,nabla12,sol1}
		\begin{aligned}
			\nabla_{e_1} e_1 &=e_2+a_{31}e_3, &\nabla_{e_3}e_1&=-e_1, &\nabla_{e_3}e_2&=-\alpha\,e_2, &\nabla_{e_3}e_3&=\tfrac{\alpha-2}{a_{31}}e_2-2\,e_3.
		\end{aligned}
	\end{equation}
	
	Consider the following automorphism
	\begin{align*}
	\Psi(e_1) &=\tfrac{\sqrt{\varepsilon\,a_{31}}}{\varepsilon} e_1,&
	\Psi(e_2) &= e_2,&
	\Psi(e_3) &=-\tfrac{1}{a_{31}} e_2+e_3,\quad\varepsilon=\pm1.
\end{align*}
	Applying it to the  connection~\eqref{g3,4,alpha,nabla12,sol1}, we obtain the equivalent connection
	\begin{equation}\label{g3,4,alpha,nabla12,sol1,1}
		\begin{aligned}
			\nabla_{e_1} e_1 &=\varepsilon\, e_3, &\nabla_{e_3}e_1&=-e_1, &\nabla_{e_3}e_2&=-\alpha\,e_2, &\nabla_{e_3}e_3&=-2\,e_3.
		\end{aligned}
	\end{equation}
	which coincides with the one associated with the flat Lie algebra $\h_{0,7}$ listed in Table~\ref{g3,4,alpha}.
	
	The second solution is given by the following flat, torsion-free connection:
	\begin{equation}\label{g3,4,alpha,nabla12,sol2}
		\begin{aligned}
			\nabla_{e_1} e_1 &=e_2, \quad\nabla_{e_1}e_3=a_{23}e_2, \quad\nabla_{e_2}e_3=(\alpha-2)e_2, &\nabla_{e_3}e_1&=-e_1+a_{23}e_2, &\nabla_{e_3}e_2&=-2\,e_2,\\
			\nabla_{e_3}e_3&=a_{23}(1-\alpha)e_1+c_{23}e_2+(\alpha-2)e_3.
		\end{aligned}
	\end{equation}
		Consider the following automorphism
	\begin{align*}
		\Psi(e_1) &= e_1,&
		\Psi(e_2) &= e_2,&
		\Psi(e_3) &=a_{23}e_1+\tfrac{a_{23}^2-c_{23}}{2}e_2+e_3.
	\end{align*}
	Applying it to the  connection~\eqref{g3,4,alpha,nabla12,sol2}, we obtain the equivalent connection
	\begin{equation}\label{g3,4,alpha,nabla12,sol2,1}
		\begin{aligned}
			\nabla_{e_1} e_1 &=e_2, &\nabla_{e_2}e_3&=(\alpha-2)e_2,&\nabla_{e_3}e_1&=-e_1, &\nabla_{e_3}e_2&=-2\,e_2, &\nabla_{e_3}e_3&=(\alpha-2)e_3.
		\end{aligned}
	\end{equation}
	which coincides with the one associated with the flat Lie algebra $\h_{2,8}$ listed in Table~\ref{g3,4,alpha}.
	
	The third  solution is given by the following flat, torsion-free connection:
		\begin{equation}\label{g3,4,alpha,nabla12,sol3}
		\begin{aligned}
			\nabla_{e_1} e_1 &=e_2, \quad\nabla_{e_1}e_3=(2-\alpha)e_1+a_{23}e_2, &\nabla_{e_3}e_1&=(1-\alpha)e_1+a_{23}e_2, &\nabla_{e_3}e_2&=-\alpha\,e_2,\\
			\nabla_{e_3}e_3&=a_{23}(1-\alpha)e_1+c_{23}e_2+(2-\alpha)e_3.
		\end{aligned}
	\end{equation}
		Consider the following automorphism
	\begin{align*}
		\Psi(e_1) &= e_1,&
		\Psi(e_2) &= e_2,&
		\Psi(e_3) &=a_{23}e_1+\tfrac{a_{23}^2-c_{23}}{2}e_2+e_3.
	\end{align*}
	Applying it to the  connection~\eqref{g3,4,alpha,nabla12,sol3}, we obtain the equivalent connection
	\begin{equation}\label{g3,4,alpha,nabla12,sol3,1}
		\begin{aligned}
			\nabla_{e_1} e_1 &=e_2, &\nabla_{e_2}e_3&=(2-\alpha)e_2,&\nabla_{e_3}e_1&=(1-\alpha)e_1, &\nabla_{e_3}e_2&=-\alpha\,e_2, &\nabla_{e_3}e_3&=(2-\alpha)e_3.
		\end{aligned}
	\end{equation}
	which coincides with the one associated with the flat Lie algebra $\h_{2,9}$ listed in Table~\ref{g3,4,alpha}.

	If $\nabla^0=\nabla^{12}$ and $\delta_{\varepsilon}=\pm1$, it is evident from a straightforward computation that the flatness equations do not admit any solution under $0<|\alpha|<1$ and that $\delta_{\varepsilon}=\pm1$.

		If $\nabla^0=\nabla^{13}$, then a straightforward computation shows that the flatness equations admit two solutions. The first solution is given by the following  
	flat torsion-free connection:
	\begin{equation}\label{g3,4,alpha,nabla1,sol13}
		\begin{aligned}
			\nabla_{e_2} e_2 &=e_1+e_2-\tfrac{1}{b_{23}}e_3, &\nabla_{e_2}e_3&=b_{23}e_1+b_{23}e_2-e_1, \qquad\nabla_{e_3}e_1=-e_1,\\
			\nabla_{e_3}e_2&=b_{23}e_1+(b_{23}-\alpha)e_2-e_3, &\nabla_{e_3}e_3&=b_{23}(2\,\alpha+b_{23}-1)e_1+b_{23}(\alpha+b_{23})e_2-(b_{23}+2\,\alpha)e_3.
		\end{aligned}
	\end{equation}

	Consider the following automorphism
\begin{align*}
	\Psi(e_1) &= e_1,&
	\Psi(e_2) &= \tfrac{1}{\sqrt{-\varepsilon\,b_{23}}}e_2,&
	\Psi(e_3) &=b_{23}e_1+\tfrac{b_{23}}{\sqrt{-\varepsilon\,b_{23}}}e_2+e_3,\quad\varepsilon=\pm1.
\end{align*}
Applying $\Psi$ to the connection given in~~\eqref{g3,4,alpha,nabla1,sol13} yields precisely the flat, torsion-free connection associated with the flat Lie algebra $\h_{0,23}$.

	The second solution is given by the following flat, torsion-free connection:
	\begin{equation}\label{g3,4,alpha,nabla1,sol13,1}
		\begin{aligned}
			\nabla_{e_2} e_2 &=e_1+e_2-\tfrac{1}{b_{23}}e_3, &\nabla_{e_2}e_3&=\tfrac{3\,b_{23}^2+b_{23}+3\,c_{13}}{6\,b_{23}}e_1+b_{23}e_2-e_3, \qquad\nabla_{e_3}e_1=-e_1,\\
			\nabla_{e_3}e_2&=\tfrac{3\,b_{23}^2+b_{23}+3\,c_{13}}{6\,b_{23}}e_1+(b_{23}-\tfrac{1}{3})e_2-e_3, &\nabla_{e_3}e_3&=c_{13}e_1+(b_{23}^2+\tfrac{1}{3}b_{23})e_2-(b_{23}+\tfrac{2}{3})e_3.
		\end{aligned}
	\end{equation}
	In this case, $\alpha=\frac{1}{3}$.

		Consider the following automorphism
	\begin{align*}
		\Psi(e_1) &=x\, e_1,&
		\Psi(e_2) &= \tfrac{1}{\sqrt{-\varepsilon\,b_{23}}}e_2,&
		\Psi(e_3) &=x\,b_{23}e_1+\tfrac{b_{23}}{\sqrt{-\varepsilon\,b_{23}}}e_2+e_3,\quad\varepsilon=\pm1.
	\end{align*}
For a suitable choice of the parameter $x\in\R^\ast$, applying $\Psi$  to the connection given in~~\eqref{g3,4,alpha,nabla1,sol13} yields the following equivalent connection:

		\begin{equation}\label{g3,4,alpha,nabla1,sol13,1,1}
		\begin{aligned}
			\nabla_{e_2}e_2&=\varepsilon\,e_3, &\nabla_{e_2}e_3&=\delta\,e_1, &\nabla_{e_3}e_1&=-e_1, &\nabla_{e_3}e_2&=\delta\,e_1-\tfrac{1}{3}e_2, &\nabla_{e_3}e_3&=-\tfrac{2}{3}e_3,\quad\delta=0,1.
		\end{aligned}
	\end{equation}
	If $\delta=0$, then this connection coincides with the one associated with the flat Lie algebra $\h_{0,23}$, in which case $\alpha=\frac{1}{3}$.  If $\delta=1$,  then connection~\eqref{g3,4,alpha,nabla1,sol13,1,1} is precisely the one associated with the flat Lie algebra $\h_{0,24}$ listed in Table~\ref{g3,4,alpha}.

	If $\nabla^0=\nabla^{14}$, Thus, it can be seen that the flatness equations have exactly seven solutions if a simple calculation is performed. The first one is given by
		\begin{equation}\label{g3,4,alpha,nabla1,sol14,sol1}
		\begin{aligned}
			\nabla_{e_2}e_2&=e_1, \quad\nabla_{e_2}e_3=b_{13}e_1, \quad\nabla_{e_3}e_1=-e_1, &\nabla_{e_3}e_2&=b_{13}e_1-\tfrac{1}{2}e_2,\\ \nabla_{e_3}e_3&=c_{13}e_1-b_{13}(c_{33}+\tfrac{1}{2})e_2+c_{33}e_3.
		\end{aligned}
	\end{equation}
	In this case, $\alpha=\frac{1}{2}$. If $c_{33}\neq-1$. Applying the following automorphism
	\begin{align*}
		\Psi(e_1) &= e_1,&
		\Psi(e_2) &= e_2,&
		\Psi(e_3) &=\tfrac{b_{13}^2-c_{13}}{c_{33}+1}e_1+b_{13}e_2+e_3,
	\end{align*}
	to the connection given in~\eqref{g3,4,alpha,nabla1,sol14,sol1} yields the following equivalent connection:
		\begin{equation}\label{g3,4,alpha,nabla1,sol14,sol1,1}
		\begin{aligned}
			\nabla_{e_2}e_2&=e_1, &\nabla_{e_3}e_1&=-e_1, &\nabla_{e_3}e_2&=-\tfrac{1}{2}e_2, &\nabla_{e_3}e_3&=\lambda\,e_3.
		\end{aligned}
	\end{equation}
	In this case, we have $c_{33}=\lambda\neq-1$.
	Now, if $c_{33}=-1$. Consider the following automorphism
	\begin{align*}
		\Psi(e_1) &=x^2\, e_1,&
		\Psi(e_2) &=x\, e_2,&
		\Psi(e_3) &=x\,b_{13}e_2+e_3.
	\end{align*}
	For a suitable choice of the parameter $x\in\R^\ast$, applying $\Psi$ to the connection given in \eqref{g3,4,alpha,nabla1,sol14,sol1} yields the following equivalent connection:
	\begin{equation}\label{g3,4,alpha,nabla1,sol14,sol1,2}
		\begin{aligned}
			\nabla_{e_2}e_2&=e_1, &\nabla_{e_3}e_1&=-e_1, &\nabla_{e_3}e_2&=-\tfrac{1}{2}e_2, &\nabla_{e_3}e_3&=\delta_\varepsilon\,e_1-e_3,\quad\delta_\varepsilon=0,\pm1.
		\end{aligned}
	\end{equation}
	Observe first that the flat, torsion-free connections given in \eqref{g3,4,alpha,nabla1,sol14,sol1,1} and  \eqref{g3,4,alpha,nabla1,sol14,sol1,2} are isomorphic if and only if $\delta_\varepsilon=0$ and $\lambda=-1$. For this reason, we set $\lambda\in\R$ and assume that $\delta_\varepsilon=\varepsilon=\pm1$. The previous connection then corresponds to the flat, torsion-free connections associated with the flat Lie algebras $\h_{2,2}$ and $\h_{2,3}$, respectively. The additional condition $\lambda\in\R^\ast$ comes from the fact that, otherwise, the connection would be isomorphic to the flat Lie algebra $\h_{2,1}$, as will be shown later.
	
	The second solution is given by the following flat, torsion-free connection:
	\begin{equation}\label{g3,4,alpha,nabla1,sol14,sol2}
		\begin{aligned}
			\nabla_{e_2}e_2&=e_1, \quad\nabla_{e_2}e_3=b_{13}e_1+(2\,\alpha-1)e_2, &\nabla_{e_3}e_1&=-e_1, &\nabla_{e_3}e_2&=b_{13}e_1+(\alpha-1)e_2 ,\\
			\nabla_{e_3}e_3&=c_{13}e_1+b_{13}(\alpha-1)e_2+(2\,\alpha-1)e_3.
		\end{aligned}
	\end{equation}
	 Applying the following automorphism
	\begin{align*}
		\Psi(e_1) &= e_1,&
		\Psi(e_2) &= e_2,&
		\Psi(e_3) &=\tfrac{b_{13}^2-c_{13}}{2\,\alpha}e_1+b_{13}e_2+e_3,
	\end{align*}
	to the connection given in~\eqref{g3,4,alpha,nabla1,sol14,sol2} yields the following equivalent connection:
	\begin{equation}\label{g3,4,alpha,nabla1,sol14,sol2,1}
		\begin{aligned}
			\nabla_{e_2}e_2&=e_1, &\nabla_{e_2}e_3&=(2\,\alpha-1)e_2, &\nabla_{e_3}e_1&=-e_1, &\nabla_{e_3}e_2&=(\alpha-1)e_2, &\nabla_{e_3}e_3&=(2\,\alpha-1)e_3.
		\end{aligned}
	\end{equation}
	This connection is associated with the flat Lie algebra $\h_{2,4}$, where the assumption $\alpha\neq\frac{1}{2}$ is imposed to avoid it being isomorphic to the flat Lie algebra $\h_{2,1}$.
	
	The third solution corresponds to the following flat, torsion-free connection:
	\begin{equation}\label{g3,4,alpha,nabla1,sol14,sol3}
		\begin{aligned}
		\nabla_{e_1}e_3&=(1-2\,\alpha)e_1,	\quad\nabla_{e_2}e_2=e_1,  \quad\nabla_{e_2}e_3=b_{13}e_1, &\nabla_{e_3}e_1&=-2\,\alpha\,e_1, &\nabla_{e_3}e_2&=b_{13}e_1-\alpha\,e_2,\\
		\nabla_{e_3}e_3&=c_{13}e_1+b_{13}(\alpha-1)e_2+(1-2\,\alpha)e_3.
		\end{aligned}
	\end{equation}
	 Applying the following automorphism
	\begin{align*}
		\Psi(e_1) &= e_1,&
		\Psi(e_2) &= e_2,&
		\Psi(e_3) &=\tfrac{b_{13}^2-c_{13}}{2\,\alpha}e_1+b_{13}e_2+e_3,
	\end{align*}
	to the connection given in~\eqref{g3,4,alpha,nabla1,sol14,sol3} yields the following equivalent connection:
	\begin{equation}\label{g3,4,alpha,nabla1,sol14,sol3,1}
		\begin{aligned}
			\nabla_{e_1}e_3&=(1-2\,\alpha)e_1,	\quad\nabla_{e_2}e_2=e_1, &\nabla_{e_3}e_1&=-2\,\alpha\,e_1, &\nabla_{e_3}e_2&=-\alpha\,e_2,&
			\nabla_{e_3}e_3&=(1-2\,\alpha)e_3.
		\end{aligned}
	\end{equation}
	The previous flat, torsion-free connection coincides precisely with the one associated with the flat Lie algebra algebra $\h_{2,1}$ listed in Table~\ref{g3,4,alpha}.

	The fourth solution corresponds to the following flat, torsion-free connection:
	\begin{equation}\label{g3,4,alpha,nabla1,sol14,sol4}
		\begin{aligned}
			\nabla_{e_1}e_3&=c_{33}e_1, \quad\nabla_{e_2}e_2=e_1, &\nabla_{e_2}e_3&=b_{13}e_1+c_{33}e_2, \quad\nabla_{e_3}e_1=(c_{33}-1)e_1, \\\nabla_{e_3}e_2&=b_{13}e_1+(c_{33}-\tfrac{1}{2})e_2,&
			\nabla_{e_3}e_3&=c_{13}e_1+b_{13}(c_{33}-\tfrac{1}{2})e_2+c_{33}e_3.
		\end{aligned}
	\end{equation}

		In this case, $\alpha=\frac{1}{2}$. If $c_{33}\neq1$. Applying the following automorphism
	\begin{align*}
		\Psi(e_1) &= e_1,&
		\Psi(e_2) &= e_2,&
		\Psi(e_3) &=\tfrac{c_{13}-b_{13}^2}{c_{33}-1}e_1+b_{13}e_2+e_3,
	\end{align*}
	to the connection given in~\eqref{g3,4,alpha,nabla1,sol14,sol4} yields the following equivalent connection:
	\begin{equation}\label{g3,4,alpha,nabla1,sol14,sol4,1}
		\begin{aligned}
		\nabla_{e_1}e_3&=\lambda\,e_1,	&\nabla_{e_2}e_2&=e_1, &\nabla_{e_2}e_3&=\lambda\,e_2 &\nabla_{e_3}e_1&=(\lambda-1)e_1, &\nabla_{e_3}e_2&=(\lambda-\tfrac{1}{2})e_2, &\nabla_{e_3}e_3&=\lambda\,e_3.
		\end{aligned}
	\end{equation}
	In this case, we have $c_{33}=\lambda\neq1$.
	Now, if $c_{33}=1$. Consider the following automorphism
	\begin{align*}
		\Psi(e_1) &=x^2\, e_1,&
		\Psi(e_2) &=x\, e_2,&
		\Psi(e_3) &=x\,b_{13}e_2+e_3.
	\end{align*}
	For a suitable choice of the parameter $x\in\R^\ast$, applying $\Psi$ to the connection given in \eqref{g3,4,alpha,nabla1,sol14,sol4} yields the following equivalent connection:
	\begin{equation}\label{g3,4,alpha,nabla1,sol14,sol4,2}
	\begin{aligned}
		\nabla_{e_1}e_3&=e_1,	&\nabla_{e_2}e_2&=e_1, &\nabla_{e_2}e_3&=e_2,  &\nabla_{e_3}e_2&=\tfrac{1}{2}e_2, &\nabla_{e_3}e_3&=\delta_\varepsilon\,e_1+e_3.
	\end{aligned}
\end{equation}
	Observe first that the flat, torsion-free connections given in \eqref{g3,4,alpha,nabla1,sol14,sol4,1} and  \eqref{g3,4,alpha,nabla1,sol14,sol4,2} are isomorphic if and only if $\delta_\varepsilon=0$ and $\lambda=1$. For this reason, we set $\lambda\in\R$ and assume that $\delta_\varepsilon=\varepsilon=\pm1$. The previous connection then corresponds to the flat, torsion-free connections associated with the flat Lie algebras $\h_{2,5}$ and $\h_{2,6}$, respectively. The additional condition $\lambda\in\R^\ast$ comes from the fact that, otherwise, the connection would be isomorphic to the flat Lie algebra $\h_{2,1}$.
	
	The fifth solution corresponds to the following flat, torsion-free connection:
		\begin{equation}\label{g3,4,alpha,nabla1,sol14,sol5}
		\begin{aligned}
	\nabla_{e_2}e_2&=e_1+b_{32}e_3, &\nabla_{e_3}e_1&=-e_1,  &\nabla_{e_3}e_2&=\tfrac{b_{32}c_{13}-1}{2}e_2, &\nabla_{e_3}e_3&=c_{13}e_1+(b_{32}-1)e_3.
		\end{aligned}
	\end{equation}
	In this case, $\alpha=\tfrac{1-b_{32}c_{13}}{2}$. Suppose that $b_{32}\neq0$, and set $c_{13}=\frac{1-2\,\alpha}{b_{32}}$. Consider the following automorphism
	\begin{align*}
		\Psi(e_1) &= e_1,&
		\Psi(e_2) &=x\, e_2,&
		\Psi(e_3) &=-\tfrac{1}{b_{32}}e_1+e_3.
	\end{align*}
	For a suitable choice of the parameter $x\in\R^\ast$, applying $\Psi$ to the connection given in \eqref{g3,4,alpha,nabla1,sol14,sol5} yields the following equivalent connection:
	\begin{equation}\label{g3,4,alpha,nabla1,sol14,sol5,1}
		\begin{aligned}
			\nabla_{e_2}e_2&=\varepsilon\,e_3, &\nabla_{e_3}e_1&=-e_1,  &\nabla_{e_3}e_2&=-\alpha\,e_2, &\nabla_{e_3}e_3&=-2\,\alpha\,e_3.
		\end{aligned}
	\end{equation}
	This flat, torsion-free connection is precisely the one associated with the flat Lie algebra $\h_{0,23}$.
	
	If $b_{32}=0$, then $\alpha=\frac{1}{2}$. Consider the following automorphism
	\begin{align*}
		\Psi(e_1) &=x^2\, e_1,&
		\Psi(e_2) &=x\, e_2,&
		\Psi(e_3) &=e_3.
	\end{align*}
	For a suitable choice of the parameter $x\in\R^\ast$, applying $\Psi$ to the connection given in \eqref{g3,4,alpha,nabla1,sol14,sol5} yields the following equivalent connection:
	\begin{equation}\label{g3,4,alpha,nabla1,sol14,sol5,2}
		\begin{aligned}
			\nabla_{e_2}e_2&=e_1, &\nabla_{e_3}e_1&=-e_1,  &\nabla_{e_3}e_2&=-\tfrac{1}{2} e_2, &\nabla_{e_3}e_3&=\delta_\varepsilon e_1-e_3.
		\end{aligned}
	\end{equation}
	This flat, torsion-free connection coincides with the one presented previously in \eqref{g3,4,alpha,nabla1,sol14,sol1,2}, and thus no further analysis is necessary.
	
	The sixth solution corresponds to the following flat, torsion-free connection:
	\begin{equation}\label{g3,4,alpha,nabla1,sol14,sol6}
		\begin{aligned}
			\nabla_{e_2}e_2&=e_1+b_{32}e_3, &\nabla_{e_2}e_3&=b_{13}e_1, &\nabla_{e_3}e_1&=-e_1, &\nabla_{e_3}e_2&=b_{13}e_1-\tfrac{1}{3}e_2, &\nabla_{e_3}e_3&=\tfrac{1}{3\,b_{32}}e_1-\tfrac{2}{3}e_3.
		\end{aligned}
	\end{equation}
	In this case, $\alpha=\frac{1}{3}$. Consider the following automorphism
	\begin{align*}
		\Psi(e_1) &=x\, e_1,&
		\Psi(e_2) &=\tfrac{\sqrt{\varepsilon\,b_{32}}}{\varepsilon} e_2,&
		\Psi(e_3) &=-\tfrac{x}{b_{32}}e_1+e_3,\quad\varepsilon=\pm1.
	\end{align*}
	For a suitable choice of the parameter $x\in\R^\ast$, applying $\Psi$ to the connection given in \eqref{g3,4,alpha,nabla1,sol14,sol6} yields the following equivalent connection:
	\begin{equation}\label{g3,4,alpha,nabla1,sol14,sol6,1}
		\begin{aligned}
			\nabla_{e_2}e_2&=\varepsilon\, e_3, &\nabla_{e_2}e_3&=\delta\, e_1, &\nabla_{e_3}e_1&=-e_1, &\nabla_{e_3}e_2&=\delta\,e_1-\tfrac{1}{3}e_2, &\nabla_{e_3}e_3&=-\tfrac{2}{3}e_3.
		\end{aligned}
	\end{equation}
	This flat, torsion-free connection coincides with the one presented previously in \eqref{g3,4,alpha,nabla1,sol13,1,1}, and thus no further analysis is necessary.

	The last solution in this case is given by the following flat, torsion-free connection:
	\begin{equation}\label{g3,4,alpha,nabla1,sol14,sol7}
		\begin{aligned}
			\nabla_{e_1}e_2&=a_{32}e_3, &\nabla_{e_2}e_1&=a_{32}e_3, &\nabla_{e_3}e_1&=-e_1, &\nabla_{e_3}e_2&=-\tfrac{1}{2}e_2, &\nabla_{e_3}e_3&=-\tfrac{3}{2}e_3.
		\end{aligned}
	\end{equation}
	In this case, $\alpha=\frac{1}{2}$. Consider the following automorphism
	\begin{align*}
		\Psi(e_1) &=x^2\, e_1,&
		\Psi(e_2) &=x\, e_2,&
		\Psi(e_3) &=e_3.
	\end{align*}
	For a suitable choice of the parameter $x\in\R^\ast$, applying $\Psi$ to the connection given in \eqref{g3,4,alpha,nabla1,sol14,sol7} yields the following equivalent connection 
	\begin{equation}\label{g3,4,alpha,nabla1,sol14,sol7,1}
		\begin{aligned}
			\nabla_{e_1}e_2&=\delta\,e_3, &\nabla_{e_2}e_1&=\delta\,e_3, &\nabla_{e_3}e_1&=-e_1, &\nabla_{e_3}e_2&=-\tfrac{1}{2}e_2, &\nabla_{e_3}e_3&=-\tfrac{3}{2}e_3,\quad\delta=0,1.
		\end{aligned}
	\end{equation}
	If $\delta=1$, then this connection coincides with the one associated with the flat Lie algebra $\h_{2,7}$. On the other hand, if $\delta=0$, it coincides with the one associated with the flat Lie algebra $\h_{0,10}$, where $\alpha=\frac{1}{2}$ and  $\lambda=-\frac{3}{2}$.

	If $\nabla^0=\nabla^{15}$ and $\delta=1$, then a straightforward computation shows that the flatness equations admit exactly two solutions. The first one is given by
	\begin{equation}\label{g3,4,alpha,nabla1,sol15,delta=1}
		\begin{aligned}
			\nabla_{e_2}e_2&=e_2+b_{32}e_3, &\nabla_{e_2}e_3&=-\tfrac{1}{b_{32}}e_2-e_3, &\nabla_{e_3}e_1&=-e_1,\\
			\nabla_{e_3}e_2&=-\tfrac{1+\alpha\,b_{32}}{b_{32}}e_2-e_3, &\nabla_{e_3}e_3&=\tfrac{1-\alpha\,b_{32}}{b_{32}}e_2+\tfrac{1-2\,\alpha\,b_{32}}{b_{32}}e_3.
		\end{aligned}
	\end{equation}
Applying the following automorphism
\begin{align*}
	\Psi(e_1) &= e_1,&
	\Psi(e_2) &= \tfrac{\sqrt{\varepsilon\,b_{32}}}{\varepsilon}e_2,&
	\Psi(e_3) &=-\tfrac{\sqrt{\varepsilon\,b_{23}}}{\varepsilon\,b_{32}}e_2+e_3,\quad\varepsilon=\pm1.
\end{align*}
to the connection given in~\eqref{g3,4,alpha,nabla1,sol15,delta=1} yields precisely the one associated with the flat Lie algebra $\h_{0,23}$.

The second solution is given by the following flat, torsion-free connection:
\begin{equation}\label{g3,4,alpha,nabla1,sol15,delta=1,sol2}
	\begin{aligned}
		\nabla_{e_2}e_2&=e_2+b_{32}e_3, &\nabla_{e_2}e_3&=-\tfrac{b_{32}c_{13}}{2}e_1-\tfrac{1}{b_{32}}e_2-e_3, &\nabla_{e_3}e_1&=-e_1,\\
		\nabla_{e_3}e_2&=-\tfrac{b_{32}c_{13}}{2}e_1-\tfrac{b_{32}+3}{3\,b_{32}}e_2, &\nabla_{e_3}e_3&=c_{13}e_1+\tfrac{3-b_{32}}{3\,b_{32}^2}e_2+\tfrac{3-2\,b_{32}}{3\,b_{32}}e_3.
	\end{aligned}
\end{equation}
In this case, $\alpha=\frac{1}{3}.$ Consider the following automorphism
\begin{align*}
	\Psi(e_1) &=x\, e_1,&
	\Psi(e_2) &= \tfrac{\sqrt{\varepsilon\,b_{32}}}{\varepsilon}e_2,&
	\Psi(e_3) &=\tfrac{\sqrt{\varepsilon\,b_{32}}}{\varepsilon\,b_{32}}e_2+e_3,\quad\varepsilon=\pm1.
\end{align*}
For a suitable choice of the parameter $x\in\R^\ast$, applying $\Psi$  to the connection given in~~\eqref{g3,4,alpha,nabla1,sol15,delta=1,sol2} yields the following equivalent connection:
\begin{equation}\label{g3,4,alpha,nabla1,sol15,delta=1,sol2,1}
	\begin{aligned}
		\nabla_{e_2}e_2&=\varepsilon\,e_3, &\nabla_{e_2}e_3&=\delta_1\,e_1, &\nabla_{e_3}e_1&=-e_1, &\nabla_{e_3}e_2&=\delta_1\,e_1-\tfrac{1}{3}e_2, &\nabla_{e_3}e_3&=-\tfrac{2}{3}e_3,\quad\delta_1=0,1.
	\end{aligned}
\end{equation}
This connection is, in fact, identical to the one given in \eqref{g3,4,alpha,nabla1,sol13,1,1}.

If $\nabla^0=\nabla^{15}$ and $\delta=0$, then a straightforward computation shows that the flatness-equations admit exactly twelve solutions. The first one is given by
\begin{equation}\label{g3,4,alpha,nabla1,sol15,delta=0,sol1}
	\begin{aligned}
		\nabla_{e_1}e_2&=a_{32}e_3, &\nabla_{e_2}e_1&=a_{32}e_3, &\nabla_{e_3}e_1&=-e_1, &\nabla_{e_3}e_2&=-\alpha\,e_2, &\nabla_{e_3}e_3&=-(\alpha+1)e_3.
	\end{aligned}
\end{equation}
Consider the following automorphism
\begin{align*}
	\Psi(e_1) &=x\, e_1,&
	\Psi(e_2) &= e_2,&
	\Psi(e_3) &=e_3.
\end{align*}
For a suitable choice of the parameter $x\in\R^\ast$, applying $\Psi$  to the connection given in~~\eqref{g3,4,alpha,nabla1,sol15,delta=0,sol1} yields the following equivalent connection:
\begin{equation}\label{g3,4,alpha,nabla1,sol15,delta=0,sol1,1}
	\begin{aligned}
		\nabla_{e_1}e_2&=\delta\,e_3, &\nabla_{e_2}e_1&=\delta\,e_3, &\nabla_{e_3}e_1&=-e_1, &\nabla_{e_3}e_2&=-\alpha\,e_2, &\nabla_{e_3}e_3&=-(\alpha+1)e_3.
	\end{aligned}
\end{equation}
If $\delta=1$, then this connection coincides with the one associated with the flat Lie algebra $\h_{0,22}$. On the other hand, if $\delta=0$, it coincides with the one associated with the flat Lie algebra $\h_{0,10}$, where   $\lambda=-(\alpha+1)$.

The second solution is given by the following flat, torsion-free connection:
\begin{equation}\label{g3,4,alpha,nabla1,sol15,delta=0,sol2}
	\begin{aligned}
		\nabla_{e_1}e_3&=b_{23}e_1, \quad\nabla_{e_2}e_3=b_{23}e_2, &\nabla_{e_3}e_1&=(b_{23}-1)e_1, &\nabla_{e_3}e_2&=(b_{23}-\alpha)e_2, \\\nabla_{e_3}e_3&=c_{13}e_1+c_{23}e_2+c_{33}e_3.
	\end{aligned}
\end{equation}
If $b_{23}\neq\alpha$ and $b_{23}\neq1$.
Consider the following automorphism
\begin{align*}
	\Psi(e_1) &= e_1,&
	\Psi(e_2) &= e_2,&
	\Psi(e_3) &=\tfrac{c_{13}}{b_{23}-1}e_1+\tfrac{c_{23}}{b_{23}-\alpha}e_2+e_3.
\end{align*}
Applying $\Psi$  to the connection given in~~\eqref{g3,4,alpha,nabla1,sol15,delta=0,sol2} yields the following equivalent connection:
\begin{equation}\label{g3,4,alpha,nabla1,sol15,delta=0,sol2,1}
	\begin{aligned}
		\nabla_{e_1}e_3&=\lambda\,e_1, \quad\nabla_{e_2}e_3=\lambda\,e_2, &\nabla_{e_3}e_1&=(\lambda-1)e_1, &\nabla_{e_3}e_2&=(\lambda-\alpha)e_2,& \nabla_{e_3}e_3&=\lambda\,e_3.
	\end{aligned}
\end{equation}
In this case, $b_{23}=\lambda\neq\alpha,1$.

If $b_{23}=1$. Consider the following automorphism
\begin{align*}
	\Psi(e_1) &=x\, e_1,&
	\Psi(e_2) &= e_2,&
	\Psi(e_3) &=-\tfrac{c_{23}}{\alpha-1}e_2+e_3.
\end{align*}
For a suitable choice of the parameter $x\in\R^\ast$, applying $\Psi$  to the connection given in~~\eqref{g3,4,alpha,nabla1,sol15,delta=0,sol2,1} yields the following equivalent connection:
\begin{equation}\label{g3,4,alpha,nabla1,sol15,delta=0,sol2,2}
	\begin{aligned}
		\nabla_{e_1}e_3&=e_1, \quad\nabla_{e_2}e_3=e_2,  &\nabla_{e_3}e_2&=(1-\alpha)e_2,& \nabla_{e_3}e_3&=\delta_1\,e_1+e_3,\quad\delta_1=0,1.
	\end{aligned}
\end{equation}

We now consider the case where $b_{23}=\alpha$, and consider the following automorphism
\begin{align*}
	\Psi(e_1) &= e_1,&
	\Psi(e_2) &=x\, e_2,&
	\Psi(e_3) &=\tfrac{c_{13}}{\alpha-1}e_1+e_3.
\end{align*}
For a suitable choice of the parameter $x\in\R^\ast$, applying $\Psi$  to the connection given in~~\eqref{g3,4,alpha,nabla1,sol15,delta=0,sol2,1} yields the following equivalent connection:
\begin{equation}\label{g3,4,alpha,nabla1,sol15,delta=0,sol2,3}
	\begin{aligned}
		\nabla_{e_1}e_3&=\alpha\,e_1, \quad\nabla_{e_2}e_3=\alpha\,e_2,  &\nabla_{e_3}e_1&=(\alpha-1)e_2,& \nabla_{e_3}e_3&=\delta_2\,e_2+\alpha\,e_3,\quad\delta_2=0,1.
	\end{aligned}
\end{equation}

	Observe  that the flat, torsion-free connections given in \eqref{g3,4,alpha,nabla1,sol15,delta=0,sol2,1} and  \eqref{g3,4,alpha,nabla1,sol15,delta=0,sol2,2} are isomorphic if and only if $\delta_1=0$ and $\lambda=1$. For this reason, we set $\lambda\in\R \setminus{\{\alpha\}}$ and assume that $\delta_1=1$. Similarly, one can easily show that \eqref{g3,4,alpha,nabla1,sol15,delta=0,sol2,1} and  \eqref{g3,4,alpha,nabla1,sol15,delta=0,sol2,2} are isomorphic if and only if $\delta_2=0$ and $\lambda=\alpha$.  For this reason, we set $\lambda\in\R$ and assume that $\delta_2=1$.

		The connections given in \eqref{g3,4,alpha,nabla1,sol15,delta=0,sol2,1}, \eqref{g3,4,alpha,nabla1,sol15,delta=0,sol2,2}, and \eqref{g3,4,alpha,nabla1,sol15,delta=0,sol2,3} correspond to the flat, torsion-free connections associated with  the flat Lie algebras $\h_{0,19}$, $\h_{0,21}$ and $\h_{0,20}$, respectively. The additional conditions $\lambda\neq\pm(\alpha-1)$ in the flat Lie algebra $\h_{0,19}$ and $\alpha\neq\frac{1}{2}$ in the flat Lie algebra $\h_{0,20}$ are imposed to ensure that the corresponding connections are not isomorphic to the  algebras $\h_{0,6}$ and $\h_{0,9}$ for algebra $\h_{0,19}$ (in which case $\delta=0$), and to $\h_{0,6}$ for algebra $\h_{0,20}$ (in which case $\delta=0$), respectively, as will be shown later.

		The third solution is given by the following flat, torsion-free connection:
		\begin{equation}\label{g3,4,alpha,nabla1,sol15,delta=0,sol3}
			\begin{aligned}
				\nabla_{e_1}e_3&=(1-\alpha)e_1, \quad\nabla_{e_2}e_3=b_{13}e_1+(1-\alpha)e_2, &\nabla_{e_3}e_1&=-\alpha\,e_1, &\nabla_{e_3}e_2&=b_{13}e_1+(1-2\,\alpha)e_2,\\
				\nabla_{e_3}e_3&=c_{13}e_1+c_{23}e_2+(1-\alpha)e_3.
			\end{aligned}
		\end{equation}
		Suppose that $b_{13}\neq 0$ $\alpha\neq\frac{1}{2}$,  and consider the following automorphism:
		\begin{align*}
			\Psi(e_1) &= e_1,&
			\Psi(e_2) &=b_{13}e_2,&
			\Psi(e_3) &=\tfrac{c_{13}-2\,\alpha\,c_{1 3} - 2\,b_{1 3}c_{2 3}}{(2\,\alpha-1)\alpha}e_1-\tfrac{c_{23}b_{13}}{2\,\alpha-1}e_2+e_3.
		\end{align*}
	Applying $\Psi$  to the connection given in~~\eqref{g3,4,alpha,nabla1,sol15,delta=0,sol3} yields the following equivalent connection:
	\begin{equation}\label{g3,4,alpha,nabla1,sol15,delta=0,sol3,1}
		\begin{aligned}
			\nabla_{e_1}e_3&=(1-\alpha)e_1, &\nabla_{e_2}e_3&=e_1+(1-\alpha)e_2, &\nabla_{e_3}e_1&=-\alpha\,e_1,\\\nabla_{e_3}e_2&=e_1+(1-2\,\alpha)e_2,
			&\nabla_{e_3}e_3&=(1-\alpha)e_3\quad\alpha\neq\tfrac{1}{2}.
		\end{aligned}
	\end{equation}	
If $b_{13}\neq0$ and $\alpha=\frac{1}{2}$. Consider the following automorphism
\begin{align*}
	\Psi(e_1) &=x\, e_1,&
	\Psi(e_2) &=x\,b_{13}\, e_2,&
	\Psi(e_3) &=\tfrac{x\,c_{13}}{2}e_2+e_3.
\end{align*}
For a suitable choice of the parameter $x\in\R^\ast$, applying $\Psi$  to the connection given in~~\eqref{g3,4,alpha,nabla1,sol15,delta=0,sol3} yields the following equivalent connection:
	\begin{equation}\label{g3,4,alpha,nabla1,sol15,delta=0,sol3,2}
	\begin{aligned}
		\nabla_{e_1}e_3&=\tfrac{1}{2} e_1, &\nabla_{e_2}e_3&=e_1+\tfrac{1}{2}e_2, &\nabla_{e_3}e_1&=-\tfrac{1}{2}\,e_1,&\nabla_{e_3}e_2&=e_1,
		&\nabla_{e_3}e_3&=\delta\,e_2+\tfrac{1}{2}e_3,\quad\delta=0,1.
	\end{aligned}
\end{equation}	
 Suppose now that $b_{13}=0$, and consider the following automorphism
\begin{align*}
	\Psi(e_1) &= e_1,&
	\Psi(e_2) &= e_2,&
	\Psi(e_3) &=-\tfrac{c_{13}}{\alpha}e_1+x\,e_2+e_3.
\end{align*}
For a suitable choice of the parameter $x\in\R$, applying $\Psi$  to the connection given in~~\eqref{g3,4,alpha,nabla1,sol15,delta=0,sol3} yields the following equivalent connection:
	\begin{equation}\label{g3,4,alpha,nabla1,sol15,delta=0,sol3,3}
	\begin{aligned}
		\nabla_{e_1}e_3&=(1-\alpha)e_1, &\nabla_{e_2}e_3&=(1-\alpha)e_2, &\nabla_{e_3}e_1&=-\alpha\,e_1,\\\nabla_{e_3}e_2&=(1-2\,\alpha)e_2,
		&\nabla_{e_3}e_3&=\delta\,e_2+(1-\alpha)e_3,&\delta&=0,1.
	\end{aligned}
\end{equation}

Note  that the flat, torsion-free connections given in \eqref{g3,4,alpha,nabla1,sol15,delta=0,sol3,1} and  \eqref{g3,4,alpha,nabla1,sol15,delta=0,sol3,2} are isomorphic if and only if $\delta=0$ and $\alpha=\frac{1}{2}$. For this reason, we set $0<|\alpha|<1$ in the connection defined in \eqref{g3,4,alpha,nabla1,sol15,delta=0,sol3,1} and assume that  $\delta=1$ in \eqref{g3,4,alpha,nabla1,sol15,delta=0,sol3,2}.  The connections given in \eqref{g3,4,alpha,nabla1,sol15,delta=0,sol3,1}, \eqref{g3,4,alpha,nabla1,sol15,delta=0,sol3,2}, and \eqref{g3,4,alpha,nabla1,sol15,delta=0,sol3,3} correspond to the flat, torsion-free connections associated with  the flat Lie algebras $\h_{0,4}$, $\h_{0,5}$ and $\h_{0,6}$, respectively.

The fourth solution is given by the following flat, torsion-free connection:
\begin{equation}\label{g3,4,alpha,nabla1,sol15,delta=0,sol4}
	\begin{aligned}
		\nabla_{e_1}e_3&=a_{13}e_1, &\nabla_{e_3}e_2&=(a_{13}-1)e_1, &\nabla_{e_3}e_2&=-\alpha\,e_2, &\nabla_{e_3}e_3&=c_{13}e_1+c_{23}e_2+a_{13}e_3.
	\end{aligned}
\end{equation}
If $a_{13}\neq1$ and $a_{13}\neq-\alpha$. Consider the following automorphism:
\begin{align*}
	\Psi(e_1) &= e_1,&
	\Psi(e_2) &=e_2,&
	\Psi(e_3) &=\tfrac{c_{13}}{a_{13}-1}e_1-\tfrac{c_{23}}{\alpha+a_{13}}e_2+e_3.
\end{align*}
Applying $\Psi$  to the connection given in~~\eqref{g3,4,alpha,nabla1,sol15,delta=0,sol4} yields the following equivalent connection:
\begin{equation}\label{g3,4,alpha,nabla1,sol15,delta=0,sol4,1}
	\begin{aligned}
		\nabla_{e_1}e_3&=\lambda\,e_1, &\nabla_{e_3}e_2&=(\lambda-1)e_1, &\nabla_{e_3}e_2&=-\alpha\,e_2, &\nabla_{e_3}e_3&=\lambda\,e_3.
	\end{aligned}
\end{equation}
In this case, $a_{13}=\lambda\neq-\alpha,1$.

If $a_{13}=-\alpha$. Consider the following automorphism
\begin{align*}
	\Psi(e_1) &= e_1,&
	\Psi(e_2) &=x\, e_2,&
	\Psi(e_3) &=-\tfrac{c_{13}}{\alpha+1}e_1+e_3.
\end{align*}
For a suitable choice of the parameter $x\in\R^\ast$, applying $\Psi$  to the connection given in~~\eqref{g3,4,alpha,nabla1,sol15,delta=0,sol4} yields the following equivalent connection:
\begin{equation}\label{g3,4,alpha,nabla1,sol15,delta=0,sol4,2}
	\begin{aligned}
		\nabla_{e_1}e_3&=-\alpha\,e_1, &\nabla_{e_3}e_2&=-(\alpha+1)e_1, &\nabla_{e_3}e_2&=-\alpha\,e_2, &\nabla_{e_3}e_3&=\delta_1 e_2-\alpha\,e_3,\quad\delta_1=0,1.
	\end{aligned}
\end{equation}
Suppose now that $a_{13}=1$, and consider the following automorphism
\begin{align*}
	\Psi(e_1) &=x\, e_1,&
	\Psi(e_2) &=e_2,&
	\Psi(e_3) &=-\tfrac{c_{23}}{\alpha+1}e_2+e_3.
\end{align*}
For a suitable choice of the parameter $x\in\R^\ast$, applying $\Psi$  to the connection given in~~\eqref{g3,4,alpha,nabla1,sol15,delta=0,sol4} yields the following equivalent connection:
\begin{equation}\label{g3,4,alpha,nabla1,sol15,delta=0,sol4,3}
	\begin{aligned}
		\nabla_{e_1}e_3&=e_1, &\nabla_{e_3}e_2&=-\alpha\,e_2, &\nabla_{e_3}e_3&=\delta_2 e_1+e_3,\quad\delta_1=0,1.
	\end{aligned}
\end{equation}

	Observe  that the flat, torsion-free connections given in \eqref{g3,4,alpha,nabla1,sol15,delta=0,sol4,1} and  \eqref{g3,4,alpha,nabla1,sol15,delta=0,sol4,2} are isomorphic if and only if $\delta_1=0$ and $\lambda=-\alpha$. For this reason, we set $\lambda\in\R \setminus{\{1\}}$ and assume that $\delta_1=1$. Similarly, one can easily show that \eqref{g3,4,alpha,nabla1,sol15,delta=0,sol4,1} and  \eqref{g3,4,alpha,nabla1,sol15,delta=0,sol4,2} are isomorphic if and only if $\delta_2=0$ and $\lambda=1$.  For this reason, we set $\lambda\in\R$ and assume that $\delta_2=1$.

The connections given in \eqref{g3,4,alpha,nabla1,sol15,delta=0,sol4,1}, \eqref{g3,4,alpha,nabla1,sol15,delta=0,sol4,2}, and \eqref{g3,4,alpha,nabla1,sol15,delta=0,sol4,3} correspond to the flat, torsion-free connections associated with  the flat Lie algebras $\h_{0,16}$, $\h_{0,17}$ and $\h_{0,18}$, respectively. The additional condition $\lambda\in \R^\ast$ in the flat Lie algebra $\h_{0,16}$ is imposed to ensure that the corresponding connection is not isomorphic to the flat Lie algebra $\h_{0,10}$.

The fifth solution is given by the following flat, torsion-free connection:
\begin{equation}\label{g3,4,alpha,nabla1,sol15,delta=0,sol5}
	\begin{aligned}
		\nabla_{e_1}e_3&=(\alpha-1)e_1+a_{23}e_2, &\nabla_{e_2}e_3&=(\alpha-1)e_2, &\nabla_{e_3}e_1&=(\alpha-2)e_1+a_{23}e_2, &\nabla_{e_3}e_2&=-e_1,\\
		\nabla_{e_3}e_3&=c_{13}e_1+c_{23}e_2+(\alpha-1)e_3.
	\end{aligned}
\end{equation}

Consider the following automorphism
\begin{align*}
	\Psi(e_1) &=x\, e_1,&
	\Psi(e_2) &=e_2,&
	\Psi(e_3) &=\tfrac{x\,c_{13}}{\alpha-2}e_1+\tfrac{2\,c_{13}a_{23}+(2-\alpha)c_{23}}{\alpha-2}+e_3.
\end{align*}
For a suitable choice of the parameter $x\in\R^\ast$, applying $\Psi$  to the connection given in~~\eqref{g3,4,alpha,nabla1,sol15,delta=0,sol5} yields the following equivalent connection:
\begin{equation}\label{g3,4,alpha,nabla1,sol15,delta=0,sol5,1}
	\begin{aligned}
		\nabla_{e_1}e_3&=(\alpha-1)e_1+\delta\,e_2, &\nabla_{e_2}e_3&=(\alpha-1)e_2, &\nabla_{e_3}e_1&=(\alpha-2)e_1+\delta\,e_2, &\nabla_{e_3}e_2&=-e_1,\\
		\nabla_{e_3}e_3&=(\alpha-1)e_3.
	\end{aligned}
\end{equation}
The connection given in \eqref{g3,4,alpha,nabla1,sol15,delta=0,sol5,1} corresponds exactly to the one associated with the flat Lie algebra $\h_{0,9}$.

The sixth solution is given by the following flat, torsion-free connection:
\begin{equation}\label{g3,4,alpha,nabla1,sol15,delta=0,sol6}
	\begin{aligned}
		\nabla_{e_1}e_1&=a_{31}e_3, &\nabla_{e_3}e_1&=-e_1, &\nabla_{e_3}e_2&=-\alpha\,e_2, &\nabla_{e_3}e_3&=-2\,e_3.
	\end{aligned}
\end{equation}
Consider the following automorphism
\begin{align*}
	\Psi(e_1) &=x\, e_1,&
	\Psi(e_2) &=e_2,&
	\Psi(e_3) &=e_3.
\end{align*}
For a suitable choice of the parameter $x\in\R^\ast$, applying $\Psi$  to the connection given in~~\eqref{g3,4,alpha,nabla1,sol15,delta=0,sol6} yields the following equivalent connection:
\begin{equation}\label{g3,4,alpha,nabla1,sol15,delta=0,sol6,1}
	\begin{aligned}
		\nabla_{e_1}e_1&=\delta_\varepsilon\,e_3, &\nabla_{e_3}e_1&=-e_1, &\nabla_{e_3}e_2&=-\alpha\,e_2, &\nabla_{e_3}e_3&=-2\,e_3,\quad\delta_\varepsilon=0,\pm1.
	\end{aligned}
\end{equation}

If $\delta_\varepsilon=0$, then this connection coincides with the one associated with the flat Lie algebra$\h_{0,10}$, in which case $\lambda=-2$. Otherwise, if $\delta_\varepsilon=\pm1$, then the connection given in~\eqref{g3,4,alpha,nabla1,sol15,delta=0,sol6,1} coincides with the one associated with the flat Lie algebra $\h_{0,7}$.

The seventh solution is given by the following flat, torsion-free connection:
\begin{equation}\label{g3,4,alpha,nabla1,sol15,delta=0,sol7}
	\begin{aligned}
		\nabla_{e_1}e_3&=a_{23}e_2, &\nabla_{e_3}e_1&=-e_1+a_{23}e_2, &\nabla_{e_3}e_2&=-\alpha\,e_2, &\nabla_{e_3}e_3&=c_{13}e_2+c_{23}e_2+(1-\alpha)e_3.
	\end{aligned}
\end{equation}

Consider the following automorphism
\begin{align*}
	\Psi(e_1) &=x\, e_1,&
	\Psi(e_2) &=e_2,&
	\Psi(e_3) &=\tfrac{x\,c_{13}}{\alpha-2}e_1+\tfrac{(2-\alpha)c_{23}+2\,c_{13}a_{23}}{\alpha-2}+e_3.
\end{align*}
For a suitable choice of the parameter $x\in\R^\ast$, applying $\Psi$  to the connection given in~~\eqref{g3,4,alpha,nabla1,sol15,delta=0,sol7} yields the following equivalent connection:
\begin{equation}\label{g3,4,alpha,nabla1,sol15,delta=0,sol7,1}
	\begin{aligned}
		\nabla_{e_1}e_3&=\delta\,e_2, &\nabla_{e_3}e_1&=-e_1+\delta\,e_2, &\nabla_{e_3}e_2&=-\alpha\,e_2, &\nabla_{e_3}e_3&=(1-\alpha)e_3,\quad\delta=0,1.
	\end{aligned}
\end{equation}

The connection given in \eqref{g3,4,alpha,nabla1,sol15,delta=0,sol7,1} corresponds exactly to the one associated with the flat Lie algebra $\h_{0,8}$.

The eighth solution is given by the following flat, torsion-free connection:
\begin{equation}\label{g3,4,alpha,nabla1,sol15,delta=0,sol8}
	\begin{aligned}
		\nabla_{e_1}e_3&=-e_1, &\nabla_{e_3}e_2&=-\alpha\,e_2, &\nabla_{e_3}e_3&=c_{13}e_1+c_{23}e_2+c_{33}e_3.
	\end{aligned}
\end{equation}
Suppose that $c_{33}\neq-\alpha,-1$.  Consider the following automorphism
\begin{align*}
	\Psi(e_1) &=e_1,&
	\Psi(e_2) &=e_2,&
	\Psi(e_3) &=-\tfrac{c_{13}}{c_{33}+1}e_1-\tfrac{c_{23}}{\alpha+c_{33}}+e_3.
\end{align*}
Applying $\Psi$  to the connection given in~~\eqref{g3,4,alpha,nabla1,sol15,delta=0,sol8} yields the following equivalent connection: 
\begin{equation}\label{g3,4,alpha,nabla1,sol15,delta=0,sol8,1}
	\begin{aligned}
		\nabla_{e_1}e_3&=-e_1, &\nabla_{e_3}e_2&=-\alpha\,e_2, &\nabla_{e_3}e_3&=\lambda\,e_3.
	\end{aligned}
\end{equation}
In this case, $c_{33}=\lambda\neq-\alpha,-1$.

If $c_{33}=-\alpha$. Consider the following automorphism
\begin{align*}
	\Psi(e_1) &=e_1,&
	\Psi(e_2) &=x\,e_2,&
	\Psi(e_3) &=\tfrac{c_{13}}{\alpha-1}e_1+e_3.
\end{align*}
For a suitable choice of the parameter $x\in\R^\ast$, applying $\Psi$  to the connection given in~~\eqref{g3,4,alpha,nabla1,sol15,delta=0,sol8} yields the following equivalent connection:
\begin{equation}\label{g3,4,alpha,nabla1,sol15,delta=0,sol8,2}
	\begin{aligned}
		\nabla_{e_1}e_3&=-e_1, &\nabla_{e_3}e_2&=-\alpha\,e_2, &\nabla_{e_3}e_3&=\delta_1e_2-\alpha\,e_3,\quad\delta_1=0,1.
	\end{aligned}
\end{equation}

If $c_{33}=-1$. Consider the following automorphism
\begin{align*}
	\Psi(e_1) &=x\,e_1,&
	\Psi(e_2) &=e_2,&
	\Psi(e_3) &=\tfrac{c_{23}}{1-\alpha}e_2+e_3.
\end{align*}
For a suitable choice of the parameter $x\in\R^\ast$, applying $\Psi$  to the connection given in~~\eqref{g3,4,alpha,nabla1,sol15,delta=0,sol8} yields the following equivalent connection:
\begin{equation}\label{g3,4,alpha,nabla1,sol15,delta=0,sol8,3}
	\begin{aligned}
		\nabla_{e_1}e_3&=-e_1, &\nabla_{e_3}e_2&=-\alpha\,e_2, &\nabla_{e_3}e_3&=\delta_2e_1-e_3,\quad\delta_2=0,1.
	\end{aligned}
\end{equation}

	Observe  that the flat, torsion-free connections given in \eqref{g3,4,alpha,nabla1,sol15,delta=0,sol8,1} and  \eqref{g3,4,alpha,nabla1,sol15,delta=0,sol8,2} are isomorphic if and only if $\delta_1=0$ and $\lambda=-\alpha$. For this reason, we set $\lambda\in\R \setminus{\{-1\}}$ and assume that $\delta_1=1$. Similarly, one can easily show that \eqref{g3,4,alpha,nabla1,sol15,delta=0,sol8,1} and  \eqref{g3,4,alpha,nabla1,sol15,delta=0,sol8,3} are isomorphic if and only if $\delta_2=0$ and $\lambda=-1$.  For this reason, we set $\lambda\in\R$ and assume that $\delta_2=1$.

The connections given in \eqref{g3,4,alpha,nabla1,sol15,delta=0,sol8,1}, \eqref{g3,4,alpha,nabla1,sol15,delta=0,sol8,2}, and \eqref{g3,4,alpha,nabla1,sol15,delta=0,sol8,3} correspond to the flat, torsion-free connections associated with  the flat Lie algebras $\h_{0,10}$, $\h_{0,11}$ and $\h_{0,12}$, respectively. The additional conditions $\lambda\neq\pm(\alpha-1)$ in the flat Lie algebra $\h_{0,10}$ and $\alpha\neq\frac{1}{2}$ in the flat Lie algebra $\h_{0,11}$ are imposed to ensure that the corresponding connections are not isomorphic to the  algebras $\h_{0,3}$ and $\h_{0,8}$ for algebra $\h_{0,10}$ (in which case $\delta=0$), and to $\h_{0,8}$ for algebra $\h_{0,11}$ (in which case $\delta=0$), respectively.

The ninth solution is given by the following flat, torsion-free connection:
\begin{equation}\label{g3,4,alpha,nabla1,sol15,delta=0,sol9}
	\begin{aligned}
		\nabla_{e_2}e_3&=b_{13}e_1, &\nabla_{e_3}e_2&=b_{13}e_1-\alpha\,e_2, &\nabla_{e_3}e_3&=c_{13}e_1+c_{23}e_2+(\alpha-1)e_3.
	\end{aligned}
\end{equation}
Suppose that $b_{13}\neq0$ and $\alpha\neq\frac{1}{2}$. Consider the following automorphism
\begin{align*}
	\Psi(e_1) &=x\,e_1,&
	\Psi(e_2) &=e_2,&
	\Psi(e_3) &=\tfrac{c_{23}}{1-\alpha}e_2+e_3.
\end{align*}
Applying $\Psi$  to the connection given in~~\eqref{g3,4,alpha,nabla1,sol15,delta=0,sol9} yields the following equivalent connection:
\begin{equation}\label{g3,4,alpha,nabla1,sol15,delta=0,sol9,1}
	\begin{aligned}
		\nabla_{e_2}e_3&=e_1, &\nabla_{e_3}e_2&=e_1-\alpha\,e_2, &\nabla_{e_3}e_3&=(\alpha-1)e_3,\quad\alpha\neq\tfrac{1}{2}.
	\end{aligned}
\end{equation}
If $b_{13}\neq0$ and $\alpha=\frac{1}{2}$. Consider the following automorphism
\begin{align*}
	\Psi(e_1) &=x\,e_1,&
	\Psi(e_2) &=x\,b_{13}e_2,&
	\Psi(e_3) &=\tfrac{x\,c_{13}}{2}e_1+e_3.
\end{align*}
For a suitable choice of the parameter $x\in\R^\ast$, applying $\Psi$  to the connection given in~~\eqref{g3,4,alpha,nabla1,sol15,delta=0,sol9} yields the following equivalent connection:
\begin{equation}\label{g3,4,alpha,nabla1,sol15,delta=0,sol9,2}
	\begin{aligned}
		\nabla_{e_2}e_3&=e_1, &\nabla_{e_3}e_2&=e_1-\tfrac{1}{2}\,e_2, &\nabla_{e_3}e_3&=\delta_1 e_2-\tfrac{1}{2}\,e_3,\quad\delta_1=0,1.
	\end{aligned}
\end{equation}
If $b_{13}=0$.  Consider the following automorphism
\begin{align*}
	\Psi(e_1) &=e_1,&
	\Psi(e_2) &=x\,e_2,&
	\Psi(e_3) &=-\tfrac{c_{13}}{\alpha}e_1+e_3.
\end{align*}
For a suitable choice of the parameter $x\in\R^\ast$, applying $\Psi$  to the connection given in~~\eqref{g3,4,alpha,nabla1,sol15,delta=0,sol9} yields the following equivalent connection:
\begin{equation}\label{g3,4,alpha,nabla1,sol15,delta=0,sol9,3}
	\begin{aligned}
	\nabla_{e_3}e_1&=-e_1,&\nabla_{e_3}e_2&=-\alpha\,e_2, &\nabla_{e_3}e_3&=\delta_2 e_2-(\alpha-1)e_3,\quad\delta_2=0,1.
	\end{aligned}
\end{equation}

	Observe  that the flat, torsion-free connections given in \eqref{g3,4,alpha,nabla1,sol15,delta=0,sol9,1} and  \eqref{g3,4,alpha,nabla1,sol15,delta=0,sol9,2} are isomorphic if and only if $\delta_1=0$ and $\alpha=\frac{1}{2}$. For this reason, we set $0<|\alpha|<1$ in \eqref{g3,4,alpha,nabla1,sol15,delta=0,sol9,1}  and assume that $\delta_1=1$ in \eqref{g3,4,alpha,nabla1,sol15,delta=0,sol9,2}. The connections given in \eqref{g3,4,alpha,nabla1,sol15,delta=0,sol9,1}, \eqref{g3,4,alpha,nabla1,sol15,delta=0,sol9,2}, and \eqref{g3,4,alpha,nabla1,sol15,delta=0,sol9,3} correspond to the flat, torsion-free connections associated with  the flat Lie algebras $\h_{0,1}$, $\h_{0,2}$ and $\h_{0,3}$, respectively.

The tenth solution is given by the following flat, torsion-free connection:
\begin{equation}\label{g3,4,alpha,nabla1,sol15,delta=0,sol10}
	\begin{aligned}
		\nabla_{e_2}e_3&=c_{33}e_2, &\nabla_{e_3}e_1&=-e_1, &\nabla_{e_3}e_2&=(c_{33}-\alpha)e_2, &\nabla_{e_3}e_3&=c_{13}e_1+c_{23}e_2+c_{33}e_3.
	\end{aligned}
\end{equation}
Suppose that $c_{33}\neq \alpha,-1$,  and consider the following automorphism
\begin{align*}
	\Psi(e_1) &=e_1,&
	\Psi(e_2) &=e_2,&
	\Psi(e_3) &=-\tfrac{c_{13}}{c_{33}+1}e_1-\tfrac{c_{23}}{\alpha-c_{33}}e_2+e_3.
\end{align*}
Applying $\Psi$  to the connection given in~~\eqref{g3,4,alpha,nabla1,sol15,delta=0,sol10} yields the following equivalent connection:
\begin{equation}\label{g3,4,alpha,nabla1,sol15,delta=0,sol10,1}
	\begin{aligned}
		\nabla_{e_2}e_3&=\lambda\,e_2, &\nabla_{e_3}e_1&=-e_1, &\nabla_{e_3}e_2&=(\lambda-\alpha)e_2, &\nabla_{e_3}e_3&=\lambda\,e_3.
	\end{aligned}
\end{equation}
In this case, $c_{33}=\lambda\neq\alpha,-1$.

Suppose now that  $c_{33}=\alpha$. Consider the following automorphism
\begin{align*}
	\Psi(e_1) &=e_1,&
	\Psi(e_2) &=x\,e_2,&
	\Psi(e_3) &=-\tfrac{c_{13}}{\alpha+1}e_1+e_3.
\end{align*}
For a suitable choice of the parameter $x\in\R^\ast$, applying $\Psi$  to the connection given in~~\eqref{g3,4,alpha,nabla1,sol15,delta=0,sol10} yields the following equivalent connection:
\begin{equation}\label{g3,4,alpha,nabla1,sol15,delta=0,sol10,2}
	\begin{aligned}
		\nabla_{e_2}e_3&=\alpha\,e_2, &\nabla_{e_3}e_1&=-e_1, &\nabla_{e_3}e_3&=\delta_1e_2+\alpha\,e_3,\quad\delta_1=0,1.
	\end{aligned}
\end{equation}

If $c_{33}=-1$.  Consider the following automorphism
\begin{align*}
	\Psi(e_1) &=x\,e_1,&
	\Psi(e_2) &=e_2,&
	\Psi(e_3) &=-\tfrac{c_{23}}{\alpha+1}e_2+e_3.
\end{align*}
For a suitable choice of the parameter $x\in\R^\ast$, applying $\Psi$  to the connection given in~~\eqref{g3,4,alpha,nabla1,sol15,delta=0,sol10} yields the following equivalent connection:
\begin{equation}\label{g3,4,alpha,nabla1,sol15,delta=0,sol10,3}
	\begin{aligned}
		\nabla_{e_2}e_3&=-e_2, &\nabla_{e_3}e_1&=-e_1,
		&\nabla_{e_3}e_2&=-(\alpha+1)e_2, &\nabla_{e_3}e_3&=\delta_2e_1-e_3,\quad\delta_1=0,1.
	\end{aligned}
\end{equation}

Note  that the flat, torsion-free connections given in \eqref{g3,4,alpha,nabla1,sol15,delta=0,sol10,1} and  \eqref{g3,4,alpha,nabla1,sol15,delta=0,sol10,2} are isomorphic if and only if $\delta_1=0$ and $\lambda=\alpha$. For this reason, we set $\lambda\in\R \setminus{\{-1\}}$ and assume that $\delta_1=1$. Similarly, one can easily show that \eqref{g3,4,alpha,nabla1,sol15,delta=0,sol10,1} and  \eqref{g3,4,alpha,nabla1,sol15,delta=0,sol10,3} are isomorphic if and only if $\delta_2=0$ and $\lambda=-1$.  For this reason, we set $\lambda\in\R$ and assume that $\delta_2=1$.

The connections given in \eqref{g3,4,alpha,nabla1,sol15,delta=0,sol10,1}, \eqref{g3,4,alpha,nabla1,sol15,delta=0,sol10,2}, and \eqref{g3,4,alpha,nabla1,sol15,delta=0,sol10,3} correspond to the flat, torsion-free connections associated with  the flat Lie algebras $\h_{0,13}$, $\h_{0,14}$ and $\h_{0,15}$, respectively.   The additional condition $\lambda\neq0$ in the flat Lie algebra $\h_{0,13}$ is imposed to ensure that the corresponding connection is not isomorphic to the flat Lie algebra $\h_{0,10}$ (in which case $\lambda=0$).

The eleventh solution is given by the following flat, torsion-free connection:
\begin{equation}\label{g3,4,alpha,nabla1,sol15,delta=0,sol11}
	\begin{aligned}
		\nabla_{e_2}e_2&=b_{32}e_3, &\nabla_{e_3}e_1&=-e_1, &\nabla_{e_3}e_2&=-\alpha\,e_2, &\nabla_{e_3}e_3&=-2\,\alpha\,e_3.
	\end{aligned}
\end{equation}
Consider the following automorphism
\begin{align*}
	\Psi(e_1) &=e_1,&
	\Psi(e_2) &=x\,e_2,&
	\Psi(e_3) &=e_3.
\end{align*}
For a suitable choice of the parameter $x\in\R^\ast$, applying $\Psi$  to the connection given in~~\eqref{g3,4,alpha,nabla1,sol15,delta=0,sol11} yields the following equivalent connection:
\begin{equation}\label{g3,4,alpha,nabla1,sol15,delta=0,sol11}
	\begin{aligned}
		\nabla_{e_2}e_2&=\delta_\varepsilon\, e_3, &\nabla_{e_3}e_1&=-e_1, &\nabla_{e_3}e_2&=-\alpha\,e_2, &\nabla_{e_3}e_3&=-2\,\alpha\,e_3,\quad\delta_\varepsilon=0,\pm1.
	\end{aligned}
\end{equation}

Observe that if $\delta_\varepsilon=0$, then this connection coincides with the one associated with the flat Lie algebra $\h_{0,10}$, in which case $\lambda=-2\alpha$. Otherwise, if $\delta_\varepsilon=\pm1$, then this connection coincides exactly with the one associated with the affine algebra $\h_{0,23}$.

We now complete the classification of flat, torsion-free connections on the flat Lie algebra $\G_{3,4}^{\alpha\neq-1}$. Under the corresponding assumptions on their parameters, all flat, torsion-free connections presented in Table~\ref{g3,4,alpha} are pairwise non-isomorphic.

Let $\alpha=-1$. In the basis $\lbrace e_1, e_2, e_3 \rbrace$, the operators $\nabla_{e_1}$, $\nabla_{e_2}$ and $\nabla_{e_3}$ are given respectively by: 
\begin{equation}
	\nabla_{e_1}=\left( \begin {array}{ccc} a_{11}&a_{12}&a_{13}\\ \noalign{\medskip}
	a_{21}&a_{22}&a_{23}\\ \noalign{\medskip}a_{31}&a_{32}&a_{33}\end {array} \right),\quad
	\nabla_{e_2}=\left( \begin {array}{ccc} a_{12}&b_{12}&b_{13}\\ \noalign{\medskip}
	a_{22}&b_{22}&b_{23}\\ \noalign{\medskip}a_{32}&b_{32}&b_{33}\end {array} \right),\quad
	\nabla_{e_3}=\left( \begin {array}{ccc} a_{13}-1 &b_{13}&c_{13}\\ \noalign{\medskip}
	a_{23}&b_{23}+1&c_{23}\\ \noalign{\medskip}a_{33}&b_{33}&c_{33}\end {array} \right),\quad
\end{equation}
where $a_{ij}$, $b_{ij}$, $c_{ij}\in \mathbb{R}$.

Assume that $\nabla^0$ is a  torsion-free connection. Then $\nabla^0$ is equivalent to one of the connections listed in Lemma~\ref{Lemg3j} under the Lie algebra $\G_{3,4}$.  We first consider the case $\nabla^0=\nabla^1$. A straightforward calculation produces a unique real solution for the flatness-equations. Following is a flat, torsion-free connection that provides this solution:
\begin{equation}\label{g3,4,alpha=-1,nabla1,sol1}
	\begin{aligned}
		\nabla_{e_1}e_2&=e_1+e_2+a_{32}e_3, &\nabla_{e_1}e_3&=-\tfrac{1}{a_{32}}e_1-\tfrac{1}{a_{32}}e_2-e_3, &\nabla_{e_2}e_1&=e_1+e_2+a_{32}e_3,\\
		\nabla_{e_2}e_3&=-\tfrac{1}{a_{32}}e_1-\tfrac{1}{a_{32}}e_2-e_3, &\nabla_{e_3}e_1&=-\tfrac{a_{32}+1}{a_{32}}e_1-\tfrac{1}{a_{32}}e_2-e_3, &\nabla_{e_3}e_2&=-\tfrac{1}{a_{32}}e_1+\tfrac{a_{32}-1}{a_{32}}e_2-e_3,\\
		\nabla_{e_3}e_3&=\tfrac{a_{32}+2}{a_{32}^2}e_1+\tfrac{2-a_{32}}{a_{32}^2}e_3+\tfrac{2}{a_{32}}e_3.
	\end{aligned}
\end{equation}
Consider the following automorphism
\begin{align*}
	\Psi(e_1) &=a_{32}\,e_1,&
	\Psi(e_2) &=e_2,&
	\Psi(e_3) &=-e_1+\tfrac{1}{a_{32}}e_2+e_3.
\end{align*}
Applying $\Psi$  to the connection given in~~\eqref{g3,4,alpha=-1,nabla1,sol1} yields the following equivalent connection:
\begin{equation}\label{g3,4,alpha=-1,nabla1,sol1,1}
	\begin{aligned}
		\nabla_{e_1}e_2&=e_3, &\nabla_{e_2}e_1&=e_3, &\nabla_{e_3}e_1&=-e_1, &\nabla_{e_3}e_2&=e_2.
	\end{aligned}
\end{equation}
This flat, torsion-free connection coincides exactly with the one associated with the flat  Lie algebra $\h_{0,11}$.

If $\nabla^0=\nabla^2$,  $\nabla^0=\nabla^3$, or $\nabla^0=\nabla^4$, then a straightforward computation shows that the flatness equations admit no solutions for any value of  the parameters defining the torsion-free connections on $\nabla^2$, $\nabla^3$ and $\nabla^4$.

If $\nabla^0=\nabla^5$. A straightforward calculation produces a unique real solution for the flatness-equations. Following is a flat, torsion-free connection that provides this solution:
\begin{equation}\label{g3,4,alpha=-1,nabla5,sol1}
	\begin{aligned}
		\nabla_{e_1}e_2&=e_1+\tfrac{1}{c_{13}}e_3, &\nabla_{e_2}e_1&=e_1+\tfrac{1}{c_{13}}e_3, &\nabla_{e_2}e_3&=-c_{13}e_1-e_3, &\nabla_{e_3}e_1&=-e_1, \\\nabla_{e_3}e_2&=-c_{13}e_1+e_2-e_3, &\nabla_{e_3}e_3&=c_{13}e_1.
	\end{aligned}
\end{equation}
Consider the following automorphism
\begin{align*}
	\Psi(e_1) &=e_1,&
	\Psi(e_2) &=\tfrac{1}{c_{13}} e_2,&
	\Psi(e_3) &=-c_{13}e_1+e_3.
\end{align*}
Applying $\Psi$  to the connection given in~~\eqref{g3,4,alpha=-1,nabla5,sol1} yields the following equivalent connection:
\begin{equation}\label{g3,4,alpha=-1,nabla1,sol5,1}
	\begin{aligned}
		\nabla_{e_1}e_2&=e_3, &\nabla_{e_2}e_1&=e_3, &\nabla_{e_3}e_1&=-e_1, &\nabla_{e_3}e_2&=e_2.
	\end{aligned}
\end{equation}
This flat, torsion-free connection coincides exactly with the one associated with the flat  Lie algebra $\h_{0,11}$.

If $\nabla^0=\nabla^6$,  or  $\nabla^0=\nabla^7$, then a straightforward computation shows that the flatness equations admit no solutions for any value of  the parameters defining the torsion-free connections on  $\nabla^6$ and $\nabla^7$.

If $\nabla^0=\nabla^8$. A straightforward calculation produces a unique real solution for the flatness-equations. Following is a flat, torsion-free connection that provides this solution:
\begin{equation}\label{g3,4,alpha=-1,nabla8,sol1}
	\begin{aligned}
		\nabla_{e_1}e_2&=e_2+a_{32}e_3, &\nabla_{e_1}e_3&=-\tfrac{1}{a_{32}}e_2-e_3, &\nabla_{e_2}e_1&=e_2+a_{32}e_3, &\nabla_{e_3}e_1&=-e_1-\tfrac{1}{a_{32}}e_2-e_3,\\
		\nabla_{e_3}e_2&=e_2, &\nabla_{e_3}e_3&=-\tfrac{1}{a_{32}}e_2.
	\end{aligned}
\end{equation}
Consider the following automorphism
\begin{align*}
	\Psi(e_1) &=a_{32}e_1,&
	\Psi(e_2) &= e_2,&
	\Psi(e_3) &=-\tfrac{1}{a_{32}}e_2+e_3.
\end{align*}
Applying $\Psi$ to the connection given in~\eqref{g3,4,alpha=-1,nabla8,sol1} yields exactly the connection associated with the flat Lie algebra $\h_{0,11}$.

If $\nabla^0=\nabla^9$. A straightforward calculation produces a unique real solution for the flatness-equations. Following is a flat, torsion-free connection that provides this solution:
\begin{equation}\label{g3,4,alpha=-1,nabla9,sol1}
	\begin{aligned}
		\nabla_{e_1}e_1&=e_1+e_2+a_{31}e_3, &\nabla_{e_1}e_3&=-\tfrac{1}{a_{31}}e_1-\tfrac{1}{a_{31}}e_2-e_3, &\nabla_{e_3}e_1&=-\tfrac{1+a_{31}}{a_{31}}e_1-\tfrac{1}{a_{31}}e_2-e_3,\\
		\nabla_{e_3}e_2&=e_2, &\nabla_{e_3}e_3&=\tfrac{1-a_{31}}{a_{31}^2}e_1+\tfrac{1-3\,a_{31}}{a_{31}^2}e_2+\tfrac{1-2\,a_{31}}{a_{31}} e_3.
	\end{aligned}
\end{equation}
Consider the following automorphism
\begin{align*}
	\Psi(e_1) &=\tfrac{\sqrt{\varepsilon\,a_{31}}}{\varepsilon} e_1,&
	\Psi(e_2) &= e_2,&
	\Psi(e_3) &=-\tfrac{\sqrt{\varepsilon\,a_{31}}}{\varepsilon\,a_{31}}e_1-\tfrac{1}{a_{31}}e_2+e_3,\quad\varepsilon=\pm1.
\end{align*}
Applying $\Psi$ to the connection given in~\eqref{g3,4,alpha=-1,nabla9,sol1} yields exactly the connection associated with the flat Lie algebra $\h_{0,3}$.
\begin{equation}\label{g3,4,alpha=-1,nabla9,sol1,1}
	\begin{aligned}
		\nabla_{e_1}e_1&=\varepsilon\, e_3,  &\nabla_{e_3}e_1&=-e_1,
		&\nabla_{e_3}e_2&=e_2, &\nabla_{e_3}e_3&=-2\,e_3.
	\end{aligned}
\end{equation}
This flat, torsion-free connection coincides exactly with the one associated with the flat  Lie algebra $\h_{0,3}$.

If $\nabla^0=\nabla^{10}$,  or  $\nabla^0=\nabla^{11}$, then a straightforward computation shows that the flatness equations admit no solutions for any value of  the parameters defining the torsion-free connections on  $\nabla^{10}$ and $\nabla^{11}$.

If $\nabla^0=\nabla^{12}$.  A straightforward calculation shows that the flatness-equations admit a real solution if and only if $\delta_\varepsilon=0$. Moreover, the flatness-equations admit three solutions; the first is given by the following flat, torsion-free connection:
\begin{equation}\label{g3,4,alpha=-1,nabla12,sol1}
	\begin{aligned}
		\nabla_{e_1}e_1&=e_2, \quad\nabla_{e_1}e_3=3\,e_1+a_{23}e_2, &\nabla_{e_3}e_1&=2\,e_1+a_{23}e_2, &\nabla_{e_3}e_2&=e_2, \\\nabla_{e_3}e_3&=2\,a_{23}e_1+c_{23}e_2+3\,e_3.
	\end{aligned}
\end{equation}
Consider the following automorphism
\begin{align*}
	\Psi(e_1) &=e_2,&
	\Psi(e_2) &=e_1,&
	\Psi(e_3) &=\tfrac{a_{23}^2-c_{23}}{2} e_1+a_{23}e_2+e_3.
\end{align*}
Applying $\Psi$  to the connection given in~~\eqref{g3,4,alpha=-1,nabla12,sol1} yields the following equivalent connection:
\begin{equation}\label{g3,4,alpha=-1,nabla12,sol1,1}
	\begin{aligned}
		\nabla_{e_2}e_2&=e_1, &\nabla_{e_2}e_3&=-3\,e_2, &\nabla_{e_3}e_1&=-e_1, &\nabla_{e_3}e_2&=-2\,e_2, &\nabla_{e_3}e_3&=-3\,e_3.
	\end{aligned}
\end{equation}
This flat, torsion-free connection coincides exactly with the one associated with the flat  Lie algebra $\h_{2,2}$.

The second is given by the following flat, torsion-free connection:
\begin{equation}\label{g3,4,alpha=-1,nabla12,sol2}
	\begin{aligned}
		\nabla_{e_1}e_1&=e_2, \quad\nabla_{e_1}e_3=a_{23}e_2, &\nabla_{e_2}e_3&=-3\,e_2, &\nabla_{e_3}e_1&=-e_1+a_{23}e_2, &\nabla_{e_3}e_2&=-2\,e_2,\\
		\nabla_{e_3}e_3&=2\,a_{23}e_1+c_{23}e_2-3\,e_3.
	\end{aligned}
\end{equation}
Consider the following automorphism
\begin{align*}
	\Psi(e_1) &=e_2,&
	\Psi(e_2) &=e_1,&
	\Psi(e_3) &=\tfrac{a_{23}^2-c_{23}}{2} e_1+a_{23}e_2-e_3.
\end{align*}
Applying $\Psi$  to the connection given in~~\eqref{g3,4,alpha=-1,nabla12,sol2} yields the following equivalent connection:
\begin{equation}\label{g3,4,alpha=-1,nabla12,sol2,2}
	\begin{aligned}
		\nabla_{e_1}e_3&=3\,e_1, &\nabla_{e_2}e_2&=e_1, &\nabla_{e_3}e_1&=2\,e_1, &\nabla_{e_3}e_2&=e_2, &\nabla_{e_3}e_3&=3\,e_3.
	\end{aligned}
\end{equation}
This flat, torsion-free connection coincides exactly with the one associated with the flat  Lie algebra $\h_{2,1}$.

	The third solution is given by the following flat, torsion-free connection:
	\begin{equation}\label{g3,4,alpha=-1,nabla12,sol3}
		\begin{aligned}
			\nabla_{e_1}e_1&=e_2-\tfrac{3}{c_{23}}e_3, &\nabla_{e_3}e_1&=-e_1, &\nabla_{e_3}e_2&=e_2, &\nabla_{e_3}e_3&=c_{23}e_2-2\,e_3.
		\end{aligned}
	\end{equation}
	Consider the following automorphism
	\begin{align*}
		\Psi(e_1) &=\tfrac{\sqrt{3}}{\sqrt{-\varepsilon\,c_{23}}} e_1,&
		\Psi(e_2) &= e_2,&
		\Psi(e_3) &=\tfrac{c_{23}}{3}e_2+e_3,\quad\varepsilon=\pm1.
	\end{align*}
	Applying $\Psi$ to the connection given in~\eqref{g3,4,alpha=-1,nabla12,sol3} yields exactly the connection associated with the flat Lie algebra $\h_{0,3}$.

	If $\nabla^0=\nabla^{13}$. A straightforward calculation produces a unique real solution for the flatness-equations. Following is a flat, torsion-free connection that provides this solution:
	\begin{equation}\label{g3,4,alpha=-1,nabla13,sol1}
		\begin{aligned}
			\nabla_{e_2}e_2&=e_1+e_2-\tfrac{1}{b_{23}}e_3, \quad\nabla_{e_2}e_3=b_{23}e_1+b_{23}e_2-e_3, \quad\nabla_{e_3}e_1=-e_1, \\\nabla_{e_3}e_2&=b_{23}e_1+(b_{23}+1)e_2-e_3,\quad\nabla_{e_3}e_3=b_{23}(b_{23}-3)e_1+(b_{23}^2-b_{23})e_2+(2-b_{23})e_3.
		\end{aligned}
	\end{equation}
	Consider the following automorphism
	\begin{align*}
		\Psi(e_1) &=e_2,&
		\Psi(e_2) &=x\,e_1,&
		\Psi(e_3) &=x\,b_{23}e_1+b_{23}e_2-e_3.
	\end{align*}
	For a suitable choice of the parameter $x\in\R^\ast$, applying $\Psi$  to the connection given in~~\eqref{g3,4,alpha=-1,nabla13,sol1} yields the following equivalent connection:
		\begin{equation}\label{g3,4,alpha=-1,nabla13,sol1,1}
		\begin{aligned}
			\nabla_{e_2}e_2&=\delta_\varepsilon\,e_3,& &\nabla_{e_3}e_1=-e_1, &\nabla_{e_3}e_2&=e_2,&\nabla_{e_3}e_3&=-2\,e_3,\quad\delta_\varepsilon=0,\pm1.
		\end{aligned}
	\end{equation}
	Observe that if $\delta_\varepsilon=0$, then this connection coincides with the one associated with the flat Lie algebra $\h_{0,1}$, in which case $\delta=0$. Otherwise, if $\delta_\varepsilon=\varepsilon=\pm1$, then this connection coincides with the one associated with the flat Lie algebra  $\h_{0,3}$.

	If $\nabla^0=\nabla^{14}$. Then, 
	\[\nabla^0_{e_1}=\begin{pmatrix}
		0&0\\
		0&0
	\end{pmatrix},\quad\nabla^0_{e_2}=\begin{pmatrix}
	0&1\\
	0&0
	\end{pmatrix}.\]
	Applying the automorphism \[\Phi:=\begin{pmatrix}
		0&y\\
		y^2&0
	\end{pmatrix}\] to the previous torsion-free connection yields the following equivalent connection:
		\[\nabla^0_{e_1}=\begin{pmatrix}
		0&0\\
		1&0
	\end{pmatrix},\quad\nabla^0_{e_2}=\begin{pmatrix}
		0&0\\
		0&0
	\end{pmatrix}.\]
		This connection coincides with the torsion-free connection $\nabla^{12}$, in which case $\delta_\varepsilon=0$. It is therefore unnecessary to analyze this case further as it has already been treated.

		If $\nabla^0=\nabla^{15}$ and $\delta=1$. A straightforward calculation produces a unique real solution for the flatness-equations. Following is a flat, torsion-free connection that provides this solution:
		\begin{equation}\label{g3,4,alpha=-1,nabla15,sol1}
		\begin{aligned}
			\nabla_{e_2}e_2&=e_2+b_{32}e_3, &\nabla_{e_2}e_3&=-\tfrac{1}{b_{32}}e_2-e_3, &\nabla{e_3}e_1&=-e_1, &\nabla_{e_3}e_2&=\tfrac{b_{32}-1}{b_{32}}e_2-e_3,\\
			\nabla_{e_3}e_3&=\tfrac{b_{32}+1}{b_{32}^2}e_2+\tfrac{1+2\,b_{32}}{b_{32}}e_3.
					\end{aligned}
	\end{equation}

	Consider the following automorphism
	\begin{align*}
		\Psi(e_1) &=e_2,&
		\Psi(e_2) &=-\tfrac{b_{32}}{\sqrt{-\varepsilon\,b_{32}}} e_1,&
		\Psi(e_3) &=\tfrac{1}{\sqrt{-\varepsilon\,b_{32}}} e_1-e_3,\quad\varepsilon=\pm1.
	\end{align*}
	Applying $\Psi$ to the connection given in~\eqref{g3,4,alpha=-1,nabla12,sol3} yields exactly the connection associated with the flat Lie algebra $\h_{0,3}$.

		If $\nabla^0=\nabla^{15}$ and $\delta=0$, then a straightforward calculation shows that the flatness equations admit eleven real solutions. The first one is given by the following flat, torsion-free connection:
	\begin{equation}\label{g3,4,alpha=-1,nabla15,delta=0,sol1}
		\begin{aligned}
			\nabla_{e_1}e_2&=a_{32}e_3, &\nabla_{e_2}e_1&=a_{32}e_3, &\nabla_{e_3}e_1&=-e_1, &\nabla_{e_3}e_2&=e_2.
		\end{aligned}
	\end{equation}
		Consider the following automorphism
		\begin{align*}
			\Psi(e_1) &=x\,e_2,&
			\Psi(e_2) &=-e_1,&
			\Psi(e_3) &=-e_3.
		\end{align*}
		For a suitable choice of the parameter $x\in\R^\ast$, applying $\Psi$  to the connection given in~~\eqref{g3,4,alpha=-1,nabla15,delta=0,sol1} yields the following equivalent connection:
		\begin{equation}\label{g3,4,alpha=-1,nabla15,delta=0,sol1,1}
			\begin{aligned}
				\nabla_{e_1}e_2&=\delta\,e_3, &\nabla_{e_2}e_1&=\delta\,e_3, &\nabla_{e_3}e_1&=-e_1, &\nabla_{e_3}e_2&=e_2,\quad\delta=0,1.
			\end{aligned}
		\end{equation}		
		If $\delta=1$, then this connection coincides with the one associated with the  flat Lie algebra $\h_{0,11}$. Otherwise, if $\delta=0$, one can easily show that this connection coincides with the one associated with the flat Lie algebra $\h_{0,4}$, in which case $\lambda=0$.

The second solution is given by the following flat, torsion-free connection:
		\begin{equation}\label{g3,4,alpha=-1,nabla15,delta=0,sol2}
			\begin{aligned}
				\nabla_{e_1}e_3&=a_{13}e_1, &\nabla_{e_3}e_1&=(a_{13}-1)e_1, &\nabla_{e_3}e_2&=e_2, &\nabla_{e_3}e_3&=c_{13}e_1+c_{23}e_2+a_{13}e_3.
			\end{aligned}
		\end{equation}
Suppose that $a_{13}\neq1$, and consider the following automorphism		
		\begin{align*}
			\Psi(e_1) &=e_2,&
			\Psi(e_2) &=e_1,&
			\Psi(e_3) &=\tfrac{c_{23}}{1-a_{13}}e_1+\tfrac{c_{13}}{a_{13}-1}e_1-e_3.
		\end{align*}
Applying $\Psi$  to the connection given in~~\eqref{g3,4,alpha=-1,nabla15,delta=0,sol2} yields the following equivalent connection:		
		\begin{equation}\label{g3,4,alpha=-1,nabla15,delta=0,sol2,1}
			\begin{aligned}
				\nabla_{e_2}e_3&=\lambda\,e_2, &\nabla_{e_3}e_1&=-e_1, &\nabla_{e_3}e_2&=(\lambda+1)e_2, &\nabla_{e_3}e_3&=\lambda\,e_3.
			\end{aligned}
		\end{equation}
In this case, $\lambda=-a_{13}\neq-1$. Suppose now that $a_{13}=1$, and consider the following automorphism		
		\begin{align*}
			\Psi(e_1) &=z\,e_2,&
			\Psi(e_2) &=y\,e_1,&
			\Psi(e_3) &=-e_3.
		\end{align*}
		
For a suitable choice of the parameters $y,z\in\R^\ast$, applying $\Psi$  to the connection given in~~\eqref{g3,4,alpha=-1,nabla15,delta=0,sol2} yields the following equivalent connection:		
		\begin{equation}\label{g3,4,alpha=-1,nabla15,delta=0,sol2,2}
			\begin{aligned}
				\nabla_{e_2}e_3&=-e_2, &\nabla_{e_3}e_1&=-e_1,  &\nabla_{e_3}e_3&=\delta_1\,e_1+\delta_2\,e_2-e_3,\quad\delta_1,\delta_2=0,1.
			\end{aligned}
		\end{equation}
		Note  that the flat, torsion-free connections given in \eqref{g3,4,alpha=-1,nabla15,delta=0,sol2,1} and  \eqref{g3,4,alpha=-1,nabla15,delta=0,sol2,2} are isomorphic if and only if $\delta_1=\delta_2=0$ and $\lambda=-1$. For this reason, we set $\lambda\in\R$ and assume that $\delta_1^2+\delta_2^2\neq0$. The connections given in \eqref{g3,4,alpha=-1,nabla15,delta=0,sol2,1},  and \eqref{g3,4,alpha=-1,nabla15,delta=0,sol2,2} correspond to the flat, torsion-free connections associated with  the flat Lie algebras $\h_{0,7}$, $\h_{0,8}$, respectively.   The additional condition $\lambda\neq0$ in the flat Lie algebra $\h_{0,7}$ is imposed to ensure that the corresponding connection is not isomorphic to the flat Lie algebra $\h_{0,4}$ (in which case $\lambda=0$).

		The third solution is given by the following flat, torsion-free connection:
		\begin{equation}\label{g3,4,alpha=-1,nabla15,delta=0,sol3}
			\begin{aligned}
				\nabla_{e_1}e_3&=b_{23}e_1, \quad\nabla_{e_2}e_3=b_{23}e_2, &\nabla_{e_3}e_1&=(b_{23}-1)e_1, &\nabla_{e_3}e_2&=(b_{23}+1)e_2, \\\nabla_{e_3}e_3&=c_{13}e_1+c_{23}e_2+b_{23}e_3.
			\end{aligned}
		\end{equation}
Suppose that $b_{23}\neq\pm1$, and consider the following automorphism		
\begin{align*}
	\Psi(e_1) &=e_2,&
	\Psi(e_2) &=e_1,&
	\Psi(e_3) &=\tfrac{c_{23}}{1+b_{23}}e_1+\tfrac{c_{13}}{b_{23}-1}e_1-e_3.
\end{align*}		
		Applying $\Psi$  to the connection given in~~\eqref{g3,4,alpha=-1,nabla15,delta=0,sol3} yields the following equivalent connection:		
	\begin{equation}\label{g3,4,alpha=-1,nabla15,delta=0,sol3,1}
	\begin{aligned}
		\nabla_{e_1}e_3&=\lambda\,e_1, \quad\nabla_{e_2}e_3=\lambda\,e_2, &\nabla_{e_3}e_1&=(\lambda-1)e_1, &\nabla_{e_3}e_2&=(\lambda+1)e_2, &\nabla_{e_3}e_3&=\lambda\,e_3.
	\end{aligned}
\end{equation}		
In this case, $\lambda=-b_{23}\neq\pm1$.

If $b_{23}=1$. Consider the following automorphism		
	\begin{align*}
		\Psi(e_1) &=x\,e_2,&
		\Psi(e_2) &=e_1,&
		\Psi(e_3) &=\tfrac{c_{23}}{2}e_1-e_3.
	\end{align*}	
For a suitable choice of the parameter $x\in\R^\ast$, applying $\Psi$  to the connection given in~~\eqref{g3,4,alpha=-1,nabla15,delta=0,sol3} yields the following equivalent connection:			
		\begin{equation}\label{g3,4,alpha=-1,nabla15,delta=0,sol3,2}
			\begin{aligned}
				\nabla_{e_1}e_3&=-e_1, \quad\nabla_{e_2}e_3=-e_2, &\nabla_{e_3}e_1&=-2\,e_1, &\nabla_{e_3}e_3&=\delta_1e_2-e_3,\quad\delta_1=0,1.
			\end{aligned}
		\end{equation}

If $b_{23}=-1$. Consider the following automorphism		
\begin{align*}
	\Psi(e_1) &=e_2,&
	\Psi(e_2) &=x\,e_1,&
	\Psi(e_3) &=-\tfrac{c_{13}}{2}e_2-e_3.
\end{align*}	
For a suitable choice of the parameter $x\in\R^\ast$, applying $\Psi$  to the connection given in~~\eqref{g3,4,alpha=-1,nabla15,delta=0,sol3} yields the following equivalent connection:			
\begin{equation}\label{g3,4,alpha=-1,nabla15,delta=0,sol3,3}
	\begin{aligned}
		\nabla_{e_1}e_3&=e_1, \quad\nabla_{e_2}e_3=e_2, &\nabla_{e_3}e_2&=2\,e_2, &\nabla_{e_3}e_3&=\delta_2e_1+e_3,\quad\delta_2=0,1.
	\end{aligned}
\end{equation}

Note  that the flat, torsion-free connections given in \eqref{g3,4,alpha=-1,nabla15,delta=0,sol3,1} and  \eqref{g3,4,alpha=-1,nabla15,delta=0,sol3,2} are isomorphic if and only if $\delta_1=0$ and $\lambda=-1$. For this reason, we set $\lambda\in\R \setminus{\{1\}}$ and assume that $\delta_1=1$. Similarly, one can easily show that \eqref{g3,4,alpha=-1,nabla15,delta=0,sol3,1} and  \eqref{g3,4,alpha=-1,nabla15,delta=0,sol3,3} are isomorphic if and only if $\delta_2=0$ and $\lambda=1$.  For this reason, we set $\lambda\in\R$ and assume that $\delta_2=1$. If $\delta_1=\delta_2=1$, 
one can easily show that the connections defined in \eqref{g3,4,alpha=-1,nabla15,delta=0,sol3,2} and \eqref{g3,4,alpha=-1,nabla15,delta=0,sol3,3} are isomorphic via the following automorphism:
\begin{align*}
	\Psi(e_1) &=e_2,&
	\Psi(e_2) &=e_1,&
	\Psi(e_3) &=-e_3.
\end{align*}	

The connection given in \eqref{g3,4,alpha=-1,nabla15,delta=0,sol3,2} is therefore eliminated from our classification and only the one defined in \eqref{g3,4,alpha=-1,nabla15,delta=0,sol3,3} kept under the assumption that $\delta_2=1$. The connections given in \eqref{g3,4,alpha=-1,nabla15,delta=0,sol3,1} and \eqref{g3,4,alpha=-1,nabla15,delta=0,sol3,3} correspond to the flat, torsion-free connections associated with  the flat Lie algebras $\h_{0,9}$ and $\h_{0,10}$, respectively.   The additional condition $\lambda\neq\pm2$ in the flat Lie algebra $\h_{0,9}$ is imposed to ensure that the corresponding connection is not isomorphic to the flat Lie algebra $\h_{0,2}$ (in which case $\delta=0$).

The fourth solution is given by the following flat, torsion-free connection:		

\begin{equation}\label{g3,4,alpha=-1,nabla15,delta=0,sol4}
	\begin{aligned}
		\nabla_{e_1}e_3&=2\,e_1, \quad\nabla_{e_2}e_3=b_{13}e_1+2\,e_2, &\nabla_{e_3}e_1&=e_1, &\nabla_{e_3}e_2&=b_{13}e_1+3\,e_3, \\\nabla_{e_3}e_3&=c_{13}e_1+c_{23}e_2+2\,e_3.
	\end{aligned}
\end{equation}	
	Consider the following automorphism
\begin{align*}
	\Psi(e_1) &=e_1,&
	\Psi(e_2) &=x\,e_2,&
	\Psi(e_3) &=(c_{13}-\tfrac{2\,c_{23}b_{13}}{3})e_1+\tfrac{x\,c_{23}}{3}e_2+e_3.
\end{align*}
For a suitable choice of the parameter $x\in\R^\ast$, applying $\Psi$  to the connection given in~~\eqref{g3,4,alpha=-1,nabla15,delta=0,sol4} yields the following equivalent connection:
\begin{equation}\label{g3,4,alpha=-1,nabla15,delta=0,sol4,1}
	\begin{aligned}
		\nabla_{e_1}e_3&=2\,e_1, \quad\nabla_{e_2}e_3=\delta\,e_1+2\,e_2, &\nabla_{e_3}e_1&=e_1, &\nabla_{e_3}e_2&=\delta\,e_1+3\,e_2, &\nabla_{e_3}e_3&=2\,e_3,\quad\delta=0,1.
	\end{aligned}
\end{equation}	
The flat, torsion-free connection defined in~\eqref{g3,4,alpha=-1,nabla15,delta=0,sol4,1} coincides exactly with the one associated with the flat Lie algebra $\h_{0,2}$.

The fifth case is described by the following flat, torsion-free connection:
\begin{equation}\label{g3,4,alpha=-1,nabla15,delta=0,sol5}
	\begin{aligned}
		\nabla_{e_1}e_3&=-2\,e_1+a_{23}e_2, \quad\nabla_{e_2}e_3=-2\,e_2, &\nabla_{e_3}e_1&=-3\,e_1+a_{23}e_2, &\nabla_{e_3}e_2&=-3e_3, \\\nabla_{e_3}e_3&=c_{13}e_1+c_{23}e_2-2\,e_3.
	\end{aligned}
\end{equation}	
Consider the following automorphism
\begin{align*}
	\Psi(e_1) &=e_2,&
	\Psi(e_2) &=x\,e_1,&
	\Psi(e_3) &=-\tfrac{(2\,c_{13}a_{23}+3\,c_{23})\,x}{3}e_1-\tfrac{c_{13}}{3}e_2-e_3.
\end{align*}
For a suitable choice of the parameter $x\in\R^\ast$, applying $\Psi$  to the connection given in~~\eqref{g3,4,alpha=-1,nabla15,delta=0,sol4} yields the following equivalent connection:
\begin{equation}\label{g3,4,alpha=-1,nabla15,delta=0,sol5,1}
	\begin{aligned}
		\nabla_{e_1}e_3&=2\,e_1, \quad\nabla_{e_2}e_3=\delta_1\,e_1+2\,e_2, &\nabla_{e_3}e_1&=e_1, &\nabla_{e_3}e_2&=\delta_1\,e_1+3\,e_2, &\nabla_{e_3}e_3&=2\,e_3,\quad\delta_1=0,1.
	\end{aligned}
\end{equation}	
This flat, torsion-free connection coincides with the one given in~\eqref{g3,4,alpha=-1,nabla15,delta=0,sol4,1}, in which case $\delta_1=\delta_2$.

The sixth solution corresponds to the flat, torsion-free connection:
\begin{equation}\label{g3,4,alpha=-1,nabla15,delta=0,sol6}
	\begin{aligned}
		\nabla_{e_1}e_1&=a_{31}e_3,  &\nabla_{e_3}e_1&=-e_1, &\nabla_{e_3}e_2&=e_2, &\nabla_{e_3}e_3&=-2\,e_3.
	\end{aligned}
\end{equation}	
Consider the following automorphism
\begin{align*}
	\Psi(e_1) &=x\,e_1,&
	\Psi(e_2) &=e_2,&
	\Psi(e_3) &=e_3.
\end{align*}
For a suitable choice of the parameter $x\in\R^\ast$, applying $\Psi$  to the connection given in~~\eqref{g3,4,alpha=-1,nabla15,delta=0,sol6} yields the following equivalent connection:
\begin{equation}\label{g3,4,alpha=-1,nabla15,delta=0,sol6,1}
	\begin{aligned}
		\nabla_{e_1}e_1&=\delta_\varepsilon\, e_3,  &\nabla_{e_3}e_1&=-e_1, &\nabla_{e_3}e_2&=e_2, &\nabla_{e_3}e_3&=-2\,e_3,\quad\delta_\varepsilon=0,\pm1.
	\end{aligned}
\end{equation}	
	If $\delta_\varepsilon=\pm1$, then this connection coincides with the one associated with the  flat Lie algebra $\h_{0,3}$. Otherwise, if $\delta_\varepsilon=0$, one can easily show that this connection coincides with the one associated with the flat Lie algebra $\h_{0,1}$, in which case $\lambda=0$.

The seventh case corresponds to the following flat, torsion-free connection:
\begin{equation}\label{g3,4,alpha=-1,nabla15,delta=0,sol7}
	\begin{aligned}
		\nabla_{e_1}e_3&=a_{23}e_2, &\nabla_{e_3}e_1&=-e_1+a_{23}e_2, &\nabla_{e_3}e_2&=e_2, &\nabla_{e_3}e_3&=c_{13}e_1+c_{23}e_2+a_{23}e_3.
	\end{aligned}
\end{equation}	

Consider the following automorphism
\begin{align*}
	\Psi(e_1) &=e_2,&
	\Psi(e_2) &=x\,e_1,&
	\Psi(e_3) &=-\tfrac{(2\,c_{13}a_{23}+3\,c_{23})\,x}{3}e_1-\tfrac{c_{13}}{3}e_2-e_3.
\end{align*}
For a suitable choice of the parameter $x\in\R^\ast$, applying $\Psi$  to the connection given in~~\eqref{g3,4,alpha=-1,nabla15,delta=0,sol7} yields the following equivalent connection:
\begin{equation}\label{g3,4,alpha=-1,nabla15,delta=0,sol7,1}
	\begin{aligned}
		\nabla_{e_2}e_3&=\delta\,e_1, &\nabla_{e_3}e_1&=-e_1, &\nabla_{e_3}e_2&=\delta\,e_1+e_2, &\nabla_{e_3}e_3&=-2\,e_3.
	\end{aligned}
\end{equation}	
The flat, torsion-free connection defined in~\eqref{g3,4,alpha=-1,nabla15,delta=0,sol7,1} coincides exactly with the one associated with the flat Lie algebra $\h_{0,1}$.

The eighth solution is given by the following flat, torsion-free connection:
\begin{equation}\label{g3,4,alpha=-1,nabla15,delta=0,sol8}
	\begin{aligned} 
		\nabla_{e_3}e_1&=-e_1, &\nabla_{e_3}e_2&=e_2, &\nabla_{e_3}e_3&=c_{13}e_1+c_{23}e_2+c_{33}e_3.
	\end{aligned}
\end{equation}	
Suppose that $c_{33}\neq\pm1$. Consider the following automorphism
\begin{align*}
	\Psi(e_1) &=e_2,&
	\Psi(e_2) &=e_1,&
	\Psi(e_3) &=\tfrac{c_{23}}{1-c_{33}}e_1-\tfrac{c_{13}}{1+c_{33}}e_2-e_3.
\end{align*}
Applying $\Psi$  to the connection given in~~\eqref{g3,4,alpha=-1,nabla15,delta=0,sol8} yields the following equivalent connection:
\begin{equation}\label{g3,4,alpha=-1,nabla15,delta=0,sol8,1}
	\begin{aligned} 
		\nabla_{e_3}e_1&=-e_1, &\nabla_{e_3}e_2&=e_2, &\nabla_{e_3}e_3&=\lambda\,e_3.
	\end{aligned}
\end{equation}	
In this case, $\lambda=-c_{33}\neq\pm1$.

If $c_{33}=1$.  Consider the following automorphism
\begin{align*}
	\Psi(e_1) &=e_2,&
	\Psi(e_2) &=x\,e_1,&
	\Psi(e_3) &=-\tfrac{c_{13}}{2}e_2-e_3.
\end{align*}
For a suitable choice of the parameter $x\in\R^\ast$, applying $\Psi$  to the connection given in~~\eqref{g3,4,alpha=-1,nabla15,delta=0,sol8} yields the following equivalent connection:
\begin{equation}\label{g3,4,alpha=-1,nabla15,delta=0,sol8,2}
	\begin{aligned} 
		\nabla_{e_3}e_1&=-e_1, &\nabla_{e_3}e_2&=e_2, &\nabla_{e_3}e_3&=\delta_1\,e_1-e_3,\quad\delta_1=0,1.
	\end{aligned}
\end{equation}

If $c_{33}=-1$.  Consider the following automorphism
\begin{align*}
	\Psi(e_1) &=x\,e_2,&
	\Psi(e_2) &=e_1,&
	\Psi(e_3) &=\tfrac{c_{23}}{2}e_2-e_3.
\end{align*}
For a suitable choice of the parameter $x\in\R^\ast$, applying $\Psi$  to the connection given in~~\eqref{g3,4,alpha=-1,nabla15,delta=0,sol8} yields the following equivalent connection:
\begin{equation}\label{g3,4,alpha=-1,nabla15,delta=0,sol8,3}
	\begin{aligned} 
		\nabla_{e_3}e_1&=-e_1, &\nabla_{e_3}e_2&=e_2, &\nabla_{e_3}e_3&=\delta_2\,e_2+e_3,\quad\delta_2=0,1.
	\end{aligned}
\end{equation}

Note  that the flat, torsion-free connections given in \eqref{g3,4,alpha=-1,nabla15,delta=0,sol8,1} and  \eqref{g3,4,alpha=-1,nabla15,delta=0,sol8,2} are isomorphic if and only if $\delta_1=0$ and $\lambda=-1$. For this reason, we set $\lambda\in\R \setminus{\{1\}}$ and assume that $\delta_1=1$. Similarly, one can easily show that \eqref{g3,4,alpha=-1,nabla15,delta=0,sol8,1} and  \eqref{g3,4,alpha=-1,nabla15,delta=0,sol8,3} are isomorphic if and only if $\delta_2=0$ and $\lambda=1$.  For this reason, we set $\lambda\in\R$ and assume that $\delta_2=1$. The connections given in \eqref{g3,4,alpha=-1,nabla15,delta=0,sol8,1}, \eqref{g3,4,alpha=-1,nabla15,delta=0,sol8,2}, and \eqref{g3,4,alpha=-1,nabla15,delta=0,sol8,3} correspond to the flat, torsion-free connections associated with  the flat Lie algebras $\h_{0,4}$, $\h_{0,5}$ and $\h_{0,6}$, respectively.   The additional condition $\lambda\neq\pm2$ in the flat Lie algebra $\h_{0,4}$ is imposed to ensure that the corresponding connection is not isomorphic to the flat Lie algebra $\h_{0,1}$ (in which case $\delta=0$).

The ninth solution corresponds to the following flat, torsion-free connection:
\begin{equation}\label{g3,4,alpha=-1,nabla15,delta=0,sol9}
	\begin{aligned} 
		\nabla_{e_2}e_3&=b_{13}e_1,
		&\nabla_{e_3}e_1&=-e_1, &\nabla_{e_3}e_2&=b_{13}e_1+e_2, &\nabla_{e_3}e_3&=c_{13}e_1+c_{23}e_2-2\,e_3.
	\end{aligned}
\end{equation}	
Consider the following automorphism
\begin{align*}
	\Psi(e_1) &=x\,e_1,&
	\Psi(e_2) &=e_2,&
	\Psi(e_3) &=\tfrac{(-2\,b_{1 3}c_{2 3} + 3\,c_{1 3})x}{3}+\tfrac{c_{23}}{2}e_2+e_3.
\end{align*}
For a suitable choice of the parameter $x\in\R^\ast$, applying $\Psi$  to the connection given in~~\eqref{g3,4,alpha=-1,nabla15,delta=0,sol9} yields the following equivalent connection:
\begin{equation}\label{g3,4,alpha=-1,nabla15,delta=0,sol9,1}
	\begin{aligned} 
		\nabla_{e_2}e_3&=\delta\,e_1,
		&\nabla_{e_3}e_1&=-e_1, &\nabla_{e_3}e_2&=\delta\,e_1+e_2, &\nabla_{e_3}e_3&=-2\,e_3.
	\end{aligned}
\end{equation}	
This flat, torsion-free connection coincides exactly with the one associated with the flat  Lie algebra $\h_{0,1}$.

The tenth solution corresponds to the following flat, torsion-free connection:
\begin{equation}\label{g3,4,alpha=-1,nabla15,delta=0,sol10}
	\begin{aligned} 
		\nabla_{e_2}e_3&=c_{33}e_2,
		&\nabla_{e_3}e_1&=-e_1, &\nabla_{e_3}e_2&=(c_{33}+1)e_2, &\nabla_{e_3}e_3&=c_{13}e_1+c_{23}e_2+c_{33}e_3.
	\end{aligned}
\end{equation}	
Consider the following automorphism
\begin{align*}
	\Psi(e_1) &=e_2,&
	\Psi(e_2) &=e_1,&
	\Psi(e_3) &-e_3.
\end{align*}
For a suitable choice of the parameter $x\in\R^\ast$, applying $\Psi$  to the connection given in~~\eqref{g3,4,alpha=-1,nabla15,delta=0,sol10} yields the following equivalent connection:
\begin{equation}\label{g3,4,alpha=-1,nabla15,delta=0,sol10,1}
	\begin{aligned} 
		\nabla_{e_1}e_3&=-c_{33}e_1,
		&\nabla_{e_3}e_1&=-(c_{33}+1)e_1, &\nabla_{e_3}e_2&=e_2, &\nabla_{e_3}e_3&=c_{23}e_1+c_{13}e_2+c_{33}e_3.
	\end{aligned}
\end{equation}	

This family of flat, torsion-free connections coincides with the second case treated in~\eqref{g3,4,alpha=-1,nabla15,delta=0,sol2}, in which case $c_{33}=-a_{13}$ and $c_{13}$ is replaced by $c_{23}$. Therefore, no further analysis is required.

The 11th solution corresponds to the following flat, torsion-free connection:
\begin{equation}\label{g3,4,alpha=-1,nabla15,delta=0,sol11}
	\begin{aligned} 
		\nabla_{e_2}e_2&=b_{23}e_3,
		&\nabla_{e_3}e_1&=-e_1, &\nabla_{e_3}e_2&=e_2, &\nabla_{e_3}e_3&=2\,e_3.
	\end{aligned}
\end{equation}	
Consider the following automorphism
\begin{align*}
	\Psi(e_1) &=e_2,&
	\Psi(e_2) &=x\,e_1,&
	\Psi(e_3) &=-e_3.
\end{align*}
For a suitable choice of the parameter $x\in\R^\ast$, applying $\Psi$  to the connection given in~~\eqref{g3,4,alpha=-1,nabla15,delta=0,sol11} yields the following equivalent connection:
\begin{equation}\label{g3,4,alpha=-1,nabla15,delta=0,sol11,1}
	\begin{aligned} 
		\nabla_{e_1}e_1&=\delta_\varepsilon\, e_3,  &\nabla_{e_3}e_1&=-e_1, &\nabla_{e_3}e_2&=e_2, &\nabla_{e_3}e_3&=-2\,e_3,\quad\delta_\varepsilon=0,\pm1.
	\end{aligned}
\end{equation}
This connection coincides with the one treated in 
\eqref{g3,4,alpha=-1,nabla15,delta=0,sol6,1}. Therefore, no further analysis is required.

This completes the classification of flat, torsion-free connections on flat Lie algebra $\G_{3,4}^{\alpha=-1}$. Under the corresponding assumptions on their parameters, all flat, torsion-free connections presented in Table~\ref{g3,4} are pairwise non-isomorphic.
\end{proof}

\begin{co}
	With the notations as above, among the flat Lie algebras on $\G_{3,4}^{\alpha=-1}$, we have
	\begin{enumerate}
		\item[i)] Associative algebras$:$\hspace{0.275cm} $\h_{0,7}^{\lambda=-1}$.
		\item[ii)] Novikov algebras$:$\hspace{0.73cm}  $\h_{0,2}^{\delta=0}$, $\h_{0,4}^{\lambda=0}$, $\h_{0,9}$, $\h_{0,10}$. 
		\item[iii)] Bi-symmetric algebras$:$ $\h_{0,7}^{\lambda=-1}$, $\h_{0,8}$.
		\item[iv)] Complete algebras$:$\hspace{0.62cm}$\h_{0,4}^{\lambda=0}$, $\h_{0,11}$.
	\end{enumerate}
\end{co}
\begin{co}
	With the notations as above, among the flat Lie algebras on $\G_{3,4}^{\alpha\neq-1}$, we have
	\begin{enumerate}
		\item[i)] Associative algebras$:$
		\item[ii)] Novikov algebras$:$\hspace{0.73cm}  $\h_{0,4}^{\alpha=\tfrac12}$, $\h_{0,5}$, $\h_{0,6}$, $\h_{0,9}^{\delta=0}$, $\h_{0,10}^{\lambda=0}$, $\h_{0,19}$, $\h_{0,20}$, $\h_{0,21}$, $\h_{0,22}$, $\h_{2,1}^{\alpha=\tfrac12}$, $\h_{2,5}$, $\h_{2,6}$.
		\item[iii)] Bi-symmetric algebras$:$
		\item[iv)] Complete algebras$:$\hspace{0.62cm}$\h_{0,10}^{\lambda=0}$, $\h_{2,1}^{\alpha=\tfrac12}$.
	\end{enumerate}
\end{co}

\begin{pr}
Let $(\G, \nabla)$ be a three-dimensional real flat  Lie algebra with $\G = \G_{3,5}$. Then $(\G, \nabla)$ is isomorphic to exactly one of the flat Lie algebras listed in Table~$\ref{g3,5}$.
{\renewcommand*{\arraystretch}{1.8}
\captionof{table}{Flat torsion-free connection on the Lie algebra $\G_{3,5}$.}
\setcounter{table}{9}
\begin{footnotesize} 
\setlength{\tabcolsep}{4pt} 
\begin{longtable}{@{}cllllllc@{}} 
			\hline
		Flat algebra&\multicolumn{2}{@{}l@{}}{~~Flat torsion-free connection}&&&& Remarks\\
			\hline
$\h_{0,1}$&$\nabla_{e_3}e_1=-\beta e_1+e_2$&$\nabla_{e_3}e_2=-e_1-\beta e_2$&$\nabla_{e_3}e_3=\lambda e_3$&&&$\lambda\in\R$&\\
$\h_{0,2}$&$\nabla_{e_1}e_3=\lambda e_1$&$\nabla_{e_2}e_3=\lambda e_2$&$\nabla_{e_3}e_1=(\lambda-\beta)e_1+e_2$&$\nabla_{e_3}e_2=-e_1+(\lambda-\beta)e_2$&$\nabla_{e_3}e_3=\lambda e_3$&$\lambda\in\R^\ast$&\\
$\h_{0,3}$&$\nabla_{e_1}e_1=\varepsilon e_3$&$\nabla_{e_2}e_2=\varepsilon e_3$&$\nabla_{e_3}e_1=-\beta e_1+e_2$&$\nabla_{e_3}e_2=-e_1-\beta e_2$&$\nabla_{e_3}e_3=-2\beta e_3$&$\varepsilon=\pm1$&
\\\hline		
			\end{longtable}
			\label{g3,5}
			\end{footnotesize}	
			}
	
\end{pr}
\begin{proof}
	As in the proof of the previous propositions. In the basis $\lbrace e_1, e_2, e_3 \rbrace$, the operators $\nabla_{e_1}$, $\nabla_{e_2}$ and $\nabla_{e_3}$ are given respectively by: 
	\begin{equation}
		\nabla_{e_1}=\left( \begin {array}{ccc} a_{11}&a_{12}&a_{13}\\ \noalign{\medskip}
		a_{21}&a_{22}&a_{23}\\ \noalign{\medskip}a_{31}&a_{32}&a_{33}\end {array} \right),\quad
		\nabla_{e_2}=\left( \begin {array}{ccc} a_{12}&b_{12}&b_{13}\\ \noalign{\medskip}
		a_{22}&b_{22}&b_{23}\\ \noalign{\medskip}a_{32}&b_{32}&b_{33}\end {array} \right),\quad
		\nabla_{e_3}=\left( \begin {array}{ccc} a_{13}-\beta &b_{13}-1&c_{13}\\ \noalign{\medskip}
		a_{23}+1&b_{23}-\beta&c_{23}\\ \noalign{\medskip}a_{33}&b_{33}&c_{33}\end {array} \right),\quad
	\end{equation}
	where $a_{ij}$, $b_{ij}$, $c_{ij}\in \mathbb{R}$ and $ \beta\geq0$.

Assume that $\nabla^0$ is a  non-flat, torsion-free connection. Then $\nabla^0$ is equivalent to one of the connections listed in Lemma~\ref{Lemg3j} under the Lie algebra $\G_{3,5}$. We first consider the case $\nabla^0=\nabla^1$. Using a straightforward computation, the flatness-equations can be solved with four complex solutions, so are not considered.

If $\nabla^0=\nabla^2$.   Then the flatness-equations admit a unique real solution given by the following flat torsion-free connection: 
\begin{equation}\label{g3,5,nabla2,sol1}
	\begin{aligned}
		\nabla_{e_1}e_1&=-a_{2 3}a_{3 1}e_1+e_2+a_{31}e_3, \hspace{3cm}\nabla_{e_1}e_3=-a_{2 3}^2a_{3 1}e_1+a_{23}e_2+a_{23}a_{31}e_3,\\
		\nabla_{e_2}e_2&=-a_{2 3}a_{3 1}e_1+e_2+a_{31}e_3, \hspace{3cm}\nabla_{e_2}e_3=a_{23}e_1-\tfrac{1}{a_{31}}e_2-e_3,\\
		\nabla_{e_3}e_1&=(-a_{2 3}^2a_{3 1} - \beta)e_1+(1 + a_{2 3})e_2+a_{2 3}a_{3 1}e_3, \hspace{0.54cm}\nabla_{e_3}e_2=(-1 + a_{2 3})e_1-\tfrac{1+\beta\,a_{31}}{a_{31}}e_2-e_3,\\
		\nabla_{e_3}e_3&=\tfrac {-{a_{{23}}^{3}}{a_{{31}}^{2}}+\beta\,a_{{23}}a_{{31}}-a_{{23}}+1}{a_{{31}}}e_1+\tfrac{1+ \left( {a_{{23}}^{2}}+a_{{23}} \right) {a_{{31}}^{2}}-
			\beta\,a_{{31}}}{{a_{{31}}^{2}}}e_2+\tfrac {{a_{{23}}^{2}}{a_{{31}}^{2}}-2\,\beta\,a_{{31}}+1}{a_{{31}
			}}e_3.			
	\end{aligned}
\end{equation}
In this case, $\eta_2=1$ and $\nu_2 =\lambda_2= -a_{2 3}a_{3 1}$.

Consider the following automorphism:
	\begin{align*}
	\Psi(e_1) &=x\,a_{31} e_2,&
	\Psi(e_2) &=x\,a_{31} e_1,&
	\Psi(e_3) &=x\, e_1+x\,a_{23}a_{31}e_2+e_3.
\end{align*}
 Applying it to the connection defined in \eqref{g3,5,nabla2,sol1} yields an equivalent connection:
 \begin{equation}\label{g3,5,nabla2,sol1,1}
 	\begin{aligned}
 		\begin{aligned}
 		\nabla_{e_1}e_1&=\tfrac{1}{x^2\,a_{31}} e_3, &\nabla_{e_2}e_2&=\tfrac{1}{x^2\,a_{31}} e_3, &\nabla_{e_3}e_1&=-\beta\, e_1+e_2, &\nabla_{e_3}e_2&=-e_1-\beta\, e_2, &\nabla_{e_3}e_3&=-2\,\beta\, e_2.
 	\end{aligned}
 	\end{aligned}
 \end{equation}
  Furthermore, by an appropriate choice of $x\in\R^\ast$, this connection coincides with the one associated with the flat Lie algebra $\h_{0,3}$ listed in Table~\ref{g3,5}.

If $\nabla^0=\nabla^3$, then the flatness-equations admit a unique real solution given by the following flat torsion-free connection: 
\begin{equation}\label{g3,5,nabla3,sol1}
\begin{aligned}
	\nabla_{e_2} e_2 &= e_1+a_{31}e_3, &\nabla_{e_1}e_3&=-\tfrac{1}{a_{31}}e_1-e_3, &\nabla_{e_2}e_2&=e_2+a_{31}e_3,\\
	\nabla_{e_3}e_1&=-\tfrac{\beta\,a_{31}+1}{a_{31}}e_1+e_2-e_3, &\nabla_{e_3}e_2&=-e_1-\beta\,e_2, &\nabla_{e_3}e_3&=\tfrac{1-\beta\,a_{31}}{a_{31}^2}e_1-\tfrac{1}{a_{31}}e_2+\tfrac{1-2\,\beta\,a_{31}}{a_{31}}e_3.
\end{aligned}
	\end{equation}
	In this case, $\eta_3=0$ and $\nu_3=1$. Next, consider the following automorphism
	\begin{align*}
		\Psi(e_1) &=x\, e_1,&
		\Psi(e_2) &=x\, e_2,&
		\Psi(e_3) &=-\frac{x}{a_{31}} e_1+e_3.
	\end{align*}
	Applying $\Psi$  to the previous connection yields the following equivalent connection:
	\begin{equation}\label{g3,5,nabla3,sol1,1}
		\begin{aligned}
	\nabla_{e_1}e_1&=\tfrac{a_{31}}{x^2} e_3, &\nabla_{e_2}e_2&=\tfrac{a_{31}}{x^2} e_3, &\nabla_{e_3}e_1&=-\beta\, e_1+e_2, &\nabla_{e_3}e_2&=-e_1-\beta\, e_2, &\nabla_{e_3}e_3&=-2\,\beta\, e_2.
\end{aligned}
\end{equation}
After normalizing the coefficient $\tfrac{a_{31}}{x^2}$ to $\pm1$, the resulting connection coincides precisely with the one associated with the flat Lie algebra $\h_{0,3}$ listed in Table~\ref{g3,5}.

If $\nabla^0=\nabla^4$, then the flatness equations admit no solutions.

Let $\nabla^0=\nabla^5$. If $\delta=0$, then the flatness-equations admit no solutions. On the other hand, when $\delta=0$, a straightforward computation shows that there exist exactly three distinct solutions. The first is determined by the following  flat torsion-free connection:
	\begin{equation}\label{g3,5,nabla5,sol1}
	\begin{aligned}
		\nabla_{e_1}e_1&=b_{32}e_3, &\nabla_{e_2}e_2&=b_{32}e_3, &\nabla_{e_3}e_1&=-\beta\,e_1+e_2, &\nabla_{e_3}e_2&=-e_1-\beta\,e_2, &\nabla_{e_3}e_3&=-2\,\beta\,e_3.
	\end{aligned}
\end{equation}
For a suitable choice of the parameter $x \in \R$, this connection is isomorphic, via the following automorphism, 
\begin{align*}
	\Psi(e_1) &=x\, e_1,&
	\Psi(e_2) &=x\, e_2,&
	\Psi(e_3) &=e_3.
\end{align*}
to the one associated with the flat Lie algebra $\h_{0,3}$.

The second solution is given by the following flat, torsion-free connection:	
	\begin{equation}\label{g3,5,nabla5,sol2}
		\begin{aligned} \nabla_{e_3}e_1&=-\beta\,e_1+e_2, &\nabla_{e_3}e_2&=-e_1-\beta\,e_2, &\nabla_{e_3}e_3&=c_{13}e_1+c_{23}e_2+c_{33}e_3.
		\end{aligned}
	\end{equation}
	Consider the following automorphism
	\begin{align*}
		\Psi(e_1) &=e_1,&
		\Psi(e_2) &=e_2,&
		\Psi(e_3) &=\tfrac{(-\beta - c_{3 3})c_{1 3} + c_{2 3}}{(\beta+c_{33})^2+1}e_1+\tfrac{(-\beta - c_{3 3})c_{23} - c_{1 3}}{(\beta+c_{33})^2+1}e_2+e_3.
	\end{align*}
	Applying $\Psi$  to the  connection \eqref{g3,5,nabla5,sol2} yields the following equivalent connection:
		\begin{equation}\label{g3,5,nabla5,sol2,1}
		\begin{aligned} \nabla_{e_3}e_1&=-\beta\,e_1+e_2, &\nabla_{e_3}e_2&=-e_1-\beta\,e_2, &\nabla_{e_3}e_3&=c_{33}e_3.
		\end{aligned}
	\end{equation}
	This connection is precisely the one associated with the flat Lie algebra $\h_{0,1}$, with $c_{33}=\lambda\in\R$.

	The third solution corresponds to the following flat, torsion-free connection:
\begin{equation}\label{g3,5,nabla5,sol3}
	\begin{aligned}
		\nabla_{e_1}e_3&=c_{33}e_1, &\nabla_{e_2}e_3&=c_{33}e_2,\\ \nabla_{e_3}e_1&=(c_{33}-\beta)\,e_1+e_2, &\nabla_{e_3}e_2&=-e_1+(c_{33}-\beta)\,e_2, &\nabla_{e_3}e_3&=c_{13}e_1+c_{23}e_2+c_{33}e_3.
	\end{aligned}
\end{equation}
	Consider the following automorphism
	\begin{align*}
		\Psi(e_1) &=e_1,&
		\Psi(e_2) &=e_2,&
		\Psi(e_3) &=\tfrac{(-\beta + c_{3 3})c_{1 3} + c_{2 3}}{(\beta+c_{33})^2+1}e_1+\tfrac{(-\beta + c_{3 3})c_{23} - c_{1 3}}{(\beta+c_{33})^2+1}e_2+e_3.
	\end{align*}
	Applying $\Psi$  to the  connection \eqref{g3,5,nabla5,sol3} yields the following equivalent connection:
	\begin{equation}\label{g3,5,nabla5,sol3,1}
		\begin{aligned}
			\nabla_{e_1}e_3&=c_{33}e_1, &\nabla_{e_2}e_3&=c_{33}e_2,\\ \nabla_{e_3}e_1&=(c_{33}-\beta)\,e_1+e_2, &\nabla_{e_3}e_2&=-e_1+(c_{33}-\beta)\,e_2, &\nabla_{e_3}e_3&=c_{33}e_3.
		\end{aligned}
	\end{equation}
	This connection is precisely the one associated with the flat Lie algebra $\h_{0,2}$, with $c_{33}=\lambda\in\R$.
	
	It is not difficult to verify that the three flat, torsion-free connections listed in Table~\ref{g3,5} are pairwise non-isomorphic under the action of the automorphism group of $\G_{3,5}$.

\end{proof}
\begin{co}
	With the notations as above, among the flat Lie algebras on $\G_{3,5}$, we have
	\begin{enumerate}
		\item[i)] Associative algebras$:$
		\item[ii)] Novikov algebras$:$\hspace{0.73cm}   $\h_{0,1}^{\lambda=0}$, $\h_{0,2}$. 
		\item[iii)] Bi-symmetric algebras$:$ 
		\item[iv)] Complete algebras$:$\hspace{0.62cm}$\h_{0,1}^{\lambda=0}$, $\h_{0,3}^{\beta=0}$.
	\end{enumerate}
\end{co}

Let $G$ be a connected and simply connected affine Lie group. 
The following result summarizes our classification: $G$ is complete, Novikov, radiant, associative  or bi-symmetric if it admits a left-invariant affine structure carrying these properties.
\begin{theo}
	Let $G$ be a connected and simply connected three-dimensional real affine Lie group. Then the following statements hold:
	\begin{enumerate}
		\item $G$ is complete;
		\item $G$ is Novikov Lie group;
		\item $G$ is radiant;
		\item Associative Lie groups$:$\hspace{0.54cm} $G^{\alpha=-1}_{3,4}$,~$G_{33}$,~$G_{31}$, ~$\mathrm{Aff}(1,\R)\oplus\R$,~$\R^3$
		\item Bi-symmetric Lie groups$:$\hspace{0.28cm} $G^{\alpha=-1}_{3,4}$,~ $G_{33}$,~$G_{31}$,~$\mathrm{Aff}(1,\R)\oplus\R$,~$\R^3$
	\end{enumerate}
	\end{theo}

\section{Appendix}\label{App}
\subsection{Automorphism Groups of three-Dimensional  Lie algebras}

\begin{align*}
\operatorname{Aut}(3\G_{1})&=
\left\{
M\in\mathcal{M}_3(\R)~:~\det(M)\neq0 \right\}, &\operatorname{Aut}(\G_{3,1})&=
\left\{
\begin{pmatrix}
	x_{11} & x_{12} & 0 \\
	x_{21} & x_{22} & 0 \\
	x_{31} & x_{32} & x_{11}x_{22}-x_{12}x_{21}
\end{pmatrix}
\;\middle|\;
(x_{11}x_{22}-x_{12}x_{21})\neq 0
\right\},
\\
\operatorname{Aut}(\G_{3,2})&=
\left\{
\begin{pmatrix}
	x_{22} & x_{12} & x_{13} \\
	0& x_{22} & x_{23} \\
	0 & 0&1
\end{pmatrix}
\;\middle|\;
x_{22}\neq 0
\right\}, &\operatorname{Aut}(\G_{3,3})&=
\left\{
\begin{pmatrix}
	x_{11} & x_{12} & x_{13} \\
	x_{21}& x_{22} & x_{23} \\
	0 & 0&1
\end{pmatrix}
\;\middle|\;
(x_{11}x_{22}-x_{12}x_{21})\neq 0
\right\}.
\\
\operatorname{Aut}(\G_{3,4}^{\alpha\neq-1})&=
\left\{
\begin{pmatrix}
	x_{11} & 0 & x_{13} \\
	0& x_{22} & x_{23} \\
	0 & 0&1
\end{pmatrix}
\;\middle|\;
x_{11}x_{22}\neq 0
\right\}, &\operatorname{Aut}(\G_{3,4}^{\alpha=-1})&=
\left\{
\begin{pmatrix}
	x_{11} & 0 & x_{13} \\
	0& x_{22} & x_{23} \\
	0 & 0&1
\end{pmatrix},\begin{pmatrix}
0& x_{12} & x_{13} \\
x_{21} & 0& x_{23} \\
0 & 0&-1
\end{pmatrix}
\;\middle|\;
x_{11}x_{22}\neq 0,~x_{12}x_{21}\neq0
\right\}.\\
\operatorname{Aut}(3\G_{2,1}\oplus\G_1)&=
\left\{
\begin{pmatrix}
	x_{11} & x_{12} & 0 \\
	0 &1& 0 \\
	0 & x_{32} & x_{33}
\end{pmatrix}
\;\middle|\;
x_{11}x_{33}\neq 0
\right\},
\end{align*}
\subsection{Flat torsion-free connections on $3\G_{2,1}\oplus\G_1$}
In the basis $\lbrace e_1, e_2, e_3 \rbrace$, the operators $\nabla_{e_1}$, $\nabla_{e_2}$ and $\nabla_{e_3}$ are given respectively by: 
\begin{equation}\label{aff3}
	\nabla_{e_1}=\left( \begin {array}{ccc} a_{11}&a_{12}&a_{13}\\ \noalign{\medskip}
	a_{21}&a_{22}&a_{23}\\ \noalign{\medskip}a_{31}&a_{32}&a_{33}\end {array} \right),\quad
	\nabla_{e_2}=\left( \begin {array}{ccc} a_{12}-1&b_{12}&b_{13}\\ \noalign{\medskip}
	a_{22}&b_{22}&b_{23}\\ \noalign{\medskip}a_{32}&b_{32}&b_{33}\end {array} \right),\quad
	\nabla_{e_3}=\left( \begin {array}{ccc} a_{13}&b_{13}&c_{13}\\ \noalign{\medskip}
	a_{23}&b_{23}&c_{23}\\ \noalign{\medskip}a_{33}&b_{33}&c_{33}\end {array} \right),\quad
\end{equation}
where $a_{ij}$, $b_{ij}$, $c_{ij}\in \mathbb{R}$.\\
\newgeometry{left=0.1cm,right=0.1cm,top=2cm,bottom=2cm}
\textbf{Case 1.} If $\nabla^{0}\equiv\nabla^1$ which corresponds to the flat Lie algebra $\mathfrak{a}_1$ (see Table \ref{Flataffine}), then the flatness equations associated
with the connection given in \eqref{aff3} can be solved directly, giving sixteen distinct solutions:
\begin{align*}
	\nabla_{e_1}^1 e_2 &= -e_1 + a_{32} e_3, &
	\nabla_{e_2}^1 e_1 &= -2e_1 + a_{32} e_3, &
	\nabla_{e_2}^1 e_2 &= -e_2 + b_{32} e_3, &
	\nabla_{e_2}^1 e_3& = \nabla_{e_3} e_2 = -e_3.\\
	\nabla_{e_1}^2 e_1 &= a_{31}e_3, &
	\nabla_{e_1}^2 e_2 &= 2e_1, &
	\nabla_{e_2}^2 e_1 &= e_1, &
	\nabla_{e_2}^2 e_2 &= 2e_1 + b_{32} e_3.
	\\
	\nabla_{e_1}^3 e_3 &= b_{23} e_1, &
	\nabla_{e_2}^3 e_1 &= -e_1, &
	\nabla_{e_2}^3 e_3 &=\nabla_{e_3}^3 e_2 = b_{23} e_2, &
	\nabla_{e_3}^3 e_1 &= b_{23} e_1,\\
	\nabla_{e_3}^3 e_3 &= b_{23} e_1 + c_{23} e_2 + b_{23} e_3.
	\\
	\nabla_{e_1}^4 e_2 &= \lambda e_1,&
	\nabla_{e_1}^4 e_3 &=\nabla_{e_3}^4 e_1 = a_{13} e_1,&
	\nabla_{e_2}^4 e_1 &= (\lambda-1) e_1, &
	\nabla_{e_2}^4 e_2 &= \lambda e_1, \\
	\nabla_{e_2}^4 e_3 &=\nabla_{e_3}^4 e_2= \lambda e_1, &
	\nabla_{e_3}^4 e_3 &= c_{23} e_2 + \tfrac{-\lambda c_{23} + a_{13}^2}{a_{13}} e_3,&& \lambda\in \mathbb{R}, a_{13}\neq0.
	\\
	\nabla_{e_1}^5 e_2 &= e_1, &
	\nabla_{e_1}^5 e_3 &=\nabla_{e_3}^5 e_1= a_{13} e_1, &
	\nabla_{e_2}^5 e_2 &= e_2, &
	\nabla_{e_2}^5 e_3 &=\nabla_{e_3}^5 e_2= b_{13} e_1 + e_3,  \\
	\nabla_{e_3}^5 e_3 &= -\tfrac{b_{13} c_{23}}{a_{13}} e_1 + c_{23} e_2 + \tfrac{a_{13}^2 - c_{23}}{a_{13}} e_3,  &\;  a_{13}\neq0.
	\\
	\nabla_{e_1}^6 e_2 &= \lambda e_1, &
	\nabla_{e_1}^6 e_3 &=\nabla_{e_3}^6 e_1= a_{13} e_1, &
	\nabla_{e_2}^6 e_1 &= (\lambda-1) e_1, &
	\nabla_{e_2}^6 e_2 &= \lambda e_2, \\
	\nabla_{e_2}^6 e_3 &=\nabla_{e_3}^6 e_2 = a_{13} e_2, &
	\nabla_{e_3}^6 e_3 &= c_{23} e_2 + \tfrac{-\lambda c_{23} + a_{13}^{2}}{a_{13}} e_3, && \lambda\in \mathbb{R}, a_{13}\neq0.
	\\
	\nabla_{e_1}^7 e_2 &= e_1, &
	\nabla_{e_1}^7 e_3 &=\nabla_{e_3}^7 e_1= a_{13} e_1, &
	\nabla_{e_2}^7 e_2 &= e_2, &
	\nabla_{e_2}^7 e_3 &=\nabla_{e_3}^7 e_2= a_{13} e_2,  \\
	\nabla_{e_3}^7 e_3 &= c_{13} e_1 + c_{23} e_2 + \tfrac{a_{13}^{2} - c_{23}}{a_{13}} e_3,&  a_{13}\neq0.
	\\
	\nabla_{e_1}^8 e_2 &= \lambda e_1, &
	\nabla_{e_2}^8 e_1 &= (\lambda-1) e_1, &
	\nabla_{e_2}^8 e_2 &= \lambda e_2 - \tfrac{b_{33}(\lambda-b_{33})}{c_{33}} e_3, &
	\nabla_{e_2}^8 e_3 &=\nabla_{e_3}^8 e_2= b_{33} e_3, \\
	\nabla_{e_3}^8 e_3 &= c_{33} e_3,\, \lambda\in \mathbb{R}, &c_{33}&\neq0.
	\\
	\nabla_{e_1}^9 e_2 &= e_1, &
	\nabla_{e_2}^9 e_2 &= e_2 + \tfrac{b_{33}(b_{33}-1)}{c_{33}} e_3, &
	\nabla_{e_2}^9 e_3 &=\nabla_{e_3}^9 e_2= \tfrac{c_{13}b_{33}}{c_{33}} e_1 + b_{33} e_3,  &
	\nabla_{e_3}^9 e_3 &= c_{13} e_1 + c_{33} e_3,\, c_{33}\neq0.
	\\
	\nabla_{e_1}^{10} e_2 &= \lambda e_1, &
	\nabla_{e_2}^{10} e_1 &= (\lambda-1) e_1, &
	\nabla_{e_2}^{10} e_2 &= \lambda e_2 + b_{32} e_3,&\lambda\in \mathbb{R}.
	\\
	\nabla_{e_1}^{11} e_2 &= e_1, &
	\nabla_{e_2}^{11} e_2 &= \nabla_{e_2}^{11} e_3 = -c_{13}b_{32} e_1, &
	\nabla_{e_3}^{11} e_2 &= -c_{13}b_{32} e_1, &
	\nabla_{e_3}^{11} e_3 &= c_{13} e_1.
	\\
	\nabla_{e_1}^{12} e_2 &= b_{33} e_1, &
	\nabla_{e_2}^{12} e_1 &= (b_{33}-1) e_1, &
	\nabla_{e_2}^{12} e_2 &= b_{33} e_2 + b_{32} e_3, &
	\nabla_{e_2}^{12} e_3 &= \nabla_{e_3}^{12} e_2 = b_{33} e_3.
	\\
	\nabla_{e_1}^{13} e_2 &= e_1, &
	\nabla_{e_2}^{13} e_2 &= e_2 + b_{32} e_3,&
	\nabla_{e_2}^{13} e_3 &= \nabla_{e_3}^{13} e_2 = b_{13} e_1 + e_3.
	\\
	\nabla_{e_2}^{14} e_1 &= -e_1, &
	\nabla_{e_2}^{14} e_2 &= b_{32} e_3, &
	\nabla_{e_2}^{14} e_3 &= 
	\nabla_{e_3}^{14} e_2 = \tfrac{b_{32}c_{33}-b_{33}^{2}}{b_{32}} e_2 + b_{33} e_3, \\
	\nabla_{e_3}^{14} e_3 &= \tfrac{(b_{32}c_{33}-b_{33}^{2})b_{33}}{b_{32}^{2}} e_2 + c_{33} e_3,&b_{32}\neq0.
	\\
	\nabla_{e_2}^{15} e_1 &= -e_1, &
	\nabla_{e_3}^{15} e_3 &= c_{23} e_2 + c_{33} e_3.
	\\
	\nabla_{e_2}^{16} e_1 &= -e_1, &
	\nabla_{e_2}^{16} e_3 &= \nabla_{e_3}^{16} e_2 = c_{33} e_2, &
	\nabla_{e_3}^{16} e_3 &= c_{23} e_2 + c_{33} e_3.
\end{align*}
\textbf{Case 2.} If $\nabla^{0}\equiv\nabla^{2}$, corresponding to the flat Lie algebra $\mathfrak{a}_2$ (see Table \ref{Flataffine}), the flatness equations associated with the connection in \eqref{aff3} admit a direct resolution, yielding fifteen distinct solutions:
\begin{align*}
	\nabla_{e_1}^1 e_3 &= \nabla_{e_3}^1 e_1 = c_{33} e_1,  &
	\nabla_{e_2}^1 e_1 &= -e_1, &
	\nabla_{e_2}^1 e_2 &= \mu e_2, &
	\nabla_{e_2}^1 e_3 &= \nabla_{e_3}^1 e_2 = b_{23} e_2, \\
	\nabla_{e_3}^1 e_3 &= \tfrac{b_{23}(b_{23}-c_{33})}{\mu} e_2 + c_{33} e_3, &\mu \in \R^*.\\
	\nabla_{e_1}^2 e_3 &= \nabla_{e_3}^2 e_1 = c_{33} e_1, &
	\nabla_{e_2}^2 e_1 &= -e_1, &
	\nabla_{e_3}^2 e_3 &= c_{23} e_2 + c_{33} e_3.
	\\
	\nabla_{e_1}^3 e_3 &= \nabla_{e_3}^3 e_1 = c_{33} e_1, &
	\nabla_{e_2}^3 e_1 &= -e_1, &
	\nabla_{e_2}^3 e_3 &= \nabla_{e_3}^3 e_2 = c_{33} e_2, &
	\nabla_{e_3}^3 e_3 &= c_{23} e_2 + c_{33} e_3.
	\\
	\nabla_{e_1}^4 e_3 &= \nabla_{e_3}^4 e_1 = c_{33} e_1 + a_{23} e_2, &
	\nabla_{e_2}^4 e_1 &= -e_1, &
	\nabla_{e_2}^4 e_2 &= -e_2, &
	\nabla_{e_3}^4 e_3 &= c_{33} e_3.
	\\
	\nabla_{e_1}^5 e_2 &= a_{32} e_3, &
	\nabla_{e_2}^5 e_1 &= -e_1 + a_{32} e_3, &
	\nabla_{e_2}^5 e_2 &= e_2 + b_{32} e_3.
	\\
	\nabla_{e_1}^6 e_3 &= \nabla_{e_3}^6 e_1 = a_{23} e_2,&
	\nabla_{e_2}^6 e_1 &= -e_1,&
	\nabla_{e_2}^6 e_2 &= -e_2,&
	\nabla_{e_2}^6 e_3 &= \nabla_{e_3}^6 e_2 = b_{23} e_2,\\
	\nabla_{e_3}^6 e_3 &= b_{23} e_3.
	\\
	\nabla_{e_1}^7 e_1 &= a_{31} e_3,&
	\nabla_{e_2}^7 e_1 &= -e_1,&
	\nabla_{e_2}^7 e_2 &= -2e_2 + b_{32} e_3,&
	\nabla_{e_2}^7 e_3 &= \nabla_{e_3}^7 e_2 = -2e_3.
	\\
	\nabla_{e_2}^8 e_1 &= -e_1, &
	\nabla_{e_2}^8 e_2 &= \mu e_2 + b_{32} e_3, \\
	\nabla_{e_2}^8 e_3 &= 
	\nabla_{e_3}^8 e_2 = \tfrac{\mu b_{33} + b_{32}c_{33} - b_{33}^{2}}{b_{32}} e_2 + b_{33} e_3,\\
	\nabla_{e_3}^8 e_3 &= \tfrac{(\mu b_{33} + b_{32}c_{33} - b_{33}^{2}) b_{33}}{b_{32}^{2}} e_2 + c_{33} e_3, &&\mu\in\R^*, \, b_{32}\neq 0. 
	\\
	\nabla_{e_2}^9 e_1 &= -e_1, &
	\nabla_{e_2}^9 e_2 &= \mu e_2, &
	\nabla_{e_2}^9 e_3 &= 
	\nabla_{e_3}^9 e_2 = b_{23} e_2, &
	\nabla_{e_3}^9 e_3 &= c_{23} e_2 + \tfrac{-\mu c_{23} + b_{23}^{2}}{b_{23}} e_3,& \, b_{32}\neq 0.
	\\
	\nabla_{e_2}^{10} e_1 &= -e_1, &
	\nabla_{e_2}^{10} e_2 &= \mu e_2, &
	\nabla_{e_3}^{10} e_3 &= c_{33} e_3.
	\\
	\nabla_{e_2}^{11} e_1 &= -e_1, &
	\nabla_{e_2}^{11} e_2 &= b_{33} e_2, &
	\nabla_{e_2}^{11} e_3 &= \nabla_{e_3}^{11} e_2 = b_{33} e_3, &
	\nabla_{e_3}^{11} e_3 &= c_{23} e_2 + c_{33} e_3.
	\\
	\nabla_{e_2}^{12} e_1 &= -e_1, &
	\nabla_{e_2}^{12} e_2 &= -e_2, &
	\nabla_{e_2}^{12} e_3 &= \nabla_{e_3}^{12} e_2 = -e_3, &
	\nabla_{e_3}^{12} e_3 &= c_{13} e_1 + c_{23} e_2 + c_{33} e_3.
	\\
	\nabla_{e_1}^{13} e_3 &= \nabla_{e_3}^{13} e_1 = a_{13} e_1 + \tfrac{a_{13}b_{23}}{b_{13}} e_2, &
	\nabla_{e_2}^{13} e_1 &= -e_1, &
	\nabla_{e_2}^{13} e_2 &= -e_2, &
	\nabla_{e_2}^{13} e_3 &= \nabla_{e_3}^{13} e_2 = b_{13} e_1 + b_{23} e_2, \\
	\nabla_{e_3}^{13} e_3 &= (a_{13}+b_{23}) e_3, \,\, b_{13}\neq 0.
	\\
	\nabla_{e_1}^{14} e_3 &=\nabla_{e_3}^{14} e_1 = a_{13} e_1, &
	\nabla_{e_2}^{14} e_1 &= -e_1, &
	\nabla_{e_2}^{14} e_2 &= -e_2, &
	\nabla_{e_2}^{14} e_3 &= 
	\nabla_{e_3}^{14} e_2 = b_{13} e_1 - \tfrac{c_{13}}{b_{13}} e_2, \\
	\nabla_{e_3}^{14} e_3 &= c_{13} e_1 - \tfrac{c_{13}(a_{13}b_{13}+c_{13})}{b_{13}^{2}} e_2 + a_{13} e_3, & \, b_{13}\neq 0.
	\\
	\nabla_{e_2}^{15} e_1 &= -e_1, &
	\nabla_{e_2}^{15} e_2 &= -e_2 + b_{32} e_3, \\
	\nabla_{e_2}^{15} e_3 &=
	\nabla_{e_3}^{15} e_2 = b_{13} e_1 + b_{23} e_2 + \tfrac{b_{32}c_{13}}{b_{13}} e_3, \\
	\nabla_{e_3}^{15} e_3 &= c_{13} e_1 + \tfrac{b_{23}c_{13}}{b_{13}} e_2 + \tfrac{b_{13}^{2}b_{23} + b_{32}c_{13}^{2} + b_{13}c_{13}}{b_{13}^{2}} e_3, & \, b_{13}\neq 0.
\end{align*}
\textbf{Case 3.} Suppose that $\nabla^{0}\equiv\nabla^{3}$, which corresponds to the flat Lie algebra $\mathfrak{a}_3$ (see Table \ref{Flataffine}). In this case, the flatness equations associated with the connection in \eqref{aff3} can be solved explicitly, leading to the following four distinct solutions:
\begin{align*}
	\nabla_{e_1}^1 e_3 = \nabla_{e_3}^1 e_1 &= b_{23} e_1, &
	\nabla_{e_2}^1 e_1 &= -e_1, &
	\nabla_{e_2}^1 e_2 &= e_1 - e_2, &
	\nabla_{e_2}^1 e_3 &= \nabla_{e_3}^1 e_2 = -\tfrac{c_{13}}{b_{23}} e_1 + b_{23} e_2, &
	\nabla_{e_3}^1 e_3 &= c_{13} e_1 + b_{23} e_3.
	\\
	\nabla_{e_2}^2 e_1 &= -e_1, &
	\nabla_{e_2}^2 e_2 &= e_1 - e_2, &
	\nabla_{e_2}^2 e_3& = \nabla_{e_3}^2 e_2 = b_{13} e_1, &
	\nabla_{e_3}^2 e_3 &= b_{13}c_{33} e_1 + c_{33} e_3.
	\\
	\nabla_{e_2}^3 e_1 &= -e_1, &
	\nabla_{e_2}^3 e_2 &= e_1 - e_2, &
	\nabla_{e_2}^3 e_3& = \nabla_{e_3}^3 e_2 = -e_3, &
	\nabla_{e_3}^3 e_3 &= c_{13} e_1 + c_{33} e_3.
	\\
	\nabla_{e_2}^4 e_1 &= -e_1, &
	\nabla_{e_2}^4 e_2 &= e_1 - e_2 + b_{32} e_3, &
	\nabla_{e_2}^4 e_3& = \nabla_{e_3}^4 e_2 = b_{13} e_1 + b_{33} e_3, &
	\nabla_{e_3}^4 e_3 &= \tfrac{b_{13}b_{33}}{b_{32}} e_1 + \tfrac{b_{33}(b_{33}+1)}{b_{32}} e_3,\,\,b_{32}\neq0.
\end{align*}
\textbf{Case 4.} Suppose that $\nabla^{0}\equiv\nabla^{4}$, which corresponds to the flat Lie algebra $\mathfrak{a}_4$ (see Table \ref{Flataffine}). In this case, the flatness equations associated with the connection in \eqref{aff3} admit four distinct solutions:
\begin{align*}
	\nabla_{e_1}^1 e_2 &= e_1, &
	\nabla_{e_2} e_2 &= e_1 + e_2 + \tfrac{b_{33}(b_{33}-1)}{c_{33}} e_3, &
	\nabla_{e_2}^1 e_3 = \nabla_{e_3}^1 e_2 &= \tfrac{b_{33}c_{13}}{c_{33}} e_1 + b_{33} e_3, &
	\nabla_{e_3} e_3 &= c_{13} e_1 + c_{33} e_3,\,\,c_{33}\neq0.
	\\
	\nabla_{e_1}^2 e_2 &= e_1, &
	\nabla_{e_2}^2 e_2 &= e_1 + e_2 + b_{32} e_3, &
	\nabla_{e_2}^2 e_3 = \nabla_{e_3}^2 e_2 &= -b_{32}c_{13} e_1, &
	\nabla_{e_3}^2 e_3 &= c_{13} e_1.
	\\
	\nabla_{e_1}^3 e_2 &= e_1, &
	\nabla_{e_2}^3 e_2 &= e_1 + e_2 + b_{32} e_3, &
	\nabla_{e_2}^3 e_3& = \nabla_{e_3}^3 e_2 = b_{13} e_1 + e_3.
	\\
	\nabla_{e_1}^4 e_2 &= e_1, &
	\nabla_{e_1}^4 e_3 &= \nabla_{e_3}^4 e_1 = c_{33} e_1, &
	\nabla_{e_2}^4 e_2 &= e_1 + e_2, &
	\nabla_{e_2}^4 e_3 &= \nabla_{e_3}^4 e_2 = c_{33} e_2, &
	\nabla_{e_3}^4 e_3 &= c_{13} e_1 + c_{33} e_3.
\end{align*}
\textbf{Case 5.} If $\nabla^{0}\equiv\nabla^5$ which corresponds to the flat Lie algebra $\mathfrak{a}_5$ (see Table \ref{Flataffine}), then the flatness equations associated
with the connection given in \eqref{aff3} can be solved directly, giving three different solutions:
\begin{align*}
	\nabla_{e_1}^1 e_1 &= \varepsilon e_2, &
	\nabla_{e_1}^1 e_3 &= \nabla_{e_3}^1 e_1 = c_{33} e_1, &
	\nabla_{e_2}^1 e_1 &= -e_1, &
	\nabla_{e_2}^1 e_2 &= -2e_2, &
	\nabla_{e_2}^1 e_3 &= \nabla_{e_3}^1 e_2 = c_{33} e_2, \\
	\nabla_{e_3}^1 e_3 &= c_{33} e_3,
	&\varepsilon&=\pm 1.
	\\
	\nabla_{e_1}^2 e_1 &= \varepsilon e_2 + a_{31} e_3, &
	\nabla_{e_2}^2 e_1 &= -e_1, &
	\nabla_{e_2}^2 e_2 &= -2e_2 + b_{32} e_3, &
	\nabla_{e_2}^2 e_3 &= \nabla_{e_3}^2 e_2 = \tfrac{-\varepsilon b_{32}-2a_{31}}{a_{31}} e_3, &
	\nabla_{e_3}^2 e_3 &= \tfrac{\varepsilon(\varepsilon b_{32}+2a_{31})}{a_{31}^2} e_3,\\
	a_{31}&\neq0, \,\,\varepsilon=\pm 1 .
	\\
	\nabla_{e_1}^3 e_1 &= \varepsilon e_2, &
	\nabla_{e_2}^3 e_1 &= -e_1, &
	\nabla_{e_2}^3 e_2 &= -2e_2, &
	\nabla_{e_3}^3 e_3 &= c_{33} e_3,\,\, \varepsilon=\pm 1.
\end{align*}
\newpage
\restoregeometry
\subsection{Proof of Proposition~$\ref{FlatinR3}$}
Let $\nabla$ be a linear connection on $\R^3$,  viewed as $\R^2\oplus\R e_3$, with basis $\{e_1,e_2,e_3\}$. Then, $\nabla$ can be expressed as 
\begin{align}\label{Connegeneral}
\begin{split}
\nabla_xy&=\nabla^0_xy+\theta(x,y)e_3,\\
\nabla_xe_3&=\beta(x)+\gamma(x)e_3,\\
\nabla_{e_3}x&=\beta(x)+\gamma(x)e_3,\\
\nabla_{e_3}e_3&=\zeta+\lambda e_3,
\end{split}
\end{align}

for all $x, y \in \R^2$, where $\zeta\in\R^2$,  $\theta \in \mathcal{S}^2(\R^2)$ is a symmetric form, $\beta, \eta: \R^2 \to \R^2$ are endomorphisms of $\R^2$, $\gamma: \R \to \R$ is a one-form, and $\nabla^0$ is a torsion-free connection on $\R^2$. The curvature tensor $\mathcal{R}^\nabla$ of $\nabla$ is given by
\begin{align}\label{flatness-equations}
\mathcal{R}^\nabla(x,y)z=\nabla_x\nabla_yz-\nabla_y\nabla_xz,\quad\quad\text{for all~} x,y,z\in\R^3.
\end{align}
The condition for $\nabla$ to be flat is $\mathcal{R}^\nabla = 0$. We will refer to the corresponding system of equations as the \textit{flatness-equations}. In the basis $\{e_1, e_2, e_3\}$, the operators $\nabla_{e_1}$, $\nabla_{e_2}$, and $\nabla_{e_3}$ are given respectively by the matrices:
\begin{align}
\nabla_{e_1}&=\left( \begin {array}{ccc} a_{11}&a_{12}&a_{{13}}\\ \noalign{\medskip}a_{21}&a_{22}&a_{{23}}\\ \noalign{\medskip}a_{{31}}&a_{{32}}&a_{{33}}\end {array}
 \right),&\nabla_{e_2}&=\left( \begin {array}{ccc} a_{12}&b_{12}&b_{{13}}\\ \noalign{\medskip}a_{22}&b_{22}&b_{{23}}\\ \noalign{\medskip}a_{{32}}&b_{{32}}&b_{{33}}\end {array}
 \right),&\nabla_{e_3}&=\left( \begin {array}{ccc} a_{{13}}&b_{{13}}&c_{{13}}
\\ \noalign{\medskip}a_{{23}}&b_{{23}}&c_{{23}}
\\ \noalign{\medskip}a_{{33}}&b_{{33}}&c_{{33}}\end {array}
 \right),
\end{align}
where, $a_{ij}, b_{ij}, c_{ij}\in\R$.

Suppose now that $\nabla^0$ is flat, i.e., it has vanishing curvature. Then $\nabla^0$ is isomorphic to one of the flat torsion-free connections listed in Table~\ref{FlatR2}.

\textbf{Case 1.} If $\nabla^0 \equiv 0$, then it is straightforward to solve the flatness-equations associated with the connection given in $(\ref{Connegeneral})$, and one finds that they have 12 distinct solutions. 

The first solution corresponds to the following flat torsion-free connection:
\begin{align}\label{Sol01}
\begin{split}
\nabla_{e_1}e_1&=\tfrac{a_{32}^2}{b_{32}}e_3,\quad\nabla_{e_1}e_2=\nabla_{e_2}e_1=a_{32}e_3,\quad\nabla_{e_1}e_3=\nabla_{e_3}e_1=\tfrac{ \left(c_{33} -b_{23} \right)a_{32}^{2}-a_{33}^{2}b_{32}}{a_{32}^{2}}e_1+\tfrac{a_{32}b_{23}}{b_{32}}e_2+a_{33}e_3,\\
\nabla_{e_2}e_2&=b_{32}e_3,\quad\nabla_{e_2}e_3=\nabla_{e_3}e_2=-\tfrac{b_{32}\left((b_{23}-c_{33})a_{32}^2+a_{33}^2b_{32}\right)}{a_{32}^3}e_1+b_{23}e_2+\tfrac{a_{33}b_{32}}{a_{32}}e_3,\\
\nabla_{e_3}e_3&=-\tfrac{b_{32}\left((b_{23}-c_{33})a_{32}^2+a_{33}^2b_{32}\right)a_{33}}{a_{32}^4}e_1+\tfrac{a_{33}b_{32}}{a_{32}}e_2+c_{33}e_3,
\end{split}
\end{align}
where, $a_{ij},b_{ij},c_{ij}\in\R$  and $b_{32}a_{32}\neq0$. Next, we classify the previously obtained flat torsion-free connections up to isomorphism. Two flat torsion-free connections $\nabla^1$ and $\nabla^2$ on $\R^3$ are defined to be isomorphic if there exists an automorphism $\Psi \in \mathrm{GL}_3(\R)$ such that
\begin{align}\label{Relation}
\nabla^2_x = \Psi \circ \nabla^1_{\Psi^{-1}(x)} \circ \Psi^{-1}, \quad \text{for all } x \in \mathbb{R}^3.
\end{align}
Assume $a_{33} \neq 0$ and $b_{23}\neq0$. Then, by applying the isomorphism $\Psi \in \mathrm{GL}_3(\R)$ defined by
\begin{align*}
    \Psi(e_1) &=-a_{33}e_1,\quad
    \Psi(e_2) =-\tfrac{a_{33}^3b_{32}^2}{a_{32}^3b_{23}}e_1-\tfrac{a_{33}b_{32}}{a_{32}}e_3,\quad
    \Psi(e_3) =\tfrac{a_{33}^2b_{32}}{a_{32}^2}\big(e_2-e_3\big),
\end{align*}
to the connection in (\ref{Sol01}) via Relation~(\ref{Relation}), we obtain the isomorphic connection:
\begin{align}\label{Sol01,1}
    \nabla_{e_2} e_2 &= e_1+\lambda_1 e_2,\quad\nabla_{e_2} e_3 = e_1+\lambda_1 e_3,\quad\nabla_{e_3}e_2=e_1+\lambda_1 e_3,\quad\nabla_{e_3}e_3=e_2-e_3,\quad\lambda_1\in\R
\end{align}
For the case $a_{33} \neq 0$ and $b_{23} = 0$, the isomorphism $\Psi \in \mathrm{GL}_3(\R)$ defined by
\begin{align*}
    \Psi(e_1) &= a_{33}e_1, \quad
    \Psi(e_2) = e_1 + \tfrac{a_{33}b_{32}}{a_{32}}e_3, \quad
    \Psi(e_3) = \tfrac{a_{33}^2b_{32}}{a_{32}^2}(e_2 + e_3),
\end{align*}
transforms the connection in (\ref{Sol01}) into the isomorphic connection:
\begin{align}\label{Sol01,2}
    \nabla_{e_2} e_2 &= \lambda_2 e_2,\quad\nabla_{e_2} e_3 = \lambda_2 e_3,\quad\nabla_{e_3}e_2=\lambda_2 e_3,\quad\nabla_{e_3}e_3=e_2+e_3,\quad\lambda_2\in\R
\end{align}
Now, if $a_{33} = 0$ and $b_{23} \neq 0$, we can apply the isomorphism $\Psi \in \mathrm{GL}_3(\mathbb{R})$ defined by
\begin{align*}
    \Psi(e_1) &= xe_1, \quad
    \Psi(e_2) =\tfrac{x^3b_{32}^2}{a_{32}^3b_{23}} e_1 + \tfrac{xb_{32}}{a_{32}}e_3, \quad
    \Psi(e_3) = \tfrac{x^2b_{32}}{a_{32}^2}e_2,
\end{align*}
to the connection in (\ref{Sol01}) to obtain the isomorphic connection (for a suitable parameter $x \in \R^\ast$):
\begin{align}\label{Sol01,3}
    \nabla_{e_2} e_2 &= \varepsilon_1 e_2,\quad\nabla_{e_2} e_3 =e_1+\varepsilon_1 e_3,\quad\nabla_{e_3}e_2=e_1+\varepsilon_1 e_3,\quad\nabla_{e_3}e_3=e_2,\quad\varepsilon_1=0,\pm1.
\end{align}
If $a_{33}=0$ and $b_{23}=0$, we can apply the isomorphism $\Psi \in \mathrm{GL}_3(\mathbb{R})$ defined by
\begin{align*}
    \Psi(e_1) &= xe_1, \quad
    \Psi(e_2) = e_1 + \tfrac{xb_{32}}{a_{32}}e_3, \quad
    \Psi(e_3) = \tfrac{x^2b_{32}}{a_{32}^2}e_2,
\end{align*}
to the connection in (\ref{Sol01}) to obtain the isomorphic connection (for a suitable parameter $x \in \R^\ast$):
\begin{align}\label{Sol01,4}
    \nabla_{e_2} e_2 &= \varepsilon_2 e_2,\quad\nabla_{e_2} e_3 =\varepsilon_2 e_3,\quad\nabla_{e_3}e_2=\varepsilon_1 e_3,\quad\nabla_{e_3}e_3=e_2,\quad\varepsilon_2=0,\pm1.
\end{align}
The automorphism $\Psi \in \mathrm{GL}_3(\mathbb{R})$, defined by
\begin{align*}
    \Psi(e_1) &= -\tfrac{\lambda_1\sqrt{
    \varepsilon_1(4\lambda_1+1)}}{2\varepsilon_1^2}
e_1, \quad
    \Psi(e_2) = \tfrac{\sqrt{
    \varepsilon_1(4\lambda_1+1)}}{2\varepsilon_1^2}
e_1+\tfrac{\lambda_1}{\varepsilon_1}e_2, \quad
    \Psi(e_3) = -\tfrac{1}{2\varepsilon_1}e_2-\tfrac{\sqrt{
    \varepsilon_1(4\lambda_1+1)}}{2\varepsilon_1}
e_3,
\end{align*}
establishes an isomorphism between the connection in (\ref{Sol01,1}) and the one in (\ref{Sol01,3}) if and only if $\varepsilon_1 = \pm 1$ and $\lambda_1 \notin \{0, -\frac{1}{4}\}$. Similarly, the isomorphism $\Psi$ defined by
\begin{align*}
    \Psi(e_1) &=e_1, \quad
    \Psi(e_2) = \tfrac{\lambda_2}{\varepsilon_1}e_2, \quad
    \Psi(e_3) =\tfrac{\sqrt{
    \varepsilon_1(4\lambda_1+1)}}{2\varepsilon_1^2}e_1 +\tfrac{1}{2\varepsilon_1}e_2+\tfrac{\sqrt{
    \varepsilon_1(4\lambda_1+1)}}{2\varepsilon_1}
e_3,
\end{align*}
establishes an isomorphism between the connection in (\ref{Sol01,2}) and the one in (\ref{Sol01,3}) if and only if $\varepsilon_1 = \pm 1$ and $\lambda_2 \notin \{0, -\frac{1}{4}\}$.

In the same manner, the automorphism $\Psi$ given by
\begin{align*}
    \Psi(e_1) &=-\varepsilon_1 e_1, \quad
    \Psi(e_2) = \tfrac{\varepsilon_1}{\varepsilon_2}e_2, \quad
    \Psi(e_3) =e_1 +\tfrac{\sqrt{
    \varepsilon_1\varepsilon_2}}{\varepsilon_2}
e_3,
\end{align*}
yields an isomorphism from the connection (\ref{Sol01,3}) to the connection (\ref{Sol01,4}) if and only if $\varepsilon_1\varepsilon_2 \neq 0$.

Now the isomorphism $\Psi$ defined by
\begin{align*}
    \Psi(e_1) &= e_1, \quad
    \Psi(e_2) =-4e_1+e_2 , \quad
    \Psi(e_3) =4e_1-e_3,
\end{align*}
yields an isomorphism from the connection (\ref{Sol01,3}) to the connection (\ref{Sol01,4}) if and only if $\lambda_1=\lambda_2=-\frac{1}{4}$.

As a conclusion, the remaining non-isomorphic flat torsion-free connections are:
\begin{itemize}
    \item Connection (\ref{Sol01,1}) with $\lambda_1 = 0$,
    \item Connection (\ref{Sol01,2}) with $\lambda_2 = 0$ or $\lambda_2 = -\frac{1}{4}$,
    \item Connection (\ref{Sol01,3}) with $\varepsilon_1 = 0, \pm 1$,
    \item Connection (\ref{Sol01,4}) with $\varepsilon_2 = 0$.
\end{itemize}

Applying the automorphism $\Psi$ given by
\begin{align*}
    \Psi(e_1) &= e_1, \quad
    \Psi(e_2) = -\tfrac{1}{4}e_3, \quad
    \Psi(e_3) = e_2+\tfrac{1}{2}e_3,
\end{align*}
to connection (\ref{Sol01,2}) with $\lambda_2 = -\frac{1}{4}$, we obtain the following simplified connection:
\begin{align}
    \nabla_{e_2}e_3=e_2,\quad\nabla_{e_3}e_2=e_2,\quad\nabla_{e_3}e_3=e_3.
\end{align}

Furthermore, by applying the automorphism $\Psi$ defined by
\begin{align*}
    \Psi(e_1) &= e_1, \quad
    \Psi(e_2) = \varepsilon_1e_3, \quad
    \Psi(e_3) = -\tfrac{1}{\epsilon_1}e_1+e_2,
\end{align*}
to the connection given in (\ref{Sol01,3}) with $\varepsilon_1=\pm1$, we can show that the latter is isomorphic to the following:
\begin{align}
\nabla_{e_2}e_2=\varepsilon e_3,\quad\nabla_{e_2}e_3=e_2,\quad\nabla_{e_3}e_2=e_2,\quad\nabla_{e_3}e_3=e_3,\quad\varepsilon=\pm1.
\end{align}
We summarize these connections with the following notations:
\begin{align}
\h_{0,1}:&~~ &\nabla_{e_2}e_2&=e_1,&\nabla_{e_2}e_3&=e_1,&\nabla_{e_3}e_2&=e_1,&\nabla_{e_3}e_3&=e_2-e_3,\\
\h_{0,2}:&~~&\nabla_{e_3}e_3&=e_2+e_3,\\
\h_{0,3}:&~~&\nabla_{e_2}e_3&=e_2,&\nabla_{e_3}e_2&=e_2,&\nabla_{e_3}e_3&=e_3,\\
\h_{0,4}:&~~&\nabla_{e_2}e_3&=e_1,&\nabla_{e_3}e_2&=e_1,&\nabla_{e_3}e_3&=e_2,\\
\h_{0,5}:&~~&\nabla_{e_2}e_2&=\varepsilon e_3,&\nabla_{e_2}e_3&=e_2,&\nabla_{e_3}e_2&=e_2,&\nabla_{e_3}e_3&=e_3,\\
\h_{0,6}:&~~&\nabla_{e_3}e_3&=e_2,
\end{align}
Let us now consider the second solution, which is given by:
\begin{align}\label{Sol02}
\begin{split}
\nabla_{e_1}e_1=a_{31}e_3,\quad\nabla_{e_1}e_2=a_{32}e_3,\quad\nabla_{e_2}e_1=a_{32}e_3,\quad\nabla_{e_2}e_2=b_{32}e_3,
\end{split}
\end{align}
$a_{ij},b_{ij}\in\R$.

If $a_{31} \neq 0$, applying the following
\begin{align*}
    \Psi(e_1) &=e_1 , \quad
    \Psi(e_2) =\tfrac{a_{32}}{a_{31}}e_1+xe_2 , \quad
    \Psi(e_3) = \tfrac{1}{a_{31}}e_3,
\end{align*}
to the connection given in (\ref{Sol02}), we obtain  (for a suitable parameter $x \in \R^\ast$):
\begin{align}\label{Sol02,1}
\nabla_{e_1}e_1&=e_3,~~\nabla_{e_2}e_2=\varepsilon e_3,~~\varepsilon=0,\pm1.
\end{align}
If $a_{31} = 0$ and $a_{32} \neq 0$, we apply the automorphism $\Psi$ defined by
\begin{align*}
    \Psi(e_1) &= a_{32}e_1, \quad
    \Psi(e_2) = \tfrac{b_{32}}{2}e_1 + x e_2, \quad
    \Psi(e_3) =e_3,
\end{align*}
to the connection in (\ref{Sol02}). This yields
\begin{align}\label{Sol02,2}
    \nabla_{e_1} e_2 &= e_3, \quad
    \nabla_{e_2} e_1 =  e_3.
\end{align}
Now, if $a_{31} = 0$ and $a_{32} = 0$, we apply the automorphism $ \Psi$ defined by
\begin{align*}
    \Psi(e_1) &= e_1, \quad
    \Psi(e_2) = e_2, \quad
    \Psi(e_3) =x e_3,
\end{align*}
to the connection in (\ref{Sol02}). This yields (for a suitable parameter $x \in \R^\ast$)
\begin{align}\label{Sol02,3}
    \nabla_{e_2} e_2 =\varepsilon e_3,\quad\varepsilon =0,1.
\end{align}
It is easy to show that the connection given in (\ref{Sol02,1}) is isomorphic only to the connection given by $\mathfrak{h}_{0,6}$ with $\varepsilon = 0$ via the automorphism defined by
\begin{align*}
    \Psi(e_1) &= e_3, \quad
    \Psi(e_2) = e_1, \quad
    \Psi(e_3) = e_2.
\end{align*}
Otherwise, we consider the new flat Lie algebra given by:
\begin{align}
\h_{0,7}:&~~&\nabla_{e_1}e_1&=e_3,\quad&\nabla_{e_2}e_2&=\varepsilon e_3,\quad&\varepsilon=\pm1&.
\end{align}
In the same manner, the automorphism $\Psi$ defined by
\begin{align*}
    \Psi(e_1) = e_1-e_2, \quad
    \Psi(e_2) = e_1+e_2, \quad
    \Psi(e_3) = 2e_3,
\end{align*}
yields an isomorphism between the connection given in (\ref{Sol02,2}) and that given in by $\mathfrak{h}_{0,7}$.

Finally, the automorphism $\Psi$ defined by
\begin{align*}
    \Psi(e_1) &= e_1, \quad
    \Psi(e_2) = e_3, \quad
    \Psi(e_3) = e_2,
\end{align*}
yields an isomorphism between the connection given in (\ref{Sol02,3}) with $\varepsilon = 1$ and that given by $\h_{0,6}$. Otherwise, we denote by $(\h_{0,0}, \nabla)$ the vanishing flat Lie algebra, i.e., $\nabla \equiv 0$.

Let us consider the third solution given by
\begin{align}\label{Sol03}
\begin{split}
\nabla_{e_1}e_1&=a_{31}e_3,~~\nabla_{e_1}e_3=\nabla_{e_3}e_1=\tfrac{a_{31}c_{33}-a_{33}^2}{a_{31}}e_1+\tfrac{a_{31}c_{23}}{a_{33}}e_2+a_{33}e_3,\\
\nabla_{e_3}e_3&=\tfrac{(a_{31}c_{33}-a_{33}^2)a_{33}}{a_{31}}e_1+c_{23}e_2+c_{33}e_3,
\end{split}
\end{align}
where, $a_{ij},c_{ij}\in\R$,  and $a_{31}a_{33}\neq0$.

If $a_{33}^2-a_{31}c_{33}\neq0$, we apply the automorphism $\Psi$ defined by
\begin{align*}
    \Psi(e_1) &=\tfrac{a_{31}c_{33}-a_{33}^2}{a_{33}} e_1-\tfrac{a_{31}c_{23}}{a_{33}(a_{31}c_{33}-a_{33}^2)}e_2-\tfrac{a_{31}c_{33}-a_{33}^2}{a_{33}}e_3, \quad
    \Psi(e_2) = e_2, \quad
    \Psi(e_3) = \tfrac{a_{31}c_{33}-a_{33}^2}{a_{31}}e_1,
\end{align*}
to the third solution (\ref{Sol03}) to obtain
\begin{align}\label{Sol03,1}
\nabla_{e_1}e_1=\lambda e_1-e_3,~~\nabla_{e_1}e_3=e_1,~~\nabla_{e_3}e_1=e_1,~~\nabla_{e_3}e_3=e_3,\quad\lambda\in\R.
\end{align}
If $a_{33}^2-a_{31}c_{33}=0$ and $c_{23}\neq0$, we apply the automorphism $\Psi$ defined by
\begin{align*}
    \Psi(e_1) &=\tfrac{a_{31}c_{33}-a_{33}^2}{a_{33}} e_1-\tfrac{a_{31}c_{23}}{a_{33}(a_{31}c_{33}-a_{33}^2)}e_2-\tfrac{a_{31}c_{33}-a_{33}^2}{a_{33}}e_3, \quad
    \Psi(e_2) = e_2, \quad
    \Psi(e_3) = \tfrac{a_{31}c_{33}-a_{33}^2}{a_{31}}e_1,
\end{align*}
to the third solution (\ref{Sol03}) to obtain
\begin{align}\label{Sol03,2}
\nabla_{e_1}e_1&=e_1+ e_2,~~\nabla_{e_3}e_3=-e_2.
\end{align}

If $a_{33}^2-a_{31}c_{33}=0$ and $c_{23}=0$, we apply the automorphism $\Psi$ defined by
\begin{align*}
    \Psi(e_1) &=a_{33}e_1+e_3, \quad
    \Psi(e_2) =e_2, \quad
    \Psi(e_3) =\tfrac{a_{33}^2}{a_{31}}e_1,
\end{align*}
to the third solution (\ref{Sol03}) to obtain
\begin{align}\label{Sol03,3}
\nabla_{e_1}e_1&=e_1.
\end{align}

We consider the connection given in (\ref{Sol03,1}). If $\lambda= \pm 2$, the automorphism $\Psi$ defined by
\begin{align*}
    \Psi(e_1) &=e_2\pm e_3, \quad
    \Psi(e_2) = e_1, \quad
    \Psi(e_3) =  e_3,
\end{align*}
yields an isomorphism to the connection given by $\h_{0,3}$. If $\lambda\neq\pm2$, the automorphism $\Psi$ defined by
\begin{align*}
    \Psi(e_1) &=\tfrac{\sqrt{\varepsilon(\lambda^2-4)}}{2\varepsilon}e_2+\tfrac{\lambda}{2}e_3, \quad
    \Psi(e_2) = e_1, \quad
    \Psi(e_3) =  e_3,
\end{align*}
yields an isomorphism to the connection given by $\h_{0,5}$. We now consider the connection given in (\ref{Sol03,2}). The automorphism $\Psi$ defined by
\begin{align*}
    \Psi(e_1) &=e_2- e_3, \quad
    \Psi(e_2) = -e_1, \quad
    \Psi(e_3) =  e_2,
\end{align*}
yields an isomorphism to the connection given by $\h_{0,1}$.

The automorphism $\Psi$, defined by
\begin{align*}
    \Psi(e_1) &= e_2 + e_3, \quad
    \Psi(e_2) = e_2, \quad
    \Psi(e_3) = e_1,
\end{align*}
establishes an isomorphism between the connection given in $(\ref{Sol03,3})$ and the connection defined by the Lie algebra $\mathfrak{h}_{0,2}$.

The fourth solution is given by the following flat, torsion-free connection:
\begin{align}\label{Sol04}
\nabla_{e_1}e_1&=a_{31}e_3,~~\nabla_{e_1}e_3=\nabla_{e_3}e_1=a_{13}e_1+a_{23}e_2,~~\nabla_{e_3}e_3=a_{13}e_3,
\end{align}
where, $a_{ij}\in\R$. 

If $a_{13}\neq0$, we apply the automorphism defined by 
\begin{align*}
\Psi(e_1) &= -\tfrac{a_{23}}{a_{13}} e_2+xe_3, \quad
    \Psi(e_2) = e_2, \quad
    \Psi(e_3) =a_{13} e_1,
\end{align*}
to the connection given in $(\ref{Sol04})$, we obtain (for a suitable parameter $x \in \R^\ast$)
\begin{align}\label{Sol04,1}
\nabla_{e_1}e_1=e_3,~~\nabla_{e_1}e_3=e_1,~~\nabla_{e_3}e_1=e_3,~~\nabla_{e_3}e_3=\varepsilon_1e_1,~~\varepsilon=0,\pm1.
\end{align}
If $a_{13}=0$ and $a_{23}a_{31}\neq0$, we apply the automorphism $\Psi$ defined by
\begin{align*}
\Psi(e_1) &= e_3, \quad
    \Psi(e_2) =\tfrac{1}{a_{23}a_{31}} e_2, \quad
    \Psi(e_3) =\tfrac{1}{a_{31}} e_1,
\end{align*}
to the connection given in $(\ref{Sol04})$, we obtain
\begin{align}\label{Sol04,2}
\nabla_{e_1}e_3&=e_2,~~\nabla_{e_3}e_1=e_2,~~\nabla_{e_3}e_3=e_1.
\end{align}
If $a_{13}=0$ and $a_{23}=0$, we apply the automorphism $\Psi$ defined by
\begin{align*}
\Psi(e_1) &= e_3, \quad
    \Psi(e_2) = e_2, \quad
    \Psi(e_3) =x e_1,
\end{align*}
to the connection given by $(\ref{Sol04})$, we obtain (for a suitable parameter $x \in \R^\ast$)
\begin{align}\label{Sol04,3}
\nabla_{e_3}e_3&=\varepsilon_1e_1,~~
\varepsilon_1=0,1.
\end{align}

If $a_{13}=0$ and $a_{31}=0$, we apply the automorphism $\Psi$ defined by
\begin{align*}
\Psi(e_1) &= e_3, \quad
    \Psi(e_2) = e_2, \quad
    \Psi(e_3) =x e_1,
\end{align*}
to the connection given by $(\ref{Sol04})$, we obtain (for a suitable parameter $x \in \R^\ast$)
\begin{align}\label{Sol04,4}
\nabla_{e_1}e_3&=\nabla_{e_3}e_1=\varepsilon_1e_2,~~
\varepsilon_1=0,1.
\end{align}

We consider the connection given in $(\ref{Sol04,1})$. If $\varepsilon_1=0$, the automorphism $\Psi$ defined by
\begin{align*}
\Psi(e_1) &= e_3, \quad
    \Psi(e_2) = e_1, \quad
    \Psi(e_3) = e_2,
\end{align*}
yields an isomorphism to the connection given by $\h_{0,3}$. Othewise, if $\varepsilon_1=\pm1$, the automorphism $\Psi$ defined by
\begin{align*}
\Psi(e_1) &= e_3, \quad
    \Psi(e_2) = e_1, \quad
    \Psi(e_3) =\tfrac{\sqrt{\varepsilon\varepsilon_1}}{\varepsilon} e_2,
\end{align*}
yields an isomorphism to the connection given by $\h_{0,5}$.

The automorphism $\Psi$, defined by
\begin{align*}
\Psi(e_1) &= e_2, \quad
    \Psi(e_2) = e_1, \quad
    \Psi(e_3) =e_3,
\end{align*}
establishes an isomorphism between the connection given in $(\ref{Sol04,2})$ and the connection defined by the flat Lie algebra $\h_{0,4}$.

Consider the connection given in $(\ref{Sol04,3})$. If $\varepsilon_1=0$, then this connection is actualy the vanishing flat torsion-free connection. Otherwise, let $\varepsilon_1=1$. The automorphism $\Psi$ defined by 
\begin{align*}
\Psi(e_1) &= e_2, \quad
    \Psi(e_2) = e_1, \quad
    \Psi(e_3) =e_3,
\end{align*}
yields an isomorphism to the connection given by $\h_{0,6}$.

Let us consider the connection given in $(\ref{Sol04,4})$. If $\varepsilon_1=0$, then this connection is actualy the vanishing flat torsion-free connection. Otherwise, let $\varepsilon_1=1$. The automorphism $\Psi$ defined by 
\begin{align*}
\Psi(e_1) &= e_1-e_2, \quad
    \Psi(e_2) = 2e_3, \quad
    \Psi(e_3) =e_1+e_2,
\end{align*}
yields an isomorphism to the connection given by $\h_{0,7}$.

The fifth case is described by the following flat, torsion-free connection:
\begin{align}\label{Sol05}
	\nabla_{e_1}e_3&=\nabla_{e_3}e_1=a_{13}e_1+a_{23}e_2,~~\nabla_{e_2}e_3=\nabla_{e_3}e_2=\tfrac{b_{23}a_{13}}{a_{23}}e_1+b_{23}e_2,~~\nabla_{e_3}e_3=c_{13}e_1+c_{23}e_2+(a_{13}+b_{23})e_3,
\end{align}
where, $a_{ij},b_{ij},c_{ij}\in\R$.

Assume that $a_{13} + b_{23} \neq 0$. In this case, we apply the automorphism $\Psi$ defined by
\begin{align*}
	\Psi(e_1) &=\tfrac{a_{23}}{a_{13}+b_{23}} e_3, \quad
	\Psi(e_2) =e_1+ \tfrac{b_{23}}{a_{13}+b_{23}} e_3, \quad
	\Psi(e_3) =\tfrac{(b_{23}-a_{13})c_{23}+2a_{23}c_{13}}{(a_{13}+b_{23})^2}e_1+(a_{13}+b_{23})e_2+\tfrac{a_{23}c_{13}+b_{23}c_{23}}{(a_{13}+b_{23})^2}e_3,
\end{align*}
to the connection given in~\eqref{Sol05}, we obtain
\begin{align}\label{Sol05,1}
	\nabla_{e_2}e_2&=e_2,~~\nabla_{e_2}e_3=e_1+e_3,~~\nabla_{e_3}e_2=e_1+e_3.
\end{align}

Assuming that $a_{13} = -b_{23}$, we apply the automorphism $\Psi$ given by
\begin{align*}
	\Psi(e_1) &=xe_3, \quad
	\Psi(e_2) =e_1+\tfrac{xb_{23}}{a_{23}}e_3, \quad
	\Psi(e_3) =\tfrac{a_{23}}{x}e_2+\tfrac{xc_{23}}{2a_{23}}e_3,
\end{align*}
to the connection given in $(\ref{Sol05})$, we obtain (for a suitable parameter $x\in\R^\ast$)
\begin{align}\label{Sol05,2}
	\nabla_{e_2}e_2&=\varepsilon_1 e_3,~~\nabla_{e_2}e_3=e_1,~~\nabla_{e_3}e_2=e_1,~~\varepsilon_1=0,1.
\end{align}

Let us consider the connection given in~\eqref{Sol05,1}. 
The automorphism $\Psi$ defined by
\begin{align*}
	\Psi(e_1) &= -e_1, &
	\Psi(e_2) &= e_3, &
	\Psi(e_3) &= e_1 + e_2
\end{align*}
induces an isomorphism to the connection corresponding to $\h_{0,3}$.

Consider the connection given in~\eqref{Sol05,2}. 

If $\varepsilon_1 = 1$, the automorphism
\begin{align*}
	\Psi(e_1) &= e_1, &
	\Psi(e_2) &= e_3, &
	\Psi(e_3) &= e_2
\end{align*}
induces an isomorphism to the connection corresponding to $\h_{0,4}$. 

If $\varepsilon_1 = 0$, the automorphism
\begin{align*}
	\Psi(e_1) &= e_3, &
	\Psi(e_2) &= \tfrac{1}{2} e_1 - \tfrac{1}{2} e_2, &
	\Psi(e_3) &= e_1 + e_2
\end{align*}
establishes an isomorphism between the connection given in~\eqref{Sol05,2} and the connection defined by the flat Lie algebra $\h_{0,7}$ with $\varepsilon = -1$.

The sixth solution corresponds to the flat, torsion-free connection:
\begin{align}\label{Sol06}
	\nabla_{e_1}e_3&=\nabla_{e_3}e_1=a_{13}e_1,~~\nabla_{e_2}e_3=\nabla_{e_3}e_2=b_{13}e_1,~~\nabla_{e_3}e_3=c_{13}e_1+c_{23}e_2+a_{13}e_3.
\end{align}
If $a_{13}\neq0$, we apply the automorphism $\Psi$ defined by
\begin{align*}
	\Psi(e_1) &=e_3, &
	\Psi(e_2)&=e_1+\tfrac{b_{13}}{a_{13}}e_3, &
	\Psi(e_3) &=-\tfrac{c_{23}}{a_{13}}e_1+a_{13}e_2+\tfrac{a_{13}c_{13}+b_{13}c_{23}}{a_{13}^2}e_3,
\end{align*}
to the connection given in $(\ref{Sol06})$, we obtain
\begin{align}\label{Sol06,1}
	\nabla_{e_2}e_2&=e_2,~~\nabla_{e_2}e_3=\nabla_{e_3}e_2=e_3.
\end{align}

If $a_{13}=0$ and $b_{13}c_{23}\neq0$, we apply the automorphism $\Psi$ defined by
\begin{align*}
	\Psi(e_1) &=\tfrac{1}{b_{13}c_{23}} e_3, &
	\Psi(e_2) &=\tfrac{1}{c_{23}}e_1-\tfrac{c_{13}}{c_{23}^2b_{13}}e_3, &
	\Psi(e_3)&=e_2,
\end{align*}
to the connection given in $(\ref{Sol06})$, we obtain
\begin{align}\label{Sol06,2}
	\nabla_{e_1}e_2&=\nabla_{e_2}e_1=e_3,~~\nabla_{e_2}e_2=e_1.
\end{align}

Assume that $a_{13} = 0$, $b_{13} \neq 0$, and $c_{23} = 0$. 
In this case, we apply the automorphism $\Psi$ defined by
\begin{align*}
	\Psi(e_1) &=\tfrac{1}{b_{13}c_{23}} e_3, &
	\Psi(e_2) &=\tfrac{1}{c_{23}}e_1-\tfrac{c_{13}}{c_{23}^2b_{13}}e_3, &
	\Psi(e_3) &=e_2,
\end{align*}
to the connection given in $(\ref{Sol06})$, we obtain
\begin{align}\label{Sol06,3}
	\nabla_{e_1}e_2&=\nabla_{e_2}e_1=e_3.
\end{align}

If $a_{13}=0$, $b_{13}=0$ and $c_{23}\neq0$, we apply the automorphism $\Psi$ defined by
\begin{align*}
	\Psi(e_1) &= e_3, &
	\Psi(e_2)& =\tfrac{1}{c_{23}}e_1-\tfrac{c_{13}}{c_{23}}e_3, &
	\Psi(e_3) &=e_2,
\end{align*}
to the connection given in $(\ref{Sol06})$, we obtain
\begin{align}\label{Sol06,4}
	\nabla_{e_2}e_2&=e_1.
\end{align}

If $a_{13}=0$, $b_{13}=0$ and $c_{23}=0$, we apply the automorphism $\Psi$ defined by
\begin{align*}
	\Psi(e_1) &= xe_3, &
	\Psi(e_2) &=e_1, &
	\Psi(e_3) &=e_2,
\end{align*}
to the connection given in $(\ref{Sol06})$, we obtain (for a suitable parameter $x\in\R^\ast$)
\begin{align}\label{Sol06,5}
	\nabla_{e_2}e_2&=\delta e_3,~~\delta=0,1.
\end{align}

On the other hand, the automorphism $\Psi$, defined by
\begin{align*}
	\Psi(e_1) &= e_1, &
	\Psi(e_2) &= e_3, &
	\Psi(e_3) &= e_2,
\end{align*}
induces an isomorphism between the connection given in~\eqref{Sol06,1} and the connection corresponding to the flat Lie algebra $\h_{0,3}$.

The automorphism $\Psi$, given by
\begin{align*}
	\Psi(e_1) &= e_2, &
	\Psi(e_2) &= e_3, &
	\Psi(e_3) &= e_1,
\end{align*}
maps the connection in~\eqref{Sol06,2} to the connection corresponding to the flat Lie algebra $\h_{0,4}$.

The automorphism $\Psi$, defined by
\begin{align*}
	\Psi(e_1) &= \tfrac{1}{2} e_1 - \tfrac{1}{2} e_2, &
	\Psi(e_2) &= e_1 + e_2, &
	\Psi(e_3) &= e_3,
\end{align*}
transforms the connection in~\eqref{Sol06,3} into the one corresponding to the flat Lie algebra $\h_{0,7}$ with $\varepsilon = -1$.

The automorphism $\Psi$, defined by
\begin{align*}
	\Psi(e_1) &= e_2, &
	\Psi(e_2) &= e_3, &
	\Psi(e_3) &= e_1,
\end{align*}
maps the connection in~\eqref{Sol06,4} to the one corresponding to the flat Lie algebra $\h_{0,6}$.

Consider the connection given in~\eqref{Sol06,5}. 

If $\delta = 0$, this connection reduces to the trivial (vanishing) flat, torsion-free connection. 
Otherwise, for $\delta = 1$, the automorphism
\begin{align*}
	\Psi(e_1) &= e_2, &
	\Psi(e_2) &= e_1, &
	\Psi(e_3) &= e_3
\end{align*}
maps the connection to the one corresponding to the flat Lie algebra $\h_{0,6}$.

The seventh case corresponds to the following flat, torsion-free connection:
\begin{align}\label{Sol07}
	\nabla_{e_1}e_3&=\nabla_{e_3}e_1=b_{23}e_1,~~\nabla_{e_2}e_3=\nabla_{e_3}e_2=b_{23}e_2,~~\nabla_{e_3}e_3=c_{13}e_1+c_{23}e_2+b_{23}e_3.
\end{align}
If $b_{23} \neq 0$, we consider the automorphism
\begin{align*}
	\Psi(e_1) &= e_3, &
	\Psi(e_2) &= e_1, &
	\Psi(e_3) &= \tfrac{c_{23}}{b_{23}}\, e_1 + b_{23}\, e_2 + \tfrac{c_{13}}{b_{23}}\, e_3
\end{align*}
and apply it to the connection in~\eqref{Sol07}. This transforms the connection into
\begin{align}\label{Sol07,1}
	\nabla_{e_1}e_2&=\nabla_{e_2}e_1=e_1,~~\nabla_{e_2}e_2=e_2,~~\nabla_{e_2}e_3=\nabla_{e_3}e_2=e_3.
\end{align}
If $b_{23} = 0$ and $c_{23} \neq 0$, we consider the automorphism
\begin{align*}
	\Psi(e_1) &= e_3, &
	\Psi(e_2) &= \tfrac{1}{c_{23}}\, e_1 - \tfrac{c_{13}}{c_{23}}\, e_3, &
	\Psi(e_3) &= e_2
\end{align*}
and apply it to the connection in~\eqref{Sol07}, thereby transforming the connection into
\begin{align}\label{Sol07,2}
	\nabla_{e_2}e_2&=e_1.
\end{align}

Now, if $b_{23} = 0$ and $c_{23} = 0$, we consider the automorphism
\begin{align*}
	\Psi(e_1) &= e_3, &
	\Psi(e_2) &= \tfrac{1}{c_{23}}\, e_1 - \tfrac{c_{13}}{c_{23}}\, e_3, &
	\Psi(e_3) &= e_2
\end{align*}
and apply it to the connection in~\eqref{Sol07}, resulting in the transformed connection
\begin{align}\label{Sol07,3}
	\nabla_{e_2}e_2&=\delta e_3,~~\delta=0,1.
\end{align}

Consider the connection given in~\eqref{Sol07,1}. 
It is easy to verify that, for any automorphism $\Psi \in \mathrm{GL}_3(\mathbb{R})$, this connection cannot be isomorphic to any of the connections corresponding to the algebras $\h_{0,1}$--$\h_{0,7}$. 

Moreover, applying the automorphism
\begin{align*}
	\Psi(e_1) &= e_1, &
	\Psi(e_2) &= e_3, &
	\Psi(e_3) &= e_2
\end{align*}
to the connection in~\eqref{Sol07,1}, we obtain the following flat, torsion-free connection, denoted by the algebra $\h_{0,8}$:
\begin{align*}
	\h_{0,8}:\quad &\nabla_{e_1}e_3=e_1,~~\nabla_{e_2}e_3=e_2,~~\nabla_{e_3}e_j=e_j,~~j=1,2,3.
\end{align*}

The automorphism $\Psi$, defined by
\begin{align*}
	\Psi(e_1) &= e_2, &
	\Psi(e_2) &= e_3, &
	\Psi(e_3) &= e_1,
\end{align*}
maps the connection in~\eqref{Sol07,2} to the one corresponding to the flat Lie algebra $\h_{0,6}$.

Consider the connection given in~\eqref{Sol07,3}. 

If $\delta = 0$, this connection reduces to the trivial (vanishing) flat, torsion-free connection. 
Otherwise, for $\delta = 1$, the automorphism
\begin{align*}
	\Psi(e_1) &= e_1, &
	\Psi(e_2) &= e_3, &
	\Psi(e_3) &= e_2
\end{align*}
maps the connection to the one corresponding to the flat Lie algebra $\h_{0,6}$.

The eighth solution is given by the following flat, torsion-free connection:
\begin{align}\label{Sol08}
	\nabla_{e_1}e_3&=\nabla_{e_3}e_1=a_{23}e_2,~~\nabla_{e_2}e_3=\nabla_{e_3}e_2=c_{33}e_2,~~\nabla_{e_3}e_3=c_{13}e_1+c_{23}e_2+c_{33}e_3.
\end{align}
If $a_{23} c_{33} \neq 0$, we consider the automorphism
\begin{align*}
	\Psi(e_1) &= \tfrac{a_{23}}{c_{33}}\, e_3, &
	\Psi(e_2) &= e_1 + e_3, &
	\Psi(e_3) &= \tfrac{2 a_{23} c_{13} + c_{23} c_{33}}{c_{33}^2}\, e_1 + c_{33}\, e_2 + \tfrac{a_{23} c_{13} + c_{23} c_{33}}{c_{33}^2}\, e_3
\end{align*}
and apply it to the connection in~\eqref{Sol08}, thereby transforming the connection into
\begin{align}\label{Sol08,1}
	\nabla_{e_2}e_2 &= e_2,~~\nabla_{e_2}e_3=\nabla_{e_3}e_2=e_1+e_3. 
\end{align}

If $a_{23} \neq 0$, $c_{33} = 0$, and $c_{13} \neq 0$, we consider the automorphism
\begin{align*}
	\Psi(e_1) &= \tfrac{1}{c_{13}}\, e_3, &
	\Psi(e_2) &= \tfrac{1}{c_{13} a_{23}}\, e_1, &
	\Psi(e_3) &= e_2 + \tfrac{c_{23}}{2\, c_{13} a_{23}}\, e_3
\end{align*}
and apply it to the connection in~\eqref{Sol08}, thereby transforming the connection into
\begin{align}\label{Sol08,2}
	\nabla_{e_2}e_2 &= e_3,~~\nabla_{e_2}e_3=\nabla_{e_3}e_2=e_1. 
\end{align}

If $a_{23} \neq 0$, $c_{33} = 0$, and $c_{13} = 0$, we consider the automorphism
\begin{align*}
	\Psi(e_1) &= a_{23}\, e_3, &
	\Psi(e_2) &= e_1, &
	\Psi(e_3) &= e_2 + \frac{c_{23}}{2}\, e_3
\end{align*}
and apply it to the connection in~\eqref{Sol08}, thereby transforming the connection into
\begin{align}\label{Sol08,3}
	\nabla_{e_2}e_3&=\nabla_{e_3}e_2=e_1. 
\end{align}

Now, if $a_{23} = 0$ and $c_{33} \neq 0$, we consider the automorphism
\begin{align*}
	\Psi(e_1) &= e_3, &
	\Psi(e_2) &= e_1, &
	\Psi(e_3) &= \tfrac{c_{23}}{c_{33}}\, e_1 + c_{33}\, e_2 - \tfrac{c_{13}}{c_{33}}\, e_3
\end{align*}
and apply it to the connection in~\eqref{Sol08}, producing the transformed connection
\begin{align}\label{Sol08,4}
	\nabla_{e_1}e_2&=\nabla_{e_2}e_1=e_1,~~\nabla_{e_2}e_2=e_2. 
\end{align}

 If $a_{23} = 0$ and $c_{33} = 0$, we consider the automorphism
 \begin{align*}
 	\Psi(e_1) &= y\, e_3, &
 	\Psi(e_2) &= x\, e_1, &
 	\Psi(e_3) &= e_2
 \end{align*}
 and apply it to the connection in~\eqref{Sol08}, resulting in the following transformed connection
  (for a suitable parameters $x,y\in\R^\ast$)
\begin{align}\label{Sol08,5}
	\nabla_{e_2}e_2&=\delta_1 e_1+\delta_2e_3,~~\delta_1,\delta_2=0,1. 
\end{align}

If $\delta_1 = 0$, the connection in~\eqref{Sol08,5} coincides with the one given in~\eqref{Sol02,3}. 
Moreover, if $\delta_2 = 0$ and $\delta_1 = 1$, it coincides with the connection in~\eqref{Sol06,2}. 

Now, consider the case $\delta_1 = \delta_2 = 1$. The automorphism
\begin{align*}
	\Psi(e_1) &= -e_1 + e_2, &
	\Psi(e_2) &= e_3, &
	\Psi(e_3) &= e_1
\end{align*}
maps the connection in~\eqref{Sol08,5} to the one corresponding to the flat Lie algebra $\h_{0,6}$.

Consider the connection given in $(\ref{Sol08,1})$. The automorphism $\Psi$ defined by 
\begin{align*}
	\Psi(e_1) &= e_1, \quad
	\Psi(e_2) = e_3, \quad
	\Psi(e_3) =-e_1+e_2,
\end{align*}
yields an isomorphism to the connection given by $\h_{0,3}$.

Let us consider the connection given in $(\ref{Sol08,2})$. The automorphism $\Psi$ defined by 
\begin{align*}
	\Psi(e_1) &= e_1, \quad
	\Psi(e_2) = e_3, \quad
	\Psi(e_3) =e_2,
\end{align*}
yields an isomorphism to the connection given by $\h_{0,4}$.

Let us consider the connection given in $(\ref{Sol08,3})$. The automorphism $\Psi$ defined by 
\begin{align*}
	\Psi(e_1) &= e_3, \quad
	\Psi(e_2) =\tfrac{1}{2} e_1-\tfrac{1}{2}e_2, \quad
	\Psi(e_3) =e_1+e_2,
\end{align*}
yields an isomorphism to the connection given by $\h_{0,7}$ with $\varepsilon=-1$.

The automorphism
\begin{align*}
	\Psi(e_1) &= e_2, &
	\Psi(e_2) &= e_3, &
	\Psi(e_3) &= e_1
\end{align*}
maps the connection in~\eqref{Sol08,4} to the one corresponding to the flat Lie algebra $\h_{0,3}$.

The ninth solution corresponds to the following flat, torsion-free connection:
\begin{align}\label{Sol09}
	\nabla_{e_2}e_2&=b_{32}e_3,~~\nabla_{e_2}e_3=\nabla_{e_3}e_2=\tfrac{b_{32}c_{13}}{b_{33}}e_1+\tfrac{c_{33}b_{32}-b_{33}^2}{b_{32}}e_2+b_{33}e_3,~~\nabla_{e_3}e_3=c_{13}e_1+\tfrac{(c_{33}b_{32}-b_{33}^2)b_{33}}{b_{32}^2}+c_{33}e_3,
\end{align}
If $c_{33} b_{32} - b_{33}^2 \neq 0$, we consider the automorphism
\begin{align*}
	\Psi(e_1) &= e_2, &
	\Psi(e_2) &= e_3, &
	\Psi(e_3) &= e_1
\end{align*}
and apply it to the connection in~\eqref{Sol09}, producing the following transformed connection
\begin{align}\label{Sol09,1}
	\nabla_{e_2}e_2&=e_2+e_3,~~\nabla_{e_2}e_3=\nabla_{e_3}e_2=\lambda e_2,~~\nabla_{e_3}e_3=\lambda e_3,~~\lambda\in\R^\ast.
\end{align}

If $c_{33} b_{32} - b_{33}^2 = 0$ and $c_{13} \neq 0$, we consider the automorphism
\begin{align*}
	\Psi(e_1) &= \tfrac{b_{33}^4}{b_{32}^2 c_{13}}\, e_1, &
	\Psi(e_2) &= b_{33}\, e_2, &
	\Psi(e_3) &= \tfrac{b_{33}^2}{b_{32}}\, e_2 + \tfrac{b_{33}^2}{b_{32}}\, e_3
\end{align*}
and apply it to the connection in~\eqref{Sol09}, which yields the following transformed connection
\begin{align}\label{Sol09,2}
	\nabla_{e_2}e_2&=e_2+e_3,~~\nabla_{e_2}e_3=\nabla_{e_3}e_2=e_1,~~\nabla_{e_3}e_3=-e_1.
\end{align}

Now, if $c_{33} b_{32} - b_{33}^2 = 0$ and $c_{13} = 0$, we consider the automorphism
\begin{align*}
	\Psi(e_1) &= e_1, &
	\Psi(e_2) &= b_{33}\, e_2, &
	\Psi(e_3) &= \tfrac{b_{33}^2}{b_{32}}\, e_2 + \tfrac{b_{33}^2}{b_{32}}\, e_3
\end{align*}
and apply it to the connection in~\eqref{Sol09}, which yields the following transformed connection:
\begin{align}\label{Sol09,3}
	\nabla_{e_2}e_2&=e_2+e_3.
\end{align}

Consider the connection in~\eqref{Sol09,1}. 

If $\lambda = -\frac{1}{4}$, the automorphism
\begin{align*}
	\Psi(e_1) &= e_1, &
	\Psi(e_2) &= e_2 + \tfrac{1}{2} e_3, &
	\Psi(e_3) &= -\tfrac{1}{4} e_3
\end{align*}
maps this connection to the one corresponding to the flat Lie algebra $\h_{0,3}$. 

If $\lambda \neq -\frac{1}{4}$, the automorphism
\begin{align*}
	\Psi(e_1) &= e_1, &
	\Psi(e_2) &= \tfrac{\sqrt{\varepsilon (4\lambda + 1)}}{2 \varepsilon}\, e_2 + \tfrac{1}{2} e_3, &
	\Psi(e_3) &= \lambda\, e_3
\end{align*}
maps the connection to the one defined by the flat Lie algebra $\h_{0,5}$.

Let us consider the connection given in $(\ref{Sol09,2})$. The automorphism $\Psi$, defined by
\begin{align*}
	\Psi(e_1) &=-e_1, \quad
	\Psi(e_2) =-e_3, \quad
	\Psi(e_3) =e_2,
\end{align*}
establishes an isomorphism between the connection given in $(\ref{Sol09,2})$ and the connection defined by the flat Lie algebra $\h_{0,1}$.

Let us consider the connection given in $(\ref{Sol09,3})$. The automorphism $\Psi$, defined by
\begin{align*}
	\Psi(e_1) &=e_1, \quad
	\Psi(e_2) =e_3, \quad
	\Psi(e_3) = e_2,
\end{align*}
establishes an isomorphism between the connection given in $(\ref{Sol09,3})$ and the connection defined by the flat Lie algebra $\h_{0,2}$.

The tenth solution corresponds to the following flat, torsion-free connection:
\begin{align}\label{Sol010}
	\nabla_{e_2}e_2&=b_{32}e_3,~~\nabla_{e_2}e_3=\nabla_{e_3}e_2=b_{13}e_1+c_{33}e_2,~~\nabla_{e_3}e_3=c_{33}e_3.
\end{align}

If $b_{32} c_{33} \neq 0$, we consider the automorphism
\begin{align*}
	\Psi(e_1) &= e_1, &
	\Psi(e_2) &= -\tfrac{b_{13}}{c_{33}}\, e_1 + x e_2, &
	\Psi(e_3) &= c_{33}\, e_3
\end{align*}
and apply it to the connection in~\eqref{Sol010}, which yields the transformed connection for a suitable $x \in \mathbb{R}^\ast$.
\begin{align}\label{Sol010,1}
	\nabla_{e_2}e_2&=\varepsilon_1e_3,~~\nabla_{e_2}e_3=\nabla_{e_3}e_2=e_2,~~\nabla_{e_3}e_3=e_3,~~\varepsilon_1=\pm1.
\end{align}
The connection in~\eqref{Sol010,1} actually coincides with the one corresponding to the flat Lie algebra $\h_{0,5}$.

If $c_{33} \neq 0$ and $b_{32} = 0$, we consider the automorphism
\begin{align*}
	\Psi(e_1) &= e_1, &
	\Psi(e_2) &= -\tfrac{b_{13}}{c_{33}}\, e_1 + e_2, &
	\Psi(e_3) &= c_{33}\, e_3
\end{align*}
and apply it to the connection in~\eqref{Sol010}, which yields the corresponding transformed connection: 
\begin{align}\label{Sol010,2}
	\nabla_{e_2}e_3&=\nabla_{e_3}e_2=e_2,~~\nabla_{e_3}e_3=e_3.
\end{align}
The connection in~\eqref{Sol010,2} actually coincides with the one corresponding to the flat Lie algebra $\h_{0,3}$.

If $c_{33} = 0$, $b_{13} \neq 0$, and $b_{32} \neq 0$, we consider the automorphism
\begin{align*}
	\Psi(e_1) &= e_1, &
	\Psi(e_2) &= \tfrac{b_{13} b_{32}}{(b_{13}^2 b_{32}^2)^{1/3}}\, e_2, &
	\Psi(e_3) &= \tfrac{(b_{13}^2 b_{32}^2)^{1/3}}{b_{32}}\, e_3
\end{align*}
and apply it to the connection in~\eqref{Sol010}, which yields the transformed connection: 
\begin{align}\label{Sol010,2}
	\nabla_{e_2}e_2&=e_3,~~\nabla_{e_2}e_3=\nabla_{e_3}e_2=e_1.
\end{align}

If $c_{33} = 0$, $b_{13} \neq 0$, and $b_{32} = 0$, we consider the automorphism
\begin{align*}
	\Psi(e_1) &= e_1, &
	\Psi(e_2) &= b_{13}\, e_2, &
	\Psi(e_3) &= e_3
\end{align*}
and apply it to the connection in~\eqref{Sol010}, which yields the corresponding transformed connection:
\begin{align}\label{Sol010,3}
	\nabla_{e_2}e_3&=\nabla_{e_3}e_2=e_1.
\end{align}

If $c_{33} = 0$ and $b_{13} = 0$, we consider the automorphism
\begin{align*}
	\Psi(e_1) &= e_1, &
	\Psi(e_2) &= e_2, &
	\Psi(e_3) &= x\, e_3
\end{align*}
and apply it to the connection in~\eqref{Sol010}, which yields the transformed connection for a suitable $x \in \mathbb{R}^\ast$:
\begin{align}\label{Sol010,4}
	\nabla_{e_2}e_2&=\delta e_3,~~\delta=0,1.
\end{align}

Consider the connection in~\eqref{Sol010,2}. The automorphism
\begin{align*}
	\Psi(e_1) &= e_1, &
	\Psi(e_2) &= e_3, &
	\Psi(e_3) &= e_2
\end{align*}
maps this connection to the one corresponding to the flat Lie algebra $\h_{0,4}$.

Consider the connection in~\eqref{Sol010,3}. The automorphism
\begin{align*}
	\Psi(e_1) &= e_3, &
	\Psi(e_2) &= \tfrac{1}{2} e_1 - \tfrac{1}{2} e_2, &
	\Psi(e_3) &= e_1 + e_2
\end{align*}
maps this connection to the one corresponding to the flat Lie algebra $\h_{0,7}$ with $\varepsilon = -1$.

Consider the connection in~\eqref{Sol010,4}. This connection coincides with the solution in~\eqref{Sol02,3}, which is isomorphic to the flat Lie algebra $\h_{0,6}$ when $\delta = 1$. Otherwise, if $\delta = 0$, it corresponds to the vanishing connection $\h_{0,0}$.

The 11th solution is given by the following flat, torsion-free connection:
\begin{align}\label{Sol011}
	\nabla_{e_2}e_3&=\nabla_{e_3}e_2=b_{13}e_1+c_{33}e_2,~~\nabla_{e_3}e_3=c_{13}e_1+c_{23}e_2+c_{33}e_3.
\end{align}

If $c_{33} \neq 0$ and $b_{13}c_{23} - c_{13}c_{33} \neq 0$, consider the automorphism
\begin{align*}
	\Psi(e_1) &= -\tfrac{c_{33}^3}{b_{13}c_{23}-c_{13}c_{33}}\, e_1, &
	\Psi(e_2) &= \tfrac{c_{33}^2 b_{13}}{b_{13}c_{23}-c_{13}c_{33}}\, e_1 + e_2, &
	\Psi(e_3) &= \tfrac{c_{23}}{c_{33}}\, e_2 + c_{33}\, e_3
\end{align*}
and apply it to the connection in~\eqref{Sol011}, yielding the corresponding transformed connection:
\begin{align}\label{Sol011,1}
	\nabla_{e_2}e_3&=\nabla_{e_3}e_2=e_2,~~\nabla_{e_3}e_3=e_1+e_3.
\end{align}

If $c_{33} \neq 0$ and $b_{13}c_{23} - c_{13}c_{33} = 0$, consider the automorphism
\begin{align*}
	\Psi(e_1) &= e_1, &
	\Psi(e_2) &= -\tfrac{b_{13}}{c_{33}}\, e_1 + e_2, &
	\Psi(e_3) &= \tfrac{c_{23}}{c_{33}}\, e_2 + c_{33}\, e_3
\end{align*}
and apply it to the connection in~\eqref{Sol011}, yielding the transformed connection:
\begin{align}\label{Sol011,2}
	\nabla_{e_2}e_3&=\nabla_{e_3}e_2=e_2,~~\nabla_{e_3}e_3=e_3.
\end{align}
The connection in~\eqref{Sol011,2} coincides with the one corresponding to the flat Lie algebra $\h_{0,3}$.

If $c_{33} = 0$, $b_{13} \neq 0$, and $c_{23} \neq 0$, consider the automorphism
\begin{align*}
	\Psi(e_1) &= e_1, &
	\Psi(e_2) &= -\tfrac{(b_{13}^2 c_{23}^2)^{1/3}}{c_{23}}\, e_2, &
	\Psi(e_3) &= \tfrac{(b_{13}^2 c_{23}^2)^{1/3} c_{13}}{2\, c_{23} b_{13}}\, e_2 + \tfrac{b_{13} c_{23}}{(b_{13}^2 c_{23}^2)^{1/3}}\, e_3
\end{align*}
and apply it to the connection in~\eqref{Sol011}, yielding the transformed connection:
\begin{align}\label{Sol011,3}
	\nabla_{e_2}e_3&=\nabla_{e_3}e_2=e_1,~~\nabla_{e_3}e_3=e_2.
\end{align}
The connection in~\eqref{Sol011,3} coincides with the one corresponding to the flat Lie algebra $\h_{0,4}$.

If $c_{33} = 0$, $b_{13} \neq 0$, and $c_{23} = 0$, consider the automorphism
\begin{align*}
	\Psi(e_1) &= e_1, &
	\Psi(e_2) &= e_2, &
	\Psi(e_3) &= \tfrac{c_{13}}{2\,b_{13}}\, e_2 + b_{13}\, e_3
\end{align*}
and apply it to the connection in~\eqref{Sol011}, yielding the transformed connection:
\begin{align}\label{Sol011,4}
	\nabla_{e_2}e_3&=\nabla_{e_3}e_2=e_1.
\end{align}

If $c_{33} = 0$, $b_{13} = 0$, and $c_{23} \neq 0$, consider the automorphism
\begin{align*}
	\Psi(e_1) &= e_1, &
	\Psi(e_2) &= -\tfrac{c_{13}}{c_{23}}\, e_1 + \tfrac{1}{c_{23}}\, e_2, &
	\Psi(e_3) &= e_3
\end{align*}
and apply it to the connection in~\eqref{Sol011}, yielding the transformed connection:
\begin{align}\label{Sol011,5}
	\nabla_{e_3}e_3&=e_2.
\end{align}
This connection coincides with the one corresponding to the flat Lie algebra $\h_{0,6}$.

If $c_{33} = 0$, $b_{13} = 0$, and $c_{23} = 0$, consider the automorphism
\begin{align*}
	\Psi(e_1) &= x\, e_1, &
	\Psi(e_2) &= e_2, &
	\Psi(e_3) &= e_3
\end{align*}
applied to the connection in~\eqref{Sol011}, where $x \in \R^\ast$ is a suitable parameter. This yields
\begin{align}\label{Sol011,6}
	\nabla_{e_3} e_3 &= \delta\, e_1, \quad \delta = 0,1.
\end{align}

Consider the connection given in~\eqref{Sol011,1}. The automorphism
\begin{align*}
	\Psi(e_1) &= e_1, &
	\Psi(e_2) &= e_2, &
	\Psi(e_3) &= -e_1 + e_3
\end{align*}
establishes an isomorphism between the connection~\eqref{Sol011,1} and the connection corresponding to the flat Lie algebra $\h_{0,3}$.

Consider the connection given in~\eqref{Sol011,4}. The automorphism
\begin{align*}
	\Psi(e_1) &= e_3, &
	\Psi(e_2) &= \tfrac{1}{2} e_1 - \tfrac{1}{2} e_2, &
	\Psi(e_3) &= e_1 + e_2
\end{align*}
establishes an isomorphism between the connection~\eqref{Sol011,4} and the connection corresponding to the flat Lie algebra $\h_{0,7}$ with $\varepsilon = -1$.

Consider the connection given in~\eqref{Sol011,6}. If $\delta = 0$, this connection coincides with the vanishing connection. Otherwise, for $\delta = 1$, the automorphism
\begin{align*}
	\Psi(e_1) &= e_2, &
	\Psi(e_2) &= e_1, &
	\Psi(e_3) &= e_3
\end{align*}
establishes an isomorphism between the connection~\eqref{Sol011,6} and the connection corresponding to the flat Lie algebra $\h_{0,6}$.

The 12th solution is given by the following flat, torsion-free connection:
\begin{align}\label{Sol012}
	\nabla_{e_3}e_3&=c_{13}e_1+c_{23}e_2+c_{33}e_3.
\end{align}
If $c_{33} \neq 0$ and $c_{23} \neq 0$, consider the automorphism
\begin{align*}
	\Psi(e_1) &=e_1 , &
	\Psi(e_2) &=-\tfrac{c_{13}}{c_{23}}e_1+e_2, &
	\Psi(e_3) &= -\tfrac{c_{23}}{c_{33}}e_1+c_{33}e_3
\end{align*}
and apply it to the connection in~\eqref{Sol012}, yielding the transformed connection:
\begin{align}\label{Sol012,1}
	\nabla_{e_3}e_3&=e_3.
\end{align}
This connection coincides with the one described in~\eqref{Sol04,3} for $\delta=1$.

Assume that $c_{33}\neq 0$ and $c_{23}=0$. Consider the automorphism
\begin{align*}
	\Psi(e_1) &= x\,e_1, &
	\Psi(e_2) &= e_2, &
	\Psi(e_3) &= c_{33}e_3,
\end{align*}
and apply it to the connection~\eqref{Sol012}. For a suitable choice of $x\in\R^\ast$, this yields the equivalent connection
\begin{align}\label{Sol012,2}
	\nabla_{e_3}e_3 = \delta e_1 + e_3, \qquad \delta = 0,1.
\end{align}
If $\delta=0$, the connection reduces to a case already treated above. 
Otherwise, when $\delta=1$, the automorphism
\begin{align*}
	\Psi(e_1) &= e_1, &
	\Psi(e_2) &= e_2, &
	\Psi(e_3) &= -e_1+e_3,
\end{align*}
defines an isomorphism between the connection~\eqref{Sol012,2} and the one associated with the flat Lie algebra $\h_{0,2}$.

Assume now that $c_{33}=0$ and $c_{23}\neq0$. Consider the automorphism
\begin{align*}
	\Psi(e_1) &= e_1, &
	\Psi(e_2) &=-\tfrac{c_{13}}{c_{23}} e_1+\tfrac{1}{c_{23}} e_2, &
	\Psi(e_3) &=e_3,
\end{align*}
and apply it to the connection~\eqref{Sol012}. 
For a suitable choice of $x\in\R^\ast$, this produces the equivalent connection
\begin{align}\label{Sol012,3}
	\nabla_{e_3}e_3 = e_2.
\end{align}
This connection coincides with the one associated with the flat Lie algebra $\h_{0,6}$.

Assume now that $c_{33}=0$ and $c_{23}=0$. Consider the automorphism
\begin{align*}
	\Psi(e_1) &= x\,e_1, &
	\Psi(e_2) &= e_2, &
	\Psi(e_3) &= e_3,
\end{align*}
and apply it to the connection~\eqref{Sol012}. 
For a suitable choice of $x\in\R^\ast$, this yields the equivalent connection
\begin{align}\label{Sol012,4}
	\nabla_{e_3}e_3 = \delta\,e_1, \qquad \delta=0,1.
\end{align}

If $\delta=0$, this is the vanishing connection.  
If $\delta=1$, the automorphism
\begin{align*}
	\Psi(e_1) &= e_2, &
	\Psi(e_2) &= e_1, &
	\Psi(e_3) &= e_3,
\end{align*}
establishes an isomorphism between the connection in~\eqref{Sol012,4} and the one corresponding to the flat Lie algebra $\h_{0,6}$.

Up to this point, we have identified the vanishing flat Lie algebra $\h_{0,0}$, as well as eight nontrivial flat Lie algebras, which we denote by $\h_{0,1}, \dots, \h_{0,8}$.

\textbf{Case 2.} If $\nabla^0 \equiv \nabla^1$, which corresponds to the Lie algebra $\mathfrak{b}_1$ (see Table~\ref{FlatR2}), then the flatness equations associated with the connection given in~\eqref{Connegeneral} can be solved directly, yielding twelve distinct solutions.

We begin with the first solution, which is given by the following flat, torsion-free connection:
\begin{equation}\label{Sol11}
	\begin{aligned}
		\nabla_{e_1}e_1 &= e_1+\tfrac{a_{32}(a_{32}b_{33}-b_{32})}{b_{32}b_{33}}e_3, 
		&\nabla_{e_1}e_2&=\nabla_{e_2}e_1=a_{32}e_3, 
		&\nabla_{e_1}e_3&=\nabla_{e_3}e_1=\tfrac{a_{32}b_{33}}{b_{32}}e_3,\\
		\nabla_{e_2}e_2&=b_{32}e_3, 
		&\nabla_{e_2}e_3&=\nabla_{e_3}e_2=b_{33}e_3, 
		&\nabla_{e_3}e_3&=\tfrac{b_{33}^2}{b_{32}}e_3.
	\end{aligned}
\end{equation}
Note that $b_{32}b_{33}\neq 0$. 
Consider the automorphism
\begin{align*}
	\Psi(e_1) &= e_1+\tfrac{a_{32}b_{33}}{b_{32}}e_2+e_3, &
	\Psi(e_2) &= b_{33}e_2, &
	\Psi(e_3) &= \tfrac{b_{33}^2}{b_{32}}e_2,
\end{align*}
and apply it to the connection~\eqref{Sol11}. 
This yields the equivalent connection
\begin{align}\label{Sol11,1}
	\nabla_{e_2}e_2&=e_2, &\nabla_{e_3}e_3&=e_1+e_3.
\end{align}
It is straightforward to verify that, for any automorphism 
$\Psi \in \mathrm{GL}_4(\mathbb{R})$, the connection given in~\eqref{Sol11,1}
is not isomorphic to any of the previously obtained connections, namely
$\h_{0,1},\dots,\h_{0,8}$.
Fixing $\nabla^0=\nabla^1$, we apply the following automorphism
\begin{align*}
	\Psi(e_1) &= e_2, &
	\Psi(e_2) &= e_1, &
	\Psi(e_3) &= -e_2+e_3,
\end{align*}
to the connection given in~\eqref{Sol11,1}. 
This transformation yields a new flat Lie algebra, denoted by
\begin{align}
	\h_{1,1}:\qquad
	\nabla_{e_1}e_1 &= e_1, &
	\nabla_{e_3}e_3 &= e_3.
\end{align}

The second solution is described by the following flat, torsion-free connection:
\begin{equation}\label{Sol12}
	\begin{aligned}
		\nabla_{e_1}e_1 &= e_1+a_{31}e_3, 
		&\nabla_{e_1}e_2&=\nabla_{e_2}e_1=a_{32}e_3, 
		&\nabla_{e_1}e_3&=\nabla_{e_3}e_1=e_3.
	\end{aligned}
\end{equation}

Consider the automorphism
\begin{align*}
	\Psi(e_1) &=a_{31}e_2+e_3, &
	\Psi(e_2) &=e_1+a_{32}e_2, &
	\Psi(e_3) &=e_2,
\end{align*}
and apply it to the connection~\eqref{Sol12}. 
This yields the equivalent connection
\begin{align}\label{Sol12,1}
	\nabla_{e_2}e_3&=e_2, &\nabla_{e_3}e_2&=e_2, &\nabla_{e_3}e_3&=e_3.
\end{align}
This connection corresponds to the flat Lie algebra $\h_{0,3}$.

The third solution is represented by the following flat, torsion-free connection:
\begin{equation}\label{Sol13}
		\begin{aligned}
	\nabla_{e_1}e_1&=e_1+a_{31}e_3, &\nabla_{e_1}e_3&=\nabla_{e_3}e_1=a_{13}e_1+a_{23}e_2+\tfrac{a_{31}c_{13}}{a_{13}}e_3,\\ \nabla_{e_3}e_3&=c_{13}e_1-\tfrac{a_{23}(a_{13}-a_{31}c_{13})}{a_{13}a_{31}}e_2+\tfrac{a_{13}^3+a_{31}c_{13}^2-a_{13}c_{13}}{a_{13}^2}e_3.
	\end{aligned}
\end{equation}

Note that $a_{31}a_{13}\neq0$. Consider the automorphism
\begin{align*}
	\Psi(e_1) &=-a_{31} e_2, &
	\Psi(e_2) &= e_1, &
	\Psi(e_3) &=e_2+a_{31}e_3,
\end{align*}
and apply it to the connection~\eqref{Sol13}. This yields the equivalent connection
\begin{equation}
\begin{aligned}\label{Sol13,1}
	\nabla_{e_2} e_2 &=e_3,\\ \nabla_{e_2}e_3&=\nabla_{e_3}e_2=-\tfrac{a_{23}}{a_{31}^2}e_1+\tfrac{a_{13}^2-c_{13}}{a_{31}a_{13}}e_2-\tfrac{a_{13}+a_{31}c_{13}}{a_{31}a_{13}}e_3,\\
	\nabla_{e_3}e_3&=\tfrac{a_{23}(c_{13}a_{31}+a_{13})}{a_{31}^3a_{13}}e_1-\tfrac{(a_{13}^2-c_{13})(c_{13}a_{31}+a_{13})}{a_{31}^2a_{13}^2}e_2+\tfrac{a_{31}^2c_{13}^2+a_{13}(a_{13}^2+c_{13})a_{31}+a_{13}^2}{a_{31}^2a_{13}^2}e_3.
\end{aligned}
\end{equation}
Observe that this connection belongs to the family of torsion-free connections with $\nabla^0 \equiv 0$,  which has already been analyzed in Case 1. Therefore, it must be isomorphic to one of the connections $\h_{0,1}, \dots, \h_{0,8}$, and no further distinct cases arise.

The fourth solution is given by the following flat, torsion-free connection:
\begin{equation}\label{Sol14}
	\begin{aligned}
		\nabla_{e_1}e_1&=e_1, &\nabla_{e_1}e_3&=\nabla_{e_3}e_1=a_{13}e_1, &\nabla_{e_3}e_3&=c_{13}e_1+c_{23}e_2+\tfrac{a_{13}^2-c_{13}}{a_{13}}e_3.
	\end{aligned}
\end{equation}
If $a_{13}^2 - c_{13} \neq 0$, consider the automorphism
\begin{align*}
	\Psi(e_1) &=  e_2, &
	\Psi(e_2) &= e_1, &
	\Psi(e_3) &=-\tfrac{c_{23}a_{13}}{a_{13}^2-c_{13}}e_1+a_{13}e_2+\tfrac{a_{13}^2-c_{13}}{a_{13}}e_3,
\end{align*}
and apply it to the connection~\eqref{Sol14}. 
This yields the equivalent connection
\begin{align}\label{Sol14,1}
	\nabla_{e_2} e_2 &= e_2, &
	\nabla_{e_3} e_3 &= e_3.
\end{align}
This connection is isomorphic to the one defined by the flat Lie algebra $\h_{1,1}$ via the automorphism
\begin{align*}
	\Psi(e_1) &= e_2, &
	\Psi(e_2) &= e_1, &
	\Psi(e_3) &= e_3.
\end{align*}
Now, if $c_{13} = a_{13}^2$, consider the automorphism
\begin{align*}
	\Psi(e_1) &= e_2, &
	\Psi(e_2) &= e_1, &
	\Psi(e_3) &= e_3.
\end{align*}
Applying it to the connection given in~\eqref{Sol14} yields the following connection for a suitable $x\in\R^\ast$:
\begin{align}\label{Sol14,1}
	\nabla_{e_2} e_2 &= e_2, &
	\nabla_{e_3} e_3 &= \delta\, e_1, \qquad \delta = 0,1.
\end{align}
If $\delta=0$, the automorphism
\begin{align*}
	\Psi(e_1) &= e_2, & 
	\Psi(e_2) &= e_2+e_3, & 
	\Psi(e_3) &= e_1
\end{align*}
applied to the connection in \eqref{Sol14,1} shows that it is isomorphic to the flat Lie algebra $\h_{0,2}$.  

If $\delta=1$, the  automorphism
\begin{align*}
	\Psi(e_1) &= e_1, & 
	\Psi(e_2) &= -e_1+e_2-e_3, & 
	\Psi(e_3) &= e_2
\end{align*}
establishes an isomorphism between the connection in \eqref{Sol14,1} and the flat Lie algebra $\h_{0,1}$.

The fifth solution is given by the following flat, torsion-free connection:
\begin{align}\label{Sol12,5}
	\nabla_{e_1}e_1&=e_1, &\nabla_{e_1}e_3&=\nabla_{e_3}e_1=a_{13}e_1, &\nabla_{e_2}e_3&=\nabla_{e_3}e_2=\tfrac{a_{13}^2-c_{13}}{a_{13}}e_2, &\nabla_{e_3}e_3&=c_{13}e_1+c_{23}e_2+\tfrac{a_{13}^2-c_{13}}{a_{13}}e_3.
\end{align}
If $c_{13}-a_{13}^2\neq0$, we apply the automorphism $\Psi$ defined by
\begin{align*}
	\Psi(e_1) &= e_2, \quad
	\Psi(e_2) = e_1 \quad
	\Psi(e_3) = \tfrac{c_{23}a_{13}}{a_{13}^2-c_{13}}e_1+a_{13}e_2+\tfrac{a_{13}^2-c_{13}}{a_{13}}e_3,
\end{align*}
to the connection given in~\eqref{Sol12,5}, which yields the following equivalent connection:
\begin{align}\label{Sol12,5,1}
	\nabla_{e_1}e_3 &=\nabla_{e_3}e_1= e_1, &\nabla_{e_2}e_2&=e_2, &\nabla_{e_3}e_3&=e_3.
\end{align}
It is straightforward to verify that this connection is not isomorphic to any of the connections associated with the flat Lie algebras
$\h_{0,1},\ldots,\h_{0,8}$ and $\h_{1,1}$.
Moreover, in order to preserve the condition $\nabla^0 \equiv \nabla^1$ corresponding to the flat Lie algebra $\mathfrak{b}_1$, we apply the automorphism $\Psi$ defined by
\begin{align*}
	\Psi(e_1) &= e_2, \quad
	&\Psi(e_2)& = e_1, \quad
	&\Psi(e_3)& = e_3,
\end{align*}
to the connection given in~\eqref{Sol12,5,1}, which yields the following connection:
\begin{align}
	\h_{1,2} : & &\nabla_{e_1}e_1 &=e_1, &\nabla_{e_2}e_3&=e_2, &\nabla_{e_3}e_2&=e_2, &\nabla_{e_3}e_3&=e_3.
\end{align}
If $c_{13}-a_{13}^2=0$, we apply the automorphism $\Psi$ defined by
\begin{align*}
	\Psi(e_1) &= e_2, \quad
	&\Psi(e_2) &= xe_1, \quad
	&\Psi(e_3)& = a_{13}e_2+e_3,
\end{align*}
for a suitable $x\in\R^\ast$, to the connection given in~\eqref{Sol12,5}. 
This yields the following equivalent connection:
\begin{align}\label{Sol12,5,2}
	\nabla_{e_2}e_2 &= e_2, &
	\nabla_{e_3}e_3 &= \delta\,e_1, \qquad \delta=0,1.
\end{align}
If $\delta=0$, the automorphism
\begin{align*}
	\Psi(e_1) &= e_1, &
	\Psi(e_2)& = e_2+e_3, &
	\Psi(e_3) &= e_2,
\end{align*}
induces an isomorphism between the flat torsion-free connection given in~\eqref{Sol12,5,2} and the one defined by the flat Lie algebra $\h_{0,2}$.

Otherwise, if $\delta=1$, the automorphism defined by
\begin{align*}
	\Psi(e_1) &= e_1, &
	\Psi(e_2)& = -e_1+e_2-e_3, &
	\Psi(e_3) = e_2&,
\end{align*}
establishes an isomorphism between the connection given in~\eqref{Sol12,5,2} and the one defined by the flat Lie algebra $\h_{0,1}$.

The sixth solution is given by the following flat, torsion-free connection:
\begin{align}\label{Sol12,6}
	\nabla_{e_1}e_1&=e_1+a_{31}e_3, &\nabla_{e_1}e_3&=\nabla_{e_3}e_1=a_{23}e_2+a_{33}e_3, &\nabla_{e_3}e_3&=\tfrac{a_{23}(a_{33}-1)}{a_{31}}e_2+\tfrac{a_{33}(a_{33}-1)}{a_{31}}e_3.
\end{align}
Note that $a_{31}\neq 0$. We therefore apply the automorphism $\Psi$ defined by
\begin{align*}
	\Psi(e_1) &= e_2, &
	\Psi(e_2)& = e_1, &
	\Psi(e_3)& = -\tfrac{1}{a_{31}}\,e_2 + \tfrac{1}{a_{31}}\,e_3,
\end{align*}
to the connection given in~\eqref{Sol12,6}. This yields the equivalent connection
\begin{equation}\label{Sol12,6,1}
\begin{aligned}
	\nabla_{e_2}e_2 &=e_3, &\nabla_{e_2}e_3&=\nabla_{e_3}e_2=a_{31}a_{32}e_1-a_{33}e_2+(1+a_{33})e_3,\\
	&&\nabla_{e_3}e_3&=a_{31}a_{23}(1+a_{33})e_1-a_{33}(1+a_{33})e_2+(a_{33}^2+a_{33}+1)e_3.
\end{aligned}
\end{equation}
Observe that the induced connection $\nabla^0$ associated with~\eqref{Sol12,6,1} satisfies
$\nabla^0 \equiv 0$. This case has already been treated in \textbf{Case~1}. Consequently,
the connection~\eqref{Sol12,6,1} is necessarily isomorphic to one of the flat Lie algebras $
\h_{0,0},\h_{0,1},\ldots,\h_{0,8}$. Hence, the classification procedure terminates here.

The seventh solution is represented by the following flat, torsion-free connection:
\begin{align}\label{Sol12,7}
	\nabla_{e_1}e_1&=e_1, &\nabla_{e_1}e_3&=\nabla_{e_3}e_1=a_{23}e_2+e_3, &\nabla_{e_3}e_3&=c_{13}e_1+a_{23}c_{33}e_2+c_{33}e_3.
\end{align}
Consider the following automorphism
\begin{align*}
	\Psi(e_1) &= e_2, 
	&\Psi(e_2)& = e_1, 
&\Psi(e_3) &=-a_{23}e_1 -\tfrac{c_{33}}{2}\,e_2 + x\,e_3,
\end{align*}
and apply it to the connection given in~\eqref{Sol12,7}. 
For a suitable choice of the parameter $x\in\R^\ast$, this yields the equivalent connection
\begin{align}\label{Sol12,7,1}
	\nabla_{e_2}e_2 &= e_2, &
	\nabla_{e_2}e_3 &= e_3, &
	\nabla_{e_3}e_2 &= e_3, &
	\nabla_{e_3}e_3 &= \varepsilon_1 e_2,
	\qquad \varepsilon_1=0,\pm1.
\end{align}
If $\varepsilon_1=\pm1$, the automorphism
\begin{align*}
	\Psi(e_1) &= e_1, &
	\Psi(e_2)& = e_3, &
	\Psi(e_3)& = e_2,
\end{align*}
establishes an isomorphism between the connection given in~\eqref{Sol12,7,1} and the one associated with the flat Lie algebra $\h_{0,5}$, where in this case $\varepsilon_1=\varepsilon$.

Otherwise, if $\varepsilon_1=0$, one easily checks that the automorphism
\begin{align*}
	\Psi(e_1) &= e_1, &
	\Psi(e_2)& = e_3, &
	\Psi(e_3)& = e_2,
\end{align*}
induces an isomorphism between the connection given in~\eqref{Sol12,7,1} and the one associated with the flat Lie algebra $\h_{0,3}$.

The eighth solution is given by the following flat, torsion-free connection:
\begin{align}\label{Sol12,8}
	\nabla_{e_1}e_1&=e_1+a_{31}e_3, &\nabla_{e_1}e_3&=\nabla_{e_3}e_1=-a_{31}c_{23}e_2, &\nabla_{e_3}e_3&=c_{23}e_2.
\end{align}
Consider the automorphism $\Psi$ defined by
\begin{align*}
	\Psi(e_1) &= x a_{31}^2 c_{23}\, e_1 - a_{31} e_2 + e_3,&
	\Psi(e_2)& = x e_1, &
	\Psi(e_3)& = e_2,
\end{align*}
and apply it to the connection given in~\eqref{Sol12,8}.  
For a suitable choice of $x \in \R^\ast$, this transformation leads to the following equivalent connection:
\begin{align}\label{Sol12,8,1}
	\nabla_{e_2}e_2&=\delta e_1,  &\nabla_{e_3}e_3&=e_3, &\delta=0,1&.
\end{align}
If $\delta = 0$, then this connection coincides with the one previously obtained and described in equation~\eqref{Sol012,2} (with $\delta=0$).   If $\delta = 1$, the automorphism $\Psi$ defined by
\begin{align*}
	\Psi(e_1) &=  e_1, &
	\Psi(e_2)& = e_2, &
	\Psi(e_3)& = -e_1+e_2-e_3,
\end{align*}
 induces an isomorphism between the connection given in~\eqref{Sol12,8,1} and the connection associated with the flat Lie algebra $\h_{0,1}$.

The ninth solution is given by the following flat, torsion-free connection:
\begin{align}\label{Sol12,9}
	\nabla_{e_1}e_1&=e_1+a_{31}e_3, &\nabla_{e_2}e_2&=b_{32}e_3.
\end{align}
Consider the automorphism $\Psi$ defined by
\begin{align*}
	\Psi(e_1) &= -a_{31}x\,e_1 + e_3, &
	\Psi(e_2) &= e_2, &
	\Psi(e_3) &= x\,e_1,
\end{align*}
and apply it to the connection given in~\eqref{Sol12,9}. 
For a suitable choice of $x \in \R^\ast$, this yields the equivalent connection
\begin{align}\label{Sol12,9,1}
	\nabla_{e_2}e_2 &=\delta_1 e_1, &
	\nabla_{e_3}e_3 &= e_3,
	\qquad \delta_1=0,1.
\end{align}

We observe that the connection~\eqref{Sol12,9,1} coincides with the one obtained in~\eqref{Sol12,8,1}, with $\delta=\delta_1\in\{0,1\}$. 
Hence, this case has already been treated, and no further analysis is required.

The tenth solution is given by the following flat, torsion-free connection:
\begin{align}\label{Sol12,10}
	\nabla_{e_1}e_1 &= e_1, &
	\nabla_{e_2}e_2 &=b_{32} e_3, &\nabla_{e_2}e_3&=\nabla_{e_3}e_2=\tfrac{b_{32}c_{33}-b_{33}^2}{b_{32}}e_2+b_{33}e_3, &\nabla_{e_3}e_3&=\tfrac{(b_{32}c_{33}-b_{33}^2)b_{33}}{b_{32}^2}e_2+c_{33}e_3.
\end{align}
 Consider the automorphism
\begin{align*}
	\Psi(e_1) &= e_1, &
	\Psi(e_2) &= x\,e_3, &
	\Psi(e_3) &= \tfrac{x^2}{b_{32}}\,e_2 + \tfrac{x\,b_{33}}{b_{32}}\,e_3,
\end{align*}
and apply it to the connection given in~\eqref{Sol12,10}. 
For a suitable choice of $x \in \R^\ast$, this yields the following equivalent connection:
\begin{align}\label{Sol12,10,1}
	\nabla_{e_1}e_1 &= e_1, &
	\nabla_{e_2}e_2 &=\varepsilon_1 e_2, &\nabla_{e_2}e_3&=\nabla_{e_3}e_2=\varepsilon_1 e_3, &\nabla_{e_3}e_3&=e_2+\lambda e_3, &\varepsilon_1=0,\pm1,~~\lambda\in\R&.
\end{align}
Here, if $\varepsilon_1 = 0$, then $\lambda = 0$ or $1$. 
Otherwise, if $\varepsilon_1 = \pm 1$, then $\lambda \in \mathbb{R}$.

We now begin by analyzing the isomorphisms between the connection given in~\eqref{Sol12,10,1} and the previously classified connections. 
We start with the case $\varepsilon_1 = 0$ and $\lambda = 0$. 
The following automorphism
\begin{align*}
	\Psi(e_1) &= -e_1+e_2-e_3, & 
	\Psi(e_2) &= e_1, & 
	\Psi(e_3) &= e_2,
\end{align*}
induces an isomorphism between the connection~\eqref{Sol12,10,1} and the one associated with the flat Lie algebra $\h_{0,1}$. If $\varepsilon_1 = 0$ and $\lambda = 1$, the automorphism
\begin{align*}
	\Psi(e_1) &= e_3, & 
	\Psi(e_2) &= -e_2, & 
	\Psi(e_3) &= e_1 + e_2,
\end{align*}
induces an isomorphism between the connection~\eqref{Sol12,10,1} and the one associated with the flat Lie algebra $\h_{0,9}$. Now assume that $\varepsilon_1=\pm 1$. We apply the automorphism
\begin{align*}
	\Psi(e_1) &= e_1, &
	\Psi(e_2) &=\varepsilon_1 e_2, &
	\Psi(e_3) &=\tfrac{\lambda}{2} e_2 + xe_3,\quad x\in\R^\ast
\end{align*}
to the connection given in~\eqref{Sol12,10,1}. This transformation simplifies the connection and yields the following equivalent form:
\begin{align}\label{Sol12,10,2}
	\nabla_{e_1}e_1 &= e_1, &
	\nabla_{e_2}e_2 &= e_2, &\nabla_{e_2}e_3&=e_3, &\nabla_{e_3}e_2&= e_3, &\nabla_{e_3}e_3&=\tfrac{\lambda^2+4\varepsilon_1}{4x^2} e_2.
\end{align}
Observe first that, if $\varepsilon_1=-1$, then this connection is actually
isomorphic to the one associated with the flat Lie algebra $\h_{1,2}$ via the
automorphism
\begin{align*}
	\Psi(e_1) &= e_1, &
	\Psi(e_2) &= \varepsilon_1 e_3, &
	\Psi(e_3) &= e_2.
\end{align*}
In this case $\lambda=\pm2$. 
Otherwise, it is straightforward to verify that the connection given in
\eqref{Sol12,10,2} is not isomorphic to any of the connections associated with
the flat Lie algebras
\[
\h_{0,1},\ldots,\h_{0,8}, \qquad \h_{1,1}, \h_{1,2}.
\]
Since the parameter $x\in\R^\ast$ can be chosen arbitrarily, we may take
\[
x=\tfrac{\sqrt{\lambda^2+4\varepsilon_1}}{2}.
\]
This yields a new flat Lie algebra, represented by the torsion-free connection
\begin{align}\label{Sol12,10,3}
\h_{1,3}&	&\nabla_{e_1}e_1 &= e_1, &
	\nabla_{e_2}e_2 &= e_2, &
	\nabla_{e_2}e_3 &= e_3, &
	\nabla_{e_3}e_2 &= e_3, &
	\nabla_{e_3}e_3 &= e_2.
\end{align}

The eleventh solution is defined by the following flat and torsion-free connection:
\begin{align}\label{Sol12,11}
	\nabla_{e_1}e_1 &= e_1, &
	\nabla_{e_3}e_3 &= c_{23}e_2+c_{33}e_3.
\end{align}
If $c_{33}\neq 0$, consider the automorphism
\begin{align*}
	\Psi(e_1) &= e_1, &
	\Psi(e_2) &= e_2, &
	\Psi(e_3) &= -\tfrac{c_{23}}{c_{33}}\,e_2 + c_{33} e_3 .
\end{align*}
Applying this automorphism to the connection given in~\eqref{Sol12,11} yields the following equivalent connection:
\begin{align}\label{Sol12,11,1}
	\nabla_{e_1}e_1 &= e_1, &
	\nabla_{e_3}e_3 &=e_3.
\end{align}
The connection~\eqref{Sol12,11,1} actually coincides with the one associated with the flat Lie algebra $\h_{1,1}$.
 Now assume that $c_{33}=0$. Consider the automorphism
 \begin{align*}
 	\Psi(e_1) &= e_1, &
 	\Psi(e_2) &= xe_2, &
 	\Psi(e_3) &= e_3,
 \end{align*}
 where $x\in\R^\ast$.
 Applying this automorphism to the connection~\eqref{Sol12,11},
 we obtain, for a suitable choice of $x$, the following equivalent connection:
 \begin{align}\label{Sol12,11,2}
 	\nabla_{e_1}e_1 &= e_1, &
 	\nabla_{e_3}e_3 &=\delta e_2,\qquad\delta=0,1.
 \end{align}
If $\delta = 1$, it is easy to check that the connection defined in~\eqref{Sol12,11,2}
is isomorphic to the one associated with the flat Lie algebra $\h_{0,1}$ via the
automorphism
\begin{align*}
	\Psi(e_1) &= -e_1 + e_2 - e_3, &
	\Psi(e_2) &= e_1, &
	\Psi(e_3) &= e_2.
\end{align*}
Otherwise, if $\delta = 0$, the connection defined in~\eqref{Sol12,11,2}
is isomorphic to the one associated with the flat Lie algebra $\h_{0,2}$
via the automorphism
\begin{align*}
	\Psi(e_1) &= e_2 + e_3, &
	\Psi(e_2) &= e_2, &
	\Psi(e_3) &= e_1.
\end{align*}
The twelfth solution is given by the following flat and torsion-free connection:
\begin{align}\label{Sol12,12}
	\nabla_{e_1}e_1 &= e_1, &
	\nabla_{e_2}e_3 &=\nabla_{e_3}e_2=c_{33}e_2, &\nabla_{e_3}e_3&=c_{23}e_2+c_{33}e_3.
\end{align}
Assume that $c_{33}\neq 0$ and consider the automorphism
\begin{align*}
	\Psi(e_1) &= e_2, &
	\Psi(e_2) &= e_1, &
	\Psi(e_3) &= \tfrac{c_{23}}{c_{33}}\, e_1 + c_{33} e_3.
\end{align*}
Applying this automorphism to the connection given in~\eqref{Sol12,12} yields the following equivalent connection:
\begin{align}\label{Sol12,12,1}
	\nabla_{e_1}e_3 &= e_1, &
	\nabla_{e_2}e_2&=e_2, &\nabla_{e_3}e_1&=e_1, &\nabla_{e_3}e_3&=e_3.
\end{align}
This connection actually coincides with the one given in~\eqref{Sol12,5,1}, which, as shown previously, is isomorphic to the connection associated with the flat Lie algebra $\h_{1,2}$.

Assume now that $c_{33}=0$, and consider the following automorphism
\begin{align*}
	\Psi(e_1) &= e_2, & 
	\Psi(e_2) &= x\, e_1, & 
	\Psi(e_3) &= e_3,
\end{align*}
where $x\in\R^\ast$ is chosen suitably. Applying this automorphism to the connection given in~\eqref{Sol12,12} yields the following equivalent connection:
\begin{align}\label{Sol12,12,1=2}
	\nabla_{e_2}e_2&=e_2, &\nabla_{e_3}e_3&=\delta e_1,\qquad\delta=0,1.
\end{align}
This connection coincides with the one previously given in~\eqref{Sol12,5,2}, which has already been treated, so no further procedure is required.

\textbf{Case 3.} If $\nabla^0 \equiv \nabla^2$, which corresponds to the Lie algebra $\mathfrak{b}_2$ (see Table~\ref{FlatR2}), then the flatness equations associated with the connection given in~\eqref{Connegeneral} can be solved directly, yielding seven distinct solutions.

The first solution is given by the following flat torsion-free connection:
\begin{equation}
\begin{aligned}\label{Sol21}
	\nabla_{e_1}e_1&= a_{32}b_{33}e_3, &\nabla_{e_1}e_2&=\nabla_{e_2}e_1=a_{32}e_3, &\nabla_{e_1}e_3&=\nabla_{e_3}e_1=b_{33}^2e_3,\\
\nabla_{e_2}e_2&=e_1, &\nabla_{e_2}e_3&=\nabla_{e_3}e_2=b_{33}e_3, &\nabla_{e_3}e_3&=\tfrac{b_{33}^2}{a_{32}}e_3.	
\end{aligned}
\end{equation}
Consider the automorphism
\begin{align*}
	\Psi(e_1) &= e_3, &
	\Psi(e_2) &= e_2, &
	\Psi(e_3) &= \tfrac{1}{a_{32}}\,e_1 + \tfrac{b_{33}}{a_{32}}\,e_3,
\end{align*}
and apply it to the connection given in~\eqref{Sol21}. This yields the following equivalent connection:
\begin{align}\label{Sol21,1}
	\nabla_{e_2}e_2 &= e_3, &
	\nabla_{e_2}e_3 &= \nabla_{e_3}e_2 = e_1 + b_{33}e_3, &
	\nabla_{e_3}e_3 &= b_{33}e_1 + b_{33}^2 e_2 .
\end{align}
Observe that the torsion-free part $\nabla^0$ of the connection defined in~\eqref{Sol21,1} coincides with the vanishing torsion-free connection, that is, $\nabla^0 \equiv 0$. Consequently, the connection~\eqref{Sol21,1} is necessarily isomorphic to one of the previously classified connections, and no further analysis is required.

The second solution is given by the following flat, torsion-free conncetion:
\begin{align}\label{Sol22}
\nabla_{e_2}e_2&=e_1, &\nabla_{e_2}e_3&=\nabla_{e_3}e_2=b_{13}e_1, &\nabla_{e_3}e_3&=c_{13}e_1-b_{13}c_{33}e_2+c_{33}e_3.
\end{align}
Assume that $b_{13}\neq 0$ and $c_{33}\neq 0$. Consider the automorphism
\begin{align*}
	\Psi(e_1) &= \tfrac{c_{33}^2}{b_{13}^2}\, e_1, &
	\Psi(e_2) &= \tfrac{c_{33}c_{13}}{b_{13}^3}\, e_1 + \tfrac{c_{33}}{b_{13}}\, e_2, &
	\Psi(e_3) &= c_{33}\, e_3,
\end{align*}
and apply it to the connection given in~\eqref{Sol22}. This yields the following flat torsion-free connection:
\begin{align}\label{Sol22,1}
	\nabla_{e_2}e_2 &= e_1, &
	\nabla_{e_2}e_3 &= \nabla_{e_3}e_2 = e_1, &
	\nabla_{e_3}e_3 &= -e_2 + e_3.
\end{align}
The automorphism defined by
\begin{align*}
	\Psi(e_1) &= e_1, &
	\Psi(e_2) &= e_2, &
	\Psi(e_3) &= 2\,e_2 - e_3,
\end{align*}
establishes an isomorphism between the connection given in~\eqref{Sol22,1}
and the one associated with the flat Lie algebra $\h_{0,1}$.

Assume now that $b_{13}\neq 0$ and $c_{33}=0$. Consider the automorphism
\begin{align*}
	\Psi(e_1) &= e_1, &
	\Psi(e_2) &= \tfrac{1}{b_{13}}\,e_2, &
	\Psi(e_3) &= e_3,
\end{align*}
and apply it to the connection given in~\eqref{Sol22}. This yields the following
equivalent flat torsion-free connection:
\begin{align}\label{Sol22,2}
	\nabla_{e_2}e_2 &= e_1, &
	\nabla_{e_2}e_3 &= \nabla_{e_3}e_2 = e_1, &
	\nabla_{e_3}e_3 &= \lambda\,e_1,\qquad \lambda\in\R.
\end{align}
We now analyze the isomorphisms between the connection given in~\eqref{Sol22,2}
and the previously classified connections. We begin with the case $\lambda = 1$.
The automorphism
\begin{align*}
	\Psi(e_1) &= e_2, &
	\Psi(e_2) &= e_3, &
	\Psi(e_3) &= e_1 + e_2,
\end{align*}
establishes an isomorphism between the connection~\eqref{Sol22,2} and the one
associated with the flat Lie algebra $\h_{0,6}$.

If $\lambda\neq1$, the following automorphism 
\begin{align*}
	\Psi(e_1) &=\varepsilon e_1, &
	\Psi(e_2) &= e_2, &
	\Psi(e_3) &= \sqrt{\varepsilon(\lambda-1)}e_1 + e_2,\qquad\varepsilon=\pm1,
\end{align*}
establishes an isomorphism between the connection~\eqref{Sol22,2} and the one
associated with the flat Lie algebra $\h_{0,7}$.

Assume that $b_{13}=0$ and $c_{33}\neq 0$. Consider the automorphism
\begin{align*}
	\Psi(e_1) &= e_1, &
	\Psi(e_2) &= e_2, &
	\Psi(e_3) &= -\tfrac{c_{13}}{c_{33}}\,e_1 + c_{33}e_3,
\end{align*}
and apply it to the connection given in~\eqref{Sol22}. This yields the following
equivalent flat torsion-free connection:
\begin{align}\label{Sol22,3}
	\nabla_{e_2}e_2 &= e_1, &
	\nabla_{e_3}e_3 &= e_3.
\end{align}
It is straightforward to verify that the automorphism
\begin{align*}
	\Psi(e_1) &= e_1, &
	\Psi(e_2) &= e_2, &
	\Psi(e_3) &= -e_1 + e_2 - e_3,
\end{align*}
induces an isomorphism between the connection given in~\eqref{Sol22,3}
and the one associated with the flat Lie algebra $\h_{0,1}$.

Assume that $b_{13}=0$ and $c_{33}=0$, and consider the automorphism
\begin{align*}
	\Psi(e_1) &= x^2 e_1, &
	\Psi(e_2) &= x e_2, &
	\Psi(e_3) &= e_3,
\end{align*}
where $x \in \R^\ast$. Applying this automorphism to the connection given in~\eqref{Sol22},
we obtain the following equivalent connection:
\begin{align}\label{Sol22,4}
	\nabla_{e_2} e_2 &= e_1, &
	\nabla_{e_3} e_3 &= \delta_1 e_1,
	\qquad \delta_1 = 0, \pm 1.
\end{align}
If $\delta_1 = 0$, it is straightforward to verify that the automorphism
\begin{align*}
	\Psi(e_1) &= e_2, &
	\Psi(e_2) &= e_2 + e_3, &
	\Psi(e_3) &= e_1,
\end{align*}
induces an isomorphism between the connection defined in~\eqref{Sol22,3}
and the one associated with the flat Lie algebra $\h_{0,6}$.
 
If $\delta_1=\pm1$, it is straightforward to verify that the automorphism
\begin{align*}
	\Psi(e_1) &= \varepsilon\, e_3, &
	\Psi(e_2) &= e_2, &
	\Psi(e_3) &= \varepsilon\, e_1,
\end{align*}
induces an isomorphism between the connection defined in~\eqref{Sol22,3}
and the one associated with the flat Lie algebra $\h_{0,7}$, where
$\delta_1=\varepsilon=\pm1$.

The third solution is given by the following flat, torsion-free connection:
\begin{equation}
\begin{aligned}\label{Sol23}
	\nabla_{e_1}e_3&=\nabla_{e_3}e_1=c_{33}e_1, &\nabla_{e_2}e_2&=e_1, &\nabla_{e_2}e_3&=\nabla_{e_3}e_2=b_{13}e_1+c_{33}e_2,\\ \nabla_{e_3}e_3&=c_{13}e_1+b_{13}c_{33}e_2+c_{33}e_3.
\end{aligned}
\end{equation}
Assume that $c_{33}\neq 0$ and $b_{13}\neq 0$. Consider the automorphism
\begin{align*}
	\Psi(e_1) &= e_1, &
	\Psi(e_2) &=-\tfrac{c_{33}c_{13}}{b_{13}^2}e_1+\tfrac{c_{33}}{b_{13}} e_2, &
	\Psi(e_3) &= c_{33} e_3.
\end{align*}
Applying this automorphism to the connection given in~\eqref{Sol23} yields the following equivalent connection:
\begin{equation}
\begin{aligned}\label{Sol23,1}
\h_{2,1}:	&~~&\nabla_{e_1}e_3&=e_1, &\nabla_{e_2}e_2&=e_1, &\nabla_{e_2}e_3&=e_1+e_2,\\&~~& \nabla_{e_3}e_1&=e_1, &\nabla_{e_3}e_2&=e_1+e_2, &\nabla_{e_3}e_3&=e_2+e_3.
\end{aligned}
\end{equation}
It is straightforward to verify that, for every automorphism $\Psi \in \mathrm{GL}_4(\mathbb{R})$, 
the connection defined in~\eqref{Sol23,1} is not isomorphic to any of the previously classified 
connections associated with the flat Lie algebras 
$\mathfrak{h}_{0,1}, \ldots, \mathfrak{h}_{0,8}$ and 
$\mathfrak{h}_{1,1}, \mathfrak{h}_{1,2}, \mathfrak{h}_{1,3}$. 
Therefore, we obtain a new flat Lie algebra, which we denote by $\mathfrak{h}_{1,4}$.

Assume that $c_{33}\neq 0$ and $b_{13}= 0$. Consider the automorphism
\begin{align*}
	\Psi(e_1) &= e_1, &
	\Psi(e_2) &= e_2, &
	\Psi(e_3) &=\tfrac{c_{13}}{c_{33}}e_1+ c_{33} e_3.
\end{align*}
Applying this automorphism to the connection given in~\eqref{Sol23} yields the following equivalent connection:
\begin{equation}
	\begin{aligned}\label{Sol23,2}
		\nabla_{e_1}e_3&=e_1, &\nabla_{e_2}e_2&=e_1, &\nabla_{e_2}e_3&=e_2,\\\nabla_{e_3}e_1&=e_1, &\nabla_{e_3}e_2&=e_2, &\nabla_{e_3}e_3&=e_3.
	\end{aligned}
\end{equation}
This connection is, in fact, isomorphic to the one associated with the flat Lie algebra $\mathfrak{h}_{1,4}$ via the following automorphism:
\begin{align*}
	\Psi(e_1) &= e_1, &
	\Psi(e_2) &=e_2, &
	\Psi(e_3) &=e_1-e_2+ e_3.
\end{align*}
Assume now that $c_{33}=0$ and $b_{13}\neq 0$. Consider the automorphism
\begin{align*}
	\Psi(e_1) &= e_1, &
	\Psi(e_2) &= e_2, &
	\Psi(e_3) &= b_{13}\, e_3.
\end{align*}
Applying this automorphism to the connection given in~\eqref{Sol23} yields the following equivalent connection:
\begin{equation}
	\begin{aligned}
		\nabla_{e_2}e_2 &= e_1, &
		\nabla_{e_2}e_3 &= e_1, &
		\nabla_{e_3}e_2 &= e_1, &
		\nabla_{e_3}e_3 &= \tfrac{c_{13}}{b_{13}^2}\, e_3.
	\end{aligned}
\end{equation}
Observe that this connection coincides with the one given in~\eqref{Sol22,2} with 
$\lambda = \tfrac{c_{13}}{b_{13}^2}$; hence, no further analysis is required.

Assume that $c_{33}=0$ and $b_{13}=0$. Consider the automorphism
\begin{align*}
	\Psi(e_1) &= x^{2} e_1, &
	\Psi(e_2) &= x e_2, &
	\Psi(e_3) &= e_3.
\end{align*}
Applying this automorphism to the connection given in~\eqref{Sol23}, for a suitable choice of 
$x \in \mathbb{R}^\ast$, yields the following equivalent connection:
\begin{align}\label{Sol23,3}
	\nabla_{e_2}e_2 &= e_1, &
	\nabla_{e_3}e_3 &= \delta_1 e_1,\qquad \delta_1 = 0,\pm 1.
\end{align}

We now analyze whether there exists an isomorphism between the connection defined in~\eqref{Sol23,3}
and the previously classified connections. We begin with the case $\delta_1 = 0$. 
The automorphism defined by
\begin{align*}
	\Psi(e_1) &=  e_2, &
	\Psi(e_2) &=  e_3, &
	\Psi(e_3) &= e_1
\end{align*}
induces an isomorphism between the connection given in~\eqref{Sol23,3} and the one associated with
the flat Lie algebra $\mathfrak{h}_{0,6}$.

Now, if $\delta_1 = \pm 1$, it is straightforward to verify that the automorphism
\begin{align*}
	\Psi(e_1) &=e_3, &
	\Psi(e_2) &=  e_1, &
	\Psi(e_3) &= e_2
\end{align*}
induces an isomorphism between the connection given in~\eqref{Sol23,3} and the one corresponding to
the flat Lie algebra $\mathfrak{h}_{0,7}$ with $\delta_1=\varepsilon=\pm1$.

The fourth solution is given by the following flat, torsion-free connection:
\begin{equation}
\begin{aligned}\label{Sol24}
	\nabla_{e_1}e_1 &= \tfrac{(b_{32}c_{33}-b_{33}^2)^2}{c_{33}}e_3, &\nabla_{e_1}e_2&=\nabla_{e_2}e_1= -\tfrac{b_{33}(b_{32}c_{33}-b_{33}^2)}{c_{33}}e_3, &\nabla_{e_1}e_3&=\nabla_{e_3}e_1=(b_{33}^2-b_{32}c_{33})e_3,\\
	\nabla_{e_2}e_2&=e_1+b_{32}e_3, &\nabla_{e_2}e_3&=\nabla_{e_3}e_2=b_{33}e_3, &\nabla_{e_3}e_3&=c_{33}e_3.
\end{aligned}
\end{equation}
Note that $c_{33} \neq 0$. Consider the automorphism
\begin{align*}
	\Psi(e_1) &= e_1 + (b_{33}^2 - b_{32} c_{33})\, e_3, &
	\Psi(e_2) &= e_2 + b_{33}\, e_3, &
	\Psi(e_3) &= c_{33}\, e_3.
\end{align*}
Applying this automorphism to the connection given in~\eqref{Sol24} yields the following equivalent connection:
\begin{align}\label{Sol24,1}
	\nabla_{e_2} e_2 &= e_1, &
	\nabla_{e_3} e_3 &= e_3.
\end{align}
Observe that this connection coincides with the one given in~\eqref{Sol22,3}, which has been treated previously,
and is isomorphic to the connection associated with the flat Lie algebra $\mathfrak{h}_{0,1}$.

The fifth solution is given by the following flat, torsion-free connection:
\begin{equation}
	\begin{aligned}\label{Sol25}
		\nabla_{e_1} e_2 &= \nabla_{e_2} e_1 = a_{32} e_3, &
		\nabla_{e_2} e_2 &= e_1 + b_{32} e_3.
	\end{aligned}
\end{equation}

Consider the automorphism
\begin{align*}
	\Psi(e_1) &= e_1 - x\, b_{32} e_3, &
	\Psi(e_2) &= e_2, &
	\Psi(e_3) &= x\, e_3.
\end{align*}

Applying this automorphism to the connection given in~\eqref{Sol25}, for a suitable choice of 
$x \in \mathbb{R}^\ast$, yields the following equivalent connection:
\begin{equation}
	\begin{aligned}\label{Sol25,1}
		\nabla_{e_1} e_2 &= \delta_2 e_3, &
		\nabla_{e_2} e_1 &= \delta_2 e_3, &
		\nabla_{e_2} e_2 &= e_1, \qquad \delta_2 = 0,1.
	\end{aligned}
\end{equation}

Observe that if $\delta_2 = 0$, this connection coincides with the one given in~\eqref{Sol23,3} 
for $\delta_1 = \delta_2 = 0$, which has been treated previously. 

Otherwise, if $\delta_2 = 1$, the automorphism
\begin{align*}
	\Psi(e_1) &= e_2, &
	\Psi(e_2) &= e_3, &
	\Psi(e_3) &= e_1
\end{align*}
induces an isomorphism between the connection given in~\eqref{Sol25,1} and the one associated 
with the flat Lie algebra $\mathfrak{h}_{0,4}$.

The sixth solution is given by the following flat, torsion-free connection:
\begin{equation}
	\begin{aligned}\label{Sol26}
		\nabla_{e_2} e_2 &=e_1+b_{32}e_3, &\nabla_{e_2}e_3&=\tfrac{b_{32}c_{33}-b_{33}^2}{b_{32}}e_2+b_{33}e_3, &\nabla_{e_3}e_3&=\tfrac{b_{32}c_{33}-b_{33}^2}{b_{32}^2}e_1+\tfrac{b_{33}(b_{32}c_{33}-b_{33}^2)}{b_{32}^2}e_2+c_{33}e_3.
	\end{aligned}
\end{equation}
Assume that $c_{33} b_{32} - b_{33}^2 \neq 0$. Consider the automorphism
\begin{align*}
	\Psi(e_1) &= e_1, &
	\Psi(e_2) &= \tfrac{b_{33}}{c_{33} b_{32} - b_{33}^2}\, e_1
	+ \tfrac{\sqrt{\varepsilon_1 (c_{33} b_{32} - b_{33}^2)}}{\varepsilon_1}\, e_2, 
	&\Psi(e_3) &= \tfrac{\sqrt{b_{33} (c_{33} b_{32} - b_{33}^2)}}{\varepsilon_1 b_{32}}\, e_2
	+ e_3.
\end{align*}
Applying this automorphism to the connection given in~\eqref{Sol26} yields the following equivalent connection:
\begin{align}\label{Sol26,1}
	\nabla_{e_2} e_2 &= \lambda e_2 + e_3, &
	\nabla_{e_2} e_3 &= \varepsilon_1 e_2, &
	\nabla_{e_3} e_2 &= \varepsilon_1 e_2, &
	\nabla_{e_3} e_3 &= \varepsilon_1 e_3,\quad\varepsilon_1=\pm1,~~\lambda\in\R.
\end{align}
Consider the following automorphism 
\begin{align*}
	\Psi(e_1) &= e_1, &
	\Psi(e_2) &= e_2, &
	\Psi(e_3) &= -\lambda e_2
	+\varepsilon_1 e_3.
\end{align*}
Applying this automorphism to the connection given in~\eqref{Sol26,1} yields the following equivalent connection:
\begin{align}\label{Sol26,1,1}
	\nabla_{e_2} e_2 &= \varepsilon_1 e_3, &
	\nabla_{e_2} e_3 &= e_2 + \lambda e_3, &
	\nabla_{e_3} e_2 &= e_2 + \lambda e_3, 
	&\nabla_{e_3} e_3 &= \tfrac{\lambda}{\varepsilon_1} e_2
	+ \tfrac{\lambda^2 + \varepsilon_1}{\varepsilon_1} e_3,
	\qquad \varepsilon_1 = \pm 1.
\end{align}

Observe that this family of flat, torsion-free connections is equivalent to the family obtained in
\textbf{Case~1}. Therefore, no further analysis is required, since these connections are necessarily
isomorphic to one or more of the connections associated with the flat Lie algebras
$\mathfrak{h}_{0,1}, \ldots, \mathfrak{h}_{0,8}$.

Assume that $c_{33} b_{32} =b_{33}^2$ and $b_{33}\neq0$. Consider the automorphism
\begin{align*}
	\Psi(e_1) &=-b_{33}^2 e_1, &
	\Psi(e_2) &=b_{33} e_2, 
	&\Psi(e_3) &=\tfrac{b_{33}^2}{b_{32}} (e_1+e_2+e_3).
\end{align*}
Applying this automorphism to the connection given in~\eqref{Sol26} yields the following equivalent connection:
\begin{align}\label{Sol26,2}
	\nabla_{e_2} e_2 &= e_2+ e_3, &
	\nabla_{e_2} e_3 &= e_1, &
	\nabla_{e_3} e_2 &= e_1, 
	&\nabla_{e_3} e_3 &=-e_1.
\end{align}
The connection defined in~\eqref{Sol26,2} is isomorphic to the connection corresponding to the flat
Lie algebra $\mathfrak{h}_{0,1}$ through the following automorphism:
\begin{align*}
	\Psi(e_1) &=-e_1, &
	\Psi(e_2) &=-e_3, 
	&\Psi(e_3) &=e_2.
\end{align*}

The seventh solution is given by the following flat, torsion-free connection:
\begin{equation}
\begin{aligned}\label{Sol27}
	\nabla_{e_2} e_2 &=e_1+b_{32}e_3, &\nabla_{e_2}e_3&=\nabla_{e_3}e_2=b_{13}e_1+b_{23}e_2+b_{33}e_3,\\ \nabla_{e_3}e_3&=\tfrac{b_{13}b_{33}+b_{23}}{b_{32}}e_1+\tfrac{b_{23}b_{33}}{b_{32}}e_2+\tfrac{b_{23}b_{32}+b_{33}^2}{b_{32}}e_3.
\end{aligned}
\end{equation}
Assume that $b_{23} \neq 0$. Consider the automorphism
\begin{align*}
	\Psi(e_1) &= e_1, &
	\Psi(e_2) &=\tfrac{b_{33}-b_{13}b_{32}}{b_{23}b_{32}} e_1+x\,e_2, &
	\Psi(e_3) &=\tfrac{b_{33}\sqrt{\varepsilon_2b_{23}b_{32}}}{\varepsilon_2b_{32}} e_2+b_{23}e_3.
\end{align*}
Applying this automorphism to the connection given in~\eqref{Sol27}, for a suitable choice of
$x \in \mathbb{R}^\ast$, yields the following equivalent connection:
\begin{align}\label{Sol27,1}
	\nabla_{e_2} e_2 &= \lambda e_2 +\varepsilon_2 e_3, &
	\nabla_{e_2} e_3 &= e_2, &
	\nabla_{e_3} e_2 &=  e_2, &
	\nabla_{e_3} e_3 &=  e_3,\quad\varepsilon_2=\pm1,~~\lambda\in\R.
\end{align}
Note that if $\varepsilon_2 = 1$, this connection coincides with the one given in~\eqref{Sol26,1}
with $\varepsilon_1 = 1$, which has been treated previously. 

Otherwise, if $\varepsilon_2 = -1$ and $\lambda \neq \pm 1$, then the automorphism
\begin{align*}
	\Psi(e_1) &= e_1, &
	\Psi(e_2) &= \tfrac{\sqrt{\varepsilon_2(\lambda^2-4)}}{2\varepsilon_2}\, e_2 + \tfrac{\lambda}{2} e_3, &
	\Psi(e_3) &= e_3
\end{align*}
induces an isomorphism between the connection given in~\eqref{Sol27,1}, under the above assumptions,
and the one associated with the flat Lie algebra $\mathfrak{h}_{0,5}$.

If, on the other hand, $\varepsilon_2 = -1$ and $\lambda = \pm 2$, then the connection given
in~\eqref{Sol27,1} under these assumptions is isomorphic to the one associated with the flat Lie algebra
$\mathfrak{h}_{0,3}$ via the following automorphism:
\begin{align*}
	\Psi(e_1) &= e_1, &
	\Psi(e_2) &= e_2 + \tfrac{\lambda}{2} e_3, &
	\Psi(e_3) &= e_3,\qquad\lambda=\pm2.
\end{align*}

Assume that $b_{23} = 0$, $b_{33} \neq 0$, and $b_{13} b_{32} - b_{33} \neq 0$. 
Consider the automorphism
\begin{align*}
	\Psi(e_1) &= \tfrac{b_{33}^3}{b_{13} b_{32} - b_{33}}\, e_1, &
	\Psi(e_2) &= b_{33}\, e_2, &
	\Psi(e_3) &= -\tfrac{b_{33}^3}{(b_{13} b_{32} - b_{33}) b_{32}}\, e_1
	+ \tfrac{b_{33}^2}{b_{32}}\, (e_2 + e_3).
\end{align*}
Applying this automorphism to the connection given in~\eqref{Sol27} yields the following equivalent connection:
\begin{align}\label{Sol27,2}
	\nabla_{e_2} e_2 &= e_2 + e_3, &
	\nabla_{e_2} e_3 &= e_1, &
	\nabla_{e_3} e_2 &= e_1, &
	\nabla_{e_3} e_3 &= -e_1.
\end{align}
As we can see, this connection coincides with the one given in~\eqref{Sol26,2}, which is isomorphic
to the connection associated with the flat Lie algebra $\mathfrak{h}_{0,1}$.

Assume that $b_{23} = 0$ and $b_{33} = 0$. Consider the automorphism
\begin{align*}
	\Psi(e_1) &= x\, e_1, &
	\Psi(e_2) &= e_2, &
	\Psi(e_3) &= -\tfrac{x}{b_{32}}\, e_1
	+ \tfrac{b_{33}}{b_{32}}\, e_2
	+ \tfrac{1}{b_{32}}\, e_3.
\end{align*}
Applying this automorphism to the connection given in~\eqref{Sol27}, for a suitable choice of
$x \in \mathbb{R}^\ast$, yields the following equivalent connection:
\begin{align}\label{Sol27,3}
	\nabla_{e_2} e_2 &= e_3, &
	\nabla_{e_2} e_3 &= \delta_1 e_1, &
	\nabla_{e_3} e_2 &= \delta_1 e_1,\qquad\delta_1=0,1.
\end{align}
We now analyze the isomorphisms between this connection and the previously classified flat,
torsion-free connections. We begin with the case $\delta_1 = 0$. The automorphism defined by
\begin{align*}
	\Psi(e_1) &= e_1, &
	\Psi(e_2) &= e_3, &
	\Psi(e_3) &= e_2
\end{align*}
induces an isomorphism between the connection given in~\eqref{Sol27,3} and the one associated
with the flat Lie algebra $\mathfrak{h}_{0,6}$.

If $\delta_1 = 1$, the same automorphism induces an isomorphism between the connection given
in~\eqref{Sol27,3} and the one associated with the flat Lie algebra $\mathfrak{h}_{0,4}$.

\textbf{Case 4.} If $\nabla^0 \equiv \nabla^3$, which corresponds to the Lie algebra $\mathfrak{b}_3$
(see Table~\ref{FlatR2}), then the flatness equations associated with the connection
given in~\eqref{Connegeneral} can be solved directly, yielding ten distinct solutions.

The first solution is given by the following flat, torsion-free connection:
\begin{equation}
\begin{aligned}\label{Sol31}
	\nabla_{e_1} e_1 &=e_1, \qquad\nabla_{e_1}e_3=\nabla_{e_3}e_1=a_{13}e_1, \qquad\nabla_{e_2}e_2=e_2, \qquad\nabla_{e_2,e_3}=\nabla_{e_3}e_2=b_{23}e_2, \\
	\nabla_{e_3}e_3&=a_{13}(a_{13}-c_{33})e_1+b_{23}(b_{23}-c_{33})+c_{33}e_3.
\end{aligned}
\end{equation}
Consider the automorphism
\begin{align*}
	\Psi(e_1) &= e_1, &
	\Psi(e_2) &= e_2, &
	\Psi(e_3) &= a_{13} e_1 + b_{23} e_2 + x\, e_3.
\end{align*}
Applying this automorphism to the connection given in~\eqref{Sol31}, for a suitable choice of
$x \in \mathbb{R}$, yields the following equivalent connection:
\begin{equation}
	\begin{aligned}\label{Sol31,1}
		\nabla_{e_1} e_1 &= e_1, &
		\nabla_{e_2} e_2 &= e_2, &
		\nabla_{e_3} e_3 &= \delta_1 e_3, \qquad \delta_1 = 0,1.
	\end{aligned}
\end{equation}

If $\delta_1 = 1$, then it is straightforward to verify that the automorphism
\begin{align*}
	\Psi(e_1) &= \tfrac{1}{2} (e_2 + e_3), &
	\Psi(e_2) &= \tfrac{1}{2} (e_2 - e_3), &
	\Psi(e_3) &= e_1
\end{align*}
induces an isomorphism between the connection given in~\eqref{Sol31,1} and the one associated
with the flat Lie algebra $\mathfrak{h}_{1,3}$.

Otherwise, if $\delta_1 = 0$, the automorphism
\begin{align*}
	\Psi(e_1) &= e_1, &
	\Psi(e_2) &= e_2 + e_3, &
	\Psi(e_3) &= e_2
\end{align*}
induces an isomorphism between the connection given in~\eqref{Sol31,1} and the one associated
with the flat Lie algebra $\mathfrak{h}_{1,1}$.

The second solution is given by the following flat, torsion-free connection:
\begin{align}\label{Sol32}
	\nabla_{e_1} e_1 &=e_1+a_{31}e_3, & \nabla_{e_2}e_2&=e_2+b_{32}e_3.
\end{align}
Consider the automorphism
\begin{align*}
	\Psi(e_1) &= e_1-a_{31}e_3, &
	\Psi(e_2) &= e_2-b_{32}e_3, &
	\Psi(e_3) &=  e_3.
\end{align*}
Applying this automorphism to the connection given in~\eqref{Sol32}, yields the following equivalent connection:
\begin{equation}
	\begin{aligned}\label{Sol31,1}
		\nabla_{e_1} e_1 &= e_1, &
		\nabla_{e_2} e_2 &= e_2.
	\end{aligned}
\end{equation}
Observe that this connection coincides with the one given in~\eqref{Sol31,1} with $\delta_1 = 0$;
hence, no further analysis is required.

The third solution is given by the following flat, torsion-free connection:
\begin{equation}
	\begin{aligned}\label{Sol33}
		\nabla_{e_1} e_1 &=e_1, &\nabla_{e_2}e_2&=e_2+b_{32}e_3, &\nabla_{e_2}e_3&=\nabla_{e_3}e_2=\tfrac{c_{33}b_{32}-b_{33}^2+b_{33}}{b_{32}}e_2+b_{33}e_3, \\& &&&\nabla_{e_3}e_3&=\tfrac{b_{33}(c_{33}b_{32}-b_{33}^2+b_{33})}{b_{32}^2}e_2+c_{33}e_3.
	\end{aligned}
\end{equation}
Note that $b_{32}\neq0$ and consider the following automorphism
\begin{align*}
	\Psi(e_1) &= e_1, &
	\Psi(e_2) &= -b_{32}e_2, &
	\Psi(e_3) &=  b_{32} e_3.
\end{align*}
Applying this automorphism to the connection given in~\eqref{Sol33}, yields the following equivalent connection:
\begin{equation}
	\begin{aligned}\label{Sol33,1}
		\nabla_{e_1} e_1 &= e_1, &
		\nabla_{e_2} e_2 &= e_3, &\nabla_{e_2}e_3&=\nabla_{e_3}e_2=\tfrac{c_{33}b_{32}-b_{33}^2}{b_{32}^2}e_2-\tfrac{1+b_{33}}{b_{32}}e_3,\\
		& && &\nabla_{e_3}e_3&=-\tfrac{(c_{33}b_{32}-b_{33}^2)(1+b_{33})}{b_{32}^3}e_2+\tfrac{c_{33}b_{32}+2\,b_{33}+1}{b_{32}^2}e_3.
	\end{aligned}
\end{equation}
Observe that this family of flat, torsion-free connections is equivalent to the family obtained in
\textbf{Case~2}. Therefore, no further analysis is required.

The fourth solution is given by the following flat, torsion-free connection:
\begin{equation}
	\begin{aligned}\label{Sol34}
		\nabla_{e_1} e_1 &= e_1, &
		\nabla_{e_1} e_3 &=\nabla_{e_3}e_1=a_{13}e_1, &\nabla_{e_2}e_2&=e_2, &\nabla_{e_2}e_3&=\nabla_{e_3}e_2=-a_{13}e_1+e_3,\\ && && &&\nabla_{e_3}e_3&= a_{13}(a_{13}-c_{33})e_1+c_{23}e_2+c_{33}e_3.
	\end{aligned}
\end{equation}
Consider the following automorphism:
\begin{align*}
	\Psi(e_1) &= e_1, &
	\Psi(e_2) &= e_2, &
	\Psi(e_3) &= a_{13}e_1 + \tfrac{c_{33}}{2}e_2 + x\,e_3.
\end{align*}
Applying this automorphism to the connection given in~\eqref{Sol34}, for a suitable choice of
$x \in \mathbb{R}^\ast$, yields the following equivalent connection:
\begin{align}\label{Sol34,1}
	\nabla_{e_1} e_1 &= e_1, &
	\nabla_{e_2} e_2 &= e_2, &
	\nabla_{e_2} e_3 &= e_3, &
	\nabla_{e_3} e_2 &= e_3, &
	\nabla_{e_3} e_3 &= \delta_1 e_2,
	\qquad \delta_1 = 0, \pm 1.
\end{align}
We now analyze the isomorphisms between the connection given in~\eqref{Sol34,1} and the previously
classified flat connections. We begin with the case $\delta_1 = 0$.
The automorphism
\begin{align*}
	\Psi(e_1) &= e_1, &
	\Psi(e_2) &= e_3, &
	\Psi(e_3) &= e_2
\end{align*}
induces an isomorphism between the connection defined in~\eqref{Sol34,1} under this assumption
and the one associated with the flat Lie algebra $\mathfrak{h}_{1,2}$.
Otherwise, if $\delta_1 = 1$, this connection in fact coincides with the one associated
with the flat Lie algebra $\mathfrak{h}_{1,3}$. 
Now, if $\delta_1 = -1$, it is straightforward to verify that, for every choice of
automorphism $\Psi \in \mathrm{GL}_3(\mathbb{R})$, this connection is not isomorphic to
any of the previously classified connections. Therefore, we introduce a new flat,
torsion-free connection given by
\begin{align}\label{Sol34,1,1}
\h_{3,1}:&	&\nabla_{e_1} e_1 &= e_1, &
	\nabla_{e_2} e_2 &= e_2, &
	\nabla_{e_2} e_3 &= e_3, &
	\nabla_{e_3} e_2 &= e_3, &
	\nabla_{e_3} e_3 &= - e_2.
\end{align}

The fifth solution is given by the following flat, torsion-free connection:
\begin{equation}
	\begin{aligned}\label{Sol35}
		\nabla_{e_1}e_1&=e_1+\tfrac{a_{32}^2(a_{33}-1)}{a_{33}b_{32}+a_{32}}e_3, &\nabla_{e_1}e_2&=\nabla_{e_2}e_1=a_{32}e_3, &\nabla_{e_1}e_3&=\nabla_{e_3}e_1=a_{33}e_3,\\
		\nabla_{e_2}e_2&=e_2+b_{32}e_3, &\nabla_{e_2}e_3&=\nabla_{e_3}e_2=\tfrac{a_{33}b_{32}+a_{32}}{a_{32}}e_3, &\nabla_{e_3}e_3&=\tfrac{a_{33}(a_{33}b_{32}+a_{32})}{a_{32}^2}e_3.
	\end{aligned}
\end{equation}
Assume that $a_{33}\neq 0$. Consider the automorphism
\begin{align*}
	\Psi(e_1) &= e_1+a_{33}e_3, &
	\Psi(e_2) &= e_2+\tfrac{a_{33}b_{32}+a_{32}}{a_{32}}e_3, &
	\Psi(e_3) &= \tfrac{a_{33}(a_{33}b_{32}+a_{32})}{a_{32}^2}e_3.
\end{align*}
Applying $\Psi$ to the connection defined in~\eqref{Sol35}
yields the following equivalent connection:
\begin{align*}\label{Sol35,1}
	\nabla_{e_1}e_1&= e_1, &
	\nabla_{e_2}e_2 &= e_2, &
	\nabla_{e_3}e_3 &= e_3.
\end{align*}
In fact, this connection is isomorphic to the flat torsion-free connection associated
with the Lie algebra $\mathfrak{h}_{1,3}$ via the following automorphism:
\begin{align*}
	\Psi(e_1) &=\tfrac{1}{2}( e_2+e_3), &
	\Psi(e_2) &= \tfrac{1}{2}( e_2-e_3), &
	\Psi(e_3) &= e_1.
\end{align*}

Now, if $a_{33}=0$, consider the following automorphism
\begin{align*}
	\Psi(e_1) &= e_1 + \tfrac{a_{32}^2}{a_{33} b_{32} + a_{32}}\, e_3, &
	\Psi(e_2) &= e_2 + b_{32} e_3, &
	\Psi(e_3) &= e_3.
\end{align*}
Applying this automorphism to the connection given in~\eqref{Sol35} yields the following
equivalent connection:
\begin{align}
	\nabla_{e_1} e_1 &= e_1, &
	\nabla_{e_2} e_2 &= e_2, &
	\nabla_{e_2} e_3 &= e_3, &
	\nabla_{e_3} e_2 &= e_3.
\end{align}
Observe that this connection coincides with the one given in~\eqref{Sol34,1} with $\delta_1=0$, which has
already been treated. Hence, no further analysis is required.

The sixth solution is given by the following flat, torsion-free connection:
\begin{equation}
\begin{aligned}\label{Sol36}
	\nabla_{e_1} e_1 &= e_1+b_{32}e_3, &
	\nabla_{e_1} e_2 &=\nabla_{e_2}e_3=-b_{32} e_3, \qquad	\nabla_{e_1} e_3 =\nabla_{e_3}e_1=a_{13}e_1+\tfrac{b_{32}c_{23}}{a_{33}-1}e_2+a_{33} e_3, \\
	\nabla_{e_2} e_2 &=e_2+b_{32} e_3, &
	\nabla_{e_2} e_3 &=\nabla_{e_3}e_2=-a_{13}e_1-\tfrac{b_{32}c_{23}}{a_{33}-1}e_2 +(1-a_{33}) e_3, \\
	&&\nabla_{e_3}e_3&=\tfrac{a_{13}a_{33}}{b_{32}}e_1+c_{23}e_2+\tfrac{a_{13}a_{33}b_{32}+a_{33}^2-b_{32}^2c_{23}-a_{13}b_{32}-2\,a_{33}^2+a_{33}}{b_{32}(a_{33}-1)}e_3.
\end{aligned}
\end{equation}
If $a_{33}\neq2$. Consider the following automorphism:
\begin{align*}
	\Psi(e_1) &=e_1-\tfrac{a_{33}-2}{2}e_2, &
	\Psi(e_2) &= \tfrac{2-a_{33}}{2}e_2, &
	\Psi(e_3) &=\tfrac{a_{33}(a_{33}-2)}{4\,b_{32}} e_2+\tfrac{(a_{33}-2)^2}{4\,b_{32}}e_3.
\end{align*}
Applying this automorphism to the connection given in \eqref{Sol36}, yields the following equivalent connection:
\begin{equation}
	\begin{aligned}\label{Sol36,1}
		\nabla_{e_1} e_1 &= e_1,  &\nabla_{e_1} e_2&=e_2, &\nabla_{e_1} e_3&=e_3,\\
		\nabla_{e_2} e_1&=e_2, &\nabla_{e_2} e_2&= e_2+e_3, &\nabla_{e_2} e_3&=\lambda_1e_1+\lambda_2e_2+e_3,\\ \nabla_{e_3} e_1&=e_3, &\nabla_{e_3}e_2&=\lambda_1e_1+\lambda_2e_2+e_3, &\nabla_{e_3}e_3&=(\lambda_1+\lambda_2)e_2+\lambda_2e_3, \quad\lambda_1,\lambda_2\in\R.
	\end{aligned}
\end{equation}
Now, we simplify the connection given in~(\ref{Sol36,1}). If $\lambda_2 \neq -\frac{1}{3}$, we consider the following automorphism
\begin{align*}
	\Psi(e_1) &=e_1, &
	\Psi(e_2) &=\tfrac{2}{3}e_1+\tfrac{\sqrt{3}\sqrt{\varepsilon_1(3\,\lambda_2+1)}}{3\varepsilon_1}e_2 , &
	\Psi(e_3) &=-\tfrac{2}{9}e_1+\tfrac{\sqrt{3}\sqrt{\varepsilon_1(3\,\lambda_2+1)}}{9\varepsilon_1}e_2+\tfrac{1+3\,\lambda_2}{3\varepsilon_1}e_3.
\end{align*}
 and apply it to~(\ref{Sol36,1}), which yields the following equivalent connection:
 \begin{equation}
 	\begin{aligned}\label{Sol36,1,1}
 		\nabla_{e_1} e_1 &= e_1,  &\nabla_{e_1} e_2&=e_2, &\nabla_{e_1} e_3&=e_3,\\
 		\nabla_{e_2} e_1&=e_2, &\nabla_{e_2} e_2&= e_3, &\nabla_{e_2} e_3&=\lambda e_1+\varepsilon_1e_2,\\ \nabla_{e_3} e_1&=e_3, &\nabla_{e_3}e_2&=\lambda e_1+\varepsilon_1e_2, &\nabla_{e_3}e_3&=\lambda e_2+\varepsilon_1e_3, \quad\lambda\in\R,~~\varepsilon_1=\pm1.
 	\end{aligned}
 \end{equation}
 If $\lambda=0$ and $\varepsilon_1=1$, then the following automorphism
 \begin{align*}
 	\Psi(e_1) &= e_1+e_2, &
 	\Psi(e_2) &= e_3, &
 	\Psi(e_3) &= e_2
 \end{align*}
 establishes an isomorphism between the connection given in~(\ref{Sol36,1,1}) and that associated with the flat Lie algebra $\h_{1,3}$.
 
 If $\lambda=0$ and $\varepsilon_1=-1$, then the following automorphism
 \begin{align*}
 	\Psi(e_1) &= e_1+e_2, &
 	\Psi(e_2) &= e_3, &
 	\Psi(e_3) &= -e_2
 \end{align*}
 establishes an isomorphism between the connection given in~(\ref{Sol36,1,1}) and that associated with the flat Lie algebra $\h_{3,1}$.

 Since we work in low dimensions, we can verify, without using Lemma~\ref{isoofconnection}, that the connection given in~(\ref{Sol36,1,1}) is not isomorphic to any of the previously classified connections. 
 
 With this analysis, we may assume that $\lambda \neq 0$. Moreover, we note that the upper block of this connection coincides with \textbf{Case~5}, which will be treated later. For this reason, we denote this connection as being associated with the flat Lie algebra $\h_{4,1}$.

If $\lambda_2 = -\frac{1}{3}$, we consider the following automorphism:
\begin{align*}
	\Psi(e_1) &=e_1 , &
	\Psi(e_2) &= \tfrac{2}{3}e_1+x\,e_2, &
	\Psi(e_3) &= -\tfrac{2}{9}e_1+\tfrac{x}{3}e_2+x^2\,e_3.
\end{align*}
We then apply it to~(\ref{Sol36,1}), for a suitable parameter $x \in \mathbb{R}^\ast$, which yields the following equivalent connection:
 \begin{equation}
	\begin{aligned}\label{Sol36,1,2}
		\nabla_{e_1} e_1 &= e_1,  &\nabla_{e_1} e_2&=e_2, &\nabla_{e_1} e_3&=e_3,\\
		\nabla_{e_2} e_1&=e_2, &\nabla_{e_2} e_2&= e_3, &\nabla_{e_2} e_3&=\delta e_1\\ \nabla_{e_3} e_1&=e_3, &\nabla_{e_3}e_2&=\delta e_1, &\nabla_{e_3}e_3&=\delta e_2, \quad \delta=0,1.
	\end{aligned}
\end{equation}
If $\delta = 0$, then the following automorphism
\begin{align*}
	\Psi(e_1) &= e_1-e_2+e_3, &
	\Psi(e_2) &= e_2, &
	\Psi(e_3) &= e_1
\end{align*}
establishes an isomorphism between the connection given in~(\ref{Sol36,1,2}) and that associated with the flat Lie algebra $\mathfrak{h}_{2,1}$. Otherwise, if $\delta=1$, then the following automorphism
\begin{align*}
	\Psi(e_1) &= e_1, &
	\Psi(e_2) &=e_1-\tfrac{\sqrt{3}}{2} e_2+\tfrac{3}{2} e_3, &
	\Psi(e_3) &= e_1+\tfrac{\sqrt{3}}{2} e_2+\tfrac{3}{2} e_3
\end{align*}
establishes an isomorphism between the connection given in~(\ref{Sol36,1,2}) and that associated with the flat Lie algebra $\h_{4,1}$ with $\lambda=0$ and $\varepsilon_1=-1$.

If $a_{33}=2$ and $a_{13}\neq0$. Consider the following automorphism:
\begin{align*}
	\Psi(e_1) &=e_1-(-a_{13}b_{32})^{\frac{1}{3}}e_2, &
	\Psi(e_2) &=(-a_{13}b_{32})^{\frac{1}{3}}e_2, &
	\Psi(e_3) &=-\tfrac{(-a_{13}b_{32})^{\frac{1}{3}}}{b_{32}}e_2+\tfrac{(-a_{13}b_{32})^{\frac{2}{3}}}{b_{32}}e_3.
\end{align*}
Applying this automorphism to the connection given in \eqref{Sol36}, yields the following equivalent connection:
\begin{equation}
	\begin{aligned}\label{Sol36,2}
		\nabla_{e_1} e_1 &= e_1,  &\nabla_{e_1} e_2&=e_2, &\nabla_{e_1} e_3&=e_3,\\
		\nabla_{e_2} e_1&=e_2, &\nabla_{e_2} e_2&= e_3, &\nabla_{e_2} e_3&= e_1+\lambda_3e_2\\ \nabla_{e_3} e_1&=e_3, &\nabla_{e_3}e_2&=e_1+\lambda_3 e_2, &\nabla_{e_3}e_3&=e_2+\lambda_3e_3, \quad \lambda_3\in\R.
	\end{aligned}
\end{equation}
If $\lambda_3 = 0$, then this connection coincides with the one treated previously in~\eqref{Sol36,1,2} with $\delta = 1$. Otherwise, the following automorphism
\begin{align*}
	\Psi(e_1) &= e_1 , &
	\Psi(e_2) &= -\tfrac{2\,\lambda_3^2}{9} e_1+\tfrac{\lambda_3^2\sqrt{3}}{9}e_2+\tfrac{\lambda_3^2}{3}e_3, &
	\Psi(e_3) &= \tfrac{2\,\lambda_3}{3}e_1+\tfrac{\sqrt{3}\lambda_3}{3}e_2
\end{align*}
establishes an isomorphism between the connection given in~\eqref{Sol36,2} and the one associated with the flat Lie algebra $\mathfrak{h}_{4,1}$ in~\eqref{Sol36,1,1}, with $\varepsilon_1 = 1$ and $\lambda = -\tfrac{\sqrt{3}\,(2\lambda_3^3 - 27)}{9\,\lambda_3^3}$.

If $a_{33}=2$ and $a_{13}=0$. Consider the following automorphism:
\begin{align*}
	\Psi(e_1) &=e_1-x\,e_2, &
	\Psi(e_2) &=x\,e_2, &
	\Psi(e_3) &=-\tfrac{x}{b_{32}}e_2+\tfrac{x^2}{b_{32}}e_3.
\end{align*}
Applying this automorphism to the connection given in \eqref{Sol36} for a suitable parameter $x\in\R^\ast$, yields the following equivalent connection:
\begin{equation}
	\begin{aligned}\label{Sol36,3}
		\nabla_{e_1} e_1 &= e_1,  &\nabla_{e_1} e_2&=e_2, &\nabla_{e_1} e_3&=e_3,\\
		\nabla_{e_2} e_1&=e_2, &\nabla_{e_2} e_2&= e_3, &\nabla_{e_2} e_3&=\varepsilon_2e_2,\\ \nabla_{e_3} e_1&=e_3, &\nabla_{e_3}e_2&=\varepsilon_2e_2, &\nabla_{e_3}e_3&=\varepsilon_2e_3, \quad \varepsilon_2=\pm1.
	\end{aligned}
\end{equation}
Note that this connection coincides with the one associated with the flat Lie algebra $\mathfrak{h}_{4,1}$ given in \eqref{Sol36,2} with $\lambda = 0$; therefore, no further analysis is required.

The seventh solution is given by the following flat, torsion-free connection:
\begin{equation}
	\begin{aligned}\label{Sol37}
		\nabla_{e_1} e_1 &= e_1+a_{31}e_3,   &\nabla_{e_1} e_3&=\tfrac{a_{31}c_{33}-a_{33}^2+a_{33}}{a_{31}}e_1+a_{33}e_3,
&		\nabla_{e_2} e_2&= e_2,\\ \nabla_{e_3} e_1&=\tfrac{a_{31}c_{33}-a_{33}^2+a_{33}}{a_{31}}e_1+a_{33}e_3, &\nabla_{e_3}e_3&=\tfrac{a_{33}(a_{31}c_{33}-a_{33}^2+a_{33})}{a_{31}^2}e_1+c_{33}e_3.
	\end{aligned}
\end{equation}
If $a_{31}c_{33}-a_{33}^2\neq0$. Consider the following automorphism
\begin{align*}
	\Psi(e_1) &=\tfrac{\sqrt{\varepsilon_4(a_{31}c_{33}-a_{33}^2)}}{\varepsilon_4}e_3, &
	\Psi(e_2) &=e_2, &
	\Psi(e_3) &=\tfrac{a_{31}c_{33}-a_{33}^2}{\varepsilon_4 a_{31}}e_1+\tfrac{a_{33}\sqrt{\varepsilon_4(a_{31}c_{33}-a_{33}^2)}}{\varepsilon_4 a_{31}}e_3.
\end{align*}
Applying this automorphism to the connection given in~\eqref{Sol37}, yields the following equivalent connection:
\begin{equation}
	\begin{aligned}\label{Sol37,1}
		\nabla_{e_1} e_1 &=\varepsilon_4 e_1, &
		\nabla_{e_1} e_3 &=\varepsilon_4 e_3, &
		\nabla_{e_2} e_2 &=  e_1, &\nabla_{e_3}e_1&=\varepsilon_4 e_3, &\nabla_{e_3}e_3&=e_1+\lambda_4 e_3, \quad \lambda_4\in\R,~~\varepsilon_4=\pm1.
	\end{aligned}
\end{equation}
Assume that $\varepsilon_4 = 1$ and $\lambda_4 \in \mathbb{R}$. Then the automorphism
\begin{align*}
	\Psi(e_1) &= -\tfrac{2\lambda_4^2}{9}e_1-\tfrac{\lambda_4\sqrt{3\lambda_4^2+9}}{9}e_2+(\tfrac{\lambda_4^2}{3}+1)e_3, &
	\Psi(e_2) &=(\tfrac{2\lambda_4^2}{9}+1)e_1+\tfrac{\lambda_4\sqrt{3\lambda_4^2+9}}{9}e_2-(\tfrac{\lambda_4^2}{3}+1)e_3, &
	\Psi(e_3) &=\tfrac{\lambda_4}{3}e_1+\tfrac{\sqrt{3\lambda_4^2+9}}{3}e_2,
\end{align*}
establishes an isomorphism between the connection given in \eqref{Sol37,1} and the one associated with the flat Lie algebra $\mathfrak{h}_{4,1}$ in \eqref{Sol36,1,1} with $\lambda=\frac{\lambda_4(2\lambda_4^2+9)\sqrt{3\lambda_4^2+9}}{9\,(\lambda_4^2+3)^2}$ and $\varepsilon_1=1$.

Suppose that $\varepsilon_4 = -1$ and $\lambda_4 \in \mathbb{R}$ with $\lambda_4 \neq \pm\sqrt{3}$. Then the following automorphism
\begin{align*}
	\Psi(e_1) &=-\tfrac{2\lambda_4^2}{9}e_1-\tfrac{\lambda_4\sqrt{3}\sqrt{\varepsilon_1(\lambda_4^2-3)}}{9\,\varepsilon_1}e_2 +\tfrac{\lambda_4^2-3}{3\,\varepsilon_1}e_3, &
	\Psi(e_2) &=(1-\tfrac{2\,\lambda_4^2}{9})e_1-\tfrac{\lambda_4\sqrt{3}\sqrt{\varepsilon_1(\lambda_4^2-3)}}{9\,\varepsilon_1}e_2+\tfrac{\lambda_4^2-3}{3\,\varepsilon_1}e_3, &
	\Psi(e_3) &=\tfrac{\lambda_4}{3}e_1+\tfrac{\sqrt{3}\sqrt{\varepsilon_1(\lambda_4^2-3)}}{3\,\varepsilon_1}e_2,
\end{align*}
establishes an isomorphism between the connection given in \eqref{Sol37,1} and the one associated with the flat Lie algebra $\mathfrak{h}_{4,1}$ in \eqref{Sol36,1,1} with $\lambda=\frac{\lambda_4\varepsilon_1^2\sqrt{3}(2\,\lambda_4^2-9)}{\sqrt{\varepsilon_1(\lambda_4^2-3)}(9\,\lambda_4^2-27)}$.

Assume now that $\varepsilon_4 = -1$ and $\lambda_4 = \pm \sqrt{3}$. Then the following automorphism
\begin{align*}
	\Psi(e_1) &=-e_2, &
	\Psi(e_2) &=e_1, &
	\Psi(e_3) &=\tfrac{\sqrt{3}\varepsilon_0}{2}e_2+\tfrac{\sqrt{4}-3\,\varepsilon_0}{2},\quad\varepsilon_0=\pm1.
\end{align*}
establishes an isomorphism between the connection given in \eqref{Sol37,1} and the one associated with the flat Lie algebra $\h_{3,1}$.

If $a_{31}c_{33}-a_{33}^2=0$. Consider the following automorphism
\begin{align*}
	\Psi(e_1) &= x\, e_3, &
	\Psi(e_2) &= e_2, &
	\Psi(e_3) &= \tfrac{x^2}{a_{31}} e_1 + \tfrac{a_{33} x}{a_{31}} e_3.
\end{align*}
Applying this automorphism to the connection given in \eqref{Sol37}, for a suitable parameter $x \in \mathbb{R}^\ast$, yields the following equivalent connection:
\begin{equation}
	\begin{aligned}\label{Sol37,2}
		\nabla_{e_2} e_2 &=e_2, &
		\nabla_{e_3} e_3 &=e_1+\delta e_3,\quad\delta=0,1.
	\end{aligned}
\end{equation}
If $\delta = 0$, then the following automorphism
\begin{align*}
	\Psi(e_1) &= e_1, &
	\Psi(e_2) &= -e_1 + e_2 - e_3, &
	\Psi(e_3) &= e_2,
\end{align*}
establishes an isomorphism between the connection given in \eqref{Sol37,2} and the one associated with the flat Lie algebra $\mathfrak{h}_{0,1}$. Otherwise, if $\delta = 1$, the following automorphism
\begin{align*}
	\Psi(e_1) &= e_1, &
	\Psi(e_2) &=\tfrac{1}{2} e_2 +\tfrac{1}{2} e_3, &
	\Psi(e_3) &= -e_1-\tfrac{1}{2} e_2 +\tfrac{1}{2} e_3,
\end{align*}
establishes an isomorphism between the connection given in \eqref{Sol37,2} and the one associated with the flat Lie algebra $\mathfrak{h}_{0,5}$

The eighth solution is given by the following flat, torsion-free connection:
\begin{equation}
	\begin{aligned}\label{Sol38}
		\nabla_{e_1} e_1 &=e_1, &
		\nabla_{e_1} e_3 &=\nabla_{e_3}e_1=-b_{23} e_2+e_3,
		&\nabla_{e_2} e_2 &=e_2, &
		\nabla_{e_2} e_3 &=\nabla_{e_3}e_2=b_{23}e_2\\
		&&\nabla_{e_3}e_3&=c_{13}e_1+b_{23}(b_{23}-c_{33})e_2+c_{33}e_3.
	\end{aligned}
\end{equation}
Consider the following automorphism:
\begin{align*}
	\Psi(e_1) &= e_1, &
	\Psi(e_2) &= e_2 , &
	\Psi(e_3) &=  \tfrac{c_{33}}{2}e_1+b_{23}e_2+x\,e_3.
\end{align*}
Applying it to the connection given in \eqref{Sol38}, for a suitable parameter $x \in \mathbb{R}^\ast$, yields the following equivalent connection:
\begin{equation}
	\begin{aligned}\label{Sol38,1}
		\nabla_{e_1} e_1 &=e_1, &
		\nabla_{e_1} e_3 &=e_3,
		&\nabla_{e_2} e_2 &=e_2,
		&\nabla_{e_3}e_1&=e_3, &\nabla_{e_3}e_3&=\varepsilon_0 e_1,\quad\varepsilon_0=0,\pm1.
	\end{aligned}
\end{equation}
If $\varepsilon_0 = 0$, then the following automorphism:
\begin{align*}
	\Psi(e_1) &= e_3, &
	\Psi(e_2) &= e_1, &
	\Psi(e_3) &= e_2
\end{align*}
establishes an isomorphism between the connection given in \eqref{Sol38,1} and the one associated with the flat Lie algebra $\h_{1,2}$.

If $\varepsilon_0 = 1$, then the following automorphism:
\begin{align*}
	\Psi(e_1) &= e_2, &
	\Psi(e_2) &= e_1, &
	\Psi(e_3) &= e_3
\end{align*}
establishes an isomorphism between the connection given in \eqref{Sol38,1} and the one associated with the flat Lie algebra $\h_{1,3}$.

If $\varepsilon_0 = -1$, then the following automorphism:
\begin{align*}
	\Psi(e_1) &= e_2, &
	\Psi(e_2) &= e_1, &
	\Psi(e_3) &= e_3
\end{align*}
establishes an isomorphism between the connection given in \eqref{Sol38,1} and the one associated with the flat Lie algebra $\h_{3,1}$.

The ninth solution is given by the following flat, torsion-free connection:
\begin{equation}
	\begin{aligned}\label{Sol39}
		\nabla_{e_1} e_1 &=e_1+b_{32}e_3, &\nabla_{e_1}e_2&=\nabla_{e_2}e_1=-b_{32}e_3, &
		\nabla_{e_1} e_3 &=\nabla_{e_3}e_1=c_{13}b_{32}e_1-b_{23}e_2+e_3,\\
		\nabla_{e_2} e_2 &=e_2+b_{32}e_3, &
		\nabla_{e_2} e_3 &=\nabla_{e_3}e_2=-c_{13}b_{32}e_1+b_{23}e_2,
	&\nabla_{e_3}e_3&=c_{13}e_1+(c_{13}b_{32}+b_{23})e_3.
	\end{aligned}
\end{equation}
If $b_{32} \neq 0$, consider the following automorphism:
\begin{align*}
	\Psi(e_1) &= e_1-e_2, &
	\Psi(e_2) &= e_2, &
	\Psi(e_3) &= \tfrac{1}{b_{32}}e_1- \tfrac{1}{b_{32}}e_2+ \tfrac{1}{b_{32}}e_3.
\end{align*}
Applying it to the connection given in \eqref{Sol39} yields the following equivalent connection:
\begin{equation}
	\begin{aligned}\label{Sol39,1}
		\nabla_{e_1} e_1 &=e_1, &\nabla_{e_1}e_2&=e_2, &
		\nabla_{e_1} e_3 &=e_3,\\
		\nabla_{e_2} e_1 &=e_2, &\nabla_{e_2}e_2&=e_1+e_3, &
		\nabla_{e_2} e_3 &=ae_1+be_2+e_3,\\
		\nabla_{e_3}e_1&=e_3, &\nabla_{e_3}e_2&=ae_1+be_2+e_3, &\nabla_{e_3}e_3&=(a+b)e_1+(a+b)e_2+be_3,\quad a,b\in\R.
	\end{aligned}
\end{equation}
It is straightforward to verify that the following automorphism:
\begin{align*}
	\Psi(e_1) &= e_1, &
	\Psi(e_2) &=-\tfrac{1}{3}e_1+ e_2, &
	\Psi(e_3) &= -\tfrac{8}{9}e_1+\tfrac{1}{3}e_2+e_3.
\end{align*}
establishes an isomorphism between the connection given in \eqref{Sol39,1} and the one treated in \eqref{Sol36,1}, with $\lambda_1 =a-\frac{b}{3}-\frac{32}{27} $ and $\lambda_2 =1+b $.

If $b_{32} = 0$, consider the following automorphism:
\begin{align*}
	\Psi(e_1) &= e_1 - e_2, &
	\Psi(e_2) &= e_2, &
	\Psi(e_3) &= \tfrac{b_{23}}{2} e_1 - \tfrac{b_{23}}{2} e_2 + x\,e_3.
\end{align*}
Applying it to the connection given in \eqref{Sol39}, for a suitable parameter $x \in \mathbb{R}^\ast$, yields the following equivalent connection:
\begin{equation}
	\begin{aligned}\label{Sol39,2}
	\nabla_{e_1} e_1 &=e_1, &\nabla_{e_1}e_2&=e_2, &
		\nabla_{e_1} e_3 &=e_3,&
		\nabla_{e_2} e_1 &=e_2, &\nabla_{e_2}e_2&=e_2,&
		\nabla_{e_3}e_1&=e_3,\\ \nabla_{e_3}e_3&=\varepsilon_0e_1-\varepsilon_0e_2, && & \varepsilon_0=0,\pm1&.
	\end{aligned}
\end{equation}
If $\varepsilon_0 = \pm 1$, then the following automorphism:
\begin{align*}
	\Psi(e_1) &= e_1, &
	\Psi(e_2) &= e_1 - \tfrac{1}{\varepsilon_1} e_3, &
	\Psi(e_3) &= \tfrac{\sqrt{\varepsilon_0 \varepsilon_1}}{\varepsilon_1} e_2.
\end{align*}
establishes an isomorphism between the connection given in \eqref{Sol39,2} and the one associated with the flat Lie algebra $\h_{4,1}$ given in \eqref{Sol36,1,1} with $\lambda=0$. Otherwise, if $\varepsilon_0 = 0$, then the following automorphism:
\begin{align*}
	\Psi(e_1) &= e_1, &
	\Psi(e_2) &=\tfrac{1}{9} e_1 + \tfrac{2\sqrt{3}}{9}e_2+\tfrac{1}{3} e_3, &
	\Psi(e_3) &=e_1+ \tfrac{\sqrt{3}}{2} e_2-\tfrac{3}{2}e_3.
\end{align*}
establishes an isomorphism between the connection given in \eqref{Sol39,2} and, once again, the one associated with the flat Lie algebra $\h_{4,1}$, with $\varepsilon = 1$ and $\lambda = \tfrac{2\sqrt{3}}{9}$.

The tenth solution is given by the following flat, torsion-free connection:
\begin{equation}
	\begin{aligned}\label{Sol310}
		\nabla_{e_1} e_1 &=e_1+a_{31}e_3, &\nabla_{e_1}e_2&=\nabla_{e_2}e_1=-b_{32}e_3, &\nabla_{e_1}e_3&=\nabla_{e_3}e_1=e_3,
		&\nabla_{e_2} e_2 &=e_2+b_{32}e_3,
	\end{aligned}
\end{equation}
If $(a_{31}-b_{32})\neq0$. Consider the following automorphism:
\begin{align*}
	\Psi(e_1) &= e_1-e_2, &
	\Psi(e_2) &= e_2, &
	\Psi(e_3) &= \tfrac{1}{a_{31}-b_{32}}e_3.
\end{align*}
Applying it to the connection given in \eqref{Sol310} yields the following equivalent connection:
\begin{equation}
	\begin{aligned}\label{Sol310,1}
		\nabla_{e_1} e_1 &=e_1+e_3, &\nabla_{e_1}e_2&=e_2, &
		\nabla_{e_1} e_3 &=e_3,&
		\nabla_{e_2} e_1 &=e_2, &\nabla_{e_2}e_2&=e_2+\lambda_0 e_3,&
		\nabla_{e_3}e_1&=e_3,\quad\lambda_0\in\R.
	\end{aligned}
\end{equation}
It is straightforward to verify that the following automorphism:
\begin{align*}
	\Psi(e_1) &=\tfrac{1}{3}e_1+\tfrac{\sqrt{3}}{3}e_2 + e_3, &
	\Psi(e_2) &=(\tfrac{2\,\lambda_0}{3}+\tfrac{1}{9})e_1-\tfrac{(3\,\lambda_0+2)\sqrt{3}}{9}e_2+(\tfrac{1}{3}-\lambda_0)e_3, &
	\Psi(e_3) &= -\tfrac{2}{3}e_1+\tfrac{\sqrt{3}}{3}e_2+ e_3.
\end{align*}
establishes an isomorphism between the connection given in \eqref{Sol310,1} and the one associated with the flat Lie algebra $\h_{4,1}$ with $\lambda = -\tfrac{2\sqrt{3}}{9}$ and $\varepsilon_1 = 1$.

If $a_{31} - b_{32} = 0$, consider the following automorphism:
\begin{align*}
	\Psi(e_1) &= e_1 - e_2, &
	\Psi(e_2) &= e_2, &
	\Psi(e_3) &= x\, e_3.
\end{align*}
Applying it to the connection given in \eqref{Sol310}, for a suitable parameter $x \in \mathbb{R}^\ast$, yields the following equivalent connection:
\begin{equation}
	\begin{aligned}\label{Sol310,2}
		\nabla_{e_1} e_1 &=e_1, &\nabla_{e_1}e_2&=e_2, &
		\nabla_{e_1} e_3 &=e_3,&
		\nabla_{e_2} e_1 &=e_2, &\nabla_{e_2}e_2&=e_2+\delta_0 e_3,&
		\nabla_{e_3}e_1&=e_3,\quad\delta_0=0,1.
	\end{aligned}
\end{equation}
If $\delta = 0$, then this connection coincides with the one given in \eqref{Sol39,2} with $\varepsilon_0 = 0$. Otherwise, if $\delta = 1$, then the following automorphism:
\begin{align*}
	\Psi(e_1) &= e_1, &
	\Psi(e_2) &=\tfrac{7}{9}e_1+\tfrac{5\sqrt{3}}{9} e_2-\tfrac{2}{3}e_3, &
	\Psi(e_3) &=-\tfrac{2}{3}e_2 -\tfrac{\sqrt{3}}{3}e_2+ e_3.
\end{align*}
establishes an isomorphism between the connection given in \eqref{Sol310,2} under the previous assumption and the one associated with the flat Lie algebra $\h_{4,1}$ with $\lambda=\frac{2\sqrt{3}}{9}$ and $\varepsilon_1=1$.\\\\

\textbf{Case 5.} If $\nabla^0 \equiv \nabla^4$, which corresponds to the Lie algebra $\mathfrak{b}_4$
(see Table~\ref{FlatR2}), then the flatness equations associated with the connection
given in~\eqref{Connegeneral} can be solved directly, yielding eight distinct solutions.

The first solution is provided by the following flat, torsion-free connection:
\begin{equation}
	\begin{aligned}\label{Sol41}
		\nabla_{e_1} e_1 &=e_1+\tfrac{a_{32}(a_{32}b_{33}+b_{32})}{b_{32}b_{33}}e_3, &\nabla_{e_1}e_2&=\nabla_{e_2}e_1=e_2+a_{32}e_3, &
		\nabla_{e_1} e_3 &=\nabla_{e_3}e_1=\tfrac{a_{32}b_{33}+b_{32}}{b_{32}}e_3,\\
		\nabla_{e_2} e_2 &=b_{32}e_3, &\nabla_{e_2}e_3&=\nabla_{e_3}e_2=b_{33} e_3,&
		\nabla_{e_3}e_3&=\tfrac{b_{33}^2}{b_{32}}e_3.
	\end{aligned}
\end{equation}
Consider the following automorphism:
\begin{align*}
	\Psi(e_1) &=\tfrac{a_{32}b_{33}}{b_{32}} e_1+e_2+\tfrac{a_{32}b_{33}}{b_{32}}e_3, &
	\Psi(e_2) &= b_{33}e_3, &
	\Psi(e_3) &= \tfrac{b_{33}^2}{b_{32}}e_1 + \tfrac{b_{33}^2}{b_{32}}e_3.
\end{align*}
Applying it to the connection given in \eqref{Sol41} yields the following equivalent connection:
\begin{equation}
	\begin{aligned}\label{Sol41,1}
	\nabla_{e_1}e_2&=\nabla_{e_2}e_1=e_1, &
		\nabla_{e_2} e_2 &=e_2, &\nabla_{e_2}e_3&=\nabla_{e_3}e_2= e_3,&
		\nabla_{e_3}e_3&=e_1+e_3.
	\end{aligned}
\end{equation}
The following automorphism
\begin{align*}
	\Psi(e_1) &=-\tfrac{2}{3}e_1+\tfrac{\sqrt{3}}{3}e_2+e_3,&
	\Psi(e_2) &= e_1,&
	\Psi(e_3) &= \tfrac{7}{9}e_1-\tfrac{5\sqrt{3}}{9}e_2-\tfrac{2}{3}e_3,
\end{align*}
establishes an isomorphism between the connection given in \eqref{Sol41,1} and that associated with the flat Lie algebra defined in \eqref{Sol36,1,1} with $\varepsilon_1=1$ and $\lambda=-\frac{2\sqrt{3}}{9}$.

The second solution is provided by the following flat, torsion-free connection:
\begin{equation}
	\begin{aligned}\label{Sol42}
	\nabla_{e_1}e_1&=e_1,~~ \nabla_{e_1}e_2=\nabla_{e_2}e_1=e_2, ~~
		\nabla_{e_1} e_3 =\nabla_{e_3}e_1=e_3,\\ \nabla_{e_2}e_2&=b_{32} e_3,~~
		\nabla_{e_2}e_3=\nabla_{e_3}e_2=(b_{32}c_{23}-b_{23}b_{33})e_1+b_{23}e_2+b_{33}e_3,\\
		\nabla_{e_3}e_3&=-\tfrac{b_{33}(b_{23}b_{33}-b_{32}c_{23})}{b_{32}}e_1+c_{23}e_2+\tfrac{b_{32}b_{23}+b_{33}^2}{b_{32}}e_3.
	\end{aligned}
\end{equation}
Consider the following automorphism:
\begin{align*}
	\Psi(e_1) &=e_1, &
	\Psi(e_2) &=\tfrac{b_{33}}{3}e_1+x\,e_2, &
	\Psi(e_3) &= \tfrac{b_{33}^2}{9b_{32}}e_1+\tfrac{2b_{33}\,x}{3b_{32}} + \tfrac{x^2}{b_{32}}e_3.
\end{align*}
Applying it to the connection given in \eqref{Sol42}, for a suitable parameter $x\in\R^\ast$, yields the following equivalent connection:
\begin{equation}
	\begin{aligned}\label{Sol42,1}
			\nabla_{e_1} e_1 &= e_1,  &\nabla_{e_1} e_2&=e_2, &\nabla_{e_1} e_3&=e_3,\\
		\nabla_{e_2} e_1&=e_2, &\nabla_{e_2} e_2&= e_3, &\nabla_{e_2} e_3&=\lambda_2 e_1+\varepsilon_2e_2,\\ \nabla_{e_3} e_1&=e_3, &\nabla_{e_3}e_2&=\lambda_2 e_1+\varepsilon_2e_2, &\nabla_{e_3}e_3&=\lambda_2 e_2+\varepsilon_2e_3, \quad\lambda_2\in\R,~~\varepsilon_2=0,\pm1.
	\end{aligned}
	\end{equation}
If $\varepsilon_2=\pm 1$, then for all $\lambda_2 \in \mathbb{R}$, this connection coincides with the one associated with the flat Lie algebra $\mathfrak{h}_{4,1}$, where $\varepsilon_1=\varepsilon_2$ and $\lambda=\lambda_2$. If $\varepsilon_2=0$ and $\lambda_2=0$, consider the following automorphism:
\begin{align*}
	\Psi(e_1) &= e_1-e_2+e_3, &
	\Psi(e_2) &= e_2, &
	\Psi(e_3) &= e_1.
\end{align*}
This automorphism establishes an isomorphism between the connection given in \eqref{Sol42,1} and the one associated with the flat Lie algebra $\h_{2,1}$.

If $\varepsilon_2=0$ and $\lambda_2\neq0$, consider the following automorphism:
\begin{align*}
	\Psi(e_1) &= e_1, &
	\Psi(e_2) &=\lambda_2^{\frac{1}{3}} e_2, &
	\Psi(e_3) &=\lambda_2^{\frac{2}{3}}  e_3.
\end{align*}
Applying this automorphism to the connection given in \eqref{Sol42,1} yields the following equivalent connection:
\begin{equation}
	\begin{aligned}\label{Sol42,1,1}
		\nabla_{e_1} e_1 &= e_1,  &\nabla_{e_1} e_2&=e_2, &\nabla_{e_1} e_3&=e_3,\\
		\nabla_{e_2} e_1&=e_2, &\nabla_{e_2} e_2&= e_3, &\nabla_{e_2} e_3&= e_1,\\ \nabla_{e_3} e_1&=e_3, &\nabla_{e_3}e_2&= e_1, &\nabla_{e_3}e_3&=e_2.
	\end{aligned}
\end{equation}
In fact, the connection \eqref{Sol42,1,1} is isomorphic to the one associated with the flat Lie algebra $\h_{4,1}$, given in \eqref{Sol36,1,1} with $\lambda=0$ and $\varepsilon_1=-1$, via the following automorphism:
\begin{align*}
	\Psi(e_1) &= e_1, &
	\Psi(e_2) &= e_1-\tfrac{\sqrt{3}}{2} e_2+\tfrac{3}{2}e_3, &
	\Psi(e_3) &= e_1+\tfrac{\sqrt{3}}{2} e_2+\tfrac{3}{2}e_3.
\end{align*}

The third solution is given by the following flat, torsion-free connection:
\begin{equation}
	\begin{aligned}\label{Sol43}
		\nabla_{e_1}e_1&=e_1+a_{31}e_3, &\nabla_{e_1}e_2=\nabla_{e_2}e_1=e_2&, &\nabla_{e_1}e_3=\nabla_{e_3}e_1=&e_3,
		&\nabla_{e_2}e_2&=b_{32}e_3.
	\end{aligned}
\end{equation}
Consider the following automorphism:
\begin{align*}
	\Psi(e_1) &= a_{31}\,x\,e_1, &
	\Psi(e_2) &= e_3, &
	\Psi(e_3) &=x\, e_1.
\end{align*}
Applying this automorphism to the connection given in \eqref{Sol43}, for a suitable parameter $x\in\R^\ast$, yields the following equivalent connection:
\begin{equation}
	\begin{aligned}\label{Sol43,1}
		\nabla_{e_1}e_2&=e_1, &\nabla_{e_1}e_2=e_1&, &\nabla_{e_2}e_2=&e_2,
		&\nabla_{e_2}e_3&=e_3, &\nabla_{e_3}e_2&=e_3, &\nabla_{e_3}e_3&=\delta_2 e_1,\quad\delta=0,1.
	\end{aligned}
\end{equation}
 If $\delta=0$. Consider the following automorphism:
\begin{align*}
	\Psi(e_1) &= e_1, &
	\Psi(e_2) &= e_3, &
	\Psi(e_3) &= e_2.
\end{align*}
This automorphism establishes an isomorphism between the connection given in \eqref{Sol43,1} and the one associated with the flat Lie algebra $\h_{0,8}$.

If $\delta=1$.Then, the following automorphism:
\begin{align*}
	\Psi(e_1) &= e_1, &
	\Psi(e_2) &=e_1-e_2+ e_3, &
	\Psi(e_3) &= e_2
\end{align*}
establishes an isomorphism between the connection given in \eqref{Sol43,1} and the one associated with the flat Lie algebra $\h_{2,1}$.

The fourth solution is given by the following flat, torsion-free connection:
\begin{equation}
	\begin{aligned}\label{Sol44,2}
		\nabla_{e_1}e_1&=e_1+a_{31}e_3, &\nabla_{e_1}e_2&=\nabla_{e_2}e_1=e_2, &\nabla_{e_1}e_3&=\nabla_{e_3}e_1=\tfrac{a_{31}c_{23}}{a_{33}}e_2+a_{33}e_3, \\\nabla_{e_3}e_3&=c_{23}e_2+\tfrac{a_{33}(a_{33}-1)}{a_{31}}e_3.
	\end{aligned}
\end{equation}
If $a_{33}-1\neq0$, consider the following automorphism:
\begin{align*}
	\Psi(e_1) &= e_1-\tfrac{a_{31}^2c_{23}}{a_{33}(a_{33}-1)}e_2+(a_{33}-1)e_3, &
	\Psi(e_2) &= e_2, &
	\Psi(e_3) &= -\tfrac{c_{23}a_{31}}{a_{33}(a_{33}-1)}e_2+\tfrac{a_{33}(a_{33}-1)}{a_{31}}e_3.
\end{align*}
Applying this automorphism to the connection given in \eqref{Sol44,2} yields the following equivalent connection:
\begin{equation}
	\begin{aligned}\label{Sol44,2,1}
		\nabla_{e_1}e_1&=e_1, &\nabla_{e_1}e_2&=\nabla_{e_2}e_1=e_2, &\nabla_{e_1}e_3&=\nabla_{e_3}e_1=e_3, &\nabla_{e_3}e_3&=e_3.
	\end{aligned}
\end{equation}
In fact, this connection is isomorphic to the one associated with the flat Lie algebra $\h_{1,2}$ via the following automorphism:
\begin{align*}
	\Psi(e_1) &= e_1+e_3, &
	\Psi(e_2) &= e_2, &
	\Psi(e_3) &= e_1.
\end{align*}
If $a_{33}=1$, consider the following automorphism:
\begin{align*}
	\Psi(e_1) &= e_1-a_{31}^2c_{23}e_2+a_{31}e_3, &
	\Psi(e_2) &= x\,e_2, &
	\Psi(e_3) &= e_3.
\end{align*}
Applying this automorphism to the connection given in \eqref{Sol43}, for a suitable parameter $x\in\R^\ast$, yields the following equivalent connection:
\begin{equation}
	\begin{aligned}\label{Sol44,2,2}
		\nabla_{e_1}e_1&=e_1, &\nabla_{e_1}e_2&=\nabla_{e_2}e_1=e_2, &\nabla_{e_1}e_3&=\nabla_{e_3}e_1=e_3, &\nabla_{e_3}e_3&=\delta_2e_2,\quad\delta=0,1.
	\end{aligned}
\end{equation}
If $\delta_2=0$, then the following automorphism
\begin{align*}
	\Psi(e_1) &= e_3, &
	\Psi(e_2) &= e_2, &
	\Psi(e_3) &= e_1
\end{align*}
establishes an isomorphism between the connection given in \eqref{Sol44,2,2} and the one associated with the flat Lie algebra $\h_{0,8}$.

If $\delta_2=1$, then the following automorphism
\begin{align*}
	\Psi(e_1) &= e_1-e_2+e_3, &
	\Psi(e_2) &= e_1, &
	\Psi(e_3) &= e_2
\end{align*}
establishes an isomorphism between the connection given in \eqref{Sol44,2,2} and the one associated with the flat Lie algebra $\h_{2,1}$.

The fifth solution is given by the following flat, torsion-free connection:
\begin{equation}
	\begin{aligned}\label{Sol45}
		\nabla_{e_1}e_1&=e_1, ~~\nabla_{e_1}e_2=\nabla_{e_2}e_1=e_2, ~~\nabla_{e_1}e_3=\nabla_{e_3}e_1=e_3, ~~\nabla_{e_2}e_3=\nabla_{e_3}e_2=b_{23}e_3,\\
		\nabla_{e_3}e_3&=b_{23}(b_{23}-c_{33})e_1+c_{23}e_2+c_{33}e_3.
	\end{aligned}
\end{equation}
If $2\,b_{23}-c_{33}\neq0$, consider the following automorphism:
\begin{align*}
	\Psi(e_1) &= e_1, &
	\Psi(e_2) &= e_2, &
	\Psi(e_3) &=b_{23}e_2+\tfrac{c_{23}}{2\,b_{23}-c_{33}} e_2-(2\,b_{23}-c_{33})e_3.
\end{align*}
Applying this automorphism to the connection given in \eqref{Sol45} yields the following equivalent connection:
\begin{equation}
	\begin{aligned}\label{Sol45,1}
		\nabla_{e_1}e_1&=e_1, &
		\nabla_{e_1}e_2&=\nabla_{e_2}e_1=e_2, &
		\nabla_{e_1}e_3&=\nabla_{e_3}e_1=e_3, &
		\nabla_{e_3}e_3&=e_3.
	\end{aligned}
\end{equation}
Observe that this connection coincides with the one treated in \eqref{Sol44,2,1}; therefore, no further analysis is required.

If $2\,b_{23}-c_{33}=0$, consider the following automorphism:
\begin{align*}
	\Psi(e_1) &= e_1, &
	\Psi(e_2) &= x\,e_2, &
	\Psi(e_3) &=b_{23}e_2+e_3.
\end{align*}
Applying this automorphism to the connection given in \eqref{Sol45}, for a suitable parameter $x\in\R^\ast$, yields the following equivalent connection:
\begin{equation}
	\begin{aligned}\label{Sol45,2}
		\nabla_{e_1}e_1&=e_1, &
		\nabla_{e_1}e_2&=\nabla_{e_2}e_1=e_2, &
		\nabla_{e_1}e_3&=\nabla_{e_3}e_1=e_3, &
		\nabla_{e_3}e_3&=e_3.
	\end{aligned}
\end{equation}
Observe that this connection coincides with the one treated in \eqref{Sol44,2,2}; therefore, no further analysis is required.

The sixth solution is given by the following flat, torsion-free connection:
\begin{equation}
	\begin{aligned}\label{Sol46}
		\nabla_{e_1}e_1&=e_1, ~~\nabla_{e_1}e_2=\nabla_{e_2}e_1=e_2, ~~\nabla_{e_1}e_3=\nabla_{e_3}e_1=b_{23}e_1+a_{23}e_2,  ~~\nabla_{e_2}e_3=\nabla_{e_3}e_2=b_{23}e_3,\\
		\nabla_{e_3}e_3&=b_{23}(b_{23}-c_{33})e_1+a_{23}(2\,b_{23}-c_{33})e_2+c_{33}e_3.
	\end{aligned}
\end{equation}
Consider the following automorphism:
\begin{align*}
	\Psi(e_1) &= e_1, &
	\Psi(e_2) &= e_2, &
	\Psi(e_3) &= b_{23}e_2+a_{23}e_2+x\,e_3.
\end{align*}
Applying this automorphism to the connection given in \eqref{Sol46}, for a suitable parameter $x\in\R^\ast$, yields the following equivalent connection:
\begin{equation}
	\begin{aligned}\label{Sol46,1}
		\nabla_{e_1}e_1&=e_1, &\nabla_{e_1}e_2&=\nabla_{e_2}e_1=e_2, 
		&\nabla_{e_3}e_3&=\delta_3e_3,\quad\delta_3=0,1.
	\end{aligned}
\end{equation}
If $\delta_3=0$, then this connection is isomorphic to the one associated with the flat Lie algebra $\h_{0,3}$ via the following automorphism:
\begin{align*}
	\Psi(e_1) &= e_3, &
	\Psi(e_2) &= e_2, &
	\Psi(e_3) &= e_1.
\end{align*}
Otherwise, if $\delta_3=1$, then it is isomorphic to the one associated with the flat Lie algebra $\h_{1,2}$ via the previous automorphism.

The seventh solution is given by the following flat, torsion-free connection:
\begin{equation}
	\begin{aligned}\label{Sol47}
		\nabla_{e_1}e_1&=e_1+a_{31}e_3, &\nabla_{e_1}e_2&=\nabla_{e_2}e_1=e_2+a_{32}e_3, 
	\end{aligned}
\end{equation}
Consider the following automorphism:
\begin{align*}
	\Psi(e_1) &=-a_{31} e_1+e_3, &
	\Psi(e_2) &= e_2, &
	\Psi(e_3) &= x\,e_1.
\end{align*}
Applying this automorphism to the connection given in \eqref{Sol47}, for a suitable parameter $x\in\R^\ast$, yields the following equivalent connection:
\begin{equation}
	\begin{aligned}\label{Sol47,1}
		\nabla_{e_2}e_3&=\nabla_{e_3}e_1=\delta_3e_1+e_2, &\nabla_{e_3}e_3&=e_3, \quad\delta_3=0,1.
	\end{aligned}
\end{equation}
The upper matrix $\nabla^0$ associated with this connection coincides with the vanishing connection $\nabla \equiv 0$. Therefore, this case has already been treated in \textbf{Case 1}, and no further analysis is required.

The last solution in this case is given by the following flat, torsion-free connection:
\begin{equation}
	\begin{aligned}\label{Sol48}
		\nabla_{e_1}e_1&=e_1+a_{31}e_3, &\nabla_{e_1}e_2&=\nabla_{e_2}e_1=e_2, &\nabla_{e_1}e_3&=\nabla_{e_3}e_1=a_{23}e_2.
	\end{aligned}
\end{equation}
Consider the following automorphism:
\begin{align*}
	\Psi(e_1) &=-a_{31} e_1+e_3, &
	\Psi(e_2) &= e_2, &
	\Psi(e_3) &= x\,e_1.
\end{align*}
Applying this automorphism to the connection given in \eqref{Sol48} yields the following equivalent connection:
\begin{equation}
	\begin{aligned}\label{Sol48,1}
		\nabla_{e_2}e_3&=\nabla_{e_3}e_1=e_2, &\nabla_{e_3}e_3&=e_3.
	\end{aligned}
\end{equation}
Observe that this connection coincides with the one given previously in \eqref{Sol47,1} with $\delta_3=0$; hence, no further analysis is required.

\textbf{Case 6.} If $\nabla^0 \equiv \nabla^5$, which corresponds to the Lie algebra $\mathfrak{b}_5$
(see Table~\ref{FlatR2}), The flatness equations associated with the connection given in \eqref{Connegeneral} can be solved directly, yielding six distinct solutions.

The first solution is determined by the following flat torsion-free connection:
\begin{equation}
	\begin{aligned}\label{Sol51}
		\nabla_{e_1}e_1&=e_1+(b_{32}^2c_{33}-b_{32})e_3, &\nabla_{e_1}e_2&=\nabla_{e_2}e_1=e_2, &\nabla_{e_1}e_3&=\nabla_{e_3}e_1=b_{32}c_{33}e_3, \\\nabla_{e_2}e_2&=-e_1+b_{32}e_3, &\nabla_{e_3}e_3&=c_{33}e_3. 
	\end{aligned}
\end{equation}
Consider the following automorphism:
\begin{align*}
	\Psi(e_1) &= b_{32}\,x\, e_1 + e_3, &
	\Psi(e_2) &= e_2, &
	\Psi(e_3) &= x\,e_1.
\end{align*}
Applying this automorphism to the connection given in~\eqref{Sol51}, for a suitable parameter $x \in \mathbb{R}^\ast$, yields the following equivalent connection:
\begin{equation}
	\begin{aligned}\label{Sol51,1}
		\nabla_{e_1}e_1&=\delta e_1,  &\nabla_{e_2}e_2&=-e_3, &\nabla_{e_2}e_3&=\nabla_{e_3}e_2=e_2, &\nabla_{e_3}e_3&=e_3,\quad\delta=0,1. 
	\end{aligned}
\end{equation}
Note that if $\delta = 0$, then the upper matrix associated with this connection coincides with that of \textbf{Case~1}, and hence has already been treated. If $\delta = 1$, by the same reasoning, we are reduced to \textbf{Case~2}; therefore, no further analysis is required.

The second solution is determined by the following flat torsion-free connection:
\begin{equation}
	\begin{aligned}\label{Sol52}
		\nabla_{e_1}e_1&= e_1,  ~~\nabla_{e_1}e_2=\nabla_{e_2}e_1=e_2, ~~\nabla_{e_1}e_3=\nabla_{e_3}e_1=b_{23}e_1, ~~\nabla_{e_2}e_2=-e_1, ~~\nabla_{e_2}e_3=\nabla_{e_3}e_2=b_{23}e_2, \\\nabla_{e_3}e_3&=b_{23}(b_{23}-c_{33})e_1+c_{33}e_3. 
	\end{aligned}
\end{equation}
Consider the following automorphism:
\begin{align*}
	\Psi(e_1) &= e_3, &
	\Psi(e_2) &= e_2, &
	\Psi(e_3) &= x\,e_1+(b_{23}-c_{33})e_3,
\end{align*}
where $x\in\R^\ast$ is a suitable parameter. Applying this automorphism to the connection given in~\eqref{Sol52} yields the following equivalent connection
\begin{equation}
	\begin{aligned}\label{Sol52,1}
		\nabla_{e_1}e_1&=\delta e_1,  ~~\nabla_{e_1}e_2=\nabla_{e_2}e_1=\delta e_2, ~~\nabla_{e_1}e_3=\nabla_{e_3}e_1=\delta e_3, ~~\nabla_{e_2}e_2=-e_1, ~~\nabla_{e_2}e_3=\nabla_{e_3}e_2=-e_3, \\\nabla_{e_3}e_3&=e_3,\quad \delta=0,1. 
	\end{aligned}
\end{equation}
Note that if $\delta=0$, then the upper block of this matrix coincides with \textbf{Case 1}, where $\nabla^0\equiv 0$. If $\delta=1$, then this case also coincides with \textbf{Case 4}. Therefore, no further analysis is required.

The third solution is given by the following flat torsion-free connection:
\begin{equation}
	\begin{aligned}\label{Sol53}
		\nabla_{e_1}e_1&= e_1,  ~~\nabla_{e_1}e_2=\nabla_{e_2}e_1=e_2, ~~\nabla_{e_1}e_3=\nabla_{e_3}e_1=\tfrac{b_{13}c_{33}-c_{23}}{2\, b_{13}} e_1-b_{13}e_2, ~~\nabla_{e_2}e_2=-e_1, \\\nabla_{e_2}e_3&=\nabla_{e_3}e_2=b_{13}e_1+\tfrac{b_{13}c_{33}-c_{23}}{2\,b_{13}}e_2, ~~\nabla_{e_3}e_3=\tfrac{c_{23}^2-b_{13}^2c_{33}^2-4\,b_{13}^4}{4\,b_{13}^2}e_1+c_{23}e_2+c_{33}e_3.
	\end{aligned}
\end{equation}
Consider the following automorphism:
\begin{align*}
	\Psi(e_1) &= e_1, &
	\Psi(e_2) &= e_2, &
	\Psi(e_3) &=\tfrac{ b_{13}c_{33}-c_{23}}{2\,b_{13}}e_1-b_{13}e_2-b_{13}e_3,
\end{align*}
 Applying this automorphism to the connection given in~\eqref{Sol52} yields the following equivalent connection
\begin{equation}
	\begin{aligned}\label{Sol53,1}
		\nabla_{e_1}e_1&= e_1,  &\nabla_{e_1}e_2&=\nabla_{e_2}e_1= e_2, &\nabla_{e_2}e_2&=-e_1,  &\nabla_{e_3}e_3&=\lambda_5 e_3,\quad \lambda\in\R. 
	\end{aligned}
\end{equation}
If $\lambda_5\neq 0$, then the following automorphism
\begin{align*}
	\Psi(e_1) &=-e_3, &
	\Psi(e_2) &= e_2, &
	\Psi(e_3) &= \lambda_5 e_1+\lambda_5 e_3,
\end{align*}
establishes an isomorphism between the connection given in~\eqref{Sol53,1} and the one associated with the flat Lie algebra $\h_{4,1}$ for $\lambda=0$, $\varepsilon=-1$, and $\lambda_5\neq 0$. Otherwise, if $\lambda_5=0$, then  the following automorphism
\begin{align*}
	\Psi(e_1) &=e_3, &
	\Psi(e_2) &= e_2, &
	\Psi(e_3) &= e_1,
\end{align*}
establishes an isomorphism between the connection given in~\eqref{Sol53,1} and the one associated with the flat Lie algebra $\h_{0,5}$ for  $\varepsilon=-1$.

The fourth solution is given by the following flat torsion-free connection:
\begin{equation}
	\begin{aligned}\label{Sol54}
		\nabla_{e_1}e_1&= e_1,  ~~\nabla_{e_1}e_2=\nabla_{e_2}e_1=e_2, ~~\nabla_{e_1}e_3=\nabla_{e_3}e_1=e_3, ~~\nabla_{e_2}e_2=-e_1+b_{32}e_3, \\\nabla_{e_2}e_3&=\nabla_{e_3}e_2=b_{13}e_1+b_{23}e_2+b_{33}e_3, ~~\nabla_{e_3}e_3=\tfrac{b_{13}b_{33}-b_{23}}{b_{32}}e_1+\tfrac{b_{23}b_{33}+b_{13}}{b_{32}} e_2+\tfrac{b_{23}b_{32}+b_{33}^2+1}{b_{32}} e_3.
	\end{aligned}
\end{equation}
Consider the following automorphism:
\begin{align*}
	\Psi(e_1) &= e_1, &
	\Psi(e_2) &=\tfrac{b_{33}}{3}e_1+x\, e_2, &
	\Psi(e_3) &=\tfrac{ b_{33}^2+9}{9\,b_{32}}e_1+\tfrac{2\,b_{33}\,x}{3\, b_{32}} e_2+\tfrac{x^2}{b_{32}} e_3,
\end{align*}
Applying this automorphism to the connection given in~\eqref{Sol52} for a suitable parameter $x\in\R^\ast$, yields the following equivalent connection
\begin{equation}
	\begin{aligned}\label{Sol54,1}
		\nabla_{e_1} e_1 &= e_1,  &\nabla_{e_1} e_2&=e_2, &\nabla_{e_1} e_3&=e_3,\\
		\nabla_{e_2} e_1&=e_2, &\nabla_{e_2} e_2&= e_3, &\nabla_{e_2} e_3&=\lambda_6 e_1+\delta e_2,\\ \nabla_{e_3} e_1&=e_3, &\nabla_{e_3}e_2&=\lambda_6 e_1+\delta e_2, &\nabla_{e_3}e_3&=\lambda_6 e_2+\delta e_3, \quad\lambda_6\in\R,~~\delta=0,\pm1.
	\end{aligned}
\end{equation}
In fact, this connection coincides with the one given in~\eqref{Sol36,1,1}.

The fifth solution is given by the following flat torsion-free connection:
\begin{equation}
	\begin{aligned}\label{Sol55}
		\nabla_{e_1}e_1&= e_1+\tfrac{(b_{32}c_{33}-b_{33}^2-1)(b_{32}c_{33}-b_{33}^2)}{c_{33}}e_3,  ~~\nabla_{e_1}e_2=\nabla_{e_2}e_1=e_2+\tfrac{b_{33}(b_{32}c_{33}-b_{33}^2-1)}{c_{33}}e_3, \\\nabla_{e_1}e_3&=\nabla_{e_3}e_1=(b_{32}c_{33}-b_{33}^2)e_3, ~~\nabla_{e_2}e_2=-e_1+b_{32}e_3, \\\nabla_{e_2}e_3&=\nabla_{e_3}e_2=b_{33}e_3, ~~\nabla_{e_3}e_3=c_{33} e_3.
	\end{aligned}
\end{equation}
If $b_{32}c_{33}-b_{33}^2\neq0$. Consider the following automorphism
\begin{align*}
	\Psi(e_1) &= e_1, &
	\Psi(e_2) &=e_2+\tfrac{b_{33}}{b_{32}c_{33}-b_{33}^2}e_3, &
	\Psi(e_3) &=\tfrac{c_{33}}{b_{32}c_{33}-b_{33}^2} e_3,
\end{align*}
Applying this automorphism to the connection given in~\eqref{Sol55} for a suitable parameter $x\in\R^\ast$, yields the following equivalent connection
\begin{equation}
	\begin{aligned}\label{Sol55,1}
		\nabla_{e_1}e_1&= e_1+(\lambda_7-1)e_3, &\nabla_{e_1}e_2&=\nabla_{e_2}e_1=e_2, &\nabla_{e_1}e_3&=\nabla_{e_3}e_1=\lambda_7 e_3, &\nabla_{e_2}e_2&=-e_1+e_3, &\nabla_{e_3}e_3&=\lambda_7 e_3,
	\end{aligned}
\end{equation}
where, $\lambda_7=b_{32}c_{33}-b_{33}^2\in\R^\ast$. 

Note that the connection~\eqref{Sol55,1} is isomorphic to the one associated with the flat Lie algebra $\h_{4,1}$ via the following automorphism:
\begin{align*}
	\Psi(e_1) &=\lambda_7 e_1+(\lambda_7-1)e_3, &
	\Psi(e_2) &=-e_2, &
	\Psi(e_3) &=\lambda_7 e_1+\lambda_7 e_3.
\end{align*}

 If $b_{32}c_{33}-b_{33}^2=0$. Consider the following automorphism
 \begin{align*}
 	\Psi(e_1) &= e_1, &
 	\Psi(e_2) &=e_2+b_{33}e_3, &
 	\Psi(e_3) &=c_{33} e_3,
 \end{align*}
 Applying this automorphism to the connection given in~\eqref{Sol55}, yields the following equivalent connection
 \begin{equation}
 	\begin{aligned}\label{Sol55,2}
 		\nabla_{e_1}e_1&= e_1, &\nabla_{e_1}e_2&=\nabla_{e_2}e_1=e_2, &\nabla_{e_2}e_2&=-e_1, &\nabla_{e_3}e_3&= e_3,
 	\end{aligned}
 \end{equation}
 The connection~\eqref{Sol55,2} is in fact isomorphic to the one associated with the flat Lie algebra $\h_{4,1}$ for $\lambda=0$ and $\varepsilon=-1$ via the following automorphism:
 \begin{align*}
 	\Psi(e_1) &=-e_3, &
 	\Psi(e_2) &=e_2, &
 	\Psi(e_3) &=e_1+ e_3.
 \end{align*}
 The sixth solution is given by the following flat, torsion-free connection:
 \begin{equation}
 	\begin{aligned}\label{Sol56}
 		\nabla_{e_1}e_1&= e_1-b_{32}e_3, &\nabla_{e_1}e_2&=\nabla_{e_2}e_1=e_2+a_{32}e_3, &\nabla_{e_2}e_2&=-e_1+b_{32}e_3.
 	\end{aligned}
 \end{equation}
 Consider the following automorphism
 \begin{align*}
 	\Psi(e_1) &=b_{32}e_1+e_2, &
 	\Psi(e_2) &=-a_{32}e_1+e_3, &
 	\Psi(e_3) &=e_1.
 \end{align*}
 Applying it to the connection given in \eqref{Sol56} yields the following equivalent connection:
 \begin{equation}
 	\begin{aligned}\label{Sol56,1}
 	\nabla_{e_2}e_2&=e_2, &\nabla_{e_2}e_3&=\nabla_{e_3}e_2=e_3, &\nabla_{e_3}e_3&=-e_2.
 	\end{aligned}
 \end{equation}
 This connection \eqref{Sol56,1} is in fact isomorphic to the one associated with the flat Lie algebra $\mathfrak{h}_{0,5}$ with $\varepsilon = -1$, via the following automorphism:
 \begin{align*}
 	\Psi(e_1) &=e_1, &
 	\Psi(e_2) &=e_3, &
 	\Psi(e_3) &=e_2.
 \end{align*}

\end{document}